\documentclass[reqno]{amsart}

\usepackage{a4,amsmath,amssymb,amsthm,slashed}
\usepackage[top=1in, bottom=1in, left=1.25in, right=1.25in]{geometry}
\usepackage{xfrac}
\usepackage{faktor}
\usepackage{latexsym}
\usepackage{mathrsfs}
\usepackage[numbers]{natbib}
\usepackage{leftidx}
\usepackage{inputenc}

\usepackage{scrextend}

\usepackage{calc}
\usepackage{enumitem}
\usepackage{amsaddr}

\usepackage[T1]{fontenc}

\usepackage{leftidx}
\usepackage{tikz}

\usepackage{hyperref}
\newcounter{mnotecount}[section]

\usepackage{color}

\usepackage{url}

\numberwithin{equation}{section}

\usepackage[mathscr]{eucal}

\usepackage{graphicx}
\graphicspath{pic/}

\newcommand{\sfrak}{\mathfrak{s}}

\newcommand{\B}{\mathbf}
\newcommand{\half}{\frac{1}{2}}

\newcommand{\pt}{\partial}
\newcommand{\veps}{\varepsilon}

\newcommand{\prb}{\partial_{\rb}}
\newcommand{\di}{\mathrm{d}} 

\newcommand{\Donetwo}{\DOC_{\tb_1,\tb_2}}

\newcommand{\NPR}{{\Upsilon}}
\newcommand{\NPRplus}{\NPR_{+1}}
\newcommand{\NPRminus}{\NPR_{-1}}
\newcommand{\NPRpluss}{\NPR_{+\sfrak}}
\newcommand{\NPRminuss}{\NPR_{-\sfrak}}
\newcommand{\NPRzero}{\NPR_{0,\text{mid}}}

\newcommand{\psiplus}{\psi_{+1}}
\newcommand{\psiminus}{\psi_{-1}}

\newcommand{\Psiplus}{\Psi_{+1}}
\newcommand{\Psiminus}{\Psi_{-1}}

\newcommand{\psipluss}{\psi_{+\sfrak}}
\newcommand{\psiminuss}{\psi_{-\sfrak}}

\newcommand{\Psipluss}{\Psi_{+\sfrak}}
\newcommand{\Psiminuss}{\Psi_{-\sfrak}}
\newcommand{\Psipms}{\Psi_{s}}

\newcommand{\Phiplus}{\Phi_{+1}^{(0)}}
\providecommand{\Phiminus}[1]{\Phi_{-1}^{(#1)}}

\providecommand{\tildePhiminus}[1]{\widetilde{\Phi}_{-1}^{(#1)}}
\providecommand{\tildePhiplusHigh}[1]{\widetilde{\Phi}_{+1}^{(#1)}}

\providecommand{\Phiminuss}[1]{\Phi_{-\sfrak}^{(#1)}}

\providecommand{\PhiplussHigh}[1]{\Phi_{+\sfrak}^{(#1)}}
\providecommand{\tildePhiminuss}[1]{\widetilde{\Phi}_{-\sfrak}^{(#1)}}
\providecommand{\tildePhiplussHigh}[1]{\widetilde{\Phi}_{+\sfrak}^{(#1)}}

\providecommand{\InizeroEnergypluss}[4]{\mathbf{E}_{\Sigma_{#1}}^{#2,#3,#4}[\Psipluss]}
\providecommand{\InizeroEnergyminuss}[4]{\mathbf{E}_{\Sigma_{#1}}^{#2,#3,#4}[\Psiminuss]}

\providecommand{\InizeroEnergyplussell}[4]{\mathbf{E}_{\Sigma_{#1}}^{#2,#3,#4}[(\Psipluss)_{(m,\ell)}]}
\providecommand{\InizeroEnergyplussellTI}[4]{\mathbf{E}_{\Sigma_{#1}}^{#2,#3,#4}[\Psi_{\gplus}]}

\providecommand{\InizeroEnergyplussnv}[2]{\mathbf{E}_{\text{NV},\Sigmazero}^{\ell\geq\ell_0,#1, #2}[\Psipluss]}
\providecommand{\InizeroEnergyplussv}[2]{\mathbf{E}_{\text{V},\Sigmazero}^{\ell\geq\ell_0,#1, #2}[\Psipluss]}

\providecommand{\InizeroEnergyminussnv}[2]{\mathbf{E}_{\text{NV},\Sigmazero}^{\ell\geq\ell_0,#1, #2}[\Psiminuss]}
\providecommand{\InizeroEnergyminussv}[2]{\mathbf{E}_{\text{V},\Sigmazero}^{\ell\geq\ell_0,#1, #2}[\Psiminuss]}

\providecommand{\InizeroEnergyplusshigh}[5]{\mathbf{E}_{\Sigma_{#1}}^{#2,#3,#4}[#5]}

\providecommand{\NPCP}[1]{\mathbb{Q}_{+1}^{(m,#1)}}

\providecommand{\NPCIT}[1]{\mathbb{Q}_{g_+}^{(m,#1)}}

\providecommand{\NPCPs}[1]{\mathbb{Q}_{+\sfrak}^{(m,#1)}}
\providecommand{\NPCNs}[1]{\mathbb{Q}_{-\sfrak}^{(m,#1)}}
\providecommand{\NPCITs}[1]{\mathbb{Q}_{\gplus}^{(m,#1)}}

\newcommand{\NPRplustwo}{\NPR_{+2}}
\newcommand{\NPRminustwo}{\NPR_{-2}}

\newcommand{\psiplustwo}{\psi_{+2}}
\newcommand{\psiminustwo}{\psi_{-2}}

\newcommand{\Psiminustwo}{\Psi_{-2}}

\newcommand{\Phiplustwo}{\Phi_{+2}^{(0)}}
\providecommand{\Phiminustwo}[1]{\Phi_{-2}^{(#1)}}

\providecommand{\tildePhiminustwo}[1]{\widetilde{\Phi}_{-2}^{(#1)}}
\providecommand{\tildePhiplustwoHigh}[1]{\widetilde{\Phi}_{+2}^{(#1)}}

\newcommand{\Lxi}{\partial_{\tb}}

\newcommand{\pu}{\partial_u}
\newcommand{\pv}{\partial_v}

\newcommand{\DOC}{\mathcal{D}}
\newcommand{\tb}{\tau}
\newcommand{\pb}{{\phi}}
\newcommand{\rb}{\rho}
\newcommand{\Hyper}{\Sigma}

\newcommand{\Sigmazero}{\Hyper_{\tb_0}}
\newcommand{\Sigmatb}{\Hyper_{\tb}}
\newcommand{\Sigmatwo}{\Hyper_{\tb_2}}
\newcommand{\Sigmaone}{\Hyper_{\tb_1}}

\newcommand{\Horizon}{\mathcal{H}^+}
\newcommand{\Scri}{\mathcal{I}^+}
\newcommand{\Horizononetwo}{\Horizon_{\tb_1,\tb_2}}
\newcommand{\Scrionetwo}{\Scri_{\tb_1,\tb_2}}

\newcommand{\CDeri}{\mathbb{D}}
\def\CDeris{\slashed{\mathbb{D} }}

\newcommand{\PDeri}{\mathbb{B}}

\newcommand{\PriceDeri}{\mathbb{P}}

\newcommand{\VR}{\hat{V}}
\newcommand{\curlVR}{\hat{\mathcal{V}}}
\newcommand{\edthR}{\mathring{\eth}}

\newcommand{\Sphere}{S^2}

\providecommand{\abs}[1]{\lvert#1\rvert}
\providecommand{\norm}[1]{\lVert#1\rVert}

\providecommand{\absHighOrder}[3]{\abs{#1}_{#2,#3}}

\providecommand{\absCDeri}[2]{\absHighOrder{#1}{#2}{\CDeri}}
\providecommand{\absCDeris}[2]{\absHighOrder{#1}{#2}{\slashed{\mathbb{ D}}}}

\makeatletter
  \def\moverlay{\mathpalette\mov@rlay}
  \def\mov@rlay#1#2{\leavevmode\vtop{%
     \baselineskip\z@skip \lineskiplimit-\maxdimen
     \ialign{\hfil$#1##$\hfil\cr#2\crcr}}}
\makeatother

\newcommand{\reg}{k}

\newcommand{\regl}{k'}

\newcommand{\FBT}{{\mathbb{I}_{\textup{NV},\ell}}}
\newcommand{\FBTV}{{\mathbb{I}_{\textup{V},\ell}}}

\newcommand{\FBNV}{\mathbb{I}_{\textup{NV}}}
\newcommand{\FBV}{\mathbb{I}_{\textup{V}}}

\newcommand{\hell}{h_{\sfrak,\ell}}
\newcommand{\hellone}{h_{1,\ell}}

\newcommand{\hellt}{\tilde{h}_{+\sfrak,\ell}}
\newcommand{\hellh}[1]{h_{\sfrak,\ell}^{(#1)}}
\newcommand{\hellz}{h_{\sfrak,\ell_0}}
\newcommand{\hellzp}{h_{\sfrak,\ell_0+1}}

\newcommand{\varphisell}{\varphi_{s,\ell}}
\newcommand{\varphiminussell}{\varphi_{-\sfrak,\ell}}
\newcommand{\varphiplussell}{\varphi_{+\sfrak,\ell}}
\newcommand{\hatvarphisell}{\hat\varphi_{s,\ell}}
\newcommand{\hatvarphiplussell}{\hat\varphi_{+\sfrak,\ell}}
\newcommand{\hatvarphiell}{\hat\varphi_{-\sfrak,\ell}}
\newcommand{\Hvarphisell}{H_{\varphi_{s,\ell}}}
\newcommand{\Hhatvarphisell}{H_{\hat\varphi_{s,\ell}}}
\newcommand{\Hhatvarphiell}{H_{\hat\varphi_{-\sfrak,\ell}}}

\newcommand{\Phiplussell}{\Phi_{+\sfrak,\ell}^{(0)}}

\newcommand{\gplus}{(g_{+\sfrak})_{m,\ell}}
\newcommand{\PhiplussHighTI}[1]{\Phi_{\gplus}^{(#1)}}
\newcommand{\tildePhiplussHighTI}[1]{\widetilde{\Phi}_{\gplus}^{(#1)}}

\newcommand{\Hpsi}{\mathcal{H}[\psi_{+\sfrak,\ell}]}
\newcommand{\Ppsiell}{\psi_{+\sfrak,\ell}}
\newcommand{\Phell}{h_{\sfrak,\ell}}
\newcommand{\Pthell}{\tilde{h}_{\sfrak,\ell}}
\newcommand{\Pgell}{g_{+\sfrak,\ell}}
\newcommand{\gsf}[1]{\mathfrak{g}_{+\sfrak,\ell}^{(#1)}}
\newcommand{\csf}{\mathfrak{c}_{+\sfrak,\ell}}
\newcommand{\asfl}[1]{\mathfrak{a}_{\ell-\sfrak,#1}}

\newcommand{\cfz}{\mathfrak{c}_{+0,\ell}}
\newcommand{\cfo}{\mathfrak{c}_{+1,\ell}}
\newcommand{\cft}{\mathfrak{c}_{+2,\ell}}

\newcommand{\hhyp}{h_{\text{hyp}}}
\newcommand{\Hhyp}{H_{\text{hyp}}}

\theoremstyle{plain}
\newtheorem{thm}{Theorem}[section]
\newtheorem{cor}[thm]{Corollary}
\newtheorem{lemma}[thm]{Lemma}
\newtheorem{prop}[thm]{Proposition}

\theoremstyle{definition}
\newtheorem{definition}[thm]{Definition}
\newtheorem{assump}[thm]{Assumption}
\newtheorem{remark}[thm]{Remark}

\title{Price's law for spin fields on a Schwarzschild background}

\author[S. Ma and L.Zhang]{Siyuan Ma$^\dagger$ and Lin Zhang$^\star$}
\email{$^\dagger$siyuan.ma@sorbonne-universite.fr, $^\star$linzhang2013@pku.edu.cn}
\address{$^\dagger$Laboratoire Jacques-Louis Lions,
Sorbonne Universit\'{e} Campus Jussieu,
4 place Jussieu 75005 Paris, France\\
and\\
Max Planck Institute for Gravitational Physics, Am M\"uhlenberg 1, 14476, Potsdam, Germany\\
$^\star$College of Mathematics and Statistics, Chongqing University, Chongqing 401331, China.}

\begin{document}

\allowdisplaybreaks

\begin{abstract}
In this work, we derive the globally precise late-time asymptotics for the spin-$\mathfrak{s}$ fields on a Schwarzschild background, including the scalar field $(\mathfrak{s}=0)$, the Maxwell field $(\mathfrak{s}=\pm 1)$ and the linearized gravity $(\mathfrak{s}=\pm 2)$. The conjectured Price's law in the physics literature which predicts the sharp rates of decay of the spin $s=\pm \mathfrak{s}$ components towards the future null infinity as well as in a compact region is shown. Further, we confirm the heuristic claim by Barack and Ori that the spin $+1, +2$ components have an extra power of decay at the event horizon than the conjectured Price's law. The asymptotics are derived via a unified, detailed analysis of the Teukolsky master equation that is satisfied by all these components.
\end{abstract}

\maketitle

\tableofcontents


\section{Introduction}

The metric of a Schwarzschild black hole spacetime \cite{schw1916}
takes the form of
\begin{align}\label{eq:SchwMetricinHHtetrad}
g_{ab}= & -2l_{(a}n_{b)}+2m_{(a}\bar{m}_{b)},
\end{align}
where $(l^a,n^a,m^a,\bar{m}^a)$ is a Hawking--Hartle null tetrad \cite{HHtetrad72} and reads in the Boyer-Lindquist coordinates $(t,r,\theta,\phi)$ \cite{boyer:lindquist:1967}
\begin{align}\label{eq:HHtetrad}
l^a = \tfrac{1}{\sqrt{2}}(1 , \mu , 0 ,0), \,\,
n^a = \tfrac{1}{\sqrt{2}} (\tfrac{1}{\mu} , - 1 , 0 , 0), \,\,m^a = \tfrac{1}{\sqrt{2} r}\left(0,0 , 1, \tfrac{i}{\sin{\theta}}\right),\,\,
\bar{m}^{a}= \tfrac{1}{\sqrt{2} r}\left(0,0 , 1, \tfrac{-i}{\sin{\theta}}\right).
\end{align}
Here, the function
$\mu=\mu(r,M)=\Delta r^{-2}$ with $\Delta=\Delta(r,M)= r^2 -2Mr$ and $M$ is the mass of the black hole. The larger root $r=2M$ of the function $\Delta$ is the location of the event horizon $\mathcal{H}$, and we define the domain of outer communication (DOC), denoted as $\DOC$, of a Schwarzschild black hole spacetime to be the closure of $\{(t,r,\theta,\phi)\in \mathbb{R}\times (2M,\infty)\times {S}^2\}$ in the Kruskal maximal extension.
We consider in this work only the future Cauchy problem and denote the future event horizon and the future null infinity as $\Horizon$ and $\Scri$, respectively. Denote $v$ and $u$ the forward and retarded time, respectively, and define $\tb$ a hyperboloidal time function such that the level sets of the time function are spacelike hypersurfaces,\footnote{In fact,  the estimates in our main theorem \ref{thm:main:intro} are still valid if one portion or the whole of the level sets of $\tb$ is null.} cross $\Horizon$ regularly, and are asymptotic to $\Scri$ for large $r$. We call the coordinate system $(\tb, \rb=r,\theta,\pb)$ as the hyperboloidal coordinate system and denote the level sets of $\tb$ as $\Sigmatb$. See Section \ref{sect:foliation}.


\subsection{Heuristic Price's law}

This work is for the spin-$\sfrak$ fields, which correspond to the scalar field, the Maxwell field and the linearized gravity in the case of $\sfrak=0, 1$ and $2$, respectively.
The scalar field, denoted as $\NPR_0$, solves the scalar wave equation $\Box_{g} \NPR_0=0$.
The Maxwell field, a real two-form $\mathbf{F}_{\alpha\beta}$,  satisfies the Maxwell equations
\begin{align}\label{eq:MaxwellEqs}
\nabla^{\alpha}\mathbf{F}_{\alpha\beta}&=0, & \nabla_{[\gamma}\mathbf{F}_{\alpha\beta]}&=0.
\end{align}
In the Newman--Penrose (N--P) formalism \cite{newmanpenrose62,newmanpenrose63errata}, one projects the Maxwell field onto the Hawking--Hartle tetrad \eqref{eq:HHtetrad}  and obtains the N--P components of the Maxwell field
\begin{equation}\label{eq:MaxwellNPcomponentswithnosuperscript}
\NPRplus = \mathbf{F}_{\mu\nu} l^\mu m^\nu , \quad \NPRzero= \mathbf{F}_{\mu\nu}(l^{\mu}n^{\nu}+\bar{m}^{\mu}m^{\nu}),\quad  \NPRminus = \mathbf{F}_{\mu\nu} \bar{m}^\mu n^\nu.
\end{equation}
The lower index of each N--P component indicates its spin weight,  and we call the N--P components $\NPRplus$ and $\NPRminus$ as the spin $+1$ and $-1$ components of the Maxwell field, respectively, and $\NPRzero$ as the middle component. In the same way, we define the N--P components of the linearized gravity by projecting the linearized Weyl tensor $\B{W}_{\alpha\beta\gamma\delta}$ onto
the Hawking--Hartle tetrad: the spin $+2$ and $-2$ components of the linearized gravity are given by
\begin{align}\label{def:regularNPComps}
   \NPRplustwo&=\B{W}_{{l} m{l}m}, &
     \NPRminustwo&=\B{W}_{{n}\bar{m} {n}\bar{m}}.
\end{align}
The above-defined spin $s$ components $\NPR_s$ $(s=0,\pm1, \pm2)$ are regular and non-degenerate at the future event horizon. \underline{Throughout this work, we always denote $s$ the spin weight and $\sfrak=\abs{s}$.}

As is well-known, these massless spin-$s$ fields will develop power tails in the future development in the DOC of a Schwarzschild spacetime. Physically, this is due to the backscattering caused by an effective curvature potential which is in turn related to the non-vanishing middle N--P component of the Weyl curvature tensor of the Schwarzschild background, and, in view of the expectation that these linear models provide high accuracy approximation for the nonlinear dynamics, these tails are intimately related to many important problems in General Relativity, such as the black hole exterior (in-)stability problem and the Strong Cosmic Censorship conjecture for the black hole interior.
These power tails are predicted by Price in \cite{Price1972SchwScalar,Price1972SchwIntegerSpin}, and they are conjectured to be sharp and are now called the \textit{Price's law}. In addition,  the tails of a fixed $\ell$ mode are also discussed therein.
We shall also remark that these Price's law are conjectured to be sharp as both an \emph{upper} and a \emph{lower} bound. Furthermore, Barack--Ori claimed in \cite{bo99} that the positive spin components for both  the Maxwell field and the linearized gravity have an extra $v^{-1}$ decay at the event horizon than the conjectured Price's law, which means the sharp decay for $({\Upsilon}_{+\sfrak})^{\ell=\ell_0}$, $\sfrak=1,2$ on event horizon shall be $v^{-4-2\ell_0}$.  The conjectured Price's law and the heuristic argument of Barack--Ori together for both the spin $s$ component ${\Upsilon}_s$ and its fixed $\ell=\ell_0$ mode $({\Upsilon}_s)^{\ell=\ell_0}$ are summarized in Table $1$ provided the initial data decay sufficiently fast (or have compact support) on an initial future Cauchy surface terminating at spatial infinity.  \textbf{When there is no confusion, we shall call the summarized sharp decay rates in Table \ref{tabel:PL} as the \underline{Price's law}, albeit it is in fact a combination of the conjectured Price's law in \cite{Price1972SchwScalar,Price1972SchwIntegerSpin} and the predicted asymptotics for the positive spin components by Barack--Ori in  \cite{bo99}.} There are also works discussing the power tails in a Kerr spacetime; we guide the readers to \cite{Hod99Mode,bo99,gleiser2008late,csukas2019numerical} and the references therein.

\begin{table}[htbp]
\begin{center}
\begin{tabular}{|l|l|l|l|}
\hline
&\qquad towards $\Scri$ & in a compact region & \qquad\qquad on  $\Horizon$\\
  \hline
  \qquad\quad $\NPR_s$ &\quad$r^{-1-\sfrak-s}u^{-2-\sfrak+s}$  &  \qquad \, \,$v^{-3-2\sfrak}$ & \qquad \quad $v^{-3-2\sfrak-\zeta(s)}$\\
  \hline
\quad\quad$(\NPR_s)^{\ell=\ell_0}$ & \,\,\,\,$r^{-1-\sfrak-s}u^{-2-\ell_0+s}$  & \qquad \,\,\,$v^{-3-2\ell_0}$  & \qquad \,\,\,\, $v^{-3-2\ell_0-\zeta(s)}$\\
  \hline
  total decay power&$-3-\sfrak - \max\{\ell_0,\sfrak\}$ &\,\, $-3-2\max\{\ell_0,\sfrak\}$ & $-3-2\max\{\ell_0,\sfrak\}-\zeta(s)$ \\
  \hline
\end{tabular}
\end{center}
\caption{The conjectured Price's law for the spin $s=\pm \sfrak$ components on Schwarzschild are as above. The original Price's law, or the sharp decay rates predicted by Price in \cite{Price1972SchwScalar,Price1972SchwIntegerSpin}, contains the ones towards $\Scri$ and in a compact region. The last column about the sharp decay rates on $\Horizon$ contains the predicted decay rates by Barack--Ori \cite{bo99} for the positive spin component, with $\zeta(s)$ equals $1$ if $s>0$ and $0$ if $s\leq 0$.}
\label{tabel:PL}
\end{table}

\subsection{Main results}

Our main results give an affirmative answer to the conjectured Price's law in \cite{Price1972SchwScalar,Price1972SchwIntegerSpin} and Barack--Ori's claim in  \cite{bo99} for the spin $s$ components on a Schwarzschild background by providing a rigorous mathematical proof.  

\begin{thm}[Rough version of the globally precise late-time asymptotics of the spin $s$ components]\label{thm:main:intro}
Let $j\in \mathbb{N}$, let $\sfrak=0,1,2$, let $\ell_0\geq \sfrak$, and let $\tb_0\geq 1$. Let $Y_{m,\ell}^{s}(\cos\theta)e^{im\pb}$ be the spin-weighted spherical harmonics defined as in Section \ref{sect:decompIntoModes}. Let function $h_{\sfrak,\ell_0}$ be defined as in Definition \ref{def:hell:pm}.
\begin{enumerate}[label=\arabic*)]
\item\label{point:suffidecay:intro} Assume the initial data of the spin $s=\pm \sfrak$ components on $\Sigmazero$ are smooth and decay sufficiently fast (or are of compact support away from infinity), and are supported on $\ell\geq\ell_0$ modes.  Then in the future domain of dependence of $\Sigmazero$, there exists a constant $\veps>0$,  constants $\{C_{\sfrak;m,\ell_0}\}_{m=-\ell_0,-\ell_0+1,\ldots,\ell_0}$ which can be calculated from the initial data, and $D<\infty$ which depends only on the initial data, such that
\begin{subequations}
\label{eq:mainthm:pt1}
  \begin{align}
  &\Bigg|\Lxi^j\biggl(\Upsilon_s-2^{2\ell_0+2}
  \prod_{i=\ell_0+\sfrak+1}^{2\ell_0+1} i^{-1}\sum\limits_{m=-\ell_0}^{\ell_0} C_{\sfrak;m,\ell_0}Y_{m,\ell_0}^{s}(\cos\theta)e^{im\pb}\notag\\
  &\qquad \quad \quad\,\,\times \mu^{\frac{\sfrak+s}{2}}\hellz r^{\ell_0-\sfrak}\frac{-(\ell_0-s+1)v-(\ell_0+s+1)\tau}{\tau^{\ell_0-s+2}v^{\ell_0+s+2}}\biggr)\Bigg|\notag\\
  &\leq D r^{\ell_0-\sfrak}\tau^{-\ell_0+s-2-j-\varepsilon}v^{-\ell_0-s-1},
  \end{align}
  and, on the future event horizon $\Horizon$, for $\sfrak=1,2$,
  \begin{align}
  \hspace{4ex}&\hspace{-4ex}
  \bigg|\Lxi^{j}\NPR_{+\sfrak}\Big|_{\Horizon}-\prod_{n=2\ell_0+2}^{2\ell_0+j+3} n\frac{(-1)^{\sfrak+j+1}2^{2\ell_0+3}\sfrak (\ell_0-\sfrak)!(2M)^{\ell_0-\sfrak+1}\hellz(2M)}{(2\ell_0+1)!}\notag\\
  &\qquad \quad\quad\times
 \sum\limits_{m=-\ell_0}^{\ell_0} C_{\sfrak;m,\ell_0}Y_{m,\ell_0}^{+\sfrak}(\cos\theta)e^{im\pb}  v^{-2\ell_0-4-j}\bigg|\notag\\
  \leq{}&D
  v^{-2\ell_0-4-j-\veps}.
  \end{align}
  \end{subequations}
\item\label{point:ccinfty:intro} Assume the initial data for the spin $s=\pm \sfrak$ components on $t=t_0$ hypersurface (in the Boyer--Lindquist coordinates) is smooth and supported in a compact region of $(2M,\infty)$,\footnote{This means that the initial data are compactly supported away from both the future event horizon and infinity.} and assume they are supported on $\ell\geq\ell_0$ modes. Then the estimates in point \ref{point:suffidecay:intro}  hold true with the constants $C_{\sfrak;m,\ell_0}$ given by  
\begin{align}
  \label{def:NPell0:compsupp:intro}
 C_{\sfrak;m,\ell_0}={}&\frac{(-1)^{\ell_0-\sfrak+1}(\ell_0-\sfrak)!}{(2\ell_0+1)(2\ell_0+2)}\lim_{r\to \infty}(rh_{\sfrak, \ell_0-\sfrak,\ell_0-\sfrak} )\notag\\
 &\times \int_{2M}^{\infty}\mu^{-1}h_{\sfrak,\ell_0}r^{\ell_0+\sfrak}\big(r^2\partial_t (\NPRpluss)_{m,\ell=\ell_0}+2\sfrak(r-3M)(\NPRpluss)_{m,\ell=\ell_0}\big)\di r,
 \end{align}
where the function $h_{\sfrak, \ell_0-\sfrak,\ell_0-\sfrak}$\footnote{In particular, $rh_{\sfrak, 0,0}= -(2\sfrak+1)(2\sfrak+2)M$.} is determined as in Proposition \ref{prop:wave:Phihigh:pm1}.
\end{enumerate}
\end{thm}

Let us make a few remarks on this theorem.

\begin{remark}\label{rem:mainthm:intro}
\begin{enumerate}[label=\roman*)]

\item{(General and precise versions).}
Both points  \ref{point:suffidecay:intro} and \ref{point:ccinfty:intro} are immediate corollaries (and special cases) of Theorem \ref{summary:PL:v:pm:global:HiGhmodes}, where the  asymptotic conditions at infinity for the initial data on $\Sigmazero$ are explicitly imposed. In fact, we will prove Theorem \ref{summary:PL:nv:pm:global:HM} for the non-vanishing Newman--Penrose constant (NVNPC) case in Section \ref{sect:NVNPC:PL} and Theorem \ref{summary:PL:v:pm:global:HiGhmodes} for the vanishing Newman--Penrose constant (VNPC) case in Section \ref{sect:VNPC:PL}, both theorems requiring milder assumptions on the initial data than Theorem \ref{thm:main:intro}. 

\item {(Verify the Price's law).}
This theorem verifies the sharp decay rates for the spin $s=\pm \sfrak$ components on a Schwarzschild background presented  in the above Table \ref{tabel:PL}.  In particular, it suggests that the correct asymptotic decay rates should be a combination of the conjectured Price's law in \cite{Price1972SchwScalar,Price1972SchwIntegerSpin} and the heuristic claim of Barack--Ori in \cite{bo99}.

\item{(Genericity of the decay rates and the values of the constants $D$ and $C_{\sfrak;m,\ell_0}$).}
In both points \ref{point:suffidecay:intro} and \ref{point:ccinfty:intro} of this theorem, the constant $D$ is expressed purely in terms of the information of the initial data, and the constants $C_{\sfrak;m,\ell_0}$ are equal to $\NPCIT{\ell_0}$ which are the Newman--Penrose constants of the time integral that can in turn be calculated from the initial data of the spin $+\sfrak$ component; see more in Section \ref{subsect:PL:v:MT}. In particular, this suggests that the decay rates in the estimates \eqref{eq:mainthm:pt1} and \eqref{def:NPell0:compsupp:intro} be \underline{generically} sharp.

\item{(Properties of the function $h_{\sfrak,\ell_0}$).}
Function $r^{\ell_0-\sfrak}\hellz\sum\limits_{m=-\ell_0}^{\ell_0} Y_{m,\ell_0}^{-\sfrak}(\cos\theta)e^{im\pb}$ is a stationary solution to the equation of  an $\ell=\ell_0$ mode of the spin $-\sfrak$ component $\NPR_{-\sfrak}$, with the limit $\lim_{r\to \infty}\hellz=1$, while function $\mu^{\sfrak}r^{\ell_0-\sfrak}\hellz\sum\limits_{m=-\ell_0}^{\ell_0} Y_{m,\ell_0}^{+\sfrak}(\cos\theta)e^{im\pb}$ is a stationary solution to the equation of  an $\ell=\ell_0$ mode of the spin $+\sfrak$ component $\NPR_{+\sfrak}$.

\item
   The proof in this work also implies that in point \ref{point:suffidecay:intro}, if the coefficients of the asymptotic profiles are both vanishing, the decay rate will be faster by an extra $\tb^{-1}$ globally.

Similarly, in point \ref{point:ccinfty:intro}, one finds the decay rate can be improved if and only if the integral in the second line of \eqref{def:NPell0:compsupp:intro} vanishes. For the scalar field ($\sfrak=0$), it is manifest that if, additionally, the time derivative of the initial data vanishes, then the decay rate is  faster by an extra $\tb^{-1}$ globally. Note that this has been verified in \cite{hintz2020sharp,angelopoulos2021price} as well. For $\sfrak=1,2$, this is in fact compatible with the expectation in \cite{karkowski2004comments,price2004late} that solutions to the Regge--Wheeler equation with smooth, compactly supported in a compact region in $(2M,+\infty)$, static (in the sense that its time derivative on the initial hypersurface $t=\tb_0$ vanishes) initial data  will develop power tails $\tb^{-2\ell_0-4}$ in a finite radius region. This can be seen as follows: it is certain $\sfrak$-order derivative of the spin $s$ component which satisfies the Regge--Wheeler equation \cite{ReggeWheeler1957}, and we can use the wave equation satisfied by the spin $s$ component  to rewrite its time derivative as a combination of total radial derivatives, a first order time derivative and the component itself, which will then yield that the integral in the second line of \eqref{def:NPell0:compsupp:intro} vanishes.
  \end{enumerate}
\end{remark}

\begin{remark}[Asymptotics for the middle component of the Maxwell field]
   As a byproduct, we also obtain from the proof the global late-time asymptotics for the middle component of the Maxwell field. This thus yields that the middle component asymptotically decays to a static Coulomb solution, which equals $r^{-2}$ times a complex constant that  can be calculated from any sphere of the initial hypersurface, at a rate of decay that can be explicitly calculated.
\end{remark}

\begin{remark}[Discussions on other spin fields]
   Although we consider only $\sfrak=0,1,2$ cases in this work, the proof in this work suggests that the Price's law shall also hold true for the spin $s=\pm \sfrak$ components of the spin-$\sfrak$ field for arbitrary $\sfrak\in \mathbb{N}$. Additionally, the method developed in this work can be applied together with the results in \cite{MaZhang20sharp} to obtain the Price's law, or the global asymptotic profiles, for the massless Dirac field, which is also called as the massless spin-$\frac{1}{2}$ field.  Further, as pointed out in \cite{MaZhang20sharp}, Barack--Ori's claim about faster decay for the positive integer-spin component on the future event horizon can not be extended to the Dirac field.
\end{remark}

\begin{remark}[Generalization to wave equations with a potential]
It is without much effort to extract from the proof a general sharp decay estimate for a wave equation with a potential: Assume $\varphi$ solves $r^2\Box_g \varphi+\mathbf{P}\varphi=0$ in a Schwarzschild spacetime, where $\mathbf{P}=\mathbf{P}(r)=O(r^{-1})$ is smooth  and $\{\lim_{r\to \infty}(r^2 \partial_r)^i \mathbf{P}\}_{i=0,\ldots, n}$ exist for a sufficiently large $n$, assume there is a uniform energy boundedness and an integrated local energy decay (Morawetz) estimate for $\varphi$, and assume the initial dat on $\Sigmazero$ are smooth, decay sufficiently fast, and are supported on $\ell\geq \ell_0$ modes, $\ell_0\in \mathbb{N}$, then the statements in point \ref{point:suffidecay:intro} of Theorem \ref{thm:main:intro} with $\sfrak=0$ hold true for $\varphi$ and the coefficients in the asymptotic profiles can be calculated explicitly from the initial data and the information of the potential $\mathbf{P}$.
\end{remark}

\subsection{Outline of the proof}
\label{intro:outlineproof}

In this subsection, we make a comparison to other existed results, outline the methods and ideas,  and in the end, propose the future applications.

\subsubsection{Overview of the approach}

In \cite{angelopoulos2018vector,angelopoulos2018late},  Angelopoulos--Aretakis--Gajic initiated works aiming to prove the Price's law $v^{-1}\tb^{-2}$ for the scalar field on a Reissner--Nordstr\"{o}m background by the following steps:
\begin{labeling}{alligator}  
 \item  [A)] prove uniform energy boundedness and the Morawetz estimate;
    \item  [B)] obtain the almost Price's law in both the non-vanishing Newman--Penrose constant (NVNPC) case  and the vanishing Newman--Penrose constant (VNPC) case via
\begin{labeling}{alligator}  
\item [B1)] a definition of the Newman--Penrose constant for the spherically symmetric $\ell=0$ mode, which is a conserved quantity at null infinity;
\item [B2)] an extended $r^p$ hierarchy \cite{dafermos2009new} of spatially weighted energy estimates for an optimal range of the $r$-weight $p$;
\end{labeling}
\item [C)] achieve the sharp asymptotic $cv^{-1}\tb^{-1}$ in the NVNPC case;
\item [D)] in the VNPC case, derive the leading asymptotic $cv^{-1}\tb^{-2}$ for the $\ell=0$ mode by defining the \textquotedblleft{time integral\textquotedblright} which solves the scalar wave equation and whose time derivative equals the scalar field.
\end{labeling}
In particular, the step A) and a significant part of step B) have been well-developed for the wave equation, both linear and nonlinear, in black hole spacetimes in the past two decades.

In our work, we consider arbitrary modes of the spin $s$ components, and, without loss of much information, we shall consider here only an  arbitrary $(m,\ell)$ mode of the spin $s$ components and the dependence on $m$ and $\ell$ may be suppressed. One has to generalize the above approach to some extend and develop some novel methods and ideas.
 We provide an overview of the approach in the present work which can be basically divided into the following steps:
\begin{labeling}{alligator}  
 \item  [A')] prove uniform energy boundedness and the Morawetz estimate for the spin $s$ components;
    \item  [B')] obtain the almost Price's law via
\begin{labeling}{alligator}
\item [B1')] a derivation of extended wave systems, and a definition of the corresponding Newman--Penrose constant for the $(m,\ell)$ mode, which is also a conserved quantity at null infinity;
\item [B2')]  an extended $r^p$ hierarchy \cite{dafermos2009new} of spatially weighted energy estimates for the extended wave systems for an optimal\footnote{What we mean by \textquotedblleft{optimal\textquotedblright} here is in the sense that the obtained $p$ range has a sharp upper bound.} range of the $r$-weight $p$;
\item [B3')] obtaining the almost Price's law in the exterior region $\{\rb\geq \tb\}$;
\item [B4')] obtaining the almost Price's law in the interior region $\{\rb\leq \tb\}$;
\end{labeling}
\item [C')] derive the sharp asymptotics in the NVNPC case;
\item [D')] derive the asymptotic profiles in the VNPC case by defining the time integral for the spin $s$ component.
\end{labeling}
We shall emphasize that all the estimates are derived through a unified analysis of the \textit{Teukolsky master equation} (TME)\footnote{In some works, it is called as \emph{Teukolsky equation} instead.} \cite{Teu1972PRLseparability}, a wave equation satisfied by all the spin $s$ components. See Section \ref{sect:tme}. In fact, this equation is satisfied by the spin $s$ components in any subextremal or extremal Kerr spacetime. It has shown its enormous importance in recent works \cite{dafermos2019linear,Ma17spin2Kerr,
dafermos2019boundedness,andersson2019stability,
Giorgi2019linearRNfullcharge,
klainermanszeftel2017Schw}  in proving the linear or nonlinear stability of black hole spacetimes and recent works \cite{Ma2017Maxwell,Ma17spin2Kerr} towards the (almost) Price's law for non-zero spin fields.

In the following, we discuss each step, contrast our method with the one outlined above in \cite{angelopoulos2018vector,angelopoulos2018late}, the one developed in \cite{Ma20almost,MaZhang20sharp}, as well as the ones in other related works, and present the new ideas in this work.\\\\
\textbf{Step \textup{A')}.}
As is mentioned above, the starting point is to achieve the energy and Morawetz estimates  for solutions to the TME. These estimates have been proven in a Schwarzschild spacetime for the scalar field in \cite{bluesterbenz2006,dafrod09red} and for the Maxwell field and the linearized gravity in \cite{dafermos2019linear,pasqualotto2019spin}, where a key point is to use certain differential transformations due to Chandrasekhar \cite{chandrasekhar1975linearstabSchw}. Extensions to the Kerr spacetimes include \cite{tataru2011localkerr,larsblue15hidden,dafermos2016decay} for the scalar field and \cite{Ma2017Maxwell,Ma17spin2Kerr,dafermos2019boundedness} for non-zero spin fields. For the non-zero spin fields,  we start with the energy and Morawetz estimates in \cite{Ma2017Maxwell,Ma17spin2Kerr} where the author uses the standard technique in estimating the scalar field to  treat the coupled wave systems
\begin{align}
\textbf{WS}_{s} = \left\{ \text{the wave system of } \{\curlVR^i (\mu^{\sfrak}\NPRminuss)\}_{i=0,\ldots, \sfrak} \text{ or } \{(r^2 Y)^i \NPRpluss\}_{i=0,\ldots, \sfrak}\right\},
\end{align}
where  $\curlVR=r^2 \VR$, and  $Y=\mu^{-1}\partial_t -\partial_r$ and $\VR=\mu^{-1}V=\mu^{-1}\partial_t +\partial_r$ are the ingoing and outgoing principle null vectors, respectively. In particular, the wave equations of $\curlVR^{\sfrak} (\mu^{\sfrak}\NPRminuss)$ and  $(r^2 Y)^{\sfrak} \NPRpluss$ are the Regge--Wheeler equation. \\\\
\textbf{Step \textup{B')}.}
This step has much difference from the step B) where $\ell=\sfrak=0$. The above Morawetz estimates have already encoded most of the local information, and it remains to extract out the information in a large radius region. The substeps B1'), B2'), B3') have followed closely the approach  in \cite{Ma20almost} which treats the Maxwell field in Schwarzschild. See also  \cite{MaZhang20sharp}.

\textbf{Substep \textup{B1')}.}
The spin-weighted spherical Laplacian has a nonpositive eigenvalue when acting on a fixed $(m,\ell)$ mode, and to make use of this eigenvalue gap away from zero, one can in fact commute $\curlVR$ further  $j$ times, $j\in \mathbb{N}$,  with the equation of $\curlVR^{\sfrak} (\mu^{\sfrak}\NPRminuss)$ and obtain an extended system
\begin{align*}
\textbf{WS}^{(j)}_{-\sfrak} = \left\{ \text{the system of wave equations of } \{\Phiminuss{i}\}_{i=0,\ldots, j+\sfrak}\right\}.
\end{align*}
Here, $\Phiminuss{i}\triangleq \curlVR^i (\mu^{\sfrak}\NPRminuss)$, and each wave equation $\Phiminuss{i}$ can be written in the following form
\begin{align}
\label{eq:Phiminuss:govern:intro}
&-\mu Y\curlVR\Phiminuss{i}-(\ell-\sfrak+i+1)(\ell+\sfrak-i)\Phiminuss{i}\notag\\
&\qquad
-2(i-\sfrak+1)(r-3M)r^{-2}\curlVR\Phiminuss{i}+2i (i-\sfrak)(i-2\sfrak)M\Phiminuss{i-1}={}0.
\end{align}
It is crucial in the latter application of the $r^p$ method that the coefficient of the $\Phiminuss{i}$ term is nonpositive, and this requirement, which is satisfied only for $i\leq \ell+\sfrak$, imposes an upper bound for the value of $j$, that is, $j\leq \ell$.

To properly define a conserved quantity, the N--P constant,  at null infinity, we aim to find a scalar $\tildePhiminuss{\ell+\sfrak}$ whose equation is in the following form
\begin{align}
\label{eq:tildePhiminuss:ell+sfrak:intro}
-\mu Y\curlVR\tildePhiminuss{\ell+\sfrak}=O(r^{-1})\curlVR\tildePhiminuss{\ell+\sfrak}+\sum_{i=0}^{\ell+\sfrak}O(r^{-1})\Phiminuss{i}.
\end{align}
The right-hand side will have a zero limit towards null infinity, hence $\lim_{\rb\to \infty}\curlVR\tildePhiminuss{\ell+\sfrak}$ is a constant independent of time $\tb$; we call it the N--P constant for this $(m,\ell)$ mode. Such a scalar $\tildePhiminuss{\ell+\sfrak}$ is constructed from a linear combination of $\{\Phiminuss{i}\}_{i=0,\ldots,\ell+\sfrak}$ in \cite{Ma20almost} for the Maxwell field ($\sfrak=1)$, and, in particular, the scalar $\tildePhiminuss{\ell+\sfrak}$ is unique in all such linear combinations with constant coefficients up to an overall non-vanishing constant factor. Here, we follow the idea therein and extend  to any $\sfrak\in \mathbb{N}$. More generally, one can construct a set of scalars $\{\tildePhiminuss{i}\}_{i=0,\ldots, j+\sfrak}$, each of which takes the form $\tildePhiminuss{i}=\Phiminuss{i}+\sum_{j=0}^{i-1}c_{\sfrak,i,j}M^{i-j}\tildePhiminuss{j}$, $c_{\sfrak,i,j}$ being constants,  such that they satisfy
\begin{align*}
&-\mu Y\curlVR\tildePhiminuss{i}
-\frac{2(i-\sfrak+1)(r-3M)}{r^{2}}\curlVR\tildePhiminuss{i}
-(\ell-\sfrak+i+1)(\ell+\sfrak-i)\tildePhiminuss{i}+\sum_{j=0}^iO(r^{-1})\tildePhiminuss{j}={}0.
\end{align*}
Note that $\tildePhiminuss{i}=\Phiminuss{i}$ for $i\in \{0,\ldots,2\sfrak\}$.
 As a result, we obtain different extended wave systems for $j=0,\ldots, \ell$:
\begin{align*}
\widetilde{\textbf{WS}}^{(j)}_{-\sfrak} = \left\{ \text{the system of wave equations of } \{\tildePhiminuss{i}\}_{i=0,\ldots, j+\sfrak}\right\}.
\end{align*}

Analogously, one can define scalars $\{\PhiplussHigh{i}\}_{i=0,\ldots,\ell-\sfrak}$ and $\{\tildePhiplussHigh{i}\}_{i=0,\ldots,\ell-\sfrak}$ such that their equations are similar to the above ones for $\{\Phiminuss{i}\}_{i=2\sfrak,\ldots,\ell-\sfrak}$ and $\{\tildePhiminuss{i}\}_{i=2\sfrak,\ldots,\ell-\sfrak}$, respectively, as well as the wave systems $\textbf{WS}^{(j)}_{+\sfrak}$ and $\widetilde{\textbf{WS}}^{(j)}_{+\sfrak}$ for $j=0,\ldots, \ell-\sfrak$. Note that the equation for $\tildePhiplussHigh{i}$ is exactly the same as the one for $\tildePhiminuss{i+2\sfrak}$, that is, for any $i\in \{0,\ldots, \ell-\sfrak\}$,
\begin{align}
\label{eq:tildePhiplussHigh:govern:intro}
&-\mu Y\curlVR\tildePhiplussHigh{i}
-\frac{2(i+\sfrak+1)(r-3M)}{r^{2}}\curlVR\tildePhiplussHigh{i}
-(\ell+\sfrak+i+1)(\ell+\sfrak-i)\tildePhiplussHigh{i}+\sum_{j=0}^iO(r^{-1})\tildePhiplussHigh{j}={}0.
\end{align}

Let us finally remark that the Newman--Penrose constant for the $(m,\ell)$ mode of the spin $+\sfrak$ component is simply a non-vanishing constant multiple of the one for the spin $-\sfrak$ component.

\textbf{Substep \textup{B2')}.}
We are ready to apply the $r^p$ method to achieve the energy decay. For $i\in \{\sfrak,\ldots, \ell+\sfrak\}$, the $r^p$ estimates with $p\in [0,2]$ can be easily proven for each wave equation of $\tildePhiminuss{i}$, and one can obtain $\tb^{-2}$ decay for the $p=0$ weighted energy of $\tildePhiminuss{i}$ in terms of the $p=2$ weighted energy of $\tildePhiminuss{i}$ which in turn, by definition, is equivalent to the $p=0$ weighted energy of $\tildePhiminuss{i+1}$.  This thus yields $\tb^{-2\ell}$ decay for the $p=0$ weighted energy of $\tildePhiminuss{\sfrak}$ in terms of the $p=2$ weighted energy of $\tildePhiminuss{\ell+\sfrak}$ with a loss in the order of regularity.  For the equation of $\tildePhiminuss{\ell+\sfrak}$, because of the vanishing constant coefficient of $\tildePhiminuss{\ell+\sfrak}$ in \eqref{eq:tildePhiminuss:ell+sfrak:intro}, one can in fact extend the $r^p$ estimates from $p\in [0,2]$ to $p\in [0,3)$. This provides an optimal upper bound for the weight $p$ in the NVNPC case since the energy is infinite for the $p=3$ weighted energy. In total, we achieve $\tb^{-2\ell-1+2\veps}$ decay for the weighted energy of $\tildePhiminuss{\sfrak}$ in terms of the $p=3-2\veps$ weighted energy of $\tildePhiminuss{\ell+\sfrak}$. Additionally, we remark that for any $j\in \mathbb{N}$, there is an extra $\tb^{-2j}$ decay, that is, $\tb^{-2\ell-1-2j+2\veps}$ decay in total, for the $p=0$ weighted energy of $\Lxi^j\tildePhiminuss{\sfrak}$ in terms of the $p=3-2\veps$ weighted energy of $\tildePhiminuss{\ell+\sfrak}$, but with a further loss of regularity.  There are analogous energy decay estimates for the spin $+\sfrak$ component.

In the VNPC case, we can derive the $r^p$ estimates for $p\in [0,5)$ for the wave equation of $\tildePhiminuss{\ell+\sfrak}$, hence the above energy decay will be faster by an extra $\tb^{-2}$ decay.

\textbf{Substep \textup{B3')}.}
In the NVNPC case, the above energy decay estimate yields $v^{-1}\tb^{-\ell-1-j+\veps}$ globally pointwise decay for the scalars $\{\Lxi^j(r^2 V)^i \NPRminuss\}_{i=0,\ldots,\sfrak}$. In the exterior region $\{\rb\geq \tb\}$, one needs to gain extra $\tb^{-\sfrak}$ decay for $\NPRminuss$, and this is done by appealing to a derivation of an elliptic systems from the wave system of $\{\Phiminuss{i}\}_{i=0,\ldots,\sfrak-1}$. Consider only the most complicated case $\sfrak=2$. By equation \eqref{eq:Phiminuss:govern:intro}, the wave system can be rewritten as
\begin{align*}
\mathbf{A} (\Phiminustwo{0}, \Phiminustwo{1})^T=O(r^{-1}) (rV\Phiminustwo{1},rV\Phiminustwo{2} )+ O(r^{-1})(\Phiminustwo{1},\Phiminustwo{2} )
+O(1)(\Lxi\Phiminustwo{1},\Lxi\Phiminustwo{2} ),
\end{align*}
where $\mathbf{A}$ is a strongly elliptic $2\times 2$ matrix. This clearly implies an extra $\tb^{-1}$ decay for $\{\Phiminustwo{i}\}_{i=0,1}$, and one can derive an elliptic equation for $\Phiminustwo{0}$ to achieve a further $\tb^{-1}$ decay for $\Phiminustwo{0}$. To summarize, we obtain $v^{-1}\tb^{-\ell-\sfrak-1-j+\veps}$ pointwise decay for the scalars $\Lxi^j\NPRminuss$ in the exterior region. This is exactly the \underline{almost} Price's law in the exterior region in the NVNPC case. A similar argument works in the VNPC case and one obtains extra $\tb^{-1}$ pointwise decay.

The $v^{-1-2\sfrak}\tb^{-\ell+\sfrak-1-j+\veps}$ pointwise decay estimates for the scalars $\Lxi^j\NPRpluss$ are then derived from  the above energy decay estimates for the spin $+\sfrak$ component in Substep \textup{B2')} and the pointwise decay for $\Lxi^j\NPRminuss$ together with an application of the \textbf{Teukolsky--Starobinsky identities} (TSI) \cite{TeuPress1974III,starobinsky1973amplification}, which are two $2\sfrak$-order differential identities between the spin $\pm \sfrak$ components. See Section \ref{sect:TSI} for the TSI. The TSI allow one to derive certain estimates for one spin component from the estimates of the other spin component, and they are frequently used and play a vital role in each of the following steps as well.

\textbf{Substep \textup{B4')}}
This substep is fundamentally different from the other works. It suffices to consider the spin $-\sfrak$ component, as the estimates for the spin $+\sfrak$ component can be achieved via the TSI.

A key ingredient is to define a scalar $\varphiminussell=(\hell r^{\ell-\sfrak})^{-1}\NPRminuss$ such that its equation takes the form of
\begin{align}
\prb(r^{2\ell+2}\mu^{1+\sfrak}\hell^2 \prb \varphiminussell)=\Lxi H_{\varphiminussell},
\end{align}
where $H_{\varphiminussell}$ is an expression of $\varphiminussell, \prb\varphiminussell$ and $\Lxi\varphiminussell$. Recall from Remark \ref{rem:mainthm:intro} that $\hell r^{\ell-\sfrak}$ is a zero energy mode solution to the TME of $\NPRminuss$. This is essentially a degenerate (at horizon) elliptic equation with the RHS viewed as a source. One can derive elliptic estimates for this equation which basically says that $2\beta$-weighted energy of $\Lxi^j\varphiminussell$ in the interior region is bounded by a weighted energy of the source, which is in turn bounded by $2\beta+2$-weighted energy of $\Lxi^{j+1}\varphiminussell$, plus some boundary flux at $\rb=\tb$, for any $\beta\in [0, \ell-\sfrak-1]$. A simple iteration in $\beta$ for these elliptic estimates then yields $\tb^{-3-4\ell-2j+2\veps}$ decay for the energy $\int_{2M}^{\rb}(\abs{\Lxi^j\varphiminussell}^2+\abs{\Lxi^j (r\prb\varphiminussell)}^2) \di \rb$ in terms of the $p=3-2\veps$ weighted energy of $\tildePhiminuss{\ell+\sfrak}$; therefore, in the NVNPC case, we prove $\tb^{-2-2\ell-j+\veps}$ pointwise decay for $\Lxi^j\varphiminussell$ in the interior region, and in the VNPC case, the decay is faster by $\tb^{-1}$ in terms of  the $p=5-2\veps$ weighted energy of $\tildePhiminuss{\ell+\sfrak}$. As an extra benefit, we achieve an extra $\tb^{-1}$ decay for $\Lxi^j\prb\varphiminussell$ compared to $\Lxi^j\varphiminussell$  in both the NVNPC and VNPC cases.\\\\
\textbf{Step \textup{C')}}
First, we  follow the idea in \cite{angelopoulos2018late} to derive the asymptotic profiles for $\Lxi^{j_1}\pu^{j_2}\big((v-u)^{-2\ell}\tildePhiplussHigh{\ell-\sfrak}\big)$ in the region $\{v-u\geq v^{\alpha}\}$ where $\alpha$ is a constant strictly less than $1$ but supposed to be sufficiently close to $1$. This is done as follows: One can first integrate along constant $v$ for the wave equation \eqref{eq:tildePhiplussHigh:govern:intro} of $\tildePhiplussHigh{\ell-\sfrak}$ to obtain the asymptotic profile of $V\tildePhiplussHigh{\ell-\sfrak}$, then integrate along constant $u$ from the hypersurface $\{v-u=v^{\alpha}\}$ to achieve the asymptotic profile of $\tildePhiplussHigh{\ell-\sfrak}$. We can commute $\pv$ with the wave equation for $\tildePhiplussHigh{\ell-\sfrak}$ and run through the above argument again, which then enables us to derive the asymptotic profiles of $\Lxi^{j_1}\pu^{j_2}\big((v-u)^{-2\ell}\tildePhiplussHigh{\ell-\sfrak}\big)$.

In the next step, we make use of equation \eqref{eq:tildePhiplussHigh:govern:intro} for $i\in \{0,\ldots, \ell-\sfrak-1\}$ and obtain the following system
\begin{align*}
\hspace{4ex}&\hspace{-4ex}
(\ell-i-\sfrak)(\ell+i+\sfrak+1)(v-u)^{-2(i+\sfrak)}\tildePhiplussHigh{i}\notag\\
={}&
-2(v-u)^2\pu\big((v-u)^{-2(i+\sfrak+1)}\tildePhiplussHigh{i+1}\big)
\notag\\
&
+(v-u)^{-2(i+\sfrak)}\Big(O(1) \pu\tildePhiplussHigh{i}
+\sum_{j=0}^{i}O(r^{-1}) \tildePhiplussHigh{j}
+O(r^{-1})r^{-1}\log r \tildePhiplussHigh{i+1}\Big).
\end{align*}
The last line has faster fall-off in the region $v-u\geq v^{\alpha}$, hence, one can determine the asymptotic profiles for $(v-u)^{-2(i+\sfrak)}\tildePhiplussHigh{i}$ iteratively, including in particular the one for $(v-u)^{-2\sfrak}\tildePhiplussHigh{0}\sim 2^{-2\sfrak} r \NPRpluss$, which then yields the asymptotic profile of $\Lxi^j\NPRpluss$ in  the region $\{v-u\geq v^{\alpha}\}$.
In the remaining region $\{v-u\leq v^{\alpha}\}$, one can integrate $\prb\Lxi^j\varphiminussell$ from the hypersurface $\{v-u=v^{\alpha}\}$ and make use of the better decay for $\prb\Lxi^j\varphiminussell$ proven in Substep \textup{B4')}; we thus conclude that the asymptotic profile of $\Lxi^j\varphiminussell$ remains the same in the region $\{v-u\leq v^{\alpha}\}$. We shall comment that arbitrary $j\in \mathbb{N}$ times $\Lxi$ derivative yields an extra $\tb^{-j}$ pointwise decay, a fact which will show its importance in the next step. An application of the TSI then yields the asymptotics of the spin $+\sfrak$ component. In particular, for the spin $+\sfrak$ component, both the conjectured Price's law outside the black hole and the claim of Barack--Ori about faster decay on $\Horizon$ can be shown via the TSI. \\\\
\textbf{Step \textup{D')}}
The basic idea in this step is in the same spirit of the one in \cite{angelopoulos2018late} by reducing the VNPC problem to a NVNPC problem via defining a time integral of the solution. Again, in view of the TSI, we consider only the spin $+\sfrak$ component. The time integral is defined such that it solves also the TME of the spin $+\sfrak$ component and, more importantly, the time derivative of the time integral equals the spin $+\sfrak$ component. By this property,  one gains an extra $\tb^{-1}$ decay for the spin $+\sfrak$ component than the decay of the time integral, hence than the one in NVNPC case. The main task of the remaining discussions is to calculate the associated N--P constant of the time integral.

\subsubsection{Remarks on other most relevant works}

Throughout the above discussions, we have been contrasting our results with the ones in \cite{angelopoulos2018vector,angelopoulos2018late,Ma20almost}  for a couple of reasons: on one hand, our approach follows partly from the ones in these works; on the other hand, we can contrast ours with these works in different steps in the proof easily since they are all based on the vector field method. We shall now briefly discuss the other works aiming at proving the Price's law and make comparisons with our results.

The work \cite{metcalfe2017pointwise} by Metcalfe--Tataru--Tohaneanu followed the authors' earlier results \cite{tataru2013local,metcalfe2012price} and derived, under assumptions of both integrated local energy decay estimates and stationary local energy bounds, a global $v^{-2-s}\tb^{-2+s}$ decay for the spin $s=\pm 1$ components  and the middle component, for which $s=0$,  of the Maxwell field in a class of non-stationary asymptotically flat spacetimes. The backgrounds under consideration in this work  include in particular the Schwarzschild spacetimes and the family of the Kerr spacetimes, and the assumptions are known to hold true  in a Schwarzschild spacetime or a slowly rotating Kerr spacetime after subtracting the static/stationary Coulomb solution but unknown in more general spacetimes.  If in a Schwarzschild or a Kerr spacetime, the proven decay rates in \cite{metcalfe2017pointwise} are slower than the Price's law by $\tb^{-1}$.

The works \cite{donninger2011proof,donninger2012pointwise} by Donninger--Schlag--Soffer treated a Regge--Wheeler equation by constructing the Green's function for its solution and obtained on Schwarzschild $t^{-2\ell-2}$ for a fixed $\ell$ mode of the solution, $t^{-3}$ decay for the scalar field, $t^{-4}$ and $t^{-6}$ decay for the part of Maxwell field and linearized gravity which solves this equation, respectively. Further, they showed that the decay rate is faster by $t^{-1}$ for initially static initial data. As was discussed in Remark \ref{rem:mainthm:intro}, by applying certain derivatives, known as the  Chandrasekhar transformation \cite{chandrasekhar1975linearstabSchw}, on the spin $s$ components, the obtained scalars satisfy the Regge--Wheeler equation, the potential in which is propositional to $s^2$.  These estimates are valid  in a compact region but are not uniform in the future Cauchy development of a future Cauchy surface; the decay estimates are sharp for the entire scalar field,  but  have one less power of decay in time compared to Price's law in a compact region in the remaining cases.

Very recently, much important progress were made in proving the Price's law for the spin-$\sfrak$ fields on Schwarzschild, Reissner--Nordstr\"{o}m, and Kerr backgrounds. For the scalar field,  Hintz
\cite{hintz2020sharp} computed the $v^{-1}\tb^{-2}$ leading order term  on both Schwarzschild and subextremal Kerr spacetimes and obtained $t^{-2\ell_0-3}$ sharp asymptotics for  $\ell\geq \ell_0$ modes in a compact region on Schwarzschild; Angelopoulos--Aretakis--Gajic derived in \cite{angelopoulos2021late}  the asymptotic profiles of the $\ell=0$, $\ell=1$, and $\ell\geq 2$ modes in a subextremal Kerr spacetime and computed in \cite{angelopoulos2021price} the $v^{-1}\tb^{-2\ell_0-2}$ asymptotics for  $\ell\geq \ell_0$ modes on a  subextremal Reissner--Nordstr\"{o}m background.  For non-zero spin fields,
 in an earlier work \cite{Ma20almost} of the first author of our current work,  $v^{-2-s}\tb^{-\frac{3}{2}+s}$ decay in non-static Kerr and  $v^{-2-s}\tb^{-3+s+\epsilon}$ decay in Schwarzschild towards a stationary/static Coulomb solution are proven, and it also proves the almost Price's law $v^{-2-s}\tb^{-2-\ell_0+s+\veps}$ for any $\ell\geq \ell_0$ modes for the Maxwell field in the region $\rb\geq \tb$ on a Schwarzschild background; the authors of this current work obtained in  \cite{MaZhang20sharp}  the energy and Morawetz estimates and calculated the asymptotic profiles with decay $v^{-\frac{3}{2}-s}\tb^{-\frac{5}{2}+s}$ for the spin $s=\pm \half$ components of the massless Dirac field on Schwarzschild.

\subsubsection{Future applications}

We propose some future applications of the methods developed here:
\begin{enumerate}
\item a generalization of the results herein to the Reissner--Nordstr\"{o}m spacetime, which will be a cornerstone in a proof of the Strong Cosmic Censorship for the linearized gravity of this spacetime. We note from \cite{Giorgi2019linearRNfullcharge} that the spin $\pm 2$ components of the linearized gravity and the spin $\pm 1$ components of the Maxwell field in their TME-like wave equations. Given now a unified treatment for the TME of both the Maxwell field and the linearized gravity in a Schwarzschild spacetime, we expect that the Price's law, or the asymptotic profiles, can also be proven for both the linearized gravity and the Maxwell field in a Reissner--Nordstr\"{o}m spacetime;
\item a generalization to the non-static Kerr spacetimes where the angular momentum per mass $\mathbf{a}\neq 0$. In the non-vanishing $\mathbf{a}$ case, decoupling between different spherical symmetric modes is no longer valid in the evolution. It is quite interesting to analyze this coupling between different modes and investigate how the parameter $\mathbf{a}$ affects the asymptotic profiles of the spin $s$ components or of a higher mode of them.
\end{enumerate}

\subsection{Other relevant works}

Apart from the above mentioned works which are most relevant to the Price's law topic, we list here some other related works.

For the works on wave equations in Minkowski, we refer to the pioneering ones \cite{morawetz1968time,klainerman1986null,christodoulou1986global,CK93global,lindblad2010global} and the references therein.
There is a large amount of works in the literature about scalar field on a black hole background:  Morawetz estimates as well as pointwise decay estimates are obtained in Schwarzschild in \cite{bluesoffer03mora,blue:soffer:integral,bluesterbenz2006,dafrod09red} using a Morawetz type multiplier, in slowly rotating Kerr in \cite{larsblue15hidden,tataru2011localkerr}, and the estimates are further extended to subextremal Kerr in \cite{dafermos2016decay}.  Strichartz estimates are shown in \cite{MMTT,tohaneanu2012strichartz}. See also \cite{Finster2006} for local decay of the scalar field on Kerr.

Morawetz estimates and decay estimates for Maxwell field  are  obtained in a Schwarzschild spacetime in  \cite{blue08decayMaxSchw,pasqualotto2019spin}, in some general family of spherically symmetric stationary spacetimes in \cite{sterbenz2015decayMaxSphSym}, and in slowly rotating Kerr in  \cite{larsblue15Maxwellkerr,Ma2017Maxwell}.
We note also that a conserved, positive definite energy has been constructed for the Maxwell field in Schwarzschild in \cite{andersson16decayMaxSchw} and on Kerr and Kerr-de Sitter backgrounds but under axial symmetry in \cite{gudapati2017positive,gudapati2019conserved}. Morawetz estimates and decay estimates for the spin $\pm 2$ components of the linearized gravity in a slowly rotating Kerr spacetime are proven in \cite{Ma17spin2Kerr,dafermos2019boundedness}. See also  \cite{finster2016linear,shlapentokh2020boundedness}  for some results in a subextremal Kerr spacetime. The mode stability results in a Kerr spacetime have been obtained for the scalar field in \cite{2015AnHP...16..289S}  and  for general spin-$\sfrak$ fields in \cite{whiting1989mode,andersson2017mode,da2019mode}.

Linear stability of a Schwarzschild or a subextremal  Reissner--Nordstr\"{o}m spacetime has been shown by \cite{dafermos2019linear,hung2017linearstabSchw,johnson2019linear,Hung18odd, Hung19even,Giorgi2019linearRNfullcharge}. See also \cite{Jinhua17LinGraSchw}. Linear stability of a slowly rotating Kerr spacetime is proven in \cite{andersson2019stability,hafner2019linear}. 

Finally, we mention two nonlinear stability results \cite{klainermanszeftel2017Schw,HintzKds2018}: the first one
for Schwarzschild under polarized axisymmetry, and the second one for  slowly rotating Kerr-de Sitter spacetimes.

\subsection*{Overview of the paper}

We give some preliminaries in Section \ref{sect:prel}, and then introduce the Teukolsky master equation and Teukolsky--Starobinsky identities in Section \ref{sect:TMETSI}. Afterwards, we obtain the almost Price's law in Section \ref{sect:APL:all}. In Section \ref{sect:NVNPC:PL}, we compute the late-time asymptotics for the spin $\pm\sfrak$ components in the non-vanishing Newman--Penrose constant case. In the end, Section \ref{sect:VNPC:PL} is devoted to deriving the late-time asymptotics for the spin $\pm\sfrak$ components in the vanishing Newman--Penrose constant case and proving Theorem \ref{thm:main:intro}.


\section{Preliminaries}
\label{sect:prel}

We introduce in this section some preliminaries in geometric and analytic aspects. We first discuss the geometry of the Schwarzschild spacetime in Section \ref{sect:foliation} and make some general conventions in Section \ref{sect:convention}. Next, we define a few operators and norms in Section \ref{sect:OpersNorms}, followed by discussions on the spin-weighted spherical harmonic decomposition in Section \ref{sect:decompIntoModes}. The last sections \ref{sect:HardySobolev}--\ref{sect:rplemma} are constituted by purely analytic estimates: some Hardy and Sobolev inequalities, decay estimates followed from a hierarchy of estimates,  and $r^p$ estimates for wave equations.

\subsection{Coordinates and foliation of the spacetime}
\label{sect:foliation}

Define a tortoise coordinate $r^*$ by
\begin{align}
\di r^*=\mu^{-1}\di r,\qquad r^*(3M)=0.
\end{align}
Denote the retarded and forward double null coordinates by
\begin{align}
u=t-r^*, \qquad v=t+r^*,
\end{align}
respectively. Define  a \textit{hyperboloidal coordinate system} $(\tb,\rb,\theta,\phi)$ as in \cite{andersson2019stability}, where $\rb=r$, $\tb=v-\hhyp$ and $\hhyp=\hhyp(r)$, such that the level sets of the time function $\tb$ are strictly spacelike with
\begin{align}
c(M)r^{-2}\leq -g(\nabla \tb,\nabla\tb)\leq C(M) r^{-2}
\end{align}
for two positive universal constants $c(M)$ and $C(M)$
and they are transverse to the future event horizon regularly and asymptotic to future null infinity $\Scri$, and for large $r$, $1\lesssim \lim_{\rb\to \infty}r^2 (\partial_r\hhyp - 2\mu^{-1})\vert_{\Sigmatb}<\infty$.

Let $\Sigmatb$ be the constant $\tb$ hypersurface in the domain of outer communication $\DOC$.  Let $\tb_0\geq 1$, and let $\Sigmazero$ be our initial hypersurface on which the initial data are imposed. For any $\tb_0\leq \tb_1<\tb_2$, let $\Donetwo$, $\Scrionetwo$ and $\Horizononetwo$ be the truncated parts of $\DOC$, $\Scri$ and $\Horizon$ on $\tb\in [\tb_1,\tb_2]$, respectively. See Figures \ref{fig:2} and  \ref{fig:1}.

\begin{figure}[htbp]
\begin{minipage}[t]{0.5\linewidth}
\begin{center}
\begin{tikzpicture}[scale=0.8]
\draw[thin]
(0,0)--(2.45,2.45);
\draw[very thin]
(2.5,2.5) circle (0.05);
\coordinate [label=90:$i_+$] (a) at (2.5,2.5);
\draw[thin]
(0,0)--(2.45,-2.45);
\draw[dashed]
(2.55,2.45)--(4.95,0.05);
\draw[very thin]
(5,0) circle (0.05);
\coordinate [label=360:$i_0$] (a) at (5,0);
\draw[dashed]
(4.95,-0.05)--(2.55,-2.45);
\draw[very thin]
(2.5,-2.5) circle (0.05);
\coordinate [label=270:$i_-$] (a) at (2.5,-2.5);
\draw[thin]
(0.9,0.9) arc (215:320:2.1 and 1.7);
\node at (2.5,-0.15) {\small $\Sigma_{\tau_1}$};
\draw[thin]
(1.5,1.5) arc (212:323:1.25 and 0.9);
\node at (2.5,1.35) {\small $\Sigma_{\tau_2}$};
\draw[very thick]
(0.9,0.9)--(1.5,1.5);
\draw[dashed, very thick]
(3.57,1.43)--(4.28,0.72);
\node at (0.95,1.45) [rotate=45] {\small $\mathcal{H}_{\tau_1,\tau_2}^+$};
\node at (4.15,1.35) [rotate=-45] {\small $\mathcal{I}_{\tau_1,\tau_2}^+$};
\node at (2.5,0.7) {\small $\Donetwo$};
\end{tikzpicture}
\end{center}
\caption{Hyperboloidal foliation and related definitions.}
\label{fig:2}
\end{minipage}%
\begin{minipage}[t]{0.5\linewidth}
\begin{center}
\begin{tikzpicture}[scale=0.8]
\draw[thin]
(0,0)--(2.45,2.45);
\draw[very thin]
(2.5,2.5) circle (0.05);
\coordinate [label=90:$i_+$] (a) at (2.5,2.5);
\draw[thin]
(0,0)--(2.45,-2.45);
\draw[dashed]
(2.55,2.45)--(4.95,0.05);
\draw[very thin]
(5,0) circle (0.05);
\coordinate [label=360:$i_0$] (a) at (5,0);
\draw[dashed]
(4.95,-0.05)--(2.55,-2.45);
\draw[very thin]
(2.5,-2.5) circle (0.05);
\coordinate [label=270:$i_-$] (a) at (2.5,-2.5);
\draw[thin]
(0.6,0.6) arc (209:335:2.1 and 1.1);
\node at (1.08,1.42) [rotate=45] {\small $\mathcal{H}^+$};
\node at (3.9,1.4) [rotate=-45] {\small $\mathcal{I}^+$};
\node at (2.5,-0.25) {\small $\Sigma_{\tau_0}$};
\end{tikzpicture}
\end{center}
\caption{Initial hypersurface $\Sigma_{\tau_0}$ on which initial data will be imposed.}
\label{fig:1}
\end{minipage}
\end{figure}

For latter convenience, we define 
a function related to the hyperboloidal foliation:
\begin{align}
\Hhyp=2\mu^{-1}-\partial_r \hhyp.
\end{align}
By our choice of the hyperboloidal coordinates, it holds that \begin{align}
\label{eq:propertyofHfunction}
r^2 \Hhyp\lesssim 1 \quad
\text{for  } r  \text{ large}, \quad \text{and} \quad \abs{ \Hhyp-2\mu^{-1}}\lesssim 1  \quad
\text{as } r\to r_+.
\end{align}

We end this subsection with a remark on differential coordinate derivatives.
Since we are using a few different coordinate systems, there is potential confusion about the coordinate derivatives. We shall emphasis that the partial derivatives $\pt_{t}$ and $\pt_r$ are always in terms of the Boyer--Lindquist coordinate system $(t,r,\theta,\phi)$. Similarly, the partial derivative $\pt_{r^*}$ is in terms of the tortoise coordinates $(t, r^*, \theta, \phi)$, the partial derivatives $\pt_u$ and $\pt_v$ are in terms of the double null coordinates $(u,v,\theta,\phi)$, and the partial derivatives $\pt_{\tb}$ and $\prb$ are in terms of the hyperboloidal coordinate system $(\tb,\rb,\theta,\phi)$. Note that the coordinate derivative $\pt_{\theta}$ in each above coordinate systems is equal, and so does $\pt_{\phi}$.


\subsection{General conventions}
\label{sect:convention}

$\mathbb{N}$ is denoted as the natural number set $\{0,1,\ldots\}$, $\mathbb{N}^+$ the positive natural number set, $\mathbb{R}$ the real number set, and $\mathbb{R}^+$ the positive real number set.  

Denote $\Re(\cdot)$ as the real part.

For any $x\in \mathbb{R}$, let the Japanese bracket $\langle \cdot \rangle$ be defined by $\langle x\rangle=\sqrt{x^2+1}$.

LHS and RHS are short for left-hand side and right-hand side, respectively.

Constants in this work may depend on the hyperboloidal foliation via the function $\hhyp$. For simplicity, we shall always suppress this dependence throughout this work as one can fix this function once for all. For the same reason, the dependence on the mass parameter $M$ is always suppressed as well.

We denote a universal constant by $C$ if it depends only on the hyperboloidal foliation and the mass $M$. If a universal constant depends on a set of parameters $\mathbf{P}$, we denote it by $C(\mathbf{P})$. We denote $\reg$ as a general regularity parameter,  and denote $\regl$ as a universal constant which may change from line to line. Also, $\regl(\mathbf{P})$ means a regularity constant depending on the parameters in the set $\mathbf{P}$.

We say $F_1\lesssim F_2$ if there exists a universal constant $C$ such that $F_1\leq CF_2$. Similarly for $F_1\gtrsim F_2$. If both $F_1\lesssim F_2$ and $F_1\gtrsim F_2$ hold, we say $F_1\sim F_2$.

Let $\mathbf{P}$ be a set of parameters. We say $F_1\lesssim_{\mathbf{P}} F_2$ if there exists a universal constant $C(\mathbf{P})$ such that $F_1\leq C(\mathbf{P})F_2$. Similarly for $F_1\gtrsim_{\mathbf{P}} F_2$. If both $F_1\lesssim_{\mathbf{P}} F_2$ and $F_1\gtrsim_{\mathbf{P}} F_2$ hold, then we say $F_1\sim_{\mathbf{P}} F_2$.

For any $\alpha \in \mathbb{R}^+$, we say a function $f(r,\theta,\pb)$ is $O(r^{-\alpha})$ if  for any $j\in \mathbb{N}$,
$\abs{(\partial_r)^j f_2}\leq C_j r^{-\alpha-j}$.  Further, we say a function $f(r,\theta,\pb)$  is $O(1)$ if $\abs{f}\leq C_0$ and $\abs{(\partial_r)^j f }\leq C_jr^{-1-j}$ for any $j\in \mathbb{N}^+$.

Let  $\chi_1$ be a standard smooth cutoff function that is decreasing, equals $1$ on $(-\infty,0)$, and equals $0$ on $(1,\infty)$. Let $\chi=\chi_1((R_0-r)/M)$, with $R_0$ being suitably large and to be fixed in the proof, so $\chi=1$ for $r\geq R_0$ and vanishes identically for $r\leq R_0-M$.

\subsection{Operators and norms}
\label{sect:OpersNorms}

In this subsection, we define a few differential operators and use them to define the energy norms and spacetime Morawetz norms.

To begin with, we shall recall from \cite{geroch1973space} and \cite[Chapter 4.12]{penroserindlerI} the standard definitions of spin-weighted scalars and spin-weighted operators. 

\begin{definition}
\begin{itemize}
\item A scalar which has proper spin weight and zero boost weight in the sense of Geroch, Held and Penrose \cite{geroch1973space} is called a \textit{spin-weighted scalar}.\footnote{According to Geroch--Held--Penrose  (GHP) \cite{geroch1973space}, a GHP scalar can be viewed as $\mathbb{C}$-valued contractions of a tensor with elements of a local null tetrad $(l,n,m,\bar{m})$ or their derivatives. A GHP scalar $\alpha$ is called a scalar with proper spin weight $s$ and boost weight $b$ if the scalar $\alpha$ is transformed to $\lambda^be^{is\varphi}\alpha$ when the tetrad $(l,n,m,\bar{m})$ is transformed to $(\lambda l, \lambda^{-1}n, e^{i\varphi}m, e^{-i\varphi}\bar{m})$. In fact, the spin-weighted scalars are sections of complex line bundles.}
    Unless otherwise stated, we shall always denote $s$ the spin weight, and we may call a spin-weighted scalar with spin weight $s$ as a \textit{spin $s$ scalar}.
\item A differential operator is a \textit{spin-weighted operator} if it takes a spin-weighted scalar to a spin-weighted scalar.
\end{itemize}
\end{definition}

We first define operators that are aligned with or rescaled from the principal null directions.

\begin{definition}
\label{def:basic:vectorfields}
Define the differential operators
\begin{align}
Y=\sqrt{2}n^a\partial_a , \qquad V=\sqrt{2}l^a \partial_a.
\end{align}
Further, we define two rescaled operators from $V$ as follows:
\begin{align}
\VR= \mu^{-1}V,\qquad \curlVR={}r^2\VR.
\end{align}
\end{definition}

\begin{remark}
\label{rem:YVsoon}
\begin{itemize}
\item 
All these operators are spin-weighted operators, and the operation of these operators upon a spin $s$ scalar does not change the scalar's spin weight.
\item
One can express  $Y$ and $\VR$ as
\begin{align}
\label{def:vectorVRintermsofprb}
Y={}-\prb+(2\mu^{-1}-\Hhyp)\Lxi,\qquad
\VR={}\partial_{\rb}+\Hhyp \Lxi.
\end{align}
Note that $V=\mu\partial_{\rb}+\mu\Hhyp \Lxi\sim\mu\partial_{\rb}+2 \Lxi$ as $r\to r_+$, so the operator $V$ is regular near the future event horizon $\Horizon$; however, the operator $\VR$ is singular near $\Horizon$ since $\VR\sim\partial_{\rb}+2\mu^{-1} \Lxi$ as $r\to r_+$.
\item 
Further, we can relate these operators with the coordinate derivatives in double null coordinates $(u,v,\theta,\phi)$ via
\begin{align}
\pu=\half\mu Y,\qquad\pv=\half V=\half \mu \VR.
\end{align}
\end{itemize}
\end{remark}

For a spin $s$ scalar, there are associated angular operators, both of which being spin-weighted operators. The definitions and notations are both taken from the standard textbook of Penrose--Rindler \cite{penroserindlerI}. 

\begin{definition}
\label{def:setsofopers}
Let $\varphi$ be a spin $s$ scalar.
Let the \textit{spherical edth angular operators} $\edthR$ and $\edthR'$ be given in B--L coordinates by
\begin{align}
\edthR\varphi={}&\partial_{\theta}\varphi
+{i}csc\theta\partial_{\phi}\varphi
-{s}cot\theta\varphi,&
\edthR'\varphi={}&\partial_{\theta}\varphi
-{i}\csc\theta\partial_{\phi}\varphi
+{s}cot\theta\varphi.
\end{align}
\end{definition}

\begin{remark}
If $\varphi$ is a spin $s$ scalar, then $\edthR\varphi$ and $\edthR'\varphi$ are spin $s+1$ and $s-1$ scalars, respectively. That is, $\edthR$ increases the spin weight by $1$, while $\edthR'$ decreases the spin weight by $1$.
\end{remark}

We next define a couple of operator sets from these differential operators.

\begin{definition}
\label{def:setsofopers:commutators}
Let $\varphi$ be a spin $s$ scalar.
Define a set of operators
\begin{subequations}
\begin{align}
\PDeri={}\{Y,V, r^{-1}\edthR,r^{-1}\edthR'\}
\end{align}
adapted to the Hawking--Hartle tetrad. 
Define a set of operators
\begin{align}
\CDeri={}\{Y,rV, \edthR,\edthR'\},
\end{align}
which is adapted to the hyperboloidal foliation and will be used as the set of commutators.
\end{subequations}
\end{definition}

These operator sets can be used to define pointwise norms.

\begin{definition}
Let $\mathbb{X}=\{X_1, X_2, \ldots, X_n\}$, $n\in \mathbb{N}^+$, be a set of spin-weighted operators, and let a multi-index $\mathbf{a}$ be an ordered set $\mathbf{a}=(a_1,a_2,\ldots,a_m)$ with all $a_i\in \{1,\ldots, n\}$.  Define $|\mathbf{a}|=m$, and define $\mathbb{X}^{\mathbf{a}}=X_{a_1}X_{a_2}\cdots X_{a_m}$ if $m\in \mathbb{N}^+$ and $\mathbb{X}^{\mathbf{a}}$ as the identity operator if $m=0$. Let $\varphi$ be a spin-weighted scalar, and define its pointwise norm of order $m$, $m\in \mathbb{N}$, as
\begin{align}
\absHighOrder{\varphi}{m}{\mathbb{X}}={}\sqrt{\sum_{\abs{\mathbf{a}}\leq m}\abs{\mathbb{X}^{\mathbf{a}}\varphi}^2}.
\end{align}
\end{definition}

Finally, we are able to define energy norms and (spacetime) Morawetz norms.
\begin{definition}
\label{def:basicweightednorm}
Let $\varphi$ be a spin-weighted scalar. Let $k\in \mathbb{N}$ and $\gamma\in \mathbb{R}$. Let $\Omega$ be a $4$-dimensional subspace of the DOC and let $\Sigma$ be a $3$-dimensional space that can be parameterized by $(\rb,\theta,\pb)$. Denote the volume element of unit sphere by $\di^2\mu=\sin\theta \di \theta \wedge \di \pb$. We define the Morawetz norm in $\Omega$ and the energy norm in $\Sigma$ by
\begin{subequations}
\begin{align}
\norm{\varphi}_{W_{\gamma}^{\reg}(\Omega)}^2
={}&\int_{\Omega} \rb^{\gamma}\absCDeri{\varphi}{\reg}^2\di\tb\wedge\di \rb \wedge\di^2\mu,\\
\norm{\varphi}_{W_{\gamma}^{\reg}(\Sigma)}^2
={}&\int_{\Hyper} \rb^{\gamma}\absCDeri{\varphi}{\reg}^2\di \rb \wedge\di^2\mu.
\end{align}
\end{subequations}
\end{definition}

\subsection{Spin-weighted spherical harmonic decomposition}
\label{sect:decompIntoModes}

A theory of decomposing a spin $s$ scalar into spin-weighted spherical harmonics is standard. We here follow closely the standard textbook of Penrose--Rindler \cite[Chapter 4]{penroserindlerI}.

Recall that $\{Y_{m,\ell}^{s}(\cos\theta)e^{im\pb}\}_{m,\ell}$, $\ell\geq \abs{s}$ and $m\in\{-\ell. -\ell+1,\ldots, \ell\}$, are the eigenfunctions, called as  the \textquotedblleft{spin-weighted spherical harmonics,\textquotedblright}
 of a self-adjoint operator
$\edthR'\edthR$ with eigenvalues $-\Lambda_{\ell,s}=-(\ell-s)(\ell+s+1)$:
\begin{equation}
\label{eq:eigenvalueSWSHO}
\edthR'\edthR(Y_{m,\ell}^{s}(\cos\theta)e^{im\pb})=
-\Lambda_{\ell,s}
Y_{m,\ell}^{s}(\cos\theta)e^{im\pb},
\end{equation}
and they form a complete orthonormal basis on $L^2(\di^2\mu)$. Further, they satisfy
\begin{subequations}
\label{eq:ellipticop:eigenvalue:fixedmode}
\begin{align}
\edthR (Y_{m,\ell}^{s}(\cos\theta)e^{im\pb})={}&-\sqrt{(\ell-s)(\ell+s+1)}Y_{m,\ell}^{s+1}(\cos\theta)e^{im\pb}, \\
\edthR' (Y_{m,\ell}^{s}(\cos\theta)e^{im\pb})={}& \sqrt{(\ell+s)(\ell-s+1)}Y_{m,\ell}^{s-1}(\cos\theta)e^{im\pb}.
\end{align}
\end{subequations}

The spin-weighted spherical harmonic decomposition for spin $s$ scalars is provided in the following definition.

\begin{definition}[Mode decomposition]
\label{def:fixedmode}
For a spin $s$ scalar $\varphi$, let $(\varphi)^{\ell=\ell_0}$ and $(\varphi)_{m,\ell_0}$, $m\in \{-\ell_0,-\ell_0+1,\ldots, \ell_0\}$, be defined  such that the following decompositions hold in $L^2 (\Sphere)$:
\begin{subequations}
\begin{align}
\varphi=&\sum\limits_{\ell_0=\abs{s}}^{\infty}(\varphi)^{\ell=\ell_0},\\
(\varphi)^{\ell=\ell_0}=&\sum\limits_{m=-\ell_0}^{\ell_0}(\varphi)_{m,\ell_0}
Y_{m,\ell_0}^{s}(\cos\theta)e^{im\pb}.
\end{align}
\end{subequations}
\end{definition}

We collect a few properties for the edth operators $\edthR$ and $\edthR'$.  These are standard facts and can be found in \cite{E82E,penroserindlerI}.

\begin{lemma}
\label{lem:eigenvalue:edthandprime}
\begin{enumerate}[label=\arabic*)]
\item 
Let $\varphi$ be a spin $s$ scalar, then
\begin{align}
\label{eq:l=l0mode:eigenvalue}
\edthR\edthR'(\varphi)^{\ell=\ell_0}
=-(\ell_0+s)(\ell_0-s+1)(\varphi)^{\ell=\ell_0}, \quad
\edthR'\edthR(\varphi)^{\ell=\ell_0}
=-(\ell_0-s)(\ell_0+s+1)(\varphi)^{\ell=\ell_0}.
\end{align}
\item
\label{pt:eigenvalue:edthsquare}
Let $\varphi$ be a spin $s$ scalar, then
\begin{align}
\label{eq:ellipestis}
\int_{S^2}\big(\abs{\edthR'\varphi}^2
-({s+\abs{s}})\abs{\varphi}^2\big) \di^2\mu =\int_{S^2}\big(\abs{\edthR\varphi}^2
-({\abs{s}-s})\abs{\varphi}^2\big) \di^2\mu \geq 0.
\end{align}
If $\varphi$ is a spin $s$ scalar and is supported on $\ell\geq \ell_0$ modes, then
\begin{align}
\label{eq:ellip:highermodes}
\int_{S^2}\big(\abs{\edthR'\varphi}^2
-{(\ell_0+s)(\ell_0-s+1)}\abs{\varphi}^2\big) \di^2\mu
=&\int_{S^2}\big(\abs{\edthR\varphi}^2
-{(\ell_0-s)(\ell_0+s+1)}\abs{\varphi}^2\big) \di^2\mu \geq 0.
\end{align}
\end{enumerate}
\end{lemma}

\begin{proof}
By \eqref{eq:ellipticop:eigenvalue:fixedmode}, it holds that $
\edthR'\edthR(Y_{m,\ell}^{s}(\cos\theta)e^{im\pb})
=-(\ell-s)(\ell+s+1)(Y_{m,\ell}^{s}(\cos\theta)e^{im\pb})$, hence the second formula in \eqref{eq:l=l0mode:eigenvalue} holds true. Similarly for the first formula.

For point \ref{pt:eigenvalue:edthsquare}, it suffices to show \eqref{eq:ellip:highermodes}, since the other equation \eqref{eq:ellipestis} follows by simply setting $\ell_0=\sfrak$ in \eqref{eq:ellip:highermodes}. By assumption, we make a mode decomposition for $\varphi$ as in Definition \ref{def:fixedmode} and obtain $\varphi=\sum\limits_{\ell'=\ell_0}^{\infty}\sum\limits_{m=-\ell'}^{\ell'}(\varphi)_{m,\ell'}
Y_{m,\ell'}^{s}(\cos\theta)e^{im\pb}$, hence, applying $\edthR$ on both sides and using \eqref{eq:ellipticop:eigenvalue:fixedmode}, one finds
\begin{align}
\edthR\varphi=&-\sum\limits_{\ell'=\ell_0}^{\infty}\sum\limits_{m=-\ell'}^{\ell'}\sqrt{(\ell'-s)(\ell'+s+1)}(\varphi)_{m,\ell'}
Y_{m,\ell'}^{s+1}(\cos\theta)e^{im\pb}.
\end{align}
Since the functions $\{Y_{m,\ell}^{s}(\cos\theta)e^{im\pb}\}_{m,\ell}$ 
form a complete orthonormal basis on $L^2(\di^2\mu)$, the formula on $\edthR\varphi$ in \eqref{eq:ellip:highermodes} follows from the Plancherel's lemma. The other formula on $\edthR'\varphi$ in \eqref{eq:ellip:highermodes} follows in a similar manner.
\end{proof}


\subsection{Basic analytic estimates}
\label{sect:HardySobolev}

The following simple Hardy's inequality will be useful.
\begin{lemma}
Let $\varphi$ be a spin $s$ scalar. Then for any $r'>2M$,
\begin{align}
\label{eq:Hardy:trivial}
\int_{2M}^{r'}\abs{\varphi}^2\di r
\lesssim{}&\int_{2M}^{r'}\mu^2r^2\abs{\partial_r\varphi}^2\di r
+(r'-2M)\abs{\varphi(r')}^2.
\end{align}
If, moreover, $\lim\limits_{r\to \infty} r\abs{\varphi}^2 =0$, then
\begin{align}
\label{eq:Hardy:trivial:1}
\int_{2M}^{\infty}\abs{\varphi}^2\di r
\lesssim{}&\int_{2M}^{\infty}\mu^2r^2\abs{\partial_r\varphi}^2\di r.
\end{align}
\end{lemma}

\begin{proof}
It follows easily by integrating the following equation
\begin{align}
\partial_r((r-2M)\abs{\varphi}^2)=\abs{\varphi}^2
+2(r-2M)\Re(\bar{\varphi}\partial_r\varphi)
\end{align}
from $2M$ to $r'$ and applying the Cauchy-Schwarz inequality to the last product term.
\end{proof}

We will also use the following standard Hardy's inequality. Its proof can be found in, for instance, \cite[Lemma 4.30]{andersson2019stability}.
\begin{lemma}
\label{lem:HardyIneq}
Let $\alpha \in \mathbb{R}\setminus \{0\}$  and $h: [r_0,r_1] \rightarrow \mathbb{R}$ be a $C^1$ function.
\begin{enumerate}
\item \label{point:lem:HardyIneqLHS} If $r_0^{\alpha}\vert h(r_0)\vert^2 \leq D_0$ and $\alpha<0$, then
\begin{subequations}
\label{eq:HardyIneqLHSRHS}
\begin{align}\label{eq:HardyIneqLHS}
-2\alpha^{-1}r_1^{\alpha}\vert h(r_1)\vert^2+\int_{r_0}^{r_1}r^{\alpha -1} \vert h(r)\vert ^2 \di r \leq \frac{4}{\alpha^2}\int_{r_0}^{r_1}r^{\alpha +1} \vert \partial_r h(r)\vert ^2 \di r-2\alpha^{-1}D_0;
\end{align}
\item \label{point:lem:HardyIneqRHS} If $r_1^{\alpha}\vert h(r_1)\vert^2 \leq D_0$ and $\alpha>0$, then
\begin{align}\label{eq:HardyIneqRHS}
2\alpha^{-1}r_0^{\alpha}\vert h(r_0)\vert^2+\int_{r_0}^{r_1}r^{\alpha -1} \vert h(r)\vert ^2 \di r \leq \frac{4}{\alpha^2}\int_{r_0}^{r_1}r^{\alpha +1} \vert \partial_r h(r)\vert ^2 \di r +2\alpha^{-1}D_0.
\end{align}
\end{subequations}
\end{enumerate}
\end{lemma}

Recall the following Sobolev-type estimates from \cite[Lemmas 4.32 and 4.33]{andersson2019stability}. These are used to derive pointwise decay estimates from energy decay estimates.

\begin{lemma}
\label{lem:Sobolev}
Let $\varphi$ be a spin $s$ scalar. Then
\begin{align}
\label{eq:Sobolev:1}
\sup_{\Sigmatb}\abs{\varphi}^2\lesssim_{s}{} \norm{\varphi}_{W_{-1}^3(\Sigmatb)}^2.
\end{align}
If $\alpha\in (0,1]$, then
\begin{align}
\label{eq:Sobolev:2}
\sup_{\Sigmatb}\abs{\varphi}^2\lesssim_{s,\alpha} {} \Big(\norm{\varphi}_{W_{-2}^3(\Sigmatb)}^2
+\norm{rV\varphi}_{W_{-1-\alpha}^2(\Sigmatb)}^2\Big)^{\half}
\Big(\norm{\varphi}_{W_{-2}^3(\Sigmatb)}^2
+\norm{rV\varphi}_{W_{-1+\alpha}^2(\Sigmatb)}^2\Big)^{\half}.
\end{align}
If $\lim\limits_{{\tb\to\infty}}\abs{r^{-1}\varphi}=0$ pointwise in $(\rb,\theta,\pb)$, then
\begin{align}
\label{eq:Sobolev:3}
\abs{r^{-1}\varphi}^2\lesssim_{s} {}\norm{\varphi}_{W_{-3}^3(\DOC_{\tb,\infty})}
\norm{\Lxi\varphi}_{W_{-3}^3(\DOC_{\tb,\infty})}.
\end{align}
\end{lemma}

In the end, we present a lemma showing that a hierarchy of estimates implies a rate of decay for the energy in the hierarchy. This is the basic lemma that will be frequently used to derive energy decay estimates.

The way this lemma is stated is in the same spirit of \cite[Lemma 5.2]{andersson2019stability}.\footnote{In fact, the current statement exhibits only  the special and simpler case with $\gamma=0$ in \cite[Lemma 5.2]{andersson2019stability}.} The main ingredient of the proof is an application of the mean-value principle, and we guide the reader to a proper proof therein. In applications, $\reg$ represents a regularity level, $p$ represents a weight that arises from $r^p$ estimates, and $\tb$ represents a time function. In addition, $\regl$ characterizes the potential loss of regularity in the hierarchy of estimates. 

\begin{lemma}[A hierarchy of estimates implies energy decay]
\label{lem:hierarchyImpliesDecay}
Let $D\geq 0$. Let $\regl\geq 0$, let $\reg_0\in\mathbb{N}^+$ be suitably large, let $\tb_0\geq 1$, and let $p_1,p_2\in\mathbb{R}$ be such that $p_1\leq p_2-1$. Let $F:\{0,\ldots,\reg_0\}\times[p_1-1,p_2]\times[\tb_0,\infty)\rightarrow[0,\infty)$ be such that $F(\reg,p,\tb)$ is Lebesgue measurable in $\tb$ for each $p$ and $\reg$. 

If the following hierarchy of estimates hold:
\begin{subequations}
\begin{enumerate}[label=\arabic*)]
\item\label{pt:implydecay:1}{} [monotonicity] \label{assump:HierarchyToDecay(1)}for all $\reg,\reg_1,\reg_2\in\{0,\ldots,\reg_0\}$ with $\reg_1\leq \reg_2$, all $p, p',p''\in[p_1,p_2]$ with $p'\leq p''$, and all $\tb\geq \tb_0$,
\begin{align}
F(\reg_1,p,\tb)\lesssim{}& F(\reg_2,p,\tb) ,
\label{eq:Rev:HierarchyToDecayReal:MonotonicityHypothesis:j}\\
F(\reg, p',\tb)\lesssim{}& F(\reg, p'',\tb) ,
\label{eq:Rev:HierarchyToDecayReal:MonotonicityHypothesis:beta}
\end{align}

\item\label{pt:implydecay:2}{} [interpolation] \label{assump:HierarchyToDecay(2)}for all $\reg\in\{0,\ldots,\reg_0\}$, all $p, p', p''\in[p_1,p_2]$ such that $ p'\leq p \leq p''$, and all $\tb\geq \tb_0$,
\begin{align}
F(\reg,p,\tb)
\lesssim{}&
F(\reg, p',\tb)^{\frac{ p''-p}{ p''- p'}}
F(\reg, p'',\tb)^{\frac{p- p'}{ p''- p'}} ,
\label{eq:Rev:HierarchyToDecayReal:InterpolationHypothesis}
\end{align}

\item\label{pt:implydecay:3}{} [energy and Morawetz estimate] for all $\reg\in\{0,\ldots,\reg_0-\regl\}$, $p\in[p_1,p_2]$, and $\tb_2>\tb_1\geq \tb_0$,
\begin{align}
F(\reg,p,\tb_2)
+\int_{\tb_1}^{\tb_2} F(\reg-\regl,p-1,\tb) \di \tb
\lesssim F(\reg+\regl,p,\tb_1) +\tb_1^{p-p_2}D,
\label{eq:Rev:HierarchyToDecayReal:EvolutionHypothesis}
\end{align}
\end{enumerate}
\end{subequations}
then there exists a constant $C>0$ such that for all $\reg\in\{0,\ldots,\reg_0-C\regl\}$, all $p\in[p_1,p_2]$, and all $\tb_2>\tb_1\geq \tb_0$,
\begin{align}
F(\reg,p,\tb_2) \lesssim_{p_2, p_1}{}& \langle \tb_2-\tb_1\rangle^{p-p_2}  (F(\reg+C\regl,p_2,\tb_1) +D).
\label{eq:Rev:HierarchyToDecayReal}
\end{align}
\end{lemma}

\subsection{An $r^p$ lemma for wave equations}
\label{sect:rplemma}

We state here an $r^p$ estimate for a general spin-weighted wave equation, which is an analog of the one first proven for spin-$0$ scalar field in \cite{dafermos2009new}. See also \cite{andersson2019stability,MaZhang20sharp} where $r^p$ estimates are derived for a general spin-weighted wave equation on Schwarzschild and Kerr backgrounds. 

\begin{lemma}[$r^p$ lemma for wave equations]
\label{lem:wave:rp}
Let $\reg\in \mathbb{N}$, $\sfrak\in  \mathbb{N}$, and $p\in \mathbb{R}^+\cup\{0\}$. Let  $0<\delta, \varsigma, \veps<1/2$ be arbitrary. Let $\varphi$ and $\vartheta=\vartheta(\varphi)$ be spin $s$ scalars satisfying
\begin{align}
\label{eq:wave:rp}
-r^2YV\varphi
+(\edthR\edthR'-b_0)\varphi
-b_V V\varphi =\vartheta.
\end{align}
Let $b_{V}$ and $b_0$ be smooth real functions of $r$ such that
\begin{enumerate}[label=\arabic*)]
\item\label{assu:rplemma:1} $\exists b_{V,-1}\in \mathbb{R}^+\cup\{0\}$ such that $b_V=b_{V,-1} r +O(1)$, and
\item\label{assu:rplemma:2} $\exists b_{0,0}\in \mathbb{R}$ such that $b_0=b_{0,0}+O(r^{-1})$ and the eigenvalues of $\edthR\edthR'-b_{0,0}$ are non-positive, i.e. $\int_{\mathbb{S}^2}(\abs{\edthR'\varphi}^2+ b_{0,0} \abs{\varphi}^2) \di^2\mu\geq 0$.
\end{enumerate}

Then 
\begin{enumerate}[label=\roman*)]
\item\label{pt:rplemma:1}
there is a constant $\hat{R}_0=\hat{R}_0(p,b_0,b_V)$ such that for all $R_0\geq \hat{R}_0$ and $\tb_2>\tb_1\geq \tb_0$,
\begin{itemize}
\item for $p=0$ and $b_{V,-1}>0$,
\begin{subequations}
\label{eq:rp:0to2:full}
\begin{align}\label{eq:rp:p=0}
\norm{\varphi}^2_{W_{-2}^{\reg+1}(\Sigmatwo^{R_0})}
+\norm{\varphi}^2_{W_{-3}^{\reg+1}(\Donetwo^{R_0})}
&\lesssim_{[R_0-M,R_0]} {}\norm{\varphi}^2_{W_{-2}^{\reg+1}(\Sigmaone^{R_0})}
+\norm{\vartheta}^2_{W_{-3}^{\reg}(\Donetwo^{R_0})};
\end{align}
\item for $p=0$ and $b_{V,-1}=0$,
\begin{align}\label{eq:rp:p=0:v2}
\norm{\varphi}^2_{W_{-2}^{\reg+1}(\Sigmatwo^{R_0})}
+\norm{\varphi}^2_{W_{-3}^{\reg+1}(\Donetwo^{R_0})}
&\lesssim_{[R_0-M,R_0]} {}\norm{\varphi}^2_{W_{-2}^{\reg+1}(\Sigmaone^{R_0})}
+\norm{rV\varphi}^2_{W_{-3}^{\reg}(\Donetwo^{R_0})}
+\norm{\vartheta}^2_{W_{-3}^{\reg}(\Donetwo^{R_0})};
\end{align}
\item for $p\in (0, 2)$,
    \begin{align}\label{eq:rp:less2:2}
\hspace{2ex}&\hspace{-2ex}
\norm{rV\varphi}^2_{W_{p-2}^\reg(\Sigmatwo^{\geq R_0})}
+\norm{\varphi}^2_{W_{-2}^{\reg+1}(\Sigmatwo^{\geq R_0})}
+\norm{\varphi}^2_{W_{p-3}^{\reg+1}(\Donetwo^{\geq R_0})}
+\norm{Y\varphi}^2_{W_{-1-\delta}^{\reg}(\Donetwo^{\geq R_0})}
\notag\\
&\lesssim_{[R_0-M,R_0]} {}\norm{rV\varphi}^2_{W_{p-2}^\reg(\Sigmaone^{\geq R_0})}
+\norm{\varphi}^2_{W_{-2}^{\reg+1}(\Sigmaone^{\geq R_0})}
+\norm{\vartheta}^2_{W_{p-3}^{\reg}(\Donetwo^{\geq R_0})}
;
\end{align}
\item for $p=2$, both of the following two estimates hold: 
  \begin{align}\label{eq:rp:lp=2:2}
\hspace{2ex}&\hspace{-2ex}
\norm{rV\varphi}^2_{W_{0}^\reg(\Sigmatwo^{\geq R_0})}
+\norm{\varphi}^2_{W_{-2}^{\reg+1}(\Sigmatwo^{\geq R_0})}
+\norm{\varphi}^2_{W_{-1-\delta}^{\reg+1}(\Donetwo^{\geq R_0})}
+\norm{rV\varphi}^2_{W_{-1}^{\reg}(\Donetwo^{\geq R_0})}
\notag\\
&\lesssim_{[R_0-M,R_0]} {}\norm{rV\varphi}^2_{W_{0}^\reg(\Sigmaone^{\geq R_0})}
+\norm{\varphi}^2_{W_{-2}^{\reg+1}(\Sigmaone^{\geq R_0})}
+\norm{\vartheta}^2_{W_{-1}^{\reg}(\Donetwo^{\geq R_0})};\\
\label{eq:rp:lp=2:2:v22}
\hspace{2ex}&\hspace{-2ex}
\norm{rV\varphi}^2_{W_{0}^\reg(\Sigmatwo^{\geq R_0})}
+\norm{\varphi}^2_{W_{-2}^{\reg+1}(\Sigmatwo^{\geq R_0})}
+\norm{\varphi}^2_{W_{-1-\delta}^{\reg+1}(\Donetwo^{\geq R_0})}
+\norm{rV\varphi}^2_{W_{-1}^{\reg}(\Donetwo^{\geq R_0})}
\notag\\
&\lesssim_{[R_0-M,R_0]} {}\norm{rV\varphi}^2_{W_{0}^\reg(\Sigmaone^{\geq R_0})}
+\norm{\varphi}^2_{W_{-2}^{\reg+1}(\Sigmaone^{\geq R_0})}
\notag\\
&
{}+\veps\int_{\tb_1}^{\tb_2}\langle\tb-\tb_1\rangle^{-1-\varsigma}\norm{rV\varphi}^2_{W_{0}^\reg(\Sigmatb^{\geq R_0})}\di \tb
+\frac{1}{\veps}\int_{\tb_1}^{\tb_2}\langle\tb-\tb_1\rangle^{1+\varsigma}\norm{\vartheta}^2_{W_{-2}^{\reg}(\Sigmatb^{\geq R_0})}
\di \tb;
\end{align}
\end{subequations}
\end{itemize}
\item
\label{pt:rplemma:2}
if, in addition, the eigenvalue of of $\edthR\edthR'-b_{0,0}$ acting on $\varphi$ vanishes, both of the following estimates hold for $p\geq 2$:
\begin{align}\label{eq:rp:pleq4:2}
\hspace{1ex}&\hspace{-1ex}
\norm{rV\varphi}^2_{W_{p-2}^\reg(\Sigmatwo^{\geq R_0})}
+\norm{\varphi}^2_{W_{-2}^{\reg+1}(\Sigmatwo^{\geq R_0})}
+\norm{\varphi}^2_{W_{-1-\delta}^{\reg+1}(\Donetwo^{\geq R_0})}
+\norm{rV\varphi}^2_{W_{p-3}^{\reg}(\Donetwo^{\geq R_0})}
\notag\\
&\lesssim_{[R_0-M,R_0]} {}\norm{rV\varphi}^2_{W_{p-2}^\reg(\Sigmaone^{\geq R_0})}
+\norm{\varphi}^2_{W_{-2}^{\reg+1}(\Sigmaone^{\geq R_0})}\notag\\
&
+\Big(\norm{\vartheta}^2_{W_{p-3}^{\reg}(\Donetwo^{\geq R_0})}+\norm{\varphi}^2_{W_{p-5}^{\reg}(\Donetwo^{\geq R_0})}\Big);\\
\label{eq:rp:pgeq4:2}
\hspace{1ex}&\hspace{-1ex}
\norm{rV\varphi}^2_{W_{p-2}^\reg(\Sigmatwo^{\geq R_0})}
+\norm{\varphi}^2_{W_{-2}^{\reg+1}(\Sigmatwo^{\geq R_0})}
+\norm{\varphi}^2_{W_{-1-\delta}^{\reg+1}(\Donetwo^{\geq R_0})}
+\norm{rV\varphi}^2_{W_{p-3}^{\reg}(\Donetwo^{\geq R_0})}
\notag\\
&\lesssim_{[R_0-M,R_0]} {}\norm{rV\varphi}^2_{W_{p-2}^\reg(\Sigmaone^{\geq R_0})}
+\norm{\varphi}^2_{W_{-2}^{\reg+1}(\Sigmaone^{\geq R_0})}
\notag\\
&
{}+\veps\int_{\tb_1}^{\tb_2}\frac{1}{\langle\tb-\tb_1\rangle^{1+\varsigma}}\norm{rV\varphi}^2_{W_{p-2}^\reg(\Sigmatb^{\geq R_0})}\di \tb
+\frac{1}{\veps}\int_{\tb_1}^{\tb_2}\langle\tb-\tb_1\rangle^{1+\varsigma}\big(\norm{\vartheta}^2_{W_{p-4}^{\reg}(\Sigmatb^{\geq R_0})}
+\norm{\varphi}^2_{W_{p-6}^{\reg}(\Sigmatb^{\geq R_0})}\big)\di \tb.
\end{align}
\end{enumerate}
In all the above estimates, the terms $\norm{\varphi}^2_{W_{0}^{\reg+1}(\Sigmatwo^{R_0-M,R_0})}
+\norm{\varphi}^2_{W_{0}^{\reg+1}(\Sigmaone^{R_0-M,R_0})} +\norm{\varphi}^2_{W_{0}^{\reg+1}(\Donetwo^{R_0-M,R_0})}
+\norm{\vartheta}^2_{W_{0}^{\reg}(\Donetwo^{R_0-M, R_0})}$ supported on $[R_0-M, R_0]$ are implicit in the symbol $\lesssim_{[R_0-M, R_0]}$.
\end{lemma}

\begin{proof}
The estimates \eqref{eq:rp:0to2:full} in point \ref{pt:rplemma:1} are the same as the ones proven in \cite[Proposition 2.16]{MaZhang20sharp}. Here, we provide only an outline of the proof.
To describe the main idea, it suffices to prove \eqref{eq:rp:less2:2}, since the other ones are proven in a similar manner. Further, we only consider $\reg=0$ case for the estimate \eqref{eq:rp:less2:2}, and the general $\reg\in\mathbb{N}$ case follows by a standard approach of commuting the operator set $\CDeri$ with the wave equation.

The estimate \eqref{eq:rp:less2:2}  is achieved via the following steps. Recall that $\chi$ is a smooth cutoff function that satisfies $\chi=1$ for $r\geq R_0$ and $\chi$ vanishes identically for $r\leq R_0-M$. We multiply equation \eqref{eq:wave:rp} by $-2\chi^2 r^{p-2}V\bar{\varphi}$, take the real part and integrate over spheres. Then, based on the following calculations:
\begin{align*}
\int_{S^2}\Re\big((-r^2 YV\varphi)(-2\chi^2 r^{p-2}V\bar{\varphi})\big)\di^2\mu={}&\int_{S^2}\big(Y(\chi^2 r^p \abs{V\varphi}^2)+\partial_r (\chi^2 r^p)\abs{V\varphi}^2\big)\di^2\mu,\\
\int_{S^2}\Re\big((\edthR\edthR'\varphi-b_0\varphi)(2\chi^2 r^{p-2}V\bar{\varphi})\big)\di^2\mu={}&\int_{S^2}V\big(\chi^2 r^{p-2}(\abs{\edthR'\varphi}^2+b_0\abs{\varphi}^2)\big)\di^2\mu\notag\\
&
-\int_{S^2}\big(\mu \partial_r (\chi^2r^{p-2})\abs{\edthR'\varphi}^2 +\mu\partial_r (b_0\chi^2 r^{p-2})\abs{\varphi}^2\big)\di^2\mu,\\
\int_{S^2}\Re\big((-b_V V\varphi)(-2\chi^2 r^{p-2}V\bar{\varphi})\big)\di^2\mu={}&\int_{S^2}2\chi^2 b_Vr^{p-2}\abs{V\varphi}^2 \di^2\mu,
\end{align*}
we arrive at
\begin{align}
\label{eq:rplemma:0to2:general:mul}
&\int_{S^2}\Big(V(\chi^2 r^{p-2}(\abs{ \edthR'\varphi}^2
+b_{0}\abs{\varphi}^2))
+Y(\chi^2r^p\abs{V\varphi}^2)\Big)\di^2\mu
\notag\\
&
+\int_{S^2}\Big((\partial_r(\chi^2r^p)
+2\chi^2 b_{V}r^{p-2})\abs{V\varphi}^2
-\mu\partial_r(\chi ^2r^{p-2})
(\abs{\edthR'\varphi}^2
+b_{0}\abs{\varphi}^2)
-\mu\partial_r b_0 \chi^2 r^{p-2} \abs{\varphi}^2\Big)\di^2\mu\notag\\
&
={}-\int_{S^2} 2\chi^2 r^{p-2}\Re(V\overline{\varphi}\vartheta)\di^2\mu.
\end{align}
By assumption, the coefficient of $\abs{V\varphi}^2$  satisfies $\partial_r(\chi^2r^p)
+2\chi^2 b_{V}r^{p-2}\geq \frac{1}{2}pr^{p-1}$ in $[R_0, +\infty)$ for $R_0$ suitably large, and $\int_{S^2}(\abs{ \edthR'\varphi}^2
+b_{0}\abs{\varphi}^2)\geq -C r^{-1}\int_{S^2}\abs{\varphi}^2$. 
Therefore, by integrating this equation over the quotient space ${\Donetwo}\slash{S^2}$, adding into a standard Morawetz estimate near infinity and applying the Cauchy--Schwarz inequality to the integral of the last line of \eqref{eq:rplemma:0to2:general:mul}, the estimate \eqref{eq:rp:less2:2} then follows.

It remains to show point \ref{pt:rplemma:2}. 
If the eigenvalue of $\edthR\edthR'-b_{0,0}$ acting on $\varphi$ vanishes, the wave equation \eqref{eq:wave:rp} reduces to
\begin{align}
\label{eq:wave:rp:reduce}
-r^2YV\varphi
-b_V V\varphi =\vartheta+(b_0-b_{0,0})\varphi=\vartheta +O(r^{-1})\varphi\doteq \tilde{\vartheta}.
\end{align}
Applying the same approach outlined above by multiplying  equation \eqref{eq:wave:rp:reduce} by $-2\chi^2 r^{p-2}V\bar{\varphi}$ and taking the real part, we obtain
\begin{align}
Y(\chi^2r^p\abs{V\varphi}^2)
+(\partial_r(\chi^2r^p)
+2\chi^2 b_{V}r^{p-2})\abs{V\varphi}^2
=-2\chi^2 r^{p-2}\Re(V\overline{\varphi}\tilde{\vartheta}).
\end{align}
Then, by integrating this equation over the spacetime region $\Donetwo$ and adding into a standard Morawetz estimate near infinity, one arrives at
\begin{align}
\label{eq:rp:p=2:2:errorterm}
\hspace{2ex}&\hspace{-2ex}
\norm{rV\varphi}^2_{W_{p-2}^\reg(\Sigmatwo^{\geq R_0})}
+\norm{\varphi}^2_{W_{-2}^{\reg+1}(\Sigmatwo^{\geq R_0})}
+\norm{\varphi}^2_{W_{-1-\delta}^{\reg+1}(\Donetwo^{\geq R_0})}
+\norm{rV\varphi}^2_{W_{p-3}^{\reg}(\Donetwo^{\geq R_0})}
\notag\\
&\lesssim_{[R_0-M,R_0]} {}\norm{rV\varphi}^2_{W_{p-2}^\reg(\Sigmaone^{\geq R_0})}
+\norm{\varphi}^2_{W_{-2}^{\reg+1}(\Sigmaone^{\geq R_0})}
+\sum\limits_{\abs{\mathbf{a}}\leq \reg}\bigg|\int_{\Donetwo^{R_0}}r^{p-2}\Re\Big(
V\overline{\mathbb{D}^{\mathbf{a}}\varphi}
\mathbb{D}^{\mathbf{a}}\tilde{\vartheta}\Big) \di^4 \mu\bigg|.
\end{align}
Here, the general $\reg\in\mathbb{N}$ case is due to the fact that the obtained equation, after commuting equation \eqref{eq:wave:rp:reduce} with the operator set $\{\Lxi, \edthR,\edthR', rV\}$, is of the same form as \eqref{eq:wave:rp:reduce}  and satisfies also the assumptions \ref{assu:rplemma:1} and \ref{assu:rplemma:2}.
The estimates \eqref{eq:rp:pleq4:2} and \eqref{eq:rp:pgeq4:2}  then both follow by applying the Cauchy--Schwarz inequality to the last term of \eqref{eq:rp:p=2:2:errorterm}.
\end{proof}


\section{Spin $s$ components, Teukolsky master equation and Teukolsky--Starobinsky identities}
\label{sect:TMETSI}



In this section, we define a few scalars that are constructed from the spin $s$ components, introduce the Teukolsky master equation, and finally state the Teukolsky--Starobinsky identities relating the spin $+\sfrak$ and $-\sfrak$ components. As an aside, we write down the full Maxwell system of equations.


\subsection{Teukolsky master equation}
\label{sect:tme}

Teukolsky \cite{Teukolsky1973I} found the celebrated \textit{Teukolsky Master Equation} (TME) that is the governing equation for the spin-$s$ fields on a Schwarzschild background. 

\begin{prop}[Teukolsky master Equation \cite{Teukolsky1973I}]
\label{prop:TME}
The rescaled spin $s$ components $\psi_s$, $s=\pm \sfrak$, defined by 
\begin{align}
\psipluss={}&r^{2\sfrak}\NPRpluss, \qquad\qquad
\psiminuss=\NPRminuss,
\end{align}
satisfy the 
\textit{Teukolsky Master Equation}, a separable, decoupled wave equation which takes the following form in Boyer-Lindquist coordinates:
\begin{align}\label{eq:TME}
\hspace{4ex}&\hspace{-4ex}
\left(-\mu^{-1} r^2\partial^2_t
+ \partial_r \left( \Delta \partial_r \right)
+\frac{1}{\sin^2{\theta}} \partial^2_{\phi}+ \frac{1}{\sin{\theta}} \partial_\theta \left( \sin{\theta} \partial_\theta\right)+ \frac{2is\cos\theta}{\sin^2 \theta}\partial_{\phi} - (s^2\cot^2\theta +s)\right)\psi_s \notag\\
=&-2s((r-M)Y-2r\partial_t)\psi_s.
\end{align}
\end{prop}

This is a spin-weighted wave equation in the sense that the operator on the LHS of \eqref{eq:TME}, being equal to $-\mu^{-1} r^2\partial^2_t
+ \partial_r ( \Delta \partial_r )+\edthR \edthR'$, is a spin-weighted second-order wave operator.  Such a TME is actually derived  first by Bardeen--Press \cite{Bardeen73TMESchw} on a Schwarzschild background and later by Teukolsky \cite{Teukolsky1973I} for general half integer spin fields on the larger Kerr family of spacetimes \cite{kerr63}, and it serves as a starting model for quite many results  in obtaining quantitative estimates for these fields, including  the scalar field, the Maxwell field and the linearized gravity.


\subsection{Spin $s$ components}
\label{sect:spinscomp:all}

For convenience of latter discussions, we define in this subsection a few sets of scalars that are in turn defined from the spin $s$ components $\NPR_s$. 

\begin{definition}[Scalars constructed from the spin $s$ components]
\label{def:PsiPhietc}
\begin{enumerate}
\item
Recall from above equation \eqref{eq:spinsfields} that the scalars $\psi_s$, $s=\pm \sfrak$, are defined by
\begin{align}\label{eq:spinsfields}
\psipluss={}&r^{2\sfrak}\NPRpluss, \qquad\qquad
\psiminuss=\NPRminuss.
\end{align}

\item
Define the  radiation fields of the spin $s$ components $\psi_s$ $(s=\pm\sfrak)$ by 
\begin{align}
\label{def:Psiplusminuss}
\Psipluss={}& r\psipluss, \qquad\qquad
\Psiminuss= r\psiminuss.
\end{align}

\item
Let $i\in \mathbb{N}$. Define the following spin $s$ scalars\footnote{Here, the upper index $i$ inside the round brackets indicates how many times we apply the operator $\curlVR$ to $\Phi_s^{(0)}$.}: 
\begin{subequations}
\label{def:Phiplusminus:High}
\begin{align}
\label{def:PhiplussHigh0}
\PhiplussHigh{0}={}&\mu^{-\sfrak}\Psipluss=\mu^{-\sfrak}r\psipluss, \qquad\qquad \PhiplussHigh{i}={}\curlVR^i\PhiplussHigh{0}=(r^2\VR)^i \PhiplussHigh{0},\\
\label{eq:Phiminusi:def}
\Phiminuss{0}={}&\mu^{\sfrak}\Psiminuss=\mu^{\sfrak}  r \psiminuss,\qquad\qquad\, \,\quad \Phiminuss{i}={}\curlVR^i \Phiminuss{0}=(r^2\VR)^i \Phiminuss{0}.
\end{align}
\end{subequations}

\item 
Let $i\in \mathbb{N}$. Define the following scalars constructed from the spin $+ \sfrak$ component:
\begin{align}\label{eq:DefOfphi012PosiSpinS2}
    \Xi^{(0)}_{+\sfrak}={}&r^{1-2\sfrak}\psipluss, \qquad\qquad
\Xi^{(i)}_{+\sfrak}={}(-r^2 Y)^i\Xi^{(0)}_{+\sfrak}
\end{align}
\end{enumerate}
\end{definition}

\begin{remark}
 The first set \eqref{eq:spinsfields} of scalars are the ones solving the TME as shown in Proposition \ref{prop:TME}. In the second and third sets of scalars, we have used the uppercase Greek letters to indicate that they have incorporated the suitable $r$ weights such that they are regular, non-degenerate scalars near future null infinity. The second set \eqref{def:Psiplusminuss} consists of the scalars that are the radiation fields (by a pure $r$ scaling) of the ones in the first set.  The third set \eqref{def:Phiplusminus:High} consists of scalars that are used to derive the extended wave systems in Sections \ref{sect:BEAM} and \ref{sect:extendedwavesystem}. The scalars in the last set \eqref{eq:DefOfphi012PosiSpinS2} constructed from the spin $+\sfrak$ component satisfy a basic wave system for which the energy and Morawetz estimates are derived; see Section  \ref{sect:BEAM}.
\end{remark}

\begin{remark}
The first two sets and the last set of scalars are nondegenerate and regular near $\Horizon$. By Remark \ref{rem:YVsoon}, we have $\VR={}\partial_{\rb}+\Hhyp \Lxi\sim {}\partial_{\rb}+2\mu^{-1} \Lxi$ as $r\to r_+$, hence the scalars $\PhiplussHigh{i}$ are in fact singular near $\Horizon$. Nevertheless, we only do estimates for them in a region away from $\Horizon$. Similarly, the scalars $\Phiminuss{i}$ behave like $O(\mu^{\sfrak-i})$ towards $\Horizon$,  and the estimates for these scalars are again only performed in a region away from $\Horizon$.
\end{remark}

In addition, we define tilde scalars $\widetilde{\Phi}_{s}^{(i)}$, each of which is constructed to be a linear combination of the scalar $\{\Phi_s^{(i')}\}_{0\leq i'\leq i}$. As has already been discussed in Section \ref{intro:outlineproof}, these scalars are of paramount importance in both proving the almost sharp decay estimates and defining the Newman--Penrose constants for the spin $\pm\sfrak$ components!

\begin{definition}
\label{def:tildePhiplusandminusHigh}
Let $i\in \mathbb{N}$. Let $f_{\sfrak, i,1}=i(i+2\sfrak+1)$, $f_{\sfrak, i,2}=-(f_{\sfrak, i+1,1}-f_{\sfrak, i,1})=-2(\sfrak+i+1)$, $f_{\sfrak, i,3}=f_{\sfrak, i,1}+\frac{1}{3}(\sfrak+1)(2\sfrak+1)$, and $g_{\sfrak, i}=6 \sum\limits_{j=0}^{i-1} f_{\sfrak, j,3}= 2i(i+\sfrak)(i+2\sfrak)$.  Let $x_{\sfrak, i+1,i}=\frac{g_{\sfrak, i+1}}{f_{\sfrak, i+1,1} -f_{\sfrak, i,1}}=(i+1)(i+2\sfrak+1)$ and $x_{\sfrak, i+1,j}=-\frac{g_{\sfrak, i+1}x_{\sfrak, i,j}}{f_{\sfrak, i+1,1}-f_{\sfrak, j,1}}$ for $0\leq j\leq i-1$. Define the following scalars constructed from the scalars $\Phi_s^{(i)}$:
\begin{subequations}
\begin{align}
\tildePhiplussHigh{0}={}&\PhiplussHigh{0},& \tildePhiplussHigh{i+1}={}&\PhiplussHigh{i+1}
+\sum_{j=0}^ix_{\sfrak, i+1,j}M^{i+1-j}\tildePhiplussHigh{j},\\
\tildePhiminuss{2\sfrak}={}&\Phiminuss{2\sfrak},& \tildePhiminuss{i+2\sfrak+1}={}&\Phiminuss{i+2\sfrak+1}
+\sum_{j=0}^ix_{\sfrak, i+1,j}M^{i+1-j}\tildePhiminuss{j+2\sfrak}.
\end{align}
\end{subequations}
\end{definition}

All the scalars defined above are spin $s$ scalars and, without confusion, we may call all of the above spin $s$ scalars, constructed from the spin $s$ components $\NPR_s$, as the \textit{spin $s$ components} as well.

\subsection{Teukolsky--Starobinsky identities}
\label{sect:TSI}

It is surprising that the spin $\pm \sfrak$ components can actually be related to each other by purely differential identities--the \emph{Teukolsky-Starobinsky identities} (TSI), originating from \cite{TeuPress1974III,starobinsky1973amplification}. See also the covariant form of these identities in \cite{aksteiner2019new}. We state explicitly these TSI in terms of our notations, and these identities will be crucial and applied throughout the rest sections. 

\begin{lemma}[Teukolsky--Starobinsky identities]
\label{lem:TSI:gene}
\begin{enumerate}
\item
There are the following TSI for the spin $\pm 1$ components of the Maxwell field
\begin{subequations}
\begin{align}
\label{eq:TSI:simpleform}
(\edthR')^2(\Delta^{-1}\psiplus)
={}&\VR^2(\Delta\psiminus),\\
\label{eq:otherTSI:simpleform}
(\edthR)^2 \psiminus={}&
Y^2 \psiplus.
\end{align}
\end{subequations}
The first one \eqref{eq:TSI:simpleform} can also be written as
\begin{align}
\label{eq:TSI:simpleform:v2}
\Phiminus{2}={}(\edthR')^2 \Phiplus.
\end{align}
If restricted to a fixed $\ell$ mode, then equations \eqref{eq:TSI:simpleform} become
\begin{subequations}
\begin{align}
\label{eq:TSI:simpleform:l=1}
\ell(\ell+1)\Phiplus={}&\Phiminus{2},\\
\label{eq:otherTSI:simpleform:l=1}
\ell(\ell+1)\psiminus={}&  Y^2 \psiplus.
\end{align}
\end{subequations}
\item
There are the following TSI for the spin $\pm 2$ components of the linearized gravity:
\begin{subequations}
\label{eq:TSIspin2}
\begin{align}
\label{eq:TSIspin2:V-2}
\mu^2 (\VR)^4 (\mu^2 r^4\psiminustwo)={}& (\edthR')^4 (r^{-4}\psiplustwo ) -12 M \overline{\Lxi(r^{-4}\psiplustwo)},\\
\label{eq:TSIspin2:Y+2}
Y^4 (\psiplustwo)={}& (\edthR)^4 \psiminustwo +12 M\overline{\Lxi\psiminustwo}.
\end{align}
\end{subequations}
Moreover, equation \eqref{eq:TSIspin2:V-2} can also be written as
\begin{align}
\label{eq:TSIspin2:Phiminustwo4}
\Phiminustwo{4}={}& (\edthR')^4 \Phiplustwo  -12 M \overline{\Lxi\Phiplustwo}.
\end{align}
If restricted to a fixed $\ell$ mode, then equations \eqref{eq:TSIspin2} become
\begin{subequations}
\begin{align}
\label{eq:TSI:simpleform:l=2}
(\ell-1)\ell (\ell+1)(\ell+2)\Phiplustwo={}&\Phiminustwo{4}+12 M \overline{\Lxi\Phiplustwo},\\
\label{eq:otherTSI:simpleform:l=2}
(\ell-1)\ell (\ell+1)(\ell+2)\psiminustwo={}&  Y^4 \psiplustwo-12 M\overline{\Lxi\psiminustwo}.
\end{align}
\end{subequations}
\end{enumerate}
\end{lemma}

\begin{proof}
For the spin $\pm 1$ components, these TSI can be derived from the Maxwell system of equations \eqref{eq:TSIsSpin1Kerr}: the identity \eqref{eq:otherTSI:simpleform} can be obtained by applying $\half r^{-1}Y$ to \eqref{eq:TSIsSpin1KerrAngular0With1} and using \eqref{eq:TSIsSpin1KerrRadial0With-1};
the other identity \eqref{eq:TSI:simpleform} can be obtained from \eqref{eq:otherTSI:simpleform} by interchanging $l$ and $n$ and $\edthR$ and $\edthR'$. Equation \eqref{eq:TSI:simpleform:v2} is straightforward from identity \eqref{eq:TSI:simpleform}. When considering a fixed $\ell$ mode, equations \eqref{eq:TSI:simpleform:l=1} and \eqref{eq:otherTSI:simpleform:l=1} are immediate from  \eqref{eq:TSI:simpleform:v2} and \eqref{eq:otherTSI:simpleform}, respectively, together with the eigenvalues of $(\edthR')^2$ and $(\edthR)^2$ from
\eqref{eq:ellipticop:eigenvalue:fixedmode}.

For the spin $\pm 2$ components, the equations in \cite[Equations (3.26)]{andersson2019stability} can be rewritten as
\begin{subequations}
\begin{align}
Y^4 (\kappa_1^4\NPRplustwo)={}&r^{-4}(\edthR)^4(\kappa_1^4\NPRminustwo)
+\frac{M}{27}\Lxi \overline{\NPRminustwo},\\
(\VR)^4(\mu^2\kappa_1^4\NPRminustwo)={}&\mu^{-2}\bigg(r^{-4}(\edthR')^4(\kappa_1^4\NPRplustwo)-\frac{M}{27}\Lxi \overline{\NPRplustwo}\bigg),
\end{align}
\end{subequations}
where $\kappa_1=-\frac{1}{3}r$ in Schwarzschild spacetime. By substituting in $\NPRplustwo=r^{-4}\psiplustwo$ and $\NPRminustwo=\psiminustwo$, one arrives at equations
\eqref{eq:TSIspin2}.

Using Definition \ref{def:PsiPhietc}, equation \eqref{eq:TSIspin2:V-2} becomes
\begin{align}
 r^5 (\VR)^4 (r^3\Phiminustwo{0})={}& (\edthR')^4 \Phiplustwo  -12 M \overline{\Lxi\Phiplustwo}.
\end{align}
Expanding out the LHS and using Definition \ref{def:PsiPhietc} again, one finds the LHS equals $\Phiminustwo{4}$, hence completing the proof of \eqref{eq:TSIspin2:Phiminustwo4}. If restricted to a fixed $\ell$ mode, since we have from \eqref{eq:ellipticop:eigenvalue:fixedmode} that the eigenvalues of $(\edthR')^4$ acting on the spin $+2$ component and of $(\edthR)^4$ acting on the spin $-2$ component are both $(\ell-1)\ell (\ell+1)(\ell+2)$, equations \eqref{eq:TSI:simpleform:l=2} and \eqref{eq:otherTSI:simpleform:l=2} are easily justified from  \eqref{eq:TSIspin2:Y+2} and \eqref{eq:TSIspin2:V-2}.
\end{proof}


\subsection{The full Maxwell system of equations}

The full system of the Maxwell equations \eqref{eq:MaxwellEqs} can be written as a system of first-order differential equations for the spin $s$ components:
\begin{subequations}\label{eq:TSIsSpin1Kerr}
\begin{align}
\label{eq:TSIsSpin1KerrAngular0With1}
\edthR \NPRzero={}& 2 Y\left(r^{-1}{\psiplus}\right),\\
\label{eq:TSIsSpin1KerrAngular0With-1}
\edthR' \NPRzero={}&2\mu^{-1}  V\left(\mu r\psiminus\right),\\
\label{eq:TSIsSpin1KerrRadipal0With1}
V \NPRzero=
{}&2\edthR'\left(r^{-3}{\psiplus}\right),\\
\label{eq:TSIsSpin1KerrRadial0With-1}
Y\NPRzero
={}&2\edthR
\left(r^{-1}{\psiminus}\right).
\end{align}
\end{subequations}

\begin{definition}
Let $\leftidx^{\star} \mathbf{F}$ be the Hodge dual of the Maxwell field $\mathbf{F}$. 
Define the electronic and magnetic charges of a Maxwell field by
\begin{align}
q_{\mathbf{E}}={}&\frac{1}{4\pi}\int_{S^2(\tb,\rb)}\leftidx^{\star} \mathbf{F}, &q_{\mathbf{B}}={}&\frac{1}{4\pi}\int_{S^2(\tb,\rb)} \mathbf{F}.
\end{align}
\end{definition}

The following lemma is a standard statement and is taken from \cite[Proposition 2]{larsblue15Maxwellkerr}. See also  \cite{Ma20almost}. This is to decompose a Maxwell field into a static part and a radiative part.

\begin{lemma}
\label{lem:decomp:Maxwellfield}
The Maxwell field in a Schwarzschild spacetime can be decomposed into
\begin{align}
\mathbf{F}=\mathbf{F}_{\text{sta}}+
       \mathbf{F}_{\text{rad}},
\end{align}
where the part $\mathbf{F}_{\text{rad}}$ is the non-charged radiative part of the Maxwell field and the other part $\mathbf{F}_{\text{sta}}$ is the charged static Coulomb part,
such that
\begin{enumerate}
\item $\mathbf{F}_{\text{sta}}$ and $\mathbf{F}_{\text{rad}}$ are both solutions to the Maxwell equations, and the N-P components of them satisfy
\begin{subequations}
\label{eq:FradandFsta:property}
\begin{align}
\NPRplus(\mathbf{F}_{\text{sta}})
={}&\NPRminus(\mathbf{F}_{\text{sta}})=0,&
\NPRzero(\mathbf{F}_{\text{sta}})
={}&r^{-2}
\left(q_{\mathbf{E}} + iq_{\mathbf{B}}\right),\\
\NPRplus(\mathbf{F}_{\text{rad}})={}&\NPRplus(\mathbf{F}),&
\NPRminus(\mathbf{F}_{\text{rad}})={}&\NPRminus(\mathbf{F});
\end{align}
\end{subequations}
\item The charges $q_{\mathbf{E}}$ and $q_{\mathbf{B}}$ are constants at all spheres ${S}^2(\tb,\rb)$ for any $\tb\in \mathbb{R}$ and $\rb\geq 2M$ and can be calculated from the initial data;
   \item For any closed $2$-surface, say $S^2$, $\int_{S^2} \mathbf{F}_{\text{rad}}=  \int_{S^2} \leftidx{^{\star}}{\mathbf{F}}_{\text{rad}}=0$;
   \item $\Lxi\mathbf{F}_{\text{sta}}=0$.
 \end{enumerate}
 \end{lemma}

\begin{remark}
By such a decomposition, one can easily calculate the charged static Coulomb part of the Maxwell field from the initial data, and the estimates for the non-charged radiative part can be derived from the system \eqref{eq:TSIsSpin1Kerr} if given estimates for the spin $\pm 1$ components.
\end{remark}


\section{Almost Price's law}
\label{sect:APL:all}


The aim of this section is to derive the almost sharp energy decay and pointwise decay estimates for the spin $\pm \sfrak$ components solving the TME. 

In Section \ref{sect:BEAM}, we first review the basic energy and Morawetz estimates for wave systems of the spin $s$ components that have been proven in the past few years. Then, in Sections \ref{sect:extendedwavesystem}--\ref{sect:weakdecayesti}, we derive extended wave systems and deduce weak decay estimates  for the spin $s$ components. Afterwards, we define in Section \ref{sect:NPconsts:Maxwell} the Newman--Penrose constants--the most central definition in determining the sharp decay rate for the field.
In the end, in Sections \ref{sect:APL:ext}--\ref{sect:APL:IntRe}, we prove the almost Price's law in the exterior region $\{\rb\geq \tb\}$ and in the interior region $\{\rb\leq \tb\}$ respectively; the main statements on the almost Price's law are 
Theorems \ref{thm:almost:ext:ipm:nvv:higher} and \ref{thm:almost:int:ipm:nvv:hig}.

\subsection{Energy and Morawetz estimates}
\label{sect:BEAM}

We briefly review the energy and Morawetz estimates for the TME in this subsection.

To start with, we have to derive basic wave systems for the spin $s$ components $\Phi_{-\sfrak}^{(i)}$ and $\Xi_{+\sfrak}^{(i)}$ that are defined in Definition \ref{def:PsiPhietc}.

\begin{lemma}
\label{lem:basicsystem}
\begin{itemize}
\item Let $\sfrak=1$.
The equations of $\Phiminus{i}$, $i=0,1,2$, are
\begin{subequations}
\label{eq:RWschwPhi012ope}
\begin{align}
\label{eq:RWPhi01schwope}
(-r^2YV
+\edthR\edthR'-2)\Phiminus{0} ={}&-{2(r-3M)}{r^{-2}}\Phiminus{1},\\
\label{eq:RWPhi11schwope}
(-r^2YV
+\edthR\edthR'-2)\Phiminus{1} ={}&0,\\
\label{eq:RWPhi21schwope}
(-r^2YV
+\edthR\edthR'
-12M r^{-1})\Phiminus{2} -2(r-3M)r^{-2}\curlVR\Phiminus{2}={}&0.
\end{align}
\end{subequations}
The wave equations for $\{\Xi_{+1}^{(i)}\}_{i=0,1}$ are the same as \eqref{eq:RWPhi01schwope}--\eqref{eq:RWPhi11schwope} but replacing $\edthR\edthR'$ and $\{\Phiminus{i}\}_{i=0,1}$ by $\edthR'\edthR$ and $\{\Xi_{-1}^{(i)}\}_{i=0,1}$, respectively.
\item Let $\sfrak=2$. The equations of $\Phiminustwo{i}$, $i=0,1,2,3,4$, are
\begin{subequations}
\label{eq:RWschwPhi01234ope}
\begin{align}
\label{eq:RWPhi02schwope}
(-r^2YV
+\edthR\edthR'-4-6M r^{-1})\Phiminustwo{0} ={}&-{4(r-3M)}{r^{-2}}\Phiminustwo{1},\\
\label{eq:RWPhi12schwope}
(-r^2YV
+\edthR\edthR'-6+6M r^{-1})\Phiminustwo{1} ={}&-{2(r-3M)}{r^{-2}}\Phiminustwo{2}-6M\Phiminustwo{0},\\
\label{eq:RWPhi22schwope}
(-r^2YV
+\edthR\edthR'
-6+6M r^{-1})\Phiminustwo{2} ={}&0,\\
\label{eq:RWPhi32schwope}
(-r^2YV
+\edthR\edthR'
- 4 -6M r^{-1}) \Phiminustwo{3}
-2(r-3M)r^{-2}\curlVR\Phiminustwo{3}
={}&
6M\Phiminustwo{2} ,\\
\label{eq:RWPhi42schwope}
(-r^2YV
+\edthR\edthR'-{30M}{r^{-1}})\Phiminustwo{4}
-4(r-3M)r^{-2}\curlVR\Phiminustwo{4}={}&0.
\end{align}
\end{subequations}
The wave equations for $\{\Xi_{+2}^{(i)}\}_{i=0,1,2}$ are the same as \eqref{eq:RWPhi02schwope}--\eqref{eq:RWPhi22schwope} but replacing $\edthR\edthR'$ and  $\{\Phiminustwo{i}\}_{i=0,1,2}$ by $\edthR'\edthR$ and $\{\Xi_{-2}^{(i)}\}_{i=0,1,2}$, respectively.
\end{itemize}
\end{lemma}

\begin{proof}
The governing equations for the spin $+\sfrak$ scalars can be similarly derived as the ones  for the spin $-\sfrak$ scalars, hence we focus only on the spin $-\sfrak$ scalars.
Equations \eqref{eq:RWPhi01schwope} and \eqref{eq:RWPhi02schwope} are easily obtained from the TME \eqref{eq:TME} since the definition \eqref{eq:Phiminusi:def} of $\Phiminuss{0}$ involves only a rescaling of $\psiminuss$ by $\mu^{\sfrak}r$.  By utilizing the following commutation relation
\begin{align}\label{eq:commrela:withcurlVR}
[\curlVR, -r^2YV]\varphi={}&-\curlVR\bigg(\frac{2(r-3M)}{r^{2}} \curlVR\varphi\bigg)
=-\frac{2(r-3M)}{r^2}\curlVR^2\varphi
+(2-12Mr^{-1})\curlVR\varphi,
\end{align}
the derivations of the other equations are straightforward.
\end{proof}

Equations of $\Phiminuss{\sfrak}$ and $\Xi^{(\sfrak)}_{+\sfrak}$ are the scalar wave equation for $\sfrak=0$, the Fackerell--Ipser equation \cite{fackerell:ipser:EM}  for $\sfrak=1$ and the Regge--Wheeler equation  \cite{ReggeWheeler1957} for $\sfrak=2$,\footnote{In some works, these three equations are put in a wave form with an $\sfrak$-dependent potential and are all called the Regge--Wheeler equation.} respectively. These equations can be treated in a similar way as the scalar wave equation to obtain for the spin-weighted scalars under consideration a uniform bound of a nondegenerate energy and an integrated local decay estimate (also called a Morawetz estimate). We call these two types of estimates together as the energy and Morawetz estimate.  One can also treat the wave systems of $\{\Phiminuss{i}\}_{i=0,\ldots, \sfrak}$ and $\{\Xi^{(i)}_{+\sfrak}\}_{i=0,\ldots, \sfrak}$ and arrive at the energy and Morawetz estimates for both spin components. The following is a summary of the energy and Morawetz estimates proven for different spin fields in \cite{dafrod09red,Ma2017Maxwell,Ma17spin2Kerr}.

\begin{thm}[Energy and Morawetz estimates for TME]\label{thm:BEAM}
Let $\sfrak+1\leq \reg\in \mathbb{N}^+$. There exists a universal constant $C=C(\reg)$ such that the following energy and Morawetz estimate holds in the region $\Donetwo$ for any $\tb_0\leq \tb_1<\tb_2$ for the spin $s=\pm\sfrak$ components: 
\begin{align}\label{eq:BEAM}
\hspace{2ex}&\hspace{-2ex}
\mathbf{BE}^{\reg}_{\Sigmatwo}(\Psipms)+\int_{\tb_1}^{\tb_2}\mathbf{BM}^{\reg}_{\Sigmatb}(\Psipms)\di \tb\leq C\mathbf{BE}^{\reg}_{\Sigmaone}(\Psipms),
\end{align}
where the basic energies for any $\tb\geq \tb_0$ are
\begin{subequations}
\begin{align}
\mathbf{BE}^{\reg}_{\Sigmatb}(\Psipluss)={}\sum_{i=0,\ldots, \sfrak}\sum_{\abs{\mathbf{a}}\leq \reg-\sfrak-1}
\norm{\PDeri^{\mathbf{a}}\Xi^{(i)}_{+\sfrak}}^2_{W_{-2}^1(\Sigmatb)}\\
\mathbf{BE}^{\reg}_{\Sigmatb}(\Psiminuss)={}\sum_{i=0,\ldots, \sfrak}\sum_{\abs{\mathbf{a}}\leq \reg-\sfrak-1}
\norm{\PDeri^{\mathbf{a}}(\mu^{-\sfrak+i}\Phiminuss{i})}^2_{W_{-2}^1(\Sigmatb)},
\end{align}
\end{subequations}
and the basic Morawetz densities (in time) for any $\tb_2>\tb_1\geq \tb_0$ are
\begin{align*}
\mathbf{BM}^{\reg}_{\Sigmatb}(\Psipluss)={}&\sum_{i=0,\ldots, \sfrak}\sum_{\abs{\mathbf{a}}\leq \reg-\sfrak-1}\left(
\norm{\PDeri^{\mathbf{a}}\Xi^{(i)}_{+\sfrak}}^2_{W_{-3}^0(\Sigmatb)}
+\norm{\PDeri^{\mathbf{a}}\PDeri(\Xi^{(i)}_{+\sfrak})}^2_{W_{-2}^0(\Sigmatb\cap\{r\geq 4M\})}
\right)\\
\mathbf{BM}^{\reg}_{\Sigmatb}(\Psiminuss)={}&\sum_{i=0,\ldots, \sfrak}\sum_{\abs{\mathbf{a}}\leq \reg-\sfrak-1}\Big(
\norm{\PDeri^{\mathbf{a}}(\mu^{-\sfrak+i}\Phiminuss{i})}^2_{W_{-3}^0(\Sigmatb)}+\norm{\PDeri^{\mathbf{a}}\PDeri(\mu^{-\sfrak+i}\Phiminuss{i})}^2_{W_{-2}^0(\Sigmatb\cap\{r\geq 4M\})}
\Big).
\end{align*}
\end{thm}

\begin{remark}
\begin{enumerate}
\item The presence of a cutoff integral region $\Sigmatb\cap\{r\geq 4M\}$, instead of $\Sigmatb$, in the expressions of the basic Morawetz densities is due to the trapping phenomenon at the trapped surface $r=3M$ where one has to loss derivatives.
\item The presence of the factor $\mu^{-\sfrak+i}$ in the expressions of the basic energy and the basic Morawetz density for the spin $s=-\sfrak$ component is such that the scalars $\mu^{-\sfrak+i}\Phiminuss{i}$ are regular and nondegenerate at future event horizon.
\end{enumerate}
\end{remark}

\subsection{Extended wave systems}
\label{sect:extendedwavesystem}

As discussed in Section \ref{intro:outlineproof}, we shall base on the basic wave systems in Lemma \ref{lem:basicsystem} and derive extended wave systems for $\Phi_s^{(i)}$ and $\widetilde{\Phi}_s^{(i)}$ for a larger range of value of $i$ in order to show decay estimates that are close to the Price's law. This follows closely \cite{Ma20almost} where the Maxwell field ($\sfrak=1$) is treated.

Recall from Definition \ref{def:PsiPhietc} that $\Phi_{s}^{(i)}=\curlVR^i\Phi_{s}^{(0)}$. Let us also recall from Definition \ref{def:tildePhiplusandminusHigh}  that we have defined
\begin{subequations}
\label{eq:fgxconst}
\begin{align}
f_{\sfrak, i,1}={}&i(i+2\sfrak+1), & f_{\sfrak, i,2}={}&-2(\sfrak+i+1),\\
f_{\sfrak, i,3}={}&i(i+2\sfrak+1)+\frac{1}{3}(\sfrak+1)(2\sfrak+1), & g_{\sfrak, i}={}&2i(i+\sfrak)(i+2\sfrak)
\end{align}
and the following constants
\begin{align}
\label{eq:defconstants:extended}
x_{\sfrak, i+1,i}=\frac{g_{\sfrak, i+1}}{f_{\sfrak, i+1,1} -f_{\sfrak, i,1}}=(i+1)(i+2\sfrak+1),  \quad x_{\sfrak, i+1,j}=-\frac{g_{\sfrak, i+1}x_{\sfrak, i,j}}{f_{\sfrak, i+1,1}-f_{\sfrak, j,1}} \,\,\text{for } 0\leq j\leq i-1.
\end{align}
\end{subequations}
The scalars $\widetilde{\Phi}_s^{(i)}$ ($i\in \mathbb{N}$) are constructed from the scalars $\Phi_{s}^{(i)}$ in the following way:
\begin{subequations}
\label{eq:defscalar:extended}
\begin{align}
\tildePhiplussHigh{0}={}&\PhiplussHigh{0},& \tildePhiplussHigh{i+1}={}&\PhiplussHigh{i+1}
+\sum_{j=0}^ix_{\sfrak, i+1,j}M^{i+1-j}\tildePhiplussHigh{j},\\
\tildePhiminuss{2\sfrak}={}&\Phiminuss{2\sfrak},& \tildePhiminuss{i+2\sfrak+1}={}&\Phiminuss{i+2\sfrak+1}
+\sum_{j=0}^ix_{\sfrak, i+1,j}M^{i+1-j}\tildePhiminuss{j+2\sfrak}.
\end{align}
\end{subequations}

\begin{prop}[Extended wave systems for the spin $\pm \sfrak$ components]
\label{prop:wave:Phihigh:pm1}
Let $i\in \mathbb{N}$.
\begin{enumerate}
  \item The equation of $\PhiplussHigh{0}$ is
\begin{align}
\label{eq:Phi+1:Schw:generall}
-\mu Y\curlVR \PhiplussHigh{0}+\edthR'\edthR\PhiplussHigh{0}-{2(\sfrak+1)(r-3M)r^{-2}}\curlVR\PhiplussHigh{0} -{2(\sfrak+1)(2\sfrak+1)M}{r^{-1}}\PhiplussHigh{0}={}&0,
\end{align}
 the equation of $\PhiplussHigh{i}$ is
\begin{align}
\label{eq:Phiplushighi:Schw:generall}
&-\mu Y \curlVR  \PhiplussHigh{i} +(\edthR'\edthR+f_{\sfrak, i,1})\PhiplussHigh{i}
+{f_{\sfrak, i,2}(r-3M)r^{-2}}\curlVR\PhiplussHigh{i}
-6f_{\sfrak, i,3}Mr^{-1}\PhiplussHigh{i}
+g_{\sfrak, i} M\PhiplussHigh{i-1}={}0,
\end{align}
and the equation of $\tildePhiplussHigh{i}$ is
\begin{align}
\label{eq:Phiplushighi:Schw:generall:tildePhiplusi}
&-\mu Y \curlVR  \tildePhiplussHigh{i} +(\edthR'\edthR+f_{\sfrak, i,1})\tildePhiplussHigh{i}
+{f_{\sfrak, i,2}(r-3M)r^{-2}}\curlVR\tildePhiplussHigh{i}
+\sum_{j=0}^{i}h_{\sfrak, i,j} \PhiplussHigh{j}={}0,
\end{align}
with $h_{\sfrak, i,j}=O(r^{-1})$ for all $j\in \{0,1,\ldots, i\}$.
Moreover, all the functions $h_{\sfrak,i,j}$ can be determined by the following relation
\begin{align}
\sum_{j=0}^{i+1}h_{\sfrak, i+1,j} \PhiplussHigh{j}={}&-(r-3M) r^{-2}\sum_{j=0}^i (f_{\sfrak, i+1,2}-f_{\sfrak, j,2})x_{\sfrak, i+1,j}M^{i+1-j}  \curlVR\tildePhiplussHigh{j}\notag\\
&+\sum_{j=0}^i x_{\sfrak, i+1,j} M^{i+1-j}\sum_{j'=0}^{j}h_{\sfrak, j,j'} \PhiplussHigh{j'}
\end{align}
together with $h_{\sfrak,0,0}=-2(\sfrak+1)(2\sfrak+1)Mr^{-1}$.
  \item The equation of $\Phiminuss{2\sfrak}$ is
\begin{align}
\label{eq:Phiminus2:Schw:generall}
-\mu Y \curlVR\Phiminuss{2\sfrak}+\edthR\edthR'\Phiminuss{2\sfrak}
-{2(\sfrak+1)(r-3M)r^{-2}}\curlVR\Phiminuss{2\sfrak}-{2(\sfrak+1)(2\sfrak+1)M}{r^{-1}}\Phiminuss{2\sfrak}
={}0,
\end{align}
and the equation of $\Phiminuss{i+2\sfrak}$  is
\begin{align}
\label{eq:Phiminusi:Schw:generall}
&-\mu Y \curlVR  \Phiminuss{i+2\sfrak} +(\edthR\edthR'+f_{\sfrak, i,1})\Phiminuss{i+2\sfrak}
+{f_{\sfrak, i,2}(r-3M)r^{-2}}\curlVR\Phiminuss{i+2\sfrak}
\notag\\
&-6f_{\sfrak, i,3}Mr^{-1}\Phiminuss{i+2\sfrak}
+g_{\sfrak, i} M\Phiminuss{i+2\sfrak-1}={}0,
\end{align}
and the equation of $\tildePhiminuss{i+2\sfrak}$ is
\begin{align}
\label{eq:tildePhiminusi:Schw:generall:tildePhiminusi}
&-\mu Y \curlVR  \tildePhiminuss{i+2\sfrak} +(\edthR\edthR'+f_{\sfrak, i,1})\tildePhiminuss{i+2\sfrak}
+{f_{\sfrak, i,2}(r-3M)r^{-2}}\curlVR\tildePhiminuss{i+2\sfrak}
+\sum_{j=0}^{i}h_{\sfrak, i,j} \Phiminuss{j+2\sfrak}={}0,
\end{align}
with $h_{\sfrak, i,j}$ being of the same expression as the ones in \eqref{eq:Phiplushighi:Schw:generall:tildePhiplusi} and satisfying $h_{\sfrak, i,j}=O(r^{-1})$ for all $j\in \{0,1,\ldots, i\}$.
\end{enumerate}
\end{prop}

\begin{remark}
The wave equations \eqref{eq:Phiplushighi:Schw:generall} and \eqref{eq:Phiminusi:Schw:generall} for $\Phi_s^{(i)}$ are used to show that higher $\ell$ mode of the spin $s$ components decays faster. The significance of the wave equations \eqref{eq:Phiplushighi:Schw:generall:tildePhiplusi} and \eqref{eq:tildePhiminusi:Schw:generall:tildePhiminusi} for $\widetilde{\Phi}_s^{(i)}$ lies in the fact that the last term with $O(1)$ coefficient $g_{\sfrak,i}M$ in equations of $\Phi_{s}^{(i)}$ is removed with the only price of introducing new terms with purely $O(r^{-1})$ coefficients. This fact is essential in the latter sections to prove the almost Price's law, to define the Newman--Penrose constants and to derive the Price's law. 
\end{remark}

\begin{proof}
By the TME \eqref{eq:TME}, one has
\begin{align}
\big(-r^2 YV +\edthR'\edthR-2(s+1)Mr^{-1}\big)(r\psi_s)
={}&-2s((r-M)Y-2r\partial_t)(r\psi_s).
\end{align}
In view of the definition \eqref{def:PhiplussHigh0}, one therefore reaches
\begin{align}
\label{eq:Phi+1:Schw:generall:vvvv1}
-r^2 YV \PhiplussHigh{0}+\edthR'\edthR\PhiplussHigh{0}-{2\sfrak(r-3M)r^{-2}}\curlVR\PhiplussHigh{0} -{2(\sfrak+1)(2\sfrak+1)M}{r^{-1}}\PhiplussHigh{0}={}&0,
\end{align}
which is exactly equation \eqref{eq:Phi+1:Schw:generall}.

We shall inductively show 
\begin{align}
\label{eq:Phiplushighi:Schw:generall:v1}
&-r^2Y V  \PhiplussHigh{i} +(\edthR'\edthR+f_{\sfrak, i,1})\PhiplussHigh{i}
+{(f_{\sfrak, i,2}+2)(r-3M)r^{-2}}\curlVR\PhiplussHigh{i}\notag\\
&
-6f_{\sfrak, i,3}Mr^{-1}\PhiplussHigh{i}
+g_{\sfrak, i} M\PhiplussHigh{i-1}={}0,
\end{align}
 and equation \eqref{eq:Phiplushighi:Schw:generall}  then follows.
Equation \eqref{eq:Phiplushighi:Schw:generall:v1} clearly holds for $i=0$ in view of equation \eqref{eq:Phi+1:Schw:generall:vvvv1}. Assume this holds for $i=i_0\in \mathbb{N}$, and we verify that it also holds for $i=i_0+1$. Applying $\curlVR$ to equation \eqref{eq:Phiplushighi:Schw:generall:v1} with $i=i_0$ and using the formula $\PhiplussHigh{i+1}=\curlVR \PhiplussHigh{i}$, we arrive at
\begin{align}
&\curlVR(-r^2Y V  \PhiplussHigh{i_0} )+(\edthR'\edthR+f_{\sfrak, i_0,1})\PhiplussHigh{i_0+1}\notag\\
&
+(f_{\sfrak, i_0,2}+2)\big(r^2\partial_r ((r-3M)r^{-2})\PhiplussHigh{i_0+1} + (r-3M)r^{-2} \curlVR\PhiplussHigh{i_0+1}\big)\notag\\
&
-6f_{\sfrak, i_0,3}\big(Mr^2\partial_r(r^{-1})\PhiplussHigh{i_0} + Mr^{-1} \PhiplussHigh{i_0+1}\big)
+g_{\sfrak, i_0} M\PhiplussHigh{i_0}={}0.
\end{align}
Applying the formula \eqref{eq:commrela:withcurlVR} for the commutator $[\curlVR, -r^2 YV]$ and collecting the different terms in the above equation, 
 and in view of the expressions of the constants $f_{\sfrak,i,1}$, $f_{\sfrak,i,2}$, $f_{\sfrak,i,3}$ and $g_{\sfrak,i}$ in \eqref{eq:fgxconst}, one finds equation \eqref{eq:Phiplushighi:Schw:generall:v1} holds for $i=i_0+1$.

We prove equation \eqref{eq:Phiplushighi:Schw:generall:tildePhiplusi} also by induction. Assume it holds for $\tildePhiplussHigh{j}$ for all $1\leq j\leq i$, it suffices to show that it holds also for $\tildePhiplussHigh{i+1}$. By adding an $x_{\sfrak, i+1, j}M^{i+1-j}$ multiple of equation \eqref{eq:Phiplushighi:Schw:generall:tildePhiplusi} for $\tildePhiplussHigh{j}$ for all $j=0,1,\ldots,i$ to equation \eqref{eq:Phiplushighi:Schw:generall} of $\PhiplussHigh{i+1}$, one obtains
\begin{align}
\label{eq:tildePhipsHigh:choices}
&-\mu Y \curlVR  \tildePhiplussHigh{i+1} +(\edthR'\edthR+f_{i+1,1})\tildePhiplussHigh{i+1}
+{f_{\sfrak, i+1,2}(r-3M)r^{-2}}\curlVR\tildePhiplussHigh{i+1}
-6f_{\sfrak, i+1,3}Mr^{-1}\tildePhiplussHigh{i+1}\notag\\
&-\sum_{j=0}^i x_{\sfrak, i+1,j} M^{i+1-j}(f_{\sfrak, i+1,1}- f_{\sfrak, j,1})\tildePhiplussHigh{j}
+g_{\sfrak, i+1}M\PhiplussHigh{i}\notag\\
&
-\frac{r-3M}{r^2} \sum_{j=0}^i (f_{\sfrak, i+1,2}-f_{\sfrak, j,2})x_{\sfrak, i+1,j}M^{i+1-j}  \curlVR\tildePhiplussHigh{j}
+\sum_{j=0}^i x_{\sfrak, i+1,j} M^{i+1-j}\sum_{j'=0}^{j}h_{\sfrak, j,j'} \PhiplussHigh{j'}={}0.
\end{align}
We make the replacement $\PhiplussHigh{i}=\tildePhiplussHigh{i}
-\sum\limits_{j=0}^{i-1}x_{\sfrak, i,j}M^{i-j}\tildePhiplussHigh{j}$, which simply follows from formulas \eqref{eq:defscalar:extended} of $\tildePhiplussHigh{i}$, in the last term of the second line in the above equation and find the second line equals $\sum\limits_{j=1}^i e_{\sfrak, i+1,j}M^{i+1-j}\tildePhiplussHigh{j}$ with
\begin{subequations}
\begin{align}
e_{\sfrak, i+1,i}={}&-x_{\sfrak, i+1,i}(f_{\sfrak, i+1,1}-f_{\sfrak, i,1}) +g_{\sfrak, i+1},\\
e_{\sfrak, i+1,j}={}&-x_{\sfrak, i+1,j}(f_{\sfrak, i+1,1}-f_{\sfrak, j,1}) -g_{\sfrak, i+1}x_{\sfrak, i,j}, \quad \text{for} \quad 0\leq j\leq i-1.
\end{align}
\end{subequations}
By the choices of the constants $\{x_{\sfrak, i+1,j}\}_{j=0,1,\ldots, i}$ made in \eqref{eq:defconstants:extended}, the constants $\{e_{\sfrak, i+1,j}\}_{j=0,1,\ldots,i}$ all identically vanish,  hence the entire second line of \eqref{eq:tildePhipsHigh:choices} vanishes. 
One can rewrite $\curlVR\tildePhiplussHigh{j}$ using Definition \ref{def:tildePhiplusandminusHigh} as a weighted sum of $\{\tildePhiplussHigh{j'}\}|_{j'=0,1,\ldots, j+1}$ with all coefficients being $O(1)$,  and by denoting all the terms in the last line on the LHS of \eqref{eq:tildePhipsHigh:choices} as $\sum\limits_{j=0}^{i+1}h_{\sfrak, i+1,j} \PhiplussHigh{j}$, one finds  $h_{\sfrak, i+1,j}=O(r^{-1})$ for all $j\in \{0,1,\ldots, i+1\}$. All these together then prove equation \eqref{eq:Phiplushighi:Schw:generall:tildePhiplusi} for $\tildePhiplussHigh{i+1}$.

For  the spin $-\sfrak$ component,
equation \eqref{eq:Phiminus2:Schw:generall} comes directly from \eqref{eq:RWPhi21schwope} and \eqref{eq:RWPhi42schwope}.  Given that equations \eqref{eq:Phiminus2:Schw:generall} and \eqref{eq:Phi+1:Schw:generall} have the same form,
the derivation of  equations \eqref{eq:Phiminusi:Schw:generall} and \eqref{eq:tildePhiminusi:Schw:generall:tildePhiminusi} is immediate.
\end{proof}


\subsection{Initial energies}

Before proving any energy or pointwise decay estimates for the wave equations in the wave systems in Proposition \ref{prop:wave:Phihigh:pm1}, it is necessary for us to introduce in this subsection a few initial energies that will be utilized frequently in the rest of this work.

The first set of energies for the spin $\pm\sfrak$ components consists of the $r^p$-weighted energies which will show up in proving the global $r^p$ estimates and obtaining weak decay estimates in Section \ref{sect:weakdecayesti}. As stated in Proposition \ref{prop:weakdecay:summary}--the main statement of Section \ref{sect:weakdecayesti}, the weak (pointwise) decay estimates are in terms of this set of energies. 

\begin{definition}
\label{def:energyonSigmazero:Schw}
Let $i\in \mathbb{N}$ and $i\geq \min\{0, s\}$, and let $\reg\geq i$. Let $\tb\geq \tb_0$. Let $\tildePhiplussHigh{j}$ and $\tildePhiminuss{j}$ be defined as in Definition \ref{def:tildePhiplusandminusHigh}. Define on $\Sigmatb$ an energy  of  the spin $s=+\sfrak$ component
\begin{subequations}
\begin{align}
\InizeroEnergypluss{\tb}{i}{\reg}{p}={}\left\{
             \begin{array}{ll}
              \sum_{j=0}^{i-\sfrak}\norm{(r^2V)^j\Psipluss}^2_{W_{-2}^{\reg-j}(\Sigmatb)}+
\norm{rV\tildePhiplussHigh{i-\sfrak}}^2_{W_{p-2}^{\reg-i}(\Sigmatb\cap\{\rb\geq 4M\})}, & \hbox{$p>2$;} \\
              \sum_{j=0}^{i-\sfrak}\norm{(r^2V)^j\Psipluss}^2_{W_{-2}^{\reg-j}(\Sigmatb)}+
\norm{rV\PhiplussHigh{i-\sfrak}}^2_{W_{p-2}^{\reg-i}(\Sigmatb\cap\{\rb\geq 4M\})}, & \hbox{$0\leq p\leq 2$,}
             \end{array}
           \right.
\end{align}
and an energy of  the spin $s=-\sfrak$ component
\begin{align}
\InizeroEnergyminuss{\tb}{i}{\reg}{p}={}\left\{
             \begin{array}{ll}
              \sum_{j=0}^{i+\sfrak}\norm{(r^2V)^j\Psiminuss}^2_{W_{-2}^{\reg+\sfrak-j}(\Sigmatb)}+
\norm{rV\tildePhiminuss{i+\sfrak}}^2_{W_{p-2}^{\reg-i}(\Sigmatb\cap\{\rb\geq 4M\})}, & \hbox{$p>2$;} \\
               \sum_{j=0}^{i+\sfrak}\norm{(r^2V)^j\Psiminuss}^2_{W_{-2}^{\reg+\sfrak-j}(\Sigmatb)}+
\norm{rV\Phiminuss{i+\sfrak}}^2_{W_{p-2}^{\reg-i}(\Sigmatb\cap\{\rb\geq 4M\})}, & \hbox{$0\leq p\leq 2$.}
             \end{array}
           \right.
\end{align}
\end{subequations}
\end{definition}

Instead, the almost sharp energy and pointwise decay estimates for the $\ell\geq \ell_0$ modes of the spin $\pm\sfrak$ components will be stated in terms of the following two energies, the first one being the initial energy in the non-vanishing Newman--Penrose constant case and the second one being the initial energy in the vanishing Newman--Penrose constant case. Our main Theorems \ref{thm:almost:ext:ipm:nvv:higher} and \ref{thm:almost:int:ipm:nvv:hig}  in this entire section prove the almost Price's law for the spin $\pm\sfrak$ components in terms of these two energies.

\begin{definition}
\label{def:initialenergy:highMODEs}
Let $\ell_0\geq \sfrak$, and let $\vartheta\in(0,\frac{1}{2})$. Define for the spin $+\sfrak$ component the following initial energies:
\begin{align}
\label{initEner:NVNP}
\InizeroEnergyplussnv{\reg}{\vartheta}={}&\InizeroEnergyplusshigh{\tb_0}{\ell_0}{\reg}{3-\vartheta}{(\Psipluss)^{\ell=\ell_0}}+\sum_{\ell'=\ell_0+1}^{2\ell_0+1}\InizeroEnergyplusshigh{\tb_0}{\ell'}{\reg}{1+\vartheta}{(\Psipluss)^{\ell=\ell'}}
\notag\\
&
+\InizeroEnergyplusshigh{\tb_0}{2\ell_0+2}{\reg}{1+\vartheta}{(\Psipluss)^{\ell\geq 2\ell_0+2})}
,\\
\label{initEner:VNP}
\InizeroEnergyplussv{\reg}{\vartheta}={}&\InizeroEnergyplusshigh{\tb_0}{\ell_0}{\reg}{5-\vartheta}{(\Psipluss)^{\ell=\ell_0}}
+\InizeroEnergyplusshigh{\tb_0}{\ell_0+1}{\reg}{3-\vartheta}{(\Psipluss)^{\ell=\ell_0+1}}\notag\\
&+\sum_{\ell'=\ell_0+2}^{2\ell_0+2}\InizeroEnergyplusshigh{\tb_0}{\ell'}{\reg}{1+\vartheta}{(\Psipluss)^{\ell=\ell'}}
+\InizeroEnergyplusshigh{\tb_0}{2\ell_0+3}{\reg}{1+\vartheta}{(\Psipluss)^{\ell\geq 2\ell_0+3}}.
\end{align}
Similarly, we define the initial energies $\InizeroEnergyminussnv{\reg}{\vartheta}$ and $\InizeroEnergyminussv{\reg}{\vartheta}$ for the spin $-\sfrak$ component by simply replacing $\Psipluss$ by $\Psiminuss$ in formulas \eqref{initEner:NVNP} and \eqref{initEner:VNP}.
\end{definition}

\subsection{Weak decay estimates}
\label{sect:weakdecayesti}

We now use the wave systems derived in Sections \ref{sect:BEAM}--\ref{sect:extendedwavesystem} together with the energy and Morawetz estimates in Theorem \ref{thm:BEAM} to achieve decay estimates. These pointwise decay estimates are by no means optimal and will be improved in latter sections; hence we call them \textquotedblleft{\textit{weak decay estimates}.\textquotedblright}

The main statement of this subsection is included in the following proposition.

\begin{prop}[\textbf{Weak decay estimates for the spin $\pm\sfrak$ components}]
\label{prop:weakdecay:summary}
\begin{enumerate}
\item
Assume the spin $s=\pm \sfrak$ components are supported on a fixed $\ell$ mode, $\ell\geq \sfrak$. Then for any $j\in \mathbb{N}$ and $(i,p)\in \{\max\{0,s\}< i\leq \ell, 0\leq p<5\}\cup\{i=\max\{0,s\}, p>1\} $,  there exists a $\regl=\regl(j, i)$ such that
\begin{subequations}
\label{eq:weakdecay:ell:ipm}
\begin{align}
\absCDeri{\Lxi^j(r^{-1}\Psipluss)}{\reg}
\lesssim{}&(\InizeroEnergypluss{\tb}{i}{\reg+\regl}{p})^{\half}v^{-1}\tb^{-(p-1)/2-(i-\sfrak)-j},\\
\sum_{i=0}^{\sfrak}\absCDeri{\Lxi^j(r^{-1}(r^2 V)^i\Psiminuss)}{\reg}
\lesssim{}&(\InizeroEnergyminuss{\tb}{i}{\reg+\regl}{p})^{\half}v^{-1}\tb^{-(p-1)/2-i-j}.
\end{align}
\end{subequations}
\item\label{pt:WED:highmodes:+}
Assume the spin $s=\pm \sfrak$ components are supported on $\ell\geq \ell_0$ modes, $\ell_0\geq \sfrak$. Then for any $j\in \mathbb{N}$ and $ 0\leq p\leq 2$,  there exists a $\regl=\regl(j, \ell_0)$ such that
\begin{subequations}
\label{eq:weakdecay:geqell0:ipm}
\begin{align}
\absCDeri{\Lxi^j(r^{-1}\Psipluss)}{\reg}
\lesssim{}&(\InizeroEnergypluss{\tb}{\ell_0}{\reg+\regl}{p})^{\half}v^{-1}\tb^{-(p-1)/2-\ell_0-j},\\
\sum_{n=0}^{\sfrak}\absCDeri{\Lxi^j(r^{-1}(r^2 V)^n\Psiminuss)}{\reg}
\lesssim{}&(\InizeroEnergyminuss{\tb}{\ell_0}{\reg+\regl}{p})^{\half}v^{-1}\tb^{-(p-1)/2-\ell_0-j}.
\end{align}
\end{subequations}
\end{enumerate}
\end{prop}

In Sections \ref{subsect:WED:+} and \ref{subsect:WED:-minus} we show energy decay estimates for a fixed mode of the spin $+\sfrak$ and $-\sfrak$ components respectively, and these energy decay estimates are gathered in Section \ref{subsect:WED:ptw:plusminus} to complete the proof of the above Proposition \ref{prop:weakdecay:summary}.

\subsubsection{Energy decay estimates for a fixed $\ell$ mode of the spin $+\sfrak$ component}
\label{subsect:WED:+}

Throughout this subsubsection, we focus on a fixed $\ell$ mode and, unless otherwise stated, we drop the superscript, which indicates the $\ell$ mode, from the scalars constructed from the spin $+\sfrak$ component.

The energy decay estimates are derived by deducing global $r^p$ estimates for the wave systems in Sections \ref{sect:BEAM}--\ref{sect:extendedwavesystem}, thus it is of much convenience to introduce for the spin $+\sfrak$ component (as well as the spin $-\sfrak$ component in the latter Section \ref{subsect:WED:-minus}) a few weighted energies, which correspond to $r^p$-weighted energies for different subsystems of the ones in Proposition \ref{prop:wave:Phihigh:pm1} with different values of $p$. 

\begin{definition}
\label{def:Ffts:Phiplus:-1to5:all}
Define $F(\reg,i,p,\tb,\Psipluss)$\footnote{In this notation, $\reg$ is the regularity parameter, $i$ is the largest number of $j$ appearing in the upper index of $\tildePhiplussHigh{j}$ in the energy norms, $p$ is the parameter in $r^p$ estimates, $\tb$ suggests that the energy is integrated on $\Sigmatb$ hypersurface, and $\Psipluss$ indicates that the energy is for the spin $+\sfrak$ component.} as follows:
\begin{itemize}
\item if $i\leq\ell-\sfrak$,
\begin{subequations}
\label{eq:def:Ffts:Phiplus:-1to2}
\begin{align}
\label{def:Ffts:Phiplus:1:p-1}
F(\reg,i,p,\tb,\Psipluss)={}&\sum_{j=0}^{i} \norm{\tildePhiplussHigh{j}}^2_{W_{-3}^{\reg-\sfrak-j}(\Sigmatb^{\geq R_0})}
+\mathbf{BM}^{\reg}_{\Sigmatb}(\Psipluss), \qquad\text{for } p=-1,\\
\label{def:Ffts:Phiplus:1:p-10}
F(\reg,i,p,\tb,\Psipluss)={}&0, \qquad\text{for } p\in (-1,0),\\
\label{def:Ffts:Phiplus:1:p02}
F(\reg,i,p,\tb,\Psipluss)={}&\sum_{j=0}^{i} \Big(\norm{rV\tildePhiplussHigh{j}}^2_{W_{p-2}^{\reg-\sfrak-j-1}(\Sigmatb^{\geq R_0})}
+\norm{\tildePhiplussHigh{j}}^2_{W_{-2}^{\reg-\sfrak-j}(\Sigmatb^{\geq R_0})}\Big)\notag\\
&
+\mathbf{BE}^{\reg}_{\Sigmatb}(\Psipluss),
\qquad \text{for } p\in [0,2];
\end{align}
\end{subequations}

\item additionally, for $i=\ell-\sfrak$ and $p\in [2,5)$,
\begin{align}
\label{def:Ffts:Phiplus:1:p25:ell}
F(\reg,\ell-\sfrak,p,\tb,\Psipluss)={}& \sum_{j=0}^{\ell-\sfrak-1} \Big(\norm{rV\tildePhiplussHigh{j}}^2_{W_{0}^{\reg-\sfrak-j-1}(\Sigmatb^{\geq R_0})}
+\norm{\tildePhiplussHigh{j}}^2_{W_{-2}^{\reg-\sfrak-j}(\Sigmatb^{\geq R_0})}\Big)\notag\\
&+\norm{rV\tildePhiplussHigh{\ell-\sfrak}}^2_{W_{p-2}^{\reg-\ell-1}(\Sigmatb^{\geq R_0})}
+\norm{\tildePhiplussHigh{\ell-\sfrak}}^2_{W_{-2}^{\reg-\ell}(\Sigmatb^{\geq R_0})}
+\mathbf{BE}^{\reg}_{\Sigmatb}(\Psipluss).
\end{align}
\end{itemize}
We similarly define $F(\reg,i,p,\tb,\Lxi^j\Psipluss)$,  $j\in \mathbb{N}$, with all the scalars in the Sobolev norms acted by $\Lxi^j$.
\end{definition}

\begin{remark}
By Definition \ref{def:tildePhiplusandminusHigh}, one has $\tildePhiplussHigh{j}=\curlVR\PhiplussHigh{j-1}+\sum_{j'=0}^{j-1}O(1)\tildePhiplussHigh{j'}$, hence, in the case that $i=\ell-\sfrak$ and $p=2$, the two definitions \eqref{def:Ffts:Phiplus:1:p02} and \eqref{def:Ffts:Phiplus:1:p25:ell}  are equivalent.
\end{remark}

We shall now prove global $r^p$ estimates for the spin $+\sfrak$ component. This is achieved by applying the $r^p$ lemma \ref{lem:wave:rp} to each equation in the wave systems in Proposition \ref{prop:wave:Phihigh:pm1}. 

\begin{prop}\label{prop:BEDC:Phiplus:0}
For all $0\leq i\leq \ell-\sfrak$, $p\in [0,2]$ and $\tb_2>\tb_1\geq \tb_0$,
\begin{align}
\label{eq:BEDC:Phiplus:0}
F(\reg, i, p,\tb_2,\Psipluss)
+\int_{\tb_1}^{\tb_2}F(\reg-1, i, p-1,\tb,\Psipluss)\di \tb
\lesssim F(\reg, i, p,\tb_1,\Psipluss).
\end{align}
Moreover, this estimate is also valid in the case that $i= \ell-\sfrak$ and $p\in [2,4)$.
\end{prop}

\begin{proof}
We put equation \eqref{eq:Phiplushighi:Schw:generall:tildePhiplusi}  satisfied by the scalars constructed from the spin $+\sfrak$ component into the form of \eqref{eq:wave:rp}.  Since the eigenvalue of $\edthR'\edthR$ acting on an $\ell$ mode is $-(\ell-\sfrak)(\ell+\sfrak+1)$, one finds that the first assumption is trivially satisfied and the second assumption holds true as long as $i\leq \ell-\sfrak$. In particular, $b_{V,-1}=0$ happens only when $\sfrak=i=0$, and $b_{0,0}=0$ only when $i=\ell-\sfrak$. Based on these observations, we apply the corresponding estimates in Lemma \ref{lem:wave:rp} to equation \eqref{eq:Phiplushighi:Schw:generall:tildePhiplusi} in different cases.

Consider first the case that $i\leq\ell-\sfrak$ and $p\in [0,2]$. As discussed above, $b_{V,-1}=0$  holds only when $\sfrak=i=0$. If $\sfrak=i=0$, then the estimate \eqref{eq:BEDC:Phiplus:0} with  $p=0$ holds in view of the energy and Morawetz estimate \eqref{eq:BEAM}.  In the remaining cases, it holds that $b_{V,-1}>0$. We apply to equation \eqref{eq:Phiplushighi:Schw:generall:tildePhiplusi} with $\varphi=\tildePhiplussHigh{i'}$ the estimate  \eqref{eq:rp:p=0} for $p=0$,  the estimate \eqref{eq:rp:less2:2} for $p\in (0, 2)$, and the estimate \eqref{eq:rp:lp=2:2} for $p=2$ to achieve
\begin{subequations}\label{eq:rp:less2:=2:2:ip}
    \begin{align}
    \label{eq:rp:p=0:2:ip}
\hspace{2ex}&\hspace{-2ex}
\norm{\tildePhiplussHigh{i'}}^2_{W_{-2}^{\reg-\sfrak-i'}(\Sigmatwo^{\geq R_0})}
+\norm{\tildePhiplussHigh{i'}}^2_{W_{-3}^{\reg-\sfrak-i'}(\Donetwo^{\geq R_0})}
\notag\\
&\lesssim_{[R_0-M,R_0]} {}\norm{\tildePhiplussHigh{i'}}^2_{W_{-2}^{\reg-\sfrak-i'}(\Sigmaone^{\geq R_0})}
+\norm{\vartheta(\tildePhiplussHigh{i'})}^2_{W_{-3}^{\reg-\sfrak-i'-1}(\Donetwo^{\geq R_0})},\\
\label{eq:rp:less2:2:ip}
\hspace{2ex}&\hspace{-2ex}
\norm{rV\tildePhiplussHigh{i'}}^2_{W_{p-2}^{\reg-\sfrak-i'-1}(\Sigmatwo^{\geq R_0})}
+\norm{\tildePhiplussHigh{i'}}^2_{W_{-2}^{\reg-\sfrak-i'}(\Sigmatwo^{\geq R_0})}\notag\\
\hspace{2ex}&\hspace{-2ex}
+\norm{\tildePhiplussHigh{i'}}^2_{W_{p-3}^{\reg-\sfrak-i'}(\Donetwo^{\geq R_0})}
+\norm{Y\tildePhiplussHigh{i'}}^2_{W_{-1-\delta}^{\reg-\sfrak-i'-1}(\Donetwo^{\geq R_0})}
\notag\\
&\lesssim_{[R_0-M,R_0]} {}\norm{rV\tildePhiplussHigh{i'}}^2_{W_{p-2}^{\reg-\sfrak-i'-1}(\Sigmaone^{\geq R_0})}
+\norm{\tildePhiplussHigh{i'}}^2_{W_{-2}^{\reg-\sfrak-i'}(\Sigmaone^{\geq R_0})}
+\norm{\vartheta(\tildePhiplussHigh{i'})}^2_{W_{p-3}^{\reg-\sfrak-i'-1}(\Donetwo^{\geq R_0})},\\
\label{eq:rp:lp=2:2:ip}
\hspace{2ex}&\hspace{-2ex}
\norm{rV\tildePhiplussHigh{i'}}^2_{W_{0}^{\reg-\sfrak-i'-1}(\Sigmatwo^{\geq R_0})}
+\norm{\tildePhiplussHigh{i'}}^2_{W_{-2}^{\reg-\sfrak-i'}(\Sigmatwo^{\geq R_0})}
\notag\\
\hspace{2ex}&\hspace{-2ex}
+\norm{\tildePhiplussHigh{i'}}^2_{W_{-1-\delta}^{\reg-\sfrak-i'}(\Donetwo^{\geq R_0})}
+\norm{rV\tildePhiplussHigh{i'}}^2_{W_{-1}^{\reg-\sfrak-i'-1}(\Donetwo^{\geq R_0})}
\notag\\
&\lesssim_{[R_0-M,R_0]} {}\norm{rV\tildePhiplussHigh{i'}}^2_{W_{0}^{\reg-\sfrak-i'-1}(\Sigmaone^{\geq R_0})}
+\norm{\tildePhiplussHigh{i'}}^2_{W_{-2}^{\reg-\sfrak-i'}(\Sigmaone^{\geq R_0})}
+\norm{\vartheta(\tildePhiplussHigh{i'})}^2_{W_{-1}^{\reg-\sfrak-i'-1}(\Donetwo^{\geq R_0})},
\end{align}
\end{subequations}
respectively.
In view of the expression $\vartheta(\tildePhiplussHigh{i'})=\sum_{j=0}^{i'}O(r^{-1})\tildePhiplussHigh{j}$, one finds on the RHS,
\begin{align}
\label{eq:rp:pleq2:2:ip:e1}
\norm{\vartheta(\tildePhiplussHigh{i'})}^2_{W_{p-3}^{\reg-\sfrak-i'-1}(\Donetwo^{\geq R_0})}\lesssim{}& \sum_{j=0}^{i'}\norm{\tildePhiplussHigh{j}}^2_{W_{p-5}^{\reg-\sfrak-j-1}(\Donetwo^{\geq R_0})}\notag\\
\lesssim{}&R_0^{-2}\norm{\tildePhiplussHigh{i'}}^2_{W_{p-3}^{\reg-\sfrak-i'-1}(\Donetwo^{\geq R_0})}+\sum_{j=0}^{i'-1}\norm{\tildePhiplussHigh{j}}^2_{W_{p-5}^{\reg-\sfrak-j-1}(\Donetwo^{\geq R_0})} ,
\end{align}
where the term with $R_0^{-2}$ coefficient can be absorbed by choosing $R_0\geq \hat{R}_0$ sufficiently large.
We can then take a weighted sum of the estimates \eqref{eq:rp:less2:=2:2:ip} for $i'\in \{0,1,\ldots, i\}$ such that the last term in \eqref{eq:rp:pleq2:2:ip:e1} is aslo absorbed, and the error terms which are supported on $[R_0-M, R_0]$ and implicit in the symbol $\lesssim_{[R_0-M, R_0]}$ can be controlled by adding in a sufficient multiple of the energy and Morawetz estimate  \eqref{eq:BEAM}. This thus completes the proof of \eqref{eq:BEDC:Phiplus:0} for $p\in [0,2]$.

Consider next the case that $i=\ell-\sfrak$ and $2\leq p<4$. The steps are the same as the above for $p\in [0,2]$ except that we are now applying the estimate \eqref{eq:rp:pleq4:2} for $p\in [2,4)$ to equation \eqref{eq:Phiplushighi:Schw:generall:tildePhiplusi} with $i=\ell-\sfrak$; we arrive at
\begin{align}
\hspace{2ex}&\hspace{-2ex}
\norm{rV\tildePhiplussHigh{\ell-\sfrak}}^2_{W_{p-2}^{\reg-\ell-1}(\Sigmatwo^{\geq R_0})}
+\norm{\tildePhiplussHigh{\ell-\sfrak}}^2_{W_{-2}^{\reg-\ell}(\Sigmatwo^{\geq R_0})}
\notag\\
\hspace{2ex}&\hspace{-2ex}
+\norm{\tildePhiplussHigh{\ell-\sfrak}}^2_{W_{-1-\delta}^{\reg-\ell}(\Donetwo^{\geq R_0})}
+\norm{rV\tildePhiplussHigh{\ell-\sfrak}}^2_{W_{p-3}^{\reg-\ell-1}(\Donetwo^{\geq R_0})}
\notag\\
&\lesssim_{[R_0-M,R_0]} {}\norm{rV\tildePhiplussHigh{\ell-\sfrak}}^2_{W_{p-2}^{\reg-\ell-1}(\Sigmaone^{\geq R_0})}
+\norm{\tildePhiplussHigh{\ell-\sfrak}}^2_{W_{-2}^{\reg-\ell}(\Sigmaone^{\geq R_0})}
+\norm{\vartheta(\tildePhiplussHigh{\ell-\sfrak})}^2_{W_{p-3}^{\reg-\ell-1}(\Donetwo^{\geq R_0})}.
\end{align}
It remains to estimate $\norm{\vartheta(\tildePhiplussHigh{\ell-\sfrak})}^2_{W_{p-3}^{\reg-\ell-1}(\Donetwo^{\geq R_0})}$ by
\begin{align}
\label{eq:rp:p24:2:ip:e1}
\norm{\vartheta(\tildePhiplussHigh{\ell-\sfrak})}^2_{W_{p-3}^{\reg-\ell-1}(\Donetwo^{\geq R_0})}\lesssim{}& \sum_{j=0}^{\ell-\sfrak}\norm{\tildePhiplussHigh{j}}^2_{W_{p-5}^{\reg-\sfrak-j-1}(\Donetwo^{\geq R_0})}\notag\\
\lesssim{}&R_0^{p+\delta-4}\Big(\norm{\tildePhiplussHigh{\ell-\sfrak}}^2_{W_{-1-\delta}^{\reg-\ell-1}(\Donetwo^{\geq R_0})}+\sum_{j=0}^{\ell-\sfrak-1}\norm{\tildePhiplussHigh{j}}^2_{W_{-1-\delta}^{\reg-\sfrak-j-1}(\Donetwo^{\geq R_0})}\Big) .
\end{align}
Thus, by adding a suitable weighted sum of the estimate \eqref{eq:rp:lp=2:2:ip} for $i\in \{0,1,\ldots, \ell-\sfrak\}$ (e.g. by taking the coefficients of this weighted sum to satisfy $C_0\gg C_1 \gg\ldots \gg C_{\ell-\sfrak}\gg 1$) to the estimate \eqref{eq:rp:p24:2:ip:e1} and taking $\delta$ sufficiently small compared to $4-p$,
the error terms $\sum_{i'=0}^{\ell-\sfrak}C_{i'}\norm{\vartheta(\tildePhiplussHigh{i'})}^2_{W_{-1}^{\reg-\sfrak-i'-1}(\Donetwo^{\geq R_0})}+ \norm{\vartheta(\tildePhiplussHigh{\ell-\sfrak})}^2_{W_{p-3}^{\reg-\ell-1}(\Donetwo^{\geq R_0})}$ are absorbed. In the end, we add the energy and Morawetz estimate \eqref{eq:BEAM} to bound the error terms implicit in the symbol $\lesssim_{[R_0-M, R_0]}$,  thus proving the estimate \eqref{eq:BEDC:Phiplus:0} in the case that $i=\ell-\sfrak$ and $p\in [2,4)$.
\end{proof}

The above global $r^p$ estimates are readily employed to deduce energy decay estimates.

\begin{cor}
\label{cor:EnerDecay:0to4:ip:1}
Let $i_1,i_2\in \mathbb{N}$ and $p_1, p_2\in \mathbb{R}^+\cup \{0\}$ be such that either of the following  holds:
\begin{itemize}
\item $i_1=i_2<\ell-\sfrak$ and $p_1\leq p_2\leq 2$;
\item $i_1<i_2< \ell-\sfrak$, $p_1\leq 2$, $p_2\leq 2$;
\item $i_1= i_2=\ell-\sfrak$, $p_1\leq p_2<4$;
\item $i_1< i_2=\ell-\sfrak$, $p_1\leq 2$, $p_2<4$.
\end{itemize}
Then there exists a constant $\regl=\regl(j,i_2-i_1)$, which grows linearly in its arguments, such that for any $\tb_2>\tb_1\geq \tb_0$,
\begin{align}
\label{eq:EnerDecay:0to4:ip:1}
\hspace{4ex}&\hspace{-4ex}
F(\reg, i_1, p_1,\tb_2,\Lxi^j\Psipluss)
+\int_{\tb_2}^{\infty}F(\reg-1, i_1, p_1-1,\tb',\Lxi^j\Psipluss)\di \tb'\notag\\
\lesssim{}& \langle\tb_2-\tb_1\rangle^{-2(i_2-i_1)-(p_2-p_1)-2j}F(\reg+\regl, i_2, p_2,\tb_1,\Psipluss).
\end{align}
\end{cor}

\begin{proof}
For $\Psipluss$ and a fixed $i$, the proven estimate \eqref{eq:BEDC:Phiplus:0} can be put into the form of \eqref{eq:Rev:HierarchyToDecayReal:EvolutionHypothesis} with $D=0$. Function $F(\reg, i, p,\tb,\Psipluss)$ clearly satisfies assumption \ref{pt:implydecay:1} and, in view of Definition \ref{def:Ffts:Phiplus:-1to5:all} for and using the H\"older's inequality,  $F(\reg, i, p,\tb,\Psipluss)$ also satisfies assumption \ref{pt:implydecay:2}. Consequently, Lemma \ref{lem:hierarchyImpliesDecay} can be applied to the estimates in Proposition \ref{prop:BEDC:Phiplus:0} to yield that for any $0\leq i\leq \ell-\sfrak$, $0\leq p_1\leq p_2\leq 2$, and $\tb_2>\tb_1\geq \tb_0$,
\begin{align}
F(\reg-2,i, p_1,\tb_2,\Psipluss)\lesssim {}&\langle\tb_2-\tb_1\rangle^{-(p_2-p_1)}F(\reg,i, p_2,\tb_1,\Psipluss),
\end{align}
and for $i=\ell-\sfrak$, $0\leq p_1\leq p_2<4$, and $\tb_2>\tb_1\geq\tb_0$,
\begin{align}
F(\reg-4,\ell-\sfrak, p_1,\tb_2,\Psipluss)\lesssim {}&\langle\tb_2-\tb_1\rangle^{-(p_2-p_1)}F(\reg,\ell-\sfrak, p_2,\tb_1,\Psipluss),
\end{align}
Note that by Definition \ref{def:tildePhiplusandminusHigh}, we have for any $0\leq j\leq \ell-\sfrak$ that $\tildePhiplussHigh{i}=\curlVR\PhiplussHigh{i-1}+\sum_{i'=0}^{i-1}O(1)\tildePhiplussHigh{i'}$, therefore,
\begin{align}
F(\reg, i-1, 2,\tb,\Psipluss)\sim F(\reg, i, 0,\tb,\Psipluss).
\end{align}
Combining the above estimates then prove the estimate \eqref{eq:EnerDecay:0to4:ip:1} for $j=0$.

We prove the general $j$ case of the estimate \eqref{eq:EnerDecay:0to4:ip:1} by induction. Assume it holds for $j$, and it suffices to prove the $j+1$ case.
Since $\Lxi$ is a Killing vector field and commutes with the wave equations of $\tildePhiplussHigh{i}$, the above estimates in this proof are still valid by replacing $\Psipluss$ with $\Lxi^j\Psipluss$, which thus yields
\begin{align}\label{eq:BEDC:Phiplus:0:1234}
F(\reg, i_1, p_1,\tb_2,\Lxi^{j+1}\Psipluss)
\lesssim {}&\langle\tb_2-\tb_1\rangle^{-(2-p_1)}F(\reg+\regl, i_1, 2,\tb_2-(\tb_2-\tb_1)/3,\Lxi^{j+1}\Psipluss).
\end{align}
We use equation \eqref{eq:Phiplushighi:Schw:generall:tildePhiplusi} and the expression $Y=\mu^{-1}\big(2\Lxi-V\big)$ away from horizon to rewrite $r^2V\Lxi \tildePhiplussHigh{i'}$ as a weighted sum of $(rV)^2\tildePhiplussHigh{i'}$,  $((\ell-\sfrak)(\ell+\sfrak+1)-i'(i'+2\sfrak+1))\tildePhiplussHigh{i'}$,
 $rV\tildePhiplussHigh{i'}$,  and $\sum_{j=0}^{i'}O(r^{-1})\tildePhiplussHigh{j}$ all with $O(1)$ coefficients. Therefore, in the case that $0\leq i\leq \ell-\sfrak$,\begin{align}
F(\reg, i, 2,\tb,\Lxi^{j+1}\Psipluss)
\lesssim{}&
\sum_{i'=0}^{i} \norm{r^2V\Lxi^{j+1}\tildePhiplussHigh{i'}}^2_{W_{-2}^{\reg-\sfrak-i'-1}(\Sigmatb^{\geq R_0})}
+F(\reg+1, i, 0,\tb,\Lxi^{j}\Psipluss)\notag\\
\lesssim{}&F(\reg+1, i, 0,\tb,\Lxi^{j}\Psipluss);
\end{align}
and in the other case that $i=\ell-\sfrak$, $2\leq p<4$,
\begin{align}
F(\reg, \ell-\sfrak, p,\tb,\Lxi^{j+1}\Psiplus)
\lesssim{}&
 \norm{r^2V\Lxi^{j+1}\tildePhiplussHigh{i'}}^2_{W_{p-4}^{\reg-\sfrak-i'-1}(\Sigmatb^{\geq R_0})}
+F(\reg+1, \ell-\sfrak, 0,\tb,\Lxi^{j}\Psiplus)\notag\\
\lesssim{}&F(\reg+1, i, p-2,\tb,\Lxi^{j}\Psiplus)
\end{align}
where we have used in the last step that  the coefficient $(\ell-\sfrak)(\ell+\sfrak+1)-i(i+2\sfrak+1)=0$.
Consequently, the RHS of \eqref{eq:BEDC:Phiplus:0:1234} is bounded by
\begin{align}
\hspace{4ex}&\hspace{-4ex}
\langle\tb_2-\tb_1\rangle^{-(2-p_1)}F(\reg+\regl, i_1, 0,\tb_2-\frac{2}{3}(\tb_2-\tb_1),\Lxi^{j}\Psipluss)\notag\\
\lesssim{}&\langle\tb_2-\tb_1\rangle^{-2(i_2-i_1)-(p_2-p_1)-2(j+1)}F(\reg+\regl, i_2, p_2,\tb_1,\Psipluss),
\end{align}
where the last step follows by induction. This closes the proof.
\end{proof}

We can further extend the $p$-range in the above corollary by improving the upper bound from $4$ to $5$, hence obtain faster energy decay. 

\begin{lemma}
\label{eq:EnerDecay:ip:ingeneral}
Let $i_1,i_2\in \mathbb{N}$ and $p_1, p_2\in \mathbb{R}^+\cup \{0\}$ be such that either of the following  holds:
\begin{itemize}
\item $i_1=i_2<\ell-\sfrak$ and $p_1\leq p_2\leq 2$;
\item $i_1<i_2< \ell-\sfrak$, $p_1\leq 2$, $p_2\leq 2$;
\item $i_1= i_2=\ell-\sfrak$, $p_1\leq p_2<5$;
\item $i_1< i_2=\ell-\sfrak$, $p_1\leq 2$, $p_2<5$.
\end{itemize}
Then there exists a constant $\regl=\regl(j,i_2-i_1)$, which grows linearly in its arguments, such that for any $\tb_2>\tb_1\geq \tb_0$,
\begin{align}
\label{eq:EnerDecay:0to4:ip:1:0to5}
\hspace{4ex}&\hspace{-4ex}
F(\reg, i_1, p_1,\tb_2,\Lxi^j\Psipluss)
+\int_{\tb_2}^{\infty}F(\reg-1, i_1, p_1-1,\tb',\Lxi^j\Psipluss)\di \tb'\notag\\
\lesssim{}& \langle\tb_2-\tb_1\rangle^{-2(i_2-i_1)-(p_2-p_1)-2j}F(\reg+\regl, i_2, p_2,\tb_1,\Psipluss).
\end{align}
\end{lemma}

\begin{proof}
As can be seen from the proof of Corollary \ref{cor:EnerDecay:0to4:ip:1}, in order to prove this lemma, it suffices to extend the estimate \eqref{eq:BEDC:Phiplus:0} to the case that $i= \ell-\sfrak$ and $p\in [2,5)$. Specifically, we aim to prove that for any $p\in [4,5)$ and $\tb_2>\tb_1\geq \tb_0$,
\begin{align}
\label{eq:BEDC:Phiplus:0:extend}
F(\reg-\regl, \ell-\sfrak, p,\tb_2,\Psipluss)
+\int_{\tb_1}^{\tb_2}F(\reg-\regl-1, \ell-\sfrak, p-1,\tb,\Psipluss)\di \tb
\lesssim F(\reg, \ell-\sfrak, p,\tb_1,\Psipluss),
\end{align}
where $\regl$ is a finite universal constant.

By applying \eqref{eq:rp:pgeq4:2} to  equation \eqref{eq:Phiplushighi:Schw:generall:tildePhiplusi} with $i=\ell-\sfrak$, and together with the expression $\vartheta(\tildePhiplussHigh{\ell-\sfrak})=\sum_{i=0}^{\ell-\sfrak}O(r^{-1})\tildePhiplussHigh{i}$ which follows from equation \eqref{eq:Phiplushighi:Schw:generall:tildePhiplusi}, we get for  $p\in [4,5)$ and any $\varsigma>0$ that
\begin{align}
\hspace{2ex}&\hspace{-2ex}
\norm{rV\tildePhiplussHigh{\ell-\sfrak}}^2_{W_{p-2}^{\reg-\ell-1}(\Sigmatwo^{\geq R_0})}
+\norm{\tildePhiplussHigh{\ell-\sfrak}}^2_{W_{-2}^{\reg-\ell}(\Sigmatwo^{\geq R_0})}
+\norm{\tildePhiplussHigh{\ell-\sfrak}}^2_{W_{-1-\delta}^{\reg-\ell}(\Donetwo^{\geq R_0})}
+\norm{rV\tildePhiplussHigh{\ell-\sfrak}}^2_{W_{p-3}^{\reg-\ell-1}(\Donetwo^{\geq R_0})}
\notag\\
&\lesssim_{[R_0-M,R_0]} {}\norm{rV\tildePhiplussHigh{\ell-\sfrak}}^2_{W_{p-2}^{\reg-\ell-1}(\Sigmaone^{\geq R_0})}
+\norm{\tildePhiplussHigh{\ell-\sfrak}}^2_{W_{-2}^{\reg-\ell}(\Sigmaone^{\geq R_0})}
\notag\\
&+\veps\int_{\tb_1}^{\tb_2}\langle\tb-\tb_1\rangle^{-1-\varsigma}\norm{rV\tildePhiplussHigh{\ell-\sfrak}}^2_{W_{p-2}^{\reg-\ell-1}(\Sigmatb^{\geq R_0})}\di \tb\notag\\
&+\veps^{-1}\int_{\tb_1}^{\tb_2}\langle\tb-\tb_1\rangle^{1+\varsigma}\Big(\norm{\vartheta(\tildePhiplussHigh{\ell-\sfrak})}^2_{W_{p-4}^{\reg-\ell-s}(\Sigmatb^{\geq R_0})}
+\norm{\tildePhiplussHigh{\ell-\sfrak}}^2_{W_{p-6}^{\reg-\ell-1}(\Sigmatb^{\geq R_0})}\Big)\di \tb\notag\\
&\lesssim_{[R_0-M,R_0]} {}\norm{rV\tildePhiplussHigh{\ell-\sfrak}}^2_{W_{p-2}^{\reg-\ell-1}(\Sigmaone^{\geq R_0})}
+\norm{\tildePhiplussHigh{\ell-\sfrak}}^2_{W_{-2}^{\reg-\ell}(\Sigmaone^{\geq R_0})}
\notag\\
&
+\veps^{-1}\sum_{i=0}^{\ell-\sfrak}\int_{\tb_1}^{\tb_2}\langle\tb-\tb_1\rangle^{1+\varsigma}\norm{\tildePhiplussHigh{i}}^2_{W_{p-6}^{\reg-\ell-1}(\Sigmatb^{\geq R_0})}\di \tb
+\frac{\veps}{\varsigma}\sup_{\tb\in[\tb_1,\tb_2]}\norm{rV\tildePhiplussHigh{\ell-\sfrak}}^2_{W_{p-2}^{\reg-\ell-1}(\Sigmatb^{\geq R_0})}.
\end{align}
The last term can be absorbed by taking a supreme norm over the LHS and choosing $\veps$ sufficiently small compared to $\varsigma$, thus
\begin{align}
\label{eq:BEDC:Phiplus:0:extend:11}
\hspace{1ex}&\hspace{-1ex}
\norm{rV\tildePhiplussHigh{\ell-\sfrak}}^2_{W_{p-2}^{\reg-\ell-1}(\Sigmatwo^{\geq R_0})}
+\norm{\tildePhiplussHigh{\ell-\sfrak}}^2_{W_{-2}^{\reg-\ell}(\Sigmatwo^{\geq R_0})}
+\norm{\tildePhiplussHigh{\ell-\sfrak}}^2_{W_{-1-\delta}^{\reg-\ell}(\Donetwo^{\geq R_0})}
+\norm{rV\tildePhiplussHigh{\ell-\sfrak}}^2_{W_{p-3}^{\reg-\ell-1}(\Donetwo^{\geq R_0})}
\notag\\
&\lesssim_{[R_0-M,R_0]} {}\norm{rV\tildePhiplussHigh{\ell-\sfrak}}^2_{W_{p-2}^{\reg-\ell-1}(\Sigmaone^{\geq R_0})}
+\norm{\tildePhiplussHigh{\ell-\sfrak}}^2_{W_{-2}^{\reg-\ell}(\Sigmaone^{\geq R_0})}\notag\\
&\qquad\qquad\quad\,\,
+\frac{1}{\veps}\sum_{i=0}^{\ell-\sfrak}\int_{\tb_1}^{\tb_2}\langle\tb-\tb_1\rangle^{1+\varsigma}\norm{\tildePhiplussHigh{i}}^2_{W_{p-6}^{\reg-\ell-1}(\Sigmatb^{\geq R_0})}\di \tb.
\end{align}
From the Hardy's inequality \eqref{eq:HardyIneqLHS} and  the proven inequality \eqref{eq:EnerDecay:0to4:ip:1} with $i_1=i_2=\ell-\sfrak$, $p_1=p-4$ and $p_2=3.5$,  we have for $p\in [4,5)$ that
\begin{align}
\hspace{4ex}&\hspace{-4ex}\sum_{i=0}^{\ell-\sfrak}\int_{\tb_1}^{\tb_2}\langle\tb-\tb_1\rangle^{1+\varsigma}\norm{\tildePhiplussHigh{i}}^2_{W_{p-6}^{\reg-\ell-1}(\Sigmatb^{\geq R_0})}\di\tb\notag\\
\lesssim{}& \int_{\tb_1}^{\tb_2}\langle\tb-\tb_1\rangle^{1+\varsigma}F(\reg, \ell-\sfrak, p-4, \tb, \Psipluss)\di \tb\notag\\
\lesssim {}& \int_{\tb_1}^{\tb_2}\langle\tb-\tb_1\rangle^{1+\varsigma}\langle\tb-\tb_1\rangle^{p-15/2}\di \tb \times F(\reg+4, \ell-\sfrak, 3.5, \tb_1, \Psipluss)\notag\\
\lesssim{}&F(\reg+4, \ell-\sfrak, 3.5, \tb_1, \Psipluss).
\end{align}
Plugging this estimate into \eqref{eq:BEDC:Phiplus:0:extend:11} and adding the estimate \eqref{eq:BEDC:Phiplus:0} together, we prove the estimate \eqref{eq:BEDC:Phiplus:0:extend} for any $p\in [4,5)$ and $\tb_2>\tb_1\geq \tb_0$, with $\regl=4$.
\end{proof}

\subsubsection{Energy decay estimates for a fixed $\ell$ mode of the spin $-\sfrak$ component}
\label{subsect:WED:-minus}

Throughout this subsubsection, we focus on a fixed $\ell$ mode and, unless otherwise stated, we drop the subscript indicating the $\ell$ mode in the scalars constructed from the spin $-\sfrak$ component.

The presentation in this subsubsection \ref{subsect:WED:+} for the spin $-\sfrak$ component is in the same spirit of the one in the earlier subsubsection for the spin $+\sfrak$ component. 

Again, the energy decay estimates are proved by deriving global $r^p$ estimates for the wave systems in Sections \ref{sect:BEAM}--\ref{sect:extendedwavesystem} for the spin $-\sfrak$ component. In order to do so, we first introduce for the spin $-\sfrak$ component a few weighted energies that correspond to $r^p$-weighted energies with different values of $p$. 

\begin{definition}
\label{def:Ffts:Phiminus:-1to5}
Define $F(\reg,i,p,\tb,\Psiminuss)$ as follows:
\begin{itemize}
\item if $i\leq\ell$,
\begin{align*}
F(\reg,i,p,\tb,\Psiminuss)={}&0, \qquad\text{for } p\in [-1,0),\\
F(\reg,i,p,\tb,\Psiminuss)={}&\sum_{j=0}^{i+\sfrak} \Big(\norm{rV\tildePhiminuss{j}}^2_{W_{p-2}^{\reg-\sfrak-j-1}(\Sigmatb^{\geq R_0})}
+\norm{\tildePhiminuss{j}}^2_{W_{-2}^{\reg-\sfrak-j}(\Sigmatb^{\geq R_0})}\Big)\notag\\
&
+\mathbf{BE}^{\reg}_{\Sigmatb}(\Psiminuss),
\quad \text{for } p\in [0,2];
\end{align*}
\item additionally, for $i=\ell$ and $p\in [2,5)$,
\begin{align*}
F(\reg,\ell,p,\tb,\Psiminuss)={}& \sum_{j=0}^{\ell+\sfrak-1} \Big(\norm{rV\tildePhiminuss{j}}^2_{W_{0}^{\reg-\sfrak-j-1}(\Sigmatb^{\geq R_0})}
+\norm{\tildePhiminuss{j}}^2_{W_{-2}^{\reg-\sfrak-j}(\Sigmatb^{\geq R_0})}\Big)\notag\\
&+\norm{rV\tildePhiminuss{\ell+\sfrak}}^2_{W_{p-2}^{\reg-\ell-2\sfrak-1}(\Sigmatb^{\geq R_0})}
+\norm{\tildePhiminuss{\ell+\sfrak}}^2_{W_{-2}^{\reg-\ell-2\sfrak}(\Sigmatb^{\geq R_0})}
+\mathbf{BE}^{\reg}_{\Sigmatb}(\Psiminuss).
\end{align*}
\end{itemize}
For $j\in \mathbb{N}$, we similarly define $F(\reg,i,p,\tb,\Lxi^j\Psiminus)$ with all the scalars in the Sobolev norms acted by $\Lxi^j$.
\end{definition}

Next, we derive global $r^p$ estimates for the spin $-\sfrak$ component by applying the $r^p$ lemma \ref{lem:wave:rp} to each equation of the spin $-\sfrak$ component in the wave systems in Proposition \ref{prop:wave:Phihigh:pm1}, and then deduce energy decay estimates from the global $r^p$ estimates. Note that the upper bound of the range of value of $p$ is $2$, hence the achieved energy decay estimates are fairly weak. 

\begin{prop}
\label{prop:rp:im:lower:general}
For all $i< \sfrak$, $p\in [0,2]$ and $\tb_2>\tb_1\geq \tb_0$, there exists a nonnegative universal constant $\regl=\regl(\sfrak)$ such that
\begin{align}
\label{eq:BEDC:Phiminus:0:11}
F(\reg, i, p,\tb_2,\Psiminuss)
+\int_{\tb_1}^{\tb_2}F(\reg-1, i, p-1,\tb,\Psiminuss)\di \tb
\lesssim F(\reg+\regl, i, p,\tb_1,\Psiminuss).
\end{align}
Moreover, in both the case that $0\leq i_1<i_2<\sfrak$, $0\leq p_1,p_2\leq 2$ and the case that $0\leq i_1=i_2<\sfrak$, $0\leq p_1\leq p_2\leq 2$, there exists a nonnegative universal constant $\regl=\regl(j, i_2-i_1)$, which grows linearly in its arguments, such that for any $j\in \mathbb{N}$ and $\tb_2>\tb_1\geq \tb_0$,
\begin{align}\label{eq:BEDC:Phiminus:0:11:1}
F(\reg, i_1, p_1,\tb_2,\Lxi^j\Psiminuss)
\lesssim \langle \tb_2-\tb_1\rangle^{-2(i_2-i_1)-(p_2-p_1)-2j}F(\reg+\regl, i_2, p_2,\tb_1,\Psiminuss).
\end{align}
\end{prop}

\begin{proof}
In the case of $\sfrak=0$, this has been proven in Proposition \ref{prop:BEDC:Phiplus:0}. It remains to prove the estimates for $\sfrak=1,2$, and we consider the $\sfrak= 1$ and $\sfrak=2$ cases separately.

For $\sfrak=1$ case, we first consider the system of equations \eqref{eq:RWPhi01schwope} and
\eqref{eq:RWPhi11schwope}. These two subequations can be both put into the form of equation \eqref{eq:wave:rp} and the assumptions in Lemma \ref{lem:wave:rp} are satisfied. In particular, $\vartheta(\Phiminus{0})=2(r-3M)r^{-2}\Phiminus{1}$ and $\vartheta(\Phiminus{1})=0$. We apply to these two subequations the estimates of both \eqref{eq:rp:less2:2} and \eqref{eq:rp:lp=2:2} and sum them up, and the error term $\norm{\vartheta(\Phiminus{0})}^2_{W_{p-3}^{\reg}(\Donetwo^{\geq R_0})}\lesssim \norm{\Phiminus{1}}^2_{W_{-3}^{\reg}(\Donetwo^{\geq R_0})}$ can clearly be absorbed by taking $R_0\geq \hat{R}_0$ sufficiently large. Thus, this proves the estimate \eqref{eq:BEDC:Phiminus:0:11} with $\regl=0$ for all $i<1$, $p\in [0,2]$ and $\tb_2>\tb_1\geq \tb_0$, where the $p=0$ case follows from the energy and Morawetz estimate \eqref{eq:BEAM}.

Consider next the $\sfrak=2$ case. By exactly the same argument as for the $\sfrak=1$ case, one can show the estimate \eqref{eq:BEDC:Phiminus:0:11} with $\regl=0$  for all $i=0$, $p\in [0,2]$ and $\tb_2>\tb_1\geq \tb_0$.  Moreover, the estimate \eqref{eq:BEDC:Phiminus:0:11} with $\regl=0$  for $i=1$, $p=0$ and $\tb_2>\tb_1\geq \tb_0$ trivially holds due to the relation
\begin{align}\label{eq:equi:im:10}
F(\reg, 1, 0,\tb,\Psiminustwo)\sim F(\reg, 0, 2,\tb,\Psiminustwo)
\end{align}
 which is valid by Definition \ref{def:PsiPhietc}.
By putting quation \eqref{eq:RWPhi32schwope} of $\Phiminustwo{3}$ in the form of \eqref{eq:wave:rp}, one finds $\vartheta(\Phiminustwo{3})=-6M\Phiminustwo{2}$. We first apply the estimate \eqref{eq:rp:less2:2} and achieve for any $p\in (0,2)$ and $\tb_2>\tb_1\geq \tb_0$ that
  \begin{align}\label{eq:rp:less2:2:3m}
\hspace{2ex}&\hspace{-2ex}
\norm{rV\Phiminustwo{3}}^2_{W_{p-2}^{\reg-6}(\Sigmatwo^{\geq R_0})}
+\norm{\Phiminustwo{3}}^2_{W_{-2}^{\reg-5}(\Sigmatwo^{\geq R_0})}
+\norm{\Phiminustwo{3}}^2_{W_{p-3}^{\reg-5}(\Donetwo^{\geq R_0})}
\notag\\
&\lesssim_{[R_0-M,R_0]} {}\norm{rV\Phiminustwo{3}}^2_{W_{p-2}^{\reg-6}(\Sigmaone^{\geq R_0})}
+\norm{\Phiminustwo{3}}^2_{W_{-2}^{\reg-5}(\Sigmaone^{\geq R_0})}
+\norm{\Phiminustwo{2}}^2_{W_{p-3}^{{\reg-6}}(\Donetwo^{\geq R_0})}.
\end{align}
By the Hardy's inequality \eqref{eq:HardyIneqLHS}, we have for the last term
\begin{align}
\norm{\Phiminustwo{2}}^2_{W_{p-3}^{{\reg-6}}(\Donetwo^{\geq R_0})}
\lesssim_{[R_0-M,R_0]} {}\norm{rV\Phiminustwo{2}}^2_{W_{p-3}^{{\reg-6}}(\Donetwo^{\geq R_0})}\lesssim_{[R_0-M,R_0]}\norm{\Phiminustwo{3}}^2_{W_{p-5}^{{\reg-6}}(\Donetwo^{\geq R_0})}
\end{align}
which can thus be absorbed. Hence, we prove the estimate  \eqref{eq:BEDC:Phiminus:0:11} with $\regl=0$  for all $i=1$, $p\in [0,2)$ and $\tb_2>\tb_1\geq \tb_0$. Similarly to the proof of Corollary \ref{cor:EnerDecay:0to4:ip:1}, applying Lemma \ref{lem:hierarchyImpliesDecay} yields that there exists a nonnegative universal constant $\regl$ such that for all $p=2$ and $\tb_2>\tb_1\geq \tb_0$,
\begin{align}\label{eq:rp:less2:2:3m:112}
F(\reg, 0, 0,\tb_2,\Psiminustwo)\lesssim {}
\langle\tb_2-\tb_1\rangle^{-2-p}F(\reg+\regl, 1, p,\tb_1,\Psiminustwo).
\end{align}
We then apply the estimate \eqref{eq:rp:lp=2:2:v22} to equation \eqref{eq:RWPhi32schwope} of $\Phiminustwo{3}$ and achieve for any $p\in (0,2)$ and $\tb_2>\tb_1\geq \tb_0$ that
  \begin{align}\label{eq:rp:less2:2:3m:22}
\hspace{2ex}&\hspace{-2ex}
\norm{rV\Phiminustwo{3}}^2_{W_{0}^{\reg-6}(\Sigmatwo^{\geq R_0})}
+\norm{\Phiminustwo{3}}^2_{W_{-2}^{\reg-5}(\Sigmatwo^{\geq R_0})}
+\norm{\Phiminustwo{3}}^2_{W_{-1}^{\reg-5}(\Donetwo^{\geq R_0})}
\notag\\
&\lesssim_{[R_0-M,R_0]} {}\norm{rV\Phiminustwo{3}}^2_{W_{0}^{\reg-6}(\Sigmaone^{\geq R_0})}
+\norm{\Phiminustwo{3}}^2_{W_{-2}^{\reg-5}(\Sigmaone^{\geq R_0})}
\notag\\
&
{}+\veps\int_{\tb_1}^{\tb_2}\langle\tb-\tb_1\rangle^{-1-\varsigma}\norm{rV\Phiminustwo{3}}^2_{W_{0}^{\reg-6}(\Sigmatb^{\geq R_0})}\di \tb
+\frac{1}{\veps}\int_{\tb_1}^{\tb_2}\langle\tb-\tb_1\rangle^{1+\varsigma}\norm{\Phiminustwo{2}}^2_{W_{-2}^{\reg-6}(\Sigmatb^{\geq R_0})}
\di \tb.
\end{align}
The second last term on the RHS of \eqref{eq:rp:less2:2:3m:22} is absorbed by taking a supreme norm over the LHS and taking $\veps$ small enough, and, using  the estimate \eqref{eq:rp:less2:2:3m:112} with $p=1.5$, the last term is bounded by $\frac{1}{\veps}F(\reg+\regl, 1, 1.5,\tb_1,\Psiminustwo)$. Together with the proven  estimate  \eqref{eq:BEDC:Phiminus:0:11} with $\regl=0$  for all $i=1$, $p\in [0,2)$ and $\tb_2>\tb_1\geq \tb_0$, we conclude the $i=1$, $p=2$ case of  \eqref{eq:BEDC:Phiminus:0:11} and hence complete the proof of \eqref{eq:BEDC:Phiminus:0:11}.

The estimate \eqref{eq:BEDC:Phiminus:0:11:1} for the $j=0$ case follows simply by applying Lemma \ref{lem:hierarchyImpliesDecay} to \eqref{eq:BEDC:Phiminus:0:11} and the equivalence relation \eqref{eq:equi:im:10}.
The general $j$ case holds by the same way of arguing as in Corollary \ref{cor:EnerDecay:0to4:ip:1}.
\end{proof}

The final step is to extend the $p$-range in the above lemma, thus obtaining improved energy decay. 

\begin{lemma}
\label{lem:EnerDec:0to5:im}
Let $i_1,i_2\in \mathbb{N}$ and $p_1, p_2\in \mathbb{R}^+\cup \{0\}$ be such that either of the following  holds:
\begin{itemize}
\item $i_1=i_2<\ell$ and $p_1\leq p_2\leq 2$;
\item $i_1<i_2< \ell$, $p_1\leq 2$, $p_2\leq 2$;
\item $i_1= i_2=\ell$, $p_1\leq p_2<5$;
\item $i_1< i_2=\ell$, $p_1\leq 2$, $p_2<5$.
\end{itemize}
Then there exists a constant $\regl=\regl(j,i_2-i_1)$, which grows linearly in its arguments, such that for any $\tb_2>\tb_1\geq \tb_0$,
\begin{align}
\label{eq:EnerDecay:0to4:im:1}
\hspace{4ex}&\hspace{-4ex}
F(\reg, i_1, p_1,\tb_2,\Lxi^j\Psiminuss)
+\int_{\tb_2}^{\infty}F(\reg-1, i_1, p_1-1,\tb',\Lxi^j\Psiminuss)\di \tb'\notag\\
\lesssim {}&\langle\tb_2-\tb_1\rangle^{-2(i_2-i_1)-(p_2-p_1)-2j}F(\reg+\regl, i_2, p_2,\tb_1,\Psiminuss).
\end{align}
\end{lemma}

\begin{proof}
For each $j\in \mathbb{N}$,
$\tildePhiplussHigh{j}$ and $\tildePhiminuss{j+2\sfrak}$ satisfy the same equations, therefore, analogous to the discussions for the spin $-\sfrak$ component, in particular, analogous to the statement in Lemma \ref{eq:EnerDecay:ip:ingeneral}, the statements in Lemma \ref{lem:EnerDec:0to5:im} are valid but under an additional, common assumption that $i_1\geq \sfrak$. Meanwhile, the case that $i_1<\sfrak$ and $i_2<\sfrak$ has been treated in Proposition \ref{prop:rp:im:lower:general}.  In the remainder case that $i_1<\sfrak$ and $i_2\geq \sfrak$, we have from Proposition \ref{prop:rp:im:lower:general} that there exists a $\regl=\regl(j,\sfrak-i_1)$ such that
\begin{align*}
F(\reg, i_1, p_1,\tb_2,\Lxi^j\Psiminuss)
\lesssim{}& \langle \tb_2-\tb_1\rangle^{-2(\sfrak-1-i_1)-(2-p_1)-2j}F(\reg+\regl, \sfrak-1, 2,\tb_1+(\tb_2-\tb_1)/2,\Psiminuss)\\
\lesssim{}&\langle \tb_2-\tb_1\rangle^{-2(\sfrak-i_1)+p_1-2j}F(\reg+\regl, \sfrak, 0,\tb_1+(\tb_2-\tb_1)/2,\Psiminuss),
\end{align*}
where the last step is due to the relation $F(\reg, \sfrak-1, 2,\tb,\Psiminuss)\sim F(\reg, \sfrak, 0,\tb,\Psiminuss)$ by Definition \ref{def:PsiPhietc}.
On the other hand, the above discussions yield that there exists a $\regl=\regl(i_2-\sfrak)$ such that
\begin{align*}
F(\reg, \sfrak, 0,\tb_1+(\tb_2-\tb_1)/2,\Psiminuss)\lesssim{}&
\langle \tb_2-\tb_1\rangle^{-2(i_2-\sfrak)-p_2}
F(\reg+\regl, i_2, p_2,\tb_1,\Psiminuss)
\end{align*}
Combining these two estimates then implies the desired estimate.
\end{proof}

\subsubsection{Proof of the weak decay estimates in Proposition \ref{prop:weakdecay:summary}}
\label{subsect:WED:ptw:plusminus}

Given the energy decay estimates above, we now give a proof for Proposition \ref{prop:weakdecay:summary}. 

First, we notice by definition
that there exists a universal constant $\regl\geq 0$ such that for any $p\geq 0$, the following holds
\begin{subequations}
\label{eq:initialdata:equirela:pm}
\begin{align}
F(\reg-\regl, i-s, p,\tb,\Psipluss)\lesssim \InizeroEnergypluss{\tb}{i}{\reg}{p}&\lesssim F(\reg+\regl, i, p,\tb,\Psipluss) , \quad \text{for } \sfrak\leq i\leq \ell,\\
F(\reg-\regl, i, p,\tb,\Psiminuss)\lesssim \InizeroEnergyminuss{\tb}{i}{\reg}{p}&\lesssim F(\reg+\regl, i, p,\tb,\Psiminuss), \quad \text{for } 0\leq i\leq \ell.
\end{align}
\end{subequations}

Suppose  the spin $+\sfrak$ component is supported on a single $\ell$ mode. For any  $j\in \mathbb{N}$ and $\tb_2>\tb_1\geq \tb_0$, we utilize the estimates in Lemma \ref{eq:EnerDecay:ip:ingeneral} and conclude
\begin{itemize}
\item
the estimate \eqref{eq:Sobolev:3}  implies that for any $0\leq i\leq \ell-\sfrak$,  $0\leq p<5$,
\begin{subequations}
\begin{align}
\absCDeri{\Lxi^j(r^{-1}\Psipluss)}{\reg}
\lesssim{}&\bigg(\int_{\tb}^{\infty}F(\reg+\regl, 0, -1,\tb',\Lxi^j\Psipluss)\di \tb'
\int_{\tb}^{\infty}F(\reg+\regl, 0, -1,\tb',\Lxi^{j+1}\Psipluss)\di \tb'\bigg)^{\frac{1}{4}}\notag\\
\lesssim{}&(F(\reg+\regl, i, p,\tb_0,\Psipluss))^{\half}\tb^{-(p+1)/2-i-j}.
\end{align}
\item
the estimate \eqref{eq:Sobolev:1}  implies that for any $0\leq i\leq \ell-\sfrak$,  $0\leq p<5$,
\begin{align}
\absCDeri{\Lxi^j(r^{-1/2}\Psipluss)}{\reg}
\lesssim{}&\big(F(\reg+\regl, 0, 0,\tb,\Lxi^j\Psipluss)\big)^{\frac{1}{2}}
\lesssim(F(\reg+\regl, i, p,\tb_0,\Psipluss))^{\half}\tb^{-p/2-i-j},
\end{align}
\item
the estimate \eqref{eq:Sobolev:2} with $\alpha=\vartheta$, $\vartheta$ sufficiently small, implies that for any $(i,p)\in \{0< i\leq \ell-\sfrak, 0\leq p<5\}\cup\{i=0, p>1\} $
\begin{align}
\absCDeri{\Lxi^j\Psipluss}{\reg}
\lesssim{}&(F(\reg+\regl, 0, 1-\vartheta,\tb,\Psipluss)F(\reg+\regl, 0, 1+\vartheta,\tb,\Psipluss))^{\frac{1}{4}}\notag\\
\lesssim {}&\big(F(\reg+\regl, 0, 0,\tb,\Lxi^j\Psipluss)\big)^{\frac{1}{2}}
\tb^{-(p-1)/2-i-j},
\end{align}
\end{subequations}
\end{itemize}
Combining the above three estimates together yields that there exists a $\regl=\regl(j, i)$ such that for any $j\in \mathbb{N}$ and $(i,p)\in \{0< i\leq \ell-\sfrak, 0\leq p<5\}\cup\{i=0, p>1\} $,
\begin{align}
\absCDeri{\Lxi^j(r^{-1}\Psipluss)}{\reg}
\lesssim{}(F(\reg+\regl, i, p,\tb_0,\Psipluss))^{\half}v^{-1}\tb^{-(p-1)/2-i-j}.
\end{align}

Suppose  the spin $-\sfrak$ component is supported on a single $\ell$ mode.  Similarly, we employ the estimates in Lemma \ref{lem:EnerDec:0to5:im} together with the estimates \eqref{eq:Sobolev:1}--\eqref{eq:Sobolev:3} to conclude that that there exists a $\regl=\regl(j, i)$ such that for any $j\in \mathbb{N}$ and $(i,p)\in \{0< i\leq \ell, 0\leq p<5\}\cup\{i=0, p>1\} $,
\begin{align}
\sum_{n=0}^{\sfrak}\absCDeri{\Lxi^j(r^{-1}(r^2 V)^n\Psiminuss)}{\reg}
\lesssim{}(F(\reg+\regl, i, p,\tb_0,\Psiminuss))^{\half}v^{-1}\tb^{-(p-1)/2-i-j}.
\end{align}
Taking into account the relation \eqref{eq:initialdata:equirela:pm}, the above thus proves the estimate \eqref{eq:weakdecay:ell:ipm}.

In the end, we consider the case that the spin $+ \sfrak$ component is supported on $\ell\geq \ell_0$ modes with $\ell_0\geq \sfrak$, and the case for the spin $-\sfrak$ component can be analogously treated. In fact, the estimates in Lemma \ref{lem:wave:rp} can be applied to equation \eqref{eq:Phiplushighi:Schw:generall:tildePhiplusi}, hence the estimate \eqref{eq:BEDC:Phiplus:0} holds for all $0\leq i\leq \ell-\sfrak$, $p\in [0,2]$ and $\tb_2>\tb_1\geq \tb_0$. By going through the remaining proof in Section \ref{subsect:WED:+}, this yields
point \eqref{pt:WED:highmodes:+} of Proposition \ref{prop:weakdecay:summary}.

\begin{remark}
\label{rem:PhiplussHigh:almostsharp}
The above proof also yields that if the spin $s=\pm \sfrak$ components are supported on a fixed $\ell$ mode, $\ell\geq \sfrak$, then for any $j$, $0\leq i\leq \ell-\sfrak$ and $0\leq p<5$,  there exists a $\regl=\regl(j, i)$ such that in the region $r\geq 4M$,
\begin{align}
\absCDeri{\Lxi^j(r^{-1}\PhiplussHigh{i})}{\reg}
\lesssim{}&(\InizeroEnergypluss{\tb}{\ell}{\reg+\regl}{p})^{\half}v^{-1}\tb^{-(p-1)/2-(\ell-\sfrak-i)-j}.
\end{align}
\end{remark}

\subsection{Newman--Penrose constants}
\label{sect:NPconsts:Maxwell}

We construct the N--P constants for the spin $\pm \sfrak$ components in this subsection. As will be manifest soon, these N--P constants are in fact in terms of fixed modes of the components. The vanishing or non-vanishing property of these constants determines the sharp decay rate for the energy of the spin $\pm \sfrak$ components. Further, these constants are of paramount importance in characterising the precise asymptotics of any fixed mode or the field itself in proving the Price's law.

We start with defining the N--P constants.

\begin{definition}
\label{def:NPCs}
Let $\ell\in \mathbb{N}$ and $\ell\geq \sfrak$. Define the $(m,\ell)$-th N--P constants of  the spin $s=+ \sfrak$ and $s=-\sfrak$ components to be $\NPCPs{\ell}=\lim\limits_{\rb\to\infty}\curlVR(\tildePhiplussHigh{\ell-\sfrak})_{m,\ell}$ and $\NPCNs{\ell}=\lim\limits_{\rb\to\infty}\curlVR
(\tildePhiminuss{\ell+\sfrak})_{m,\ell}$, respectively, where $(\tildePhiplussHigh{\ell-\sfrak})_{m,\ell}$ and $(\tildePhiminuss{\ell+\sfrak})_{m,\ell}$ are defined as in Definition \ref{def:tildePhiplusandminusHigh}  from the $(m,\ell)$ mode of the spin $+\sfrak$ and $-\sfrak$ components, respectively. \end{definition}

\begin{remark}
As will be shown in Proposition \ref{prop:NPCsindepentonu}, these N--P constants are independent of $\tb$ under some general conditions, a fact which justifies they are indeed  real constants.
\end{remark}

These N--P constants are defined in such a way that only the $(m,\ell)$ mode is relevant to the $(m,\ell)$-th N--P constant. Hence, \textbf{in the rest of this subsection, we will always assume that the spin $\pm \sfrak$ components are supported only on the $(m,\ell)$ mode.}

The following two propositions are to discuss the properties of the N--P constants of the spin $\pm \sfrak$ components.

\begin{prop}
\label{prop:nullinfBeha:PhiplusiandPhinsHighi} Let $\ell\geq \sfrak$ and $\reg\in \mathbb{N}$.  Let $\regl=\regl(\ell)>0$ be suitably large.
\begin{enumerate}
\item Assume $\InizeroEnergypluss{\tb_0}{\ell}{\reg+\regl}{0}<\infty$
which is defined as in Definition \ref{def:energyonSigmazero:Schw}.
\begin{enumerate}
  \item[$(i)$] If
     $\lim\limits_{r\to\infty}\sum\limits_{j=0}^{\ell-\sfrak}\absCDeri{
     \PhiplussHigh{j}}{\reg}\vert_{\Sigmazero}
      <\infty$, then for any $\tb\geq \tb_0$, $\lim\limits_{r\to\infty}\sum\limits_{j=0}^{\ell-\sfrak} \absCDeri{\PhiplussHigh{j}}{\reg}\vert_{\Sigmatb}<\infty$. The same statement holds if one replaces all $\PhiplussHigh{j}$ by $\tildePhiplussHigh{j}$;

  \item[$(ii)$]  If $\lim\limits_{r\to\infty}\Big(\sum\limits_{j=0}^{\ell-\sfrak}\absCDeri{
      \PhiplussHigh{j}}{\reg}\vert_{\Sigmazero}
      +r^{-\alpha}\absCDeri{\PhiplussHigh{\ell-\sfrak+1}}{\reg}\vert_{\Sigmazero}\Big)
      <\infty$ for some $\alpha\in [0,2]$, then for any $\tb\geq \tb_0$, $\lim\limits_{r\to\infty}
      \Big(\sum\limits_{j=0}^{\ell-\sfrak}
      \absCDeri{\PhiplussHigh{j}}{\reg}\vert_{\Sigmatb}
      +r^{-\alpha}\absCDeri{\PhiplussHigh{\ell-\sfrak+1}}{\reg}\vert_{\Sigmatb}\Big)<\infty$.
      The same statement holds if one replaces all $\PhiplussHigh{j}$ and $\PhiplussHigh{\ell-\sfrak+1}$ by $\tildePhiplussHigh{j}$ and $\tildePhiplussHigh{\ell-\sfrak+1}$, respectively;
           \end{enumerate}
\item Assume $\InizeroEnergyminuss{\tb_0}{\ell}{\reg+\regl}{0}<\infty$
which is defined as in Definition \ref{def:energyonSigmazero:Schw}.
\begin{enumerate}
  \item[$(iii)$] If $\lim\limits_{r\to\infty}\sum\limits_{j=0}^{\ell+\sfrak}\absCDeri{
     \Phiminuss{j}}{\reg}\vert_{\Sigmazero}
      <\infty$, then for any $\tb\geq \tb_0$, $\lim\limits_{r\to\infty}\sum\limits_{j=0}^{\ell+\sfrak}      \absCDeri{\Phiminuss{j}}{\reg}\vert_{\Sigmatb}<\infty$.
      The same statement holds if one replaces all $\Phiminuss{j}$ by $\tildePhiminuss{j}$;
  \item[$(iv)$] If $\lim\limits_{r\to\infty}\Big(\sum\limits_{j=0}^{\ell+\sfrak}
      \absCDeri{
      \Phiminuss{j}}{\reg}\vert_{\Sigmazero}
      +r^{-\alpha}\absCDeri{
      \Phiminuss{\ell+\sfrak+1}}{\reg}\vert_{\Sigmazero}\Big)
      <\infty$  for some $\alpha\in [0,2]$, then for any $\tb\geq \tb_0$, $\lim\limits_{r\to\infty}\Big(\sum\limits_{j=0}^{\ell+\sfrak}
      \absCDeri{\Phiminuss{j}}{\reg}\vert_{\Sigmatb}
      +r^{-\alpha}\absCDeri{\Phiminuss{\ell+\sfrak+1}}{\reg}\vert_{\Sigmatb}\Big)<\infty$.
      The same statement holds if one replaces all $\Phiminuss{j}$ and $\Phiminuss{\ell+\sfrak+1}$ by $\tildePhiminuss{j}$ and $\tildePhiminuss{\ell+\sfrak+1}$, respectively.
\end{enumerate}
\end{enumerate}
\end{prop}

\begin{proof}
The assumption $\InizeroEnergypluss{\tb_0}{\ell}{\reg+\regl}{0}<\infty$ in particular yields that for any $\tb\geq \tb_0$ and any $1\leq j\leq \ell$,
\begin{align}
\label{eq:NPconstantProp:energy}
\norm{\Psipluss}^2_{W_{-2}^{\reg}(\Sigmatb)}
+\sum_{j=0}^{\ell-\sfrak}\norm{\PhiplussHigh{j}}^2_{W_{-2}^{\reg}(\Sigmatb^{\geq 4M})}
<\infty
\end{align}
and
\begin{align}
\label{eq:NPconstantProp:ptw}
\sup_{\Sigmatb}\int_{\Sphere} r^{-1}\absCDeri{\Psipluss}{\reg}^2\di^2\mu
+\sup_{\Sigmatb\cap\{\rb\geq 4M\}}\sum_{j=0}^{\ell-\sfrak}\int_{\Sphere} r^{-1}\absCDeri{\PhiplussHigh{j}}{\reg}^2\di^2\mu <\infty.
\end{align}
Note that the first estimate \eqref{eq:NPconstantProp:energy} follows from Lemma \ref{eq:EnerDecay:ip:ingeneral} and the relation  \eqref{eq:initialdata:equirela:pm}, and the second estimate \eqref{eq:NPconstantProp:ptw} follows from the Sobolev-type estimate \eqref{eq:Sobolev:1} together with the estimate \eqref{eq:NPconstantProp:energy}.
The rest of the proof is similar to the one of \cite[Propositions 3.4 and 3.5]{angelopoulos2018vector} and we omit it.
\end{proof}

\begin{prop}
\label{prop:NPCsindepentonu}
Let $\ell\geq \sfrak$, and let $\regl=\regl(\ell)>0$ be suitably large.
  \begin{itemize}
      \item[$(i)$] Assume $\InizeroEnergypluss{\tb_0}{\ell}{\regl}{0}<\infty$, and assume $\lim\limits_{r\to\infty}
          \sum\limits_{j=0}^{\ell-\sfrak}
          (\abs{
      \tildePhiplussHigh{j}}+\abs{
      \curlVR\tildePhiplussHigh{j}})\vert_{\Sigmazero}
      <\infty$, then the $(m,\ell)$-th N--P constant $\NPCPs{\ell}$ is finite and independent of $\tb$;
       \item[$(ii)$] Assume $\InizeroEnergyminuss{\tb_0}{\ell}{\regl}{0}<\infty$, and assume $\lim\limits_{r\to\infty}\sum\limits_{j=0}^{\ell+\sfrak}
          (\abs{
      \tildePhiminus{j}}+\abs{
      \curlVR\tildePhiminuss{j}})\vert_{\Sigmazero}
      <\infty$, then the $(m,\ell)$-th N--P constant $\NPCNs{\ell}$ is finite and independent of $\tb$.
      \end{itemize}
\end{prop}

\begin{proof}
Recall that we consider only the $(m,\ell)$ mode of the spin $\pm \sfrak$ component.
We have from Proposition \ref{prop:wave:Phihigh:pm1} the following equations for $\tildePhiplussHigh{\ell-\sfrak}$ and $ \tildePhiminuss{\ell+\sfrak} $:
\begin{subequations}
\begin{align}
\label{eq:Phiplushighi:Schw:generall:tildePhiplusi:223223}
&-2\pu \curlVR  \tildePhiplussHigh{\ell-\sfrak} -2(\ell+1)(r-3M)r^{-2}\curlVR\tildePhiplussHigh{\ell-\sfrak}
+\sum_{j=0}^{\ell-\sfrak}O(r^{-1})\tildePhiplussHigh{j}={}0,\\
\label{eq:tildePhiminusi:Schw:generall:tildePhiminusi:223223}
&-2\pu \curlVR  \tildePhiminuss{\ell+\sfrak} -2(\ell+1)(r-3M)r^{-2}\curlVR\tildePhiminuss{\ell+\sfrak}
+\sum_{j=2\sfrak}^{\ell+\sfrak}O(r^{-1})\tildePhiminuss{j}={}0.
\end{align}
\end{subequations}
Consider the spin $+\sfrak$ component case, the spin $-\sfrak$ component case being treated in the same fashion.
Since by Proposition \ref{prop:nullinfBeha:PhiplusiandPhinsHighi},  $\lim\limits_{r\to \infty}\big(\sum_{j=0}^{\ell-\sfrak}\abs{\tildePhiplussHigh{j}}+\absCDeri{\curlVR\tildePhiplussHigh{\ell-\sfrak}}{1}\big)\vert_{\Sigmatb}<\infty$
for any $\tb\geq \tb_0$, we conclude from \eqref{eq:Phiplushighi:Schw:generall:tildePhiplusi:223223} that
 \begin{align}
 \label{eq:LxicurlVRgotozero}
 \lim\limits_{\rb\to \infty}\Lxi \curlVR \tildePhiplussHigh{\ell-\sfrak}\vert_{\Sigmatb}=0.
 \end{align}
  By the bounded convergence theorem, the statement follows.
\end{proof}

Given the above results, and using the TSI in Lemma \ref{lem:TSI:gene}, we can show that the $(m,\ell)$-th N--P constants for the spin $+\sfrak$ and $-\sfrak$ components are in fact related to each other by a nonzero constant factor depending only on $\ell$ and $\sfrak$.

\begin{cor}
\label{cor:relationoftwoNPconsts}
Let $\ell\geq \sfrak$.
 Assume $\InizeroEnergypluss{\tb_0}{\ell}{\regl}{0}<\infty$ for a suitably large $\regl=\regl(\ell)>0$, and assume $\lim\limits_{r\to\infty}
          \sum\limits_{j=0}^{\ell-\sfrak}
          (\abs{
      \tildePhiplussHigh{j}}+\abs{
      \curlVR\tildePhiplussHigh{j}})\vert_{\Sigmazero}
      <\infty$. Then
      \begin{itemize}
      \item[$(i)$]
$\NPCNs{\ell}=\prod_{i=\ell-\sfrak+1}^{\ell+\sfrak}i\NPCPs{\ell}$;
\item[$(ii)$]if $\NPCNs{\ell}$ vanishes, then $\NPCPs{\ell}$ vanishes, and vice versa.
\end{itemize}
\end{cor}

\begin{proof}
Equation \eqref{eq:TSI:simpleform:l=1}
 is $\ell(\ell+1)\Phiplus=\Phiminus{2}$ and  equation \eqref{eq:TSI:simpleform:l=2} is
$(\ell-1)\ell (\ell+1)(\ell+2)\Phiplustwo=\Phiminustwo{4}+12 M \overline{\Lxi\Phiplustwo}$, which thus yield
\begin{align}
(\ell+1)\curlVR\tildePhiplusHigh{\ell-1}={}&\curlVR\tildePhiminus{\ell+1},\\
(\ell-1)\ell (\ell+1)(\ell+2)\curlVR\tildePhiplustwoHigh{\ell-2}={}&\curlVR\tildePhiminustwo{\ell+2}+12 M \overline{\Lxi\curlVR\tildePhiplustwoHigh{\ell-2}}.
\end{align}
By Proposition \ref{prop:NPCsindepentonu} and \eqref{eq:LxicurlVRgotozero}, the statement $(i)$ follows.  The statement $(ii)$ is manifest true.
\end{proof}

In the end, we discuss some further properties of the spin $\pm \sfrak$ components in the case that the N--P constant vanishes.

\begin{prop}
\label{prop:vanishingNPC:betterasymnearscri}
Let $\ell\geq \sfrak$ and $\reg\in \mathbb{N}$, and let $\alpha\in [0,1]$ be arbitrary. Assume the $(m,\ell)$-th N--P constant $\NPCNs{\ell}$  vanishes.
\begin{itemize}
\item[$(i)$] If $\InizeroEnergyminuss{\tb_0}{\ell}{\reg+\regl}{0}  +\lim\limits_{r\to\infty}\Big(\sum\limits_{j=0}^{\ell+\sfrak}
    \absCDeri{
      \tildePhiminuss{j}}{\reg}
      +\absCDeri{
      r^{1-\alpha}\curlVR\tildePhiminuss{\ell+\sfrak}}{\reg}
    \Big)\Big\vert_{\Sigmazero}
      <\infty$ for a suitably large $\regl=\regl(\ell)$,
      then there is a  constant $C_{\ell}(\tb)<\infty$ such that for any $\tb\geq \tb_0$, $\lim\limits_{r\to\infty}
\absCDeri{r^{1-\alpha}\curlVR\tildePhiminuss{\ell+\sfrak}}{\reg}
\vert_{\Sigmatb}<C_{\ell}(\tb)$.
In particular, if $\alpha>0$, then $\lim\limits_{r\to\infty}\absCDeri{r^{1-\alpha}
\curlVR\tildePhiminuss{\ell+\sfrak}}{\reg}
\vert_{\Sigmatb}$ is independent of $\tb$;

\item[$(ii)$] If $\InizeroEnergypluss{\tb_0}{\ell}{\reg+\regl}{0}  +\lim\limits_{r\to\infty}\Big(\sum\limits_{j=0}^{\ell-\sfrak}
    \absCDeri{
      \tildePhiplussHigh{j}}{\reg}
      +\absCDeri{
      r^{1-\alpha}\curlVR\tildePhiplussHigh{\ell-\sfrak}}{\reg}
    \Big)\Big\vert_{\Sigmazero}
      <\infty$ for a suitably large $\regl=\regl(\ell)$,
      then there is a  constant $C_{\ell}(\tb)<\infty$ such that for any $\tb\geq \tb_0$, $\lim\limits_{r\to\infty}
\absCDeri{r^{1-\alpha}\curlVR\tildePhiplussHigh{\ell-\sfrak}}{\reg}
\vert_{\Sigmatb}<C_{\ell}(\tb)$.
In particular, if $\alpha>0$, then $\lim\limits_{r\to\infty}\absCDeri{r^{1-\alpha}
\curlVR\tildePhiplussHigh{\ell-\sfrak}}{\reg}
\vert_{\Sigmatb}$ is independent of $\tb$.
\end{itemize}
\end{prop}

\begin{proof}
By Corollary \ref{cor:relationoftwoNPconsts}, the $(m,\ell)$-th N--P constant $\NPCPs{\ell}$ of the spin $+\sfrak$ component vanishes as well.
We show  the statements only for the spin $-\sfrak$ component, the proof of spin $+\sfrak$ component being the same. 

The scalar $\tildePhiminuss{\ell+\sfrak}$  satisfies equation \eqref{eq:tildePhiminusi:Schw:generall:tildePhiminusi:223223}, and hence performing an $r^{1-\alpha}$ rescaling gives
\begin{align}
-2\Lxi (r^{1-\alpha}\curlVR  \tildePhiminuss{\ell+\sfrak})={}& O(r^{-\alpha})rV\curlVR\tildePhiminuss{\ell+\sfrak}
+O(r^{-\alpha})\curlVR\tildePhiminuss{\ell+\sfrak}
+\sum_{j=2\sfrak}^{\ell+\sfrak}O(r^{-\alpha})\tildePhiminuss{j}.
\end{align}
In the case that $\alpha>0$, we have from the assumption of vanishing $(m,\ell)$-th N--P constant and Propositions \ref{prop:nullinfBeha:PhiplusiandPhinsHighi}  and \ref{prop:NPCsindepentonu} that the limit of the RHS on $\Sigmatb$ as $\rb\to \infty$ is zero for any $\tb\geq \tb_0$, hence one obtains
$\lim\limits_{r\to \infty}\Lxi (r^{1-\alpha}\curlVR \tildePhiminuss{\ell+\sfrak})\vert_{\Sigmatb}=0$. The statements for $\alpha>0$ then follow  from the bounded convergence theorem. For $\alpha=0$, the limit of the RHS on $\Sigmatb$ as $\rb\to \infty$ is now bounded by a $\tb$-dependent constant, hence $\lim\limits_{r\to\infty}
\absCDeri{r
\curlVR\tildePhiminuss{\ell+\sfrak}}{\reg}
\vert_{\Sigmatb}<C_\ell(\tb)$.
\end{proof}

\subsection{Almost Price's law in the exterior region $\{\rb\geq \tb\}$}
\label{sect:APL:ext}

We now turn to proving the almost Price's law for the spin $\pm\sfrak$ components in the region $\{\rb\geq \tb\}$ in this subsection. 

To begin with, we derive the almost sharp decay estimates for a single mode of the spin $\pm \sfrak$ components in the region $\{\rb\geq \tb\}$. This is contained in the following proposition.

\begin{prop}
\label{prop:almost:ext:ipm:nvv}
Assume the  the spin $\pm \sfrak$ components are supported on a single $(m,\ell)$ mode, $\ell\geq \sfrak$ and $m\in \{-\ell,-\ell+1,\ldots, \ell\}$. Let the $(m,\ell)$-th N--P constants $\NPCNs{\ell}$  be defined as in Definition \ref{def:NPCs}. Let $\vartheta\in (0,1/2)$ be arbitrary, and let $j\in \mathbb{N}$.
\begin{enumerate}
\item[$(i)$] \label{pt:NPCneq0:l=l0:decay:ext:Schw}
If the $\ell$-th N--P constant $\NPCNs{\ell}$  is nonzero, there exist universal constants $C=C(\vartheta, j, \reg, \ell)$ and $\regl=\regl(j,\ell)>0$ such that in the exterior region $\{\rb\geq \tb\}$,
\begin{subequations}
\label{eq:weakdecay:ell:ipm:ext:nv}
\begin{align}
\label{eq:weakdecay:ell:ipm:ext:nv:p}
 \absCDeri{\Lxi^j(r^{-2\sfrak-1}\Psipluss)}{\reg-\regl}\leq{}& Cv^{-1-2\sfrak}
  \tb^{-\frac{2\ell-2\sfrak-\vartheta}{2}-j-1}(\InizeroEnergypluss{\tb_0}{\ell}{\reg}{3-\vartheta}
 )^{\half},\\
 \label{eq:weakdecay:ell:ipm:ext:nv:m}
  \absCDeri{\Lxi^j(r^{-1}\Psiminuss)}{\reg-\regl}\leq{}& Cv^{-1}
  \tb^{-\frac{2\ell+2\sfrak-\vartheta}{2}-j-1}(\InizeroEnergypluss{\tb_0}{\ell}{\reg}{3-\vartheta})^{\half},
\end{align}
\end{subequations}
and for any $0\leq i\leq \ell-\sfrak$,
\begin{align}
\label{rem:PhiplussHigh:almostsharp:exterior}
 \absCDeri{\Lxi^j(r^{-1}\PhiplussHigh{i})}{\reg-\regl}\leq{}& Cv^{-1}
  \tb^{-\frac{2\ell-2\sfrak-2i-\vartheta}{2}-j-1}(\InizeroEnergypluss{\tb_0}{\ell}{\reg}{3-\vartheta}
 )^{\half}.
 \end{align}
\item[$(ii)$] \label{pt:NPCeq0:l=l0:decay:ext:Schw}
If the $\ell$-th N--P constant $\NPCNs{i}$ equals zero, there exist universal constants $C=C(\vartheta, j, \reg, \ell)$ and $\regl=\regl(j,\ell)>0$ such that in the exterior region $\{\rb\geq \tb\}$,
\begin{subequations}
\label{eq:weakdecay:ell:ipm:ext:v}
\begin{align}
\label{eq:weakdecay:ell:ipm:ext:v:p}
 \absCDeri{\Lxi^j(r^{-2\sfrak-1}\Psipluss)}{\reg-\regl}\leq{}& Cv^{-1-2\sfrak}
  \tb^{-\frac{2\ell-2\sfrak-\vartheta}{2}-j-2}(\InizeroEnergypluss{\tb_0}{\ell}{\reg}{5-\vartheta}
 )^{\half},\\
 \label{eq:weakdecay:ell:ipm:ext:v:m}
  \absCDeri{\Lxi^j(r^{-1}\Psiminuss)}{\reg-\regl}\leq{}& Cv^{-1}
  \tb^{-\frac{2\ell+2\sfrak-\vartheta}{2}-j-2}(\InizeroEnergypluss{\tb_0}{\ell}{\reg}{5-\vartheta})^{\half}.
\end{align}
\end{subequations}
\end{enumerate}
\end{prop}

\begin{remark}
\begin{itemize}
\item
These decay estimates, compared to the predicted Price's law, have only a $\tb^{\vartheta/2}$ loss, with $\vartheta>0$ arbitrarily small. This explains why they are called \emph{almost Price's law}.
\item In the case $(i)$ that the N--P constant $\NPCNs{\ell}$ is nonzero, the energies $\InizeroEnergypluss{\tb_0}{\ell}{\reg}{p}=\infty$ and $\InizeroEnergyminuss{\tb_0}{\ell}{\reg}{p}=\infty$ for $p\geq 3$ by Definition \ref{def:energyonSigmazero:Schw}. This is why these decay estimates are almost sharp.
\end{itemize}
\end{remark}

\begin{proof}
Consider first the case that the $\ell$-th N--P constant $\NPCNs{\ell}$ is non-vanishing.
 We have from Proposition \ref{prop:weakdecay:summary} with $p=3-\vartheta$ that
\begin{subequations}
\label{eq:weakdecay:ell:ipm:ext}
\begin{align}
\label{eq:weakdecay:ell:ipm:ext:p}
\absCDeri{\Lxi^j(r^{-1-2\sfrak}\Psipluss)}{\reg}
\leq{}&C(\vartheta, j, \reg, \ell)(\InizeroEnergypluss{\tb}{\ell}{\reg+\regl}{3-\vartheta})^{\half}r^{-2\sfrak}v^{-1}\tb^{-(2-\vartheta)/2-(\ell-\sfrak)-j},\\
\label{eq:weakdecay:ell:ipm:ext:m}
\sum_{n=0}^{\sfrak}\absCDeri{\Lxi^j(r^{-1}(r^2 V)^n\Psiminuss)}{\reg}
\leq{}&C(\vartheta, j, \reg, \ell)(\InizeroEnergyminuss{\tb}{\ell}{\reg+\regl}{3-\vartheta})^{\half}v^{-1}\tb^{-(2-\vartheta)/2-\ell-j}.
\end{align}
\end{subequations}
This already implies the $\sfrak=0$ case of \eqref{eq:weakdecay:ell:ipm:ext:nv}.
The estimate \eqref{eq:weakdecay:ell:ipm:ext:m} for the spin $-\sfrak$ component can be improved in the following way. For $\sfrak=1$, we have from equation \eqref{eq:RWPhi01schwope} that in the exterior region,
\begin{align*}
\ell(\ell+1)\Phiminus{0} ={}&{2(r-3M)}{r^{-2}}\Phiminus{1}-r^2YV\Phiminus{0}=O(r^{-1})\Phiminus{1}+O(r^{-1})rV\Phiminus{1}+O(1)\Lxi\Phiminus{1}.
\end{align*}
Since the RHS has further $\tb^{-1}$ decay in the exterior region compared to $\Phiminus{1}$, the estimate \eqref{eq:weakdecay:ell:ipm:ext:nv:m} thus follows. For $\sfrak=2$, we have from \eqref{eq:RWPhi02schwope} and \eqref{eq:RWPhi12schwope} that in the exterior region,
\begin{subequations}
\label{eq:RWPhi0212schwope:ext}
\begin{align}
\label{eq:RWPhi02schwope:ext}
\hspace{1ex}&\hspace{-1ex}((\ell-1)(\ell+2)+6M r^{-1})\Phiminustwo{0}-{4(r-3M)}{r^{-2}}\Phiminustwo{1} \notag\\
&=
-r^2YV\Phiminustwo{0}\notag\\
&=O(r^{-1})\Phiminustwo{1}+ O(r^{-1})rV\Phiminustwo{1}
+O(1)\Lxi\Phiminustwo{1},\\
\label{eq:RWPhi12schwope:ext}
\hspace{1ex}&\hspace{-1ex}(\ell(\ell+1)-6M r^{-1})\Phiminustwo{1} -6M\Phiminustwo{0}\notag\\
&={2(r-3M)}{r^{-2}}\Phiminustwo{2}-r^2YV\Phiminustwo{1}\notag\\
&=O(r^{-1})\Phiminustwo{2}+ O(r^{-1})rV\Phiminustwo{2}+O(1)\Lxi\Phiminustwo{2}.
\end{align}
\end{subequations}
When viewing equations \eqref{eq:RWPhi0212schwope:ext} as a coupled system of equations of $(\Phiminustwo{0}, \Phiminustwo{1})$,
the matrix operator of the LHS,  $\begin{pmatrix}
  (\ell-1)(\ell+2)+6M r^{-1} & -{4(r-3M)}{r^{-2}}\\
  -6M & \ell(\ell+1)-6M r^{-1}
\end{pmatrix}$, has determinant not less than $24-12Mr^{-1}+36M^2r^{-2}$ and is clearly positive definite for $r\geq 2M$, and the RHS has further $\tb^{-1}$ decay compared to $\Phiminustwo{2}$. An elliptic estimate then yields
\begin{align}
\sum_{n=0}^{1}\absCDeri{\Lxi^j(r^{-1}\Phiminustwo{n})}{\reg}
\lesssim{}&(\InizeroEnergyminuss{\tb}{\ell}{\reg+\regl}{3-\vartheta})^{\half}v^{-1}\tb^{-(2-\vartheta)/2-1-\ell-j}.
\end{align}
Plugging this estimate back into \eqref{eq:RWPhi02schwope:ext} and moving the term $-{4(r-3M)}{r^{-2}}\Phiminustwo{1}$ in \eqref{eq:RWPhi02schwope:ext} from the LHS to the RHS, we then conclude \eqref{eq:weakdecay:ell:ipm:ext:nv:m}.

The estimate \eqref{eq:weakdecay:ell:ipm:ext:nv:p}, in the exterior region $\{\rb\geq \tb\}$ is trivial in view of the estimate \eqref{eq:weakdecay:ell:ipm:ext:p} since $r\sim v$. Meanwhile, the estimate \eqref{rem:PhiplussHigh:almostsharp:exterior} holds true in the exterior region $\{\rb\geq \tb\}$ in view of  Remark \ref{rem:PhiplussHigh:almostsharp}.

Consider next the case that the $\ell$-th N--P constant $\NPCNs{\ell}$ vanishes. By Proposition \ref{prop:weakdecay:summary} with $p=5-\vartheta$ instead, and together with the following Lemma \ref{lem:Relation:pmenergies}, the proof proceeds in the same way as in the above case.
\end{proof}

\begin{lemma}
\label{lem:Relation:pmenergies}
Assume the  the spin $\pm \sfrak$ components are supported on $\ell\geq\ell_0$ mode, $\ell_0\geq \sfrak$. There exists universal constants $C=C(\vartheta, j, \reg, \ell_0)>0$ and $\regl=\regl(j,\ell_0)>0$ such that for any $\tb\geq \tb_0$ and $p\in [0, 5)$,
\begin{align}
C^{-1}\InizeroEnergypluss{\tb}{\ell_0}{\reg-\regl}{p}\leq
\InizeroEnergyminuss{\tb}{\ell_0}{\reg}{p}
\leq
C\InizeroEnergypluss{\tb}{\ell_0}{\reg+\regl}{p}.
\end{align}
\end{lemma}

\begin{proof}
This follows easily from Definition \ref{def:energyonSigmazero:Schw} and the TSI between the spin $+\sfrak$ and $-\sfrak$ components in  Lemma \ref{lem:TSI:gene}.
\end{proof}

Next, we combine the estimates for one single mode in Proposition \ref{prop:almost:ext:ipm:nvv} with the proven weak decay estimates in Proposition \ref{prop:weakdecay:summary} to conclude the almost sharp decay for the spin $\pm \sfrak$ components that are supported on $\ell\geq \ell_0$ modes. This is the main statement in this subsection. 

\begin{thm}[Almost Price's law for the spin $\pm\sfrak$ components in the exterior region $\{\rb\geq \tb\}$]
\label{thm:almost:ext:ipm:nvv:higher}
Assume the  the spin $\pm \sfrak$ components are supported on $\ell\geq \ell_0$ modes, $\ell_0\geq \sfrak$. Let the $(m,\ell_0)$-th N--P constants $\NPCNs{\ell_0}$  for the $(m,\ell_0)$ mode of the spin $+\sfrak$ component be defined as in Definition \ref{def:NPCs}. Let $\vartheta\in (0,1/2)$ be arbitrary, and let $j\in \mathbb{N}$.
\begin{enumerate}
\item\label{pt:NPCneq0:l=l0:decay:ext:Schw:higher}
If not all of the $(m,\ell_0)$-th N--P constants $\NPCNs{\ell_0}$, $m=-\ell_0, \ldots,\ell_0$, are nonzero, there exist universal constants $C=C(\vartheta, j, \reg, \ell_0)$ and $\regl=\regl(j,\ell_0)>0$ such that in the exterior region $\{\rb\geq \tb\}$,
\begin{subequations}
\label{eq:weakdecay:ell:ipm:ext:nv:higher}
\begin{align}
\label{eq:weakdecay:ell:ipm:ext:nv:p:higher}
 \absCDeri{\Lxi^j(r^{-2\sfrak-1}\Psipluss)}{\reg-\regl}\leq{}& Cv^{-1-2\sfrak}
  \tb^{-\frac{2\ell-2\sfrak-\vartheta}{2}-j-1}(\InizeroEnergyplussnv{\reg}{\vartheta}
 )^{\half},\\
 \label{eq:weakdecay:ell:ipm:ext:nv:m:higher}
  \absCDeri{\Lxi^j(r^{-1}\Psiminuss)}{\reg-\regl}\leq{}& Cv^{-1}
  \tb^{-\frac{2\ell_0+2\sfrak-\vartheta}{2}-j-1}(\InizeroEnergyplussnv{\reg}{\vartheta})^{\half},
\end{align}
\end{subequations}
and
\begin{subequations}
\label{eq:weakdecay:ell:ipm:ext:nv:higher:even}
\begin{align}
\label{eq:weakdecay:ell:ipm:ext:nv:p:higher:even}
 \absCDeri{\Lxi^j(r^{-2\sfrak-1}(\Psipluss)^{\ell\geq \ell_0+1})}{\reg-\regl}\leq{}& Cv^{-1-2\sfrak}
  \tb^{-\frac{2\ell-2\sfrak+\vartheta}{2}-j-1}(\InizeroEnergyplussnv{\reg}{\vartheta}
 )^{\half},\\
 \label{eq:weakdecay:ell:ipm:ext:nv:m:higher:even}
  \absCDeri{\Lxi^j(r^{-1}(\Psiminuss)^{\ell\geq \ell_0+1})}{\reg-\regl}\leq{}& Cv^{-1}
  \tb^{-\frac{2\ell_0+2\sfrak+\vartheta}{2}-j-1}(\InizeroEnergyplussnv{\reg}{\vartheta})^{\half};
\end{align}
\end{subequations}
\item\label{pt:NPCeq0:l=l0:decay:ext:Schw:higher}
If all the $(m,\ell_0)$-th N--P constant $\NPCNs{\ell_0}$, $m=-\ell_0, \ldots,\ell_0$, are zero, there exist universal constants $C=C(\vartheta, j, \reg, \ell_0)$ and $\regl=\regl(j,\ell_0)>0$ such that in the exterior region $\{\rb\geq \tb\}$,
\begin{subequations}
\label{eq:weakdecay:ell:ipm:ext:v:higher}
\begin{align}
\label{eq:weakdecay:ell:ipm:ext:v:p:higher}
 \absCDeri{\Lxi^j(r^{-2\sfrak-1}\Psipluss)}{\reg-\regl}\leq{}& Cv^{-1-2\sfrak}
  \tb^{-\frac{2\ell_0-2\sfrak-\vartheta}{2}-j-2}(\InizeroEnergyplussv{\reg}{\vartheta}
 )^{\half},\\
 \label{eq:weakdecay:ell:ipm:ext:v:m:higher}
  \absCDeri{\Lxi^j(r^{-1}\Psiminuss)}{\reg-\regl}\leq{}& Cv^{-1}
  \tb^{-\frac{2\ell_0+2\sfrak-\vartheta}{2}-j-2}(\InizeroEnergyplussv{\reg}{\vartheta})^{\half},\end{align}
\end{subequations}
and
\begin{subequations}
\label{eq:weakdecay:ell:ipm:ext:v:higher:even}
\begin{align}
\label{eq:weakdecay:ell:ipm:ext:v:p:higher:even}
 \absCDeri{\Lxi^j(r^{-2\sfrak-1}(\Psipluss)^{\ell\geq \ell_0+2})}{\reg-\regl}\leq{}& Cv^{-1-2\sfrak}
  \tb^{-\frac{2\ell-2\sfrak+\vartheta}{2}-j-2}(\InizeroEnergyplussnv{\reg}{\vartheta}
 )^{\half},\\
 \label{eq:weakdecay:ell:ipm:ext:v:m:higher:even}
  \absCDeri{\Lxi^j(r^{-1}(\Psiminuss)^{\ell\geq \ell_0+2})}{\reg-\regl}\leq{}& Cv^{-1}
  \tb^{-\frac{2\ell_0+2\sfrak+\vartheta}{2}-j-2}(\InizeroEnergyplussnv{\reg}{\vartheta})^{\half}.
\end{align}
\end{subequations}
\end{enumerate}
\end{thm}

\begin{proof}
If the spin $+\sfrak$ component is supported on $\ell\geq \ell_0$ modes, we can decompose it as in Definition \ref{def:fixedmode} into
\begin{align}
\label{eq:ellgelell0:p:modedecomp}
(\Psipluss)_{\ell\geq \ell_0}=\sum_{i=\ell_0}^{\infty} (\Psipluss)^{\ell=i}=\sum_{i=\ell_0}^{\infty} \sum\limits_{m=-i}^{i}(\Psipluss)_{(m,i)}Y_{m,i}^{s}(\cos\theta)e^{im\pb},
\end{align}
 which holds true in $L^2(S^2)$. Consider first  the case \ref{pt:NPCneq0:l=l0:decay:ext:Schw:higher}. For each $(\Psipluss)_{m,\ell_0}$, the estimates in Proposition \ref{prop:almost:ext:ipm:nvv} can be applied; while for the remainder $(\Psipluss)^{\ell\geq \ell_0+1}$, we can use the estimate \eqref{eq:weakdecay:geqell0:ipm} with $p=1+\vartheta$ in Proposition \ref{prop:weakdecay:summary}  to achieve the decay estimate \eqref{eq:weakdecay:ell:ipm:ext:nv:p:higher:even}. These two together then yield the estimate \eqref{eq:weakdecay:ell:ipm:ext:nv:p:higher}. Consider next the case \ref{pt:NPCeq0:l=l0:decay:ext:Schw:higher}. We apply Proposition \ref{prop:almost:ext:ipm:nvv} to each $(\Psipluss)_{m,\ell_0}$, the estimate \eqref{eq:weakdecay:ell:ipm:ext:nv:p:higher} to each $(\Psipluss)_{m,\ell_0+1}$ and the estimate \eqref{eq:weakdecay:ell:ipm:ext:nv:p:higher:even} to $(\Psipluss)^{\ell\geq\ell_0+2}$, which proves \eqref{eq:weakdecay:ell:ipm:ext:v:p:higher} and \eqref{eq:weakdecay:ell:ipm:ext:v:p:higher:even}.
 The estimates for the spin $-\sfrak$ component are similarly derived.
\end{proof}

\subsection{Wave equations for a fixed mode of the spin $s$ components}

It then remains to prove the almost Price's law in the interior region $\{\rb\leq \tb\}$. Before passing to the proof, we shall introduce some preliminaries that include wave equations for a couple of rescaled scalars constructed from a fixed mode of the spin $s=\pm\sfrak$ components. The equations for these rescaled scalars are indispensable in the latter Section \ref{sect:APL:IntRe} where the almost Price's law in the interior region is shown and Section \ref{sect:VNPC:PL} where the time integral is defined. 

The first rescaled scalar is a pure $r$-rescaling of a fixed mode of the spin $s$ components.

\begin{lemma}
\label{lem:radialeq:ell:S1}
Let the spin $s=\pm \sfrak$ components be supported on an $(m,\ell)$ mode, $\ell\geq \sfrak$.
Let
\begin{align}
\hatvarphisell=r^{-(\ell+s)}\psi_{s,\ell}.
\end{align}
 Then its governing equation is
\begin{align}
\label{eq:radialeq:ell:S1}
 \prb\left(r^{2\ell+2}\mu^{1-s} \prb\hatvarphisell\right)
- 2\ell (\ell+s)M \mu^{-s}r^{2\ell-1}\hatvarphisell={}\Lxi \Hhatvarphisell,
\end{align}
where
\begin{align}
\label{def:H:radialeq:ell:S1}
\Hhatvarphisell={}&-\mu^{-s+1}r^{2\ell+2} (\Hhyp-\partial_r\hhyp)\prb\hatvarphisell
+\mu^{-s+1}r^{2\ell+2}\Hhyp\partial_r\hhyp\Lxi\hatvarphisell\notag\\
&
+\mu^{-s}r^{2\ell}[
2(\ell+1)\mu r\partial_r\hhyp-2(s-1)M\partial_r\hhyp
+\mu r^2 \partial_r^2\hhyp-2(\ell-s+1)r]\hatvarphisell\notag\\
={}&r^{\ell-s+1}\big\{\mu \partial_r\hhyp \VR (\mu^{-s} r\psi_{s,\ell})-\mu \Hhyp \prb(\mu^{-s} r\psi_{s,\ell})\notag\\
&\qquad \qquad+([2sr-2M(3s+1)]\Hhyp r^{-2}-\mu \partial_r \Hhyp)\mu^{-s} r\psi_{s,\ell}\big\}.
\end{align}
\end{lemma}

\begin{remark}
We have incorporated all the $\Lxi$ derivatives into the RHS of the wave equation \eqref{eq:radialeq:ell:S1} since the $\Lxi$ derivative yields extra $\tb^{-1}$ decay as shown in the weak decay estimates in Proposition \ref{prop:weakdecay:summary}. This fact is fundamental in proving the almost Price's law in the interior region $\{\rb\leq \tb\}$ in Section \ref{sect:APL:IntRe}.
\end{remark}

\begin{proof}
For a fixed $(m,\ell)$ mode, the TME \eqref{eq:TME} can be simplified to
\begin{align}
-\mu^{-1}V(\mu r^2 Y\psi_{s,\ell})-(\ell(\ell+1)-(s^2-s))\psi_{s,\ell}
+2s(r-M)Y\psi_{s,\ell}-(4s-2)r\Lxi\psi_{s,\ell}
={}&0.
\end{align}
This is equivalent to
\begin{align}
-r^2VY\psi_{s,\ell}-(\ell+s)(\ell-s+1)\psi_{s,\ell}+2(s-1)(r-M)Y\psi_{s,\ell}
-(4s-2)r\Lxi\psi_{s,\ell}
={}&0.
\end{align}
Let $\psi_{s,\ell}=r^{\lambda}\hatvarphisell$, then
\begin{align*}
-r^2VY\psi_{s,\ell}={}&-r^{2+\lambda}VY\hatvarphisell
+\lambda r^{\lambda+1}(V-\mu Y)\hatvarphisell
+\lambda(\lambda-1)\mu r^{\lambda}\hatvarphisell,\\
2(s-1)(r-M)Y\psi_{s,\ell}={}&2(s-1)(r-M)r^{\lambda}Y\hatvarphisell
- 2\lambda (s-1)(r-M)r^{\lambda-1}\hatvarphisell.
\end{align*}
Hence by letting $\lambda=\ell+s$, one obtains the following equation for $\hatvarphisell$
\begin{align}
&- r^{3}VY\hatvarphisell
+(\ell+s) r^{2}(V-\mu Y)\hatvarphisell
+2(s-1)(r-M)rY\hat{\varphi}_{s,\ell}-(4s-2)r^{2}\Lxi\hatvarphisell \notag\\
&+[-(\ell+s)(\ell-s+1)r+(\ell+s)(\ell+s-1)\mu r
-2(\ell+s)(s-1)(r-M)
]\hatvarphisell={}0.
\end{align}
Calculating the coefficients and plugging in the expressions \eqref{def:vectorVRintermsofprb}, we arrive at
\begin{align}
&\mu r^3 \prb^2 \hatvarphisell
+[2(\ell+s)\mu r^2 -2(s-1)(r-M)r]\prb\hatvarphisell
- 2M\ell (\ell+s)\hatvarphisell\notag\\
&+\mu r^3 (\Hhyp-\partial_r\hhyp)\Lxi\prb\hatvarphisell
-\mu r^3\Hhyp\partial_r\hhyp\Lxi^2\hatvarphisell\notag\\
&
+[-2(\ell+s)\mu r^2 \partial_r\hhyp +2(s-1)(r-M)r\partial_r\hhyp-\mu r^3 \partial_r^2\hhyp+2(\ell-s+1)r^2]\Lxi\hatvarphisell={}0.
\end{align}
A further rescaling yields
 equation \eqref{eq:radialeq:ell:S1}.
\end{proof}

It is convenient and important in latter discussions to introduce a further rescaled quantity of $\hatvarphisell$ such that the PDE of this new rescaled quantity has no zeroth order term.

 Let us first define a function $\hell$.
\begin{definition}
\label{def:hell:pm}
Define
\begin{align}
\label{def:hell:nega}
\hell=\sum\limits_{i=0}^{\ell-\sfrak} \hellh{i}M^i r^{-i}
\end{align}
with
\begin{subequations}
\label{eq:totalderieq:generall:nega:def}
\begin{align}
\hellh{0}={}&1, \\
\hellh{i}={}&-2\hellh{i-1}\frac{(\ell-i-\sfrak+1)(\ell-i+1)}{i(2\ell+1-i)},\qquad i=1,2,\ldots,\ell-\sfrak.
\end{align}
\end{subequations}
\end{definition}

We can now introduce this rescaled scalar $\varphisell$ that satisfies a PDE of divergence form. The most fundamental property for the scalar $\varphisell$ is that $\prb \varphisell$ has faster $\tb^{-1}$ decay than $\varphisell$ in the region $\{\rb\leq \tb\}$, a fact that will be shown in Proposition \ref{prop:almost:int:ipm:nvv} and does not hold for any other $r$-rescaled scalar constructed from $\hatvarphisell$. Note that this PDE of divergence form is also essential in defining the time integral in Section \ref{subsect:TI:v}.

\begin{prop}
\label{prop:totalderieq:generall:nega}
Let the spin $\pm \sfrak$ components be supported on an $(m,\ell)$ mode, $\ell\geq \sfrak$.
Define
\begin{align}\label{def:varphiell:-1}
\varphisell=(\mu^{\frac{s+\sfrak}{2}}\hell)^{-1}\hatvarphisell=(\mu^{\frac{s+\sfrak}{2}}\hell r^{\ell+s})^{-1}\psi_{s,\ell}.
\end{align}
Then the scalar $\varphisell$ satisfies
\begin{align}
\label{eq:totalderieq:generall:nega}
\prb\left(r^{2\ell+2}\mu^{1+\sfrak}\hell^2 \prb\varphisell\right)
 ={}& \Lxi H_{{\varphi}_{s,\ell}},
\end{align}
with  $\Hvarphisell=\mu^{\frac{s+\sfrak}{2}}\hell \Hhatvarphisell$ and $\Hhatvarphisell$ as defined in \eqref{def:H:radialeq:ell:S1}.
\end{prop}

\begin{proof}
Given a general second order ODE
\begin{align}
\prb(f_1\prb \hat\varphi)+f_2\hat\varphi={}F_1,
\end{align}
our aim is to define $\varphi=h_1^{-1}\hat\varphi$ such that it satisfies an ODE with no zeroth order term. By this definition of $\varphi$, the above equation becomes
\begin{align}
\prb(f_1h_1\prb\varphi) +f_1\prb h_1 \prb \varphi +[\prb(f_1\prb h_1)+f_2 h_1]\varphi={}F_1,
\end{align}
or, equivalently
\begin{align}
\label{eq:radialeq:ell:S2:111}
\prb(f_1h_1^2\prb\varphi)  +h_1[\prb(f_1\prb h_1)+f_2 h_1]\varphi={}h_1F_1.
\end{align}
Recall our goal is to make the coefficient of $\varphi$ term vanishes, hence, it suffices that the function $h_1$ satisfies a second order ODE
\begin{align}
\prb(f_1\prb h_1)+f_2 h_1={}0.
\end{align}

Consider first $s\leq 0$ case. In view of the above argument and equation \eqref{eq:radialeq:ell:S1}, the function $\hell$ in the definition \eqref{def:varphiell:-1} is required to satisfy
\begin{align}
\label{eq:totalderieq:generall:nega:2:aaa}
 \prb\left(r^{2\ell+2}\mu^{1+\sfrak} \prb\hell\right)- 2\ell (\ell-\sfrak)M \mu^{\sfrak}r^{2\ell-1}\hell={}0,
 \end{align}
and by doing so, equation \eqref{eq:radialeq:ell:S1}  reduces to
 \begin{align}
 \prb\left(r^{2\ell+2}\mu^{1+\sfrak}\hell^2 \prb{\varphisell}\right)
 ={}& \Lxi \left(\hell \Hhatvarphisell \right)
 \end{align}
 which is exactly \eqref{eq:totalderieq:generall:nega}.

 Turn to $s>0$ case. Similarly, we need to choose a function $h_1$ in the expression $\varphiplussell=h_1^{-1}\hatvarphiplussell$ such that
 \begin{align}
\label{eq:totalderieq:generall:posi:2}
 \prb\left(r^{2\ell+2}\mu^{1-\sfrak} \prb h_1\right)- 2\ell (\ell+\sfrak)M \mu^{-\sfrak}r^{2\ell-1}h_1={}0,
 \end{align}
 Some direct calculations show that if $\hell$ solves equation \eqref{eq:totalderieq:generall:nega:2:aaa}, then $h_1=\mu^{\sfrak}\hell$ solves equation \eqref{eq:totalderieq:generall:posi:2}. Therefore, upon making such a choice of $h_1=\mu^{\sfrak}\hell$, equation \eqref{eq:radialeq:ell:S2:111} becomes \eqref{eq:radialeq:ell:S1}, which in turn takes the form of \eqref{eq:totalderieq:generall:nega}.

It remains to check that the function $\hell$ in Definition \ref{def:hell:pm} solves equation \eqref{eq:totalderieq:generall:nega:2:aaa}. By plugging the expression \eqref{def:hell:nega} of $\hell$ into equation \eqref{eq:totalderieq:generall:nega:2:aaa}, it becomes
\begin{align}
\label{eq:totalderieq:generall:nega:3}
\sum_{i=0}^{\ell-\sfrak}\Big( i(2\ell+1-i) \hellh{i} M^i r^{2\ell-i}
+2(\ell-i)(\ell-i-\sfrak) \hellh{i} M^{i+1} r^{2\ell-i-1}\Big)={}0.
\end{align}
This equation is clearly satisfied upon the choices of $\hell$  made in Definition \ref{def:hell:pm}.
\end{proof}

We end this subsection with listing a few properties of the function $\hell$.

\begin{prop}
\label{prop:totalderieq:propofh:generall:nega}
Let function $\hell$ be defined as in Definition \ref{def:hell:pm}. Then,
\begin{enumerate}
\item\label{pt:hellproperty:1}
the function $\hell$ satisfies
\begin{align}
\label{eq:totalderieq:generall:nega:2}
 \prb\left(r^{2\ell+2}\mu^{1+\sfrak} \prb\hell\right)- 2\ell (\ell-\sfrak)M \mu^{\sfrak}r^{2\ell-1}\hell={}0;
 \end{align}

\item\label{pt:hellproperty:2}
$r^{\ell-\sfrak}\hell$ is a stationary solution to the TME \eqref{eq:TME} for an $(m,\ell)$ mode of the spin $-\sfrak$ component $\psiminuss$, and $\mu^{\sfrak}r^{\ell+\sfrak}\hell$ is a stationary solution to the TME \eqref{eq:TME} for an $(m,\ell)$ mode of the spin $+\sfrak$ component $\psipluss$;

\item\label{pt:hellproperty:3} the function $\hell$ satisfies
\begin{align}\label{TSI:hell}
(r^2 \prb)^{2\sfrak}(\mu^{\sfrak}r^{\ell-\sfrak+1}\hell)=(\ell-\sfrak+1)\cdots(\ell+\sfrak) r^{\ell+\sfrak+1}\hell;
\end{align}

\item
the function $\hell$ in Proposition \ref{prop:totalderieq:generall:nega}  has no zero root  in $r\in [2M,\infty)$.
\end{enumerate}
\end{prop}

\begin{remark}
Equation \eqref{TSI:hell} can in fact be derived from the TSI. We will not discuss this in more details but give a simple remark that equation \eqref{TSI:hell}  holds true for any $\sfrak\in \mathbb{N}$.
\end{remark}

\begin{proof}
Equation \eqref{eq:totalderieq:generall:nega:2} in the first point is already shown in the proof of Proposition \ref{prop:totalderieq:generall:nega}.

Turn to the second point. Since $\hell$ solves \eqref{eq:totalderieq:generall:nega:2}, it is a solution to equation \eqref{eq:radialeq:ell:S1}, which thus yields that $r^{\ell-\sfrak}\hell$ is a stationary solution to the TME \eqref{eq:TME} for an $(m,\ell)$ mode of $\psiminuss$. The proof of the $s>0$ case in Proposition \ref{prop:totalderieq:generall:nega} has shown that $h_1=\mu^{\sfrak}\hell$ is a solution to equation \eqref{eq:totalderieq:generall:posi:2}, hence $\hat{\varphi}_{+\sfrak,\ell}=\mu^{\sfrak}\hell$ is a stationary solution to equation \eqref{eq:radialeq:ell:S1}. Consequently, $\mu^{\sfrak}r^{\ell+\sfrak}\hell$ is a stationary solution to the TME \eqref{eq:TME} for an $(m,\ell)$ mode of  $\psipluss$.

For the third point,
equation \eqref{TSI:hell} can be verified by direct calculations using equation \eqref{eq:totalderieq:generall:nega:2} for $\sfrak=1,2$.

It remains to show the fourth point. Recall that $\lim\limits_{r\to \infty} \hell=1$, hence if $\ell=\sfrak$, then clearly $\hell =1$ and the claim trivially holds. As a result, it suffices to show the non-vanishing property of function $\hell$ in $[2M, \infty)$ in the case that $\ell\geq \sfrak+1$.

By equation \eqref{eq:totalderieq:generall:nega:2} satisfied by function $\hell$, we have
\begin{align}
\label{eq:totalderieq:generall:nega:2:2}
\mu r^{2\ell+2}{\prb^2 \hell}+[2(\ell+1)r^{2\ell+1}-2M(2\ell+1-\sfrak) r^{2\ell}]{\prb \hell}-2M \ell(\ell-\sfrak)r^{2\ell-1}{\hell}=0.
\end{align}
As a result, function $\hell$ cannot reach its nonnegative maximum or nonpositive minimum in $\rb\in(2M,+\infty)$.

In addition, we claim $\hell(2M)> 0$. Suppose $\hell(2M)<0$, then the above equation \eqref{eq:totalderieq:generall:nega:2:2} implies $\prb\hell(2M)<0$, which means that $\hell$ must reach its negative minimum in $(2M,+\infty)$. This is in contradiction with the above argument. This thus yields the claim $\hell(2M)\geq 0$.
Instead, assume $\hell(2M)=0$, then we have $\prb\hell(2M)=0$ by equation \eqref{eq:totalderieq:generall:nega:2:2}. We can take further derivatives on  \eqref{eq:totalderieq:generall:nega:2:2} to obtain for
any $i\geq0$,
\begin{align*}
\mu r^3\prb^{i+2}\hell+[(2\ell+2+3i)r^2-2M(2\ell+1-\sfrak+2i)r]\prb^{i+1}\hell+\text{lower order derivatives}=0.
\end{align*}
We thus have $\prb^i \hell (2M)=0$ for all $i\geq 0$, which is clearly in contradiction with the expression of $\hell$ in Proposition \ref{prop:totalderieq:generall:nega}. Hence, $\hell(2M)>0$, and $\prb\hell(2M)>0$ by \eqref{eq:totalderieq:generall:nega:2:2}.

In the end, since $\hellh{1}=-(\ell-1)<0$, we have $\prb\hell>0$ for sufficiently large $\rho$. By the above conclusions that $\prb\hell(2M)>0$ and $\hell$ cannot reach its nonnegative maximum in $(2M, \infty)$, function $\hell$ must be monotonically increasing in $[2M, \infty)$. The non-vanishing property of function $\hell$ in $[2M,\infty)$ then follows easily from $\hell(2M)>0$.
\end{proof}

\subsection{Almost Price's law in the interior region  $\{\rb\leq \tb\}$}
\label{sect:APL:IntRe}

The main content of this subsection is devoted to obtaining in the interior region $\{\rb\leq \tb\}$ the almost Price's law for the spin $\pm \sfrak$ components.

Recall that the scalar $\hatvarphiell$ satisfies equation \eqref{eq:radialeq:ell:S1}  which reads
\begin{align}
\label{eq:radialeq:ell:S1:nega}
 \prb\left(r^{2\ell+2}\mu^{1+\sfrak} \prb\hatvarphiell\right)
- 2M \mu^{\sfrak}\lambda_{\sfrak} r^{2\ell-1}\hatvarphiell={}\Lxi \Hhatvarphiell,
\end{align}
with $\hatvarphiell=r^{-\ell+\sfrak}\psiminuss$, $\lambda_{\sfrak}=\ell(\ell-\sfrak)$, and
$\Hhatvarphiell$ being defined as in \eqref{def:H:radialeq:ell:S1} for $s=-\sfrak$.
A key estimate in achieving the almost Price's law in region $\{\rb\leq \tb\}$ is the following one for equation \eqref{eq:radialeq:ell:S1:nega}. 

\begin{lemma}
\label{lem:apl:int:n:ee}
Let $\reg\in \mathbb{N}$ and $\beta\in [0,(2\ell+1)/2]$ be arbitraty. Let $\hatvarphiell=r^{-\ell+\sfrak}\psiminuss$. Define $\CDeris=\{\Lxi, \rb\prb\}$. Then there exists a constant $C=C(\reg, j)$ such that
\begin{align}
\label{eq:key:almost:int:m:gene}
\hspace{4ex}&\hspace{-4ex}
\int_{2M}^{\tb}\left(r^{2\beta}\absCDeris{\hatvarphiell}{\reg}^2
+r^{2\beta+2}\absCDeris{\prb\hatvarphiell}{\reg}^2
+\mu^{2}r^{2\beta+4}\absCDeris{\prb^2 \hatvarphiell}{\reg}^2
\right)\di\rb\notag\\
\leq{}&C\bigg(\int_{2M}^{\tb} \left(r^{2\beta+4}\absCDeris{\Lxi\prb\hatvarphiell}{\reg}^2 + r^{2\beta}\absCDeris{\Lxi^2\hatvarphiell}{\reg}^2+r^{2\beta+2}\absCDeris{\Lxi\hatvarphiell}{\reg}^2\right)
\di \rb\notag\\
&\qquad
+\Big(r^{2\beta+1}\absCDeris{\hatvarphiell}{\reg}^2 +r^{2\beta+2}\absCDeris{\prb\hatvarphiell}{\reg}^2\Big)\Big\vert_{\rb=\tb}\bigg).
\end{align}
\end{lemma}

\begin{proof}
Taking a modulus square of both sides and multiplying by $r^{-4\ell}f^2$, we arrive at
\begin{align}
r^{-4\ell}f^2 \abs{\Lxi \Hhatvarphiell}^2
={}&
f^2 r^{-4\ell}\abs{\prb(\mu^{1+\sfrak} r^{2\ell+2}\prb \hatvarphiell)}^2
+4M^2 \lambda_{\sfrak}^2 \mu ^{2\sfrak}f^2r^{-2}\abs{\hatvarphiell}^2\notag\\
&
-4M\lambda_{\sfrak} \mu^{\sfrak} f^2r^{-2\ell-1}\Re\big(\prb(\mu^{1+\sfrak} r^{2\ell+2}\prb\hatvarphiell)\overline{\hatvarphiell}\big)\notag\\
\triangleq{}& I_1+I_2+I_3.
\end{align}
Consider the term $I_3$. We use the Leibniz rule to obtain
\begin{align}
\label{eq:ellipesti:range:I3:generall}
I_3={}&\prb\Big(-4M\lambda_{\sfrak} f^2\mu^{1+2\sfrak} r\Re(\prb\hatvarphiell\overline{\hatvarphiell})\Big)
+4M\lambda_{\sfrak} f^2\mu^{1+2\sfrak} r\abs{\prb\hatvarphiell}^2\notag\\
&
+4M\lambda_{\sfrak} \partial_r (\mu^{\sfrak} f^2r^{-2\ell-1})\mu^{1+\sfrak} r^{2\ell+2}\Re(\prb\hatvarphiell\overline\hatvarphiell)\notag\\
={}&\prb\Big(-4M\lambda_{\sfrak} f^2\mu^{1+2\sfrak} r\Re(\prb\hatvarphiell\overline{\hatvarphiell})\Big)
+\prb\Big(
2M\lambda_{\sfrak} \partial_r (\mu^{\sfrak} f^2r^{-2\ell-1})\mu^{1+\sfrak}r^{2\ell+2}\abs{\hatvarphiell}^2
\Big)\notag\\
&
+4M\lambda_{\sfrak} f^2\mu^{1+2\sfrak} r\abs{\prb\hatvarphiell}^2
-2M\lambda_{\sfrak}\partial_r \Big(\partial_r (\mu^{\sfrak} f^2r^{-2\ell-1})\mu^{1+\sfrak} r^{2\ell+2}\Big)\abs{\hatvarphiell}^2.
\end{align}
Choose $f^2 =\mu^{-2\sfrak} r^{2\beta}$, with $\beta$ to be fixed.
For the last term, it equals
\begin{align}
\hspace{4ex}&\hspace{-4ex}2M\lambda_{\sfrak}\partial_r\Big((2\ell+1-2\beta)\mu r^{\beta}
+2M\sfrak r^{2\beta-1}\Big)\abs{\hatvarphiell}^2\notag\\
={}&2M\lambda_{\sfrak}\Big(
\beta(2\ell+1-2\beta) \mu r^{2\beta-1}
+2M((2\ell+1-2\beta)+\sfrak(2\beta-1))r^{2\beta-2} \Big)\abs{\hatvarphiell}^2,
\end{align}
and the coefficient is nonnegative for any $\beta\in [0,(2\ell+1)/2]$.
As a result, we obtain
\begin{align}
\label{eq:almost:int:m:key}
\hspace{2.5ex}&\hspace{-2.5ex}
\mu^{-2\sfrak}r^{-4\ell}r^{2\beta} \abs{\Lxi \Hhatvarphiell}^2\notag\\
={}&\prb\Big(-4M\lambda_{\sfrak}\mu r^{2\beta+1}\Re(\prb\hatvarphiell\overline{\hatvarphiell})\Big)
+\prb\Big(
2M\lambda_{\sfrak} \partial_r (\mu^{-\sfrak} r^{-2\ell-2\beta-1})\mu^{1+\sfrak}r^{2\ell+2}\abs{\hatvarphiell}^2
\Big)\notag\\
&
+\mu^{-2\sfrak}r^{-4\ell+2\beta}\abs{\prb(\mu^{1+\sfrak} r^{2\ell+2}\prb \hatvarphiell)}^2
+4M\lambda_{\sfrak} \mu r^{2\beta+1}\abs{\prb\hatvarphiell}^2\notag\\
&+2M\lambda_{\sfrak}\Big(
2\beta(2\ell+1-2\beta) \mu r^{2\beta-1}
+2M((2\ell+1-2\beta)+\sfrak(2\beta-1)+\lambda_{\sfrak})r^{2\beta-2} \Big)\abs{\hatvarphiell}^2.
\end{align}
Note in particular that the coefficients of the last three terms are all nonnegative.

On any $\Sigmatb$, we integrate this equation \eqref{eq:almost:int:m:key} in $\rb$ from $2M$ to $\tb$. By the Cauchy--Schwarz inequality, the integral of the LHS is bounded by
\begin{align}
C\int_{2M}^{\tb} \left(r^{2\beta+4}\abs{\Lxi\prb\hatvarphiell}^2 + r^{2\beta}\abs{\Lxi^2\hatvarphiell}^2+r^{2\beta+2}\abs{\Lxi\hatvarphiell}^2\right)
\di \rb;
\end{align}
for the integral of the RHS,  the boundary term $(-4M\lambda_{\sfrak}\mu r^{2\beta+1}\Re(\prb\hatvarphiell\overline{\hatvarphiell}))\vert_{\rb=2M}$  vanishes due to the presence of the factor $\mu$, hence it is bounded below by
\begin{align}
\label{eq:almost:int:m:inter:12}
&c\int_{2M}^{\tb}\left(\mu^{-2\sfrak}r^{-4\ell+2\beta}\abs{\prb(\mu^{1+\sfrak} r^{2\ell+2}\prb \hatvarphiell)}^2+\mu M\lambda_\sfrak r^{2\beta+1}\abs{\prb\hatvarphiell}^2
+M^2 \lambda_{\sfrak}r^{2\beta-2}\abs{\hatvarphiell}^2\right) \di \rb\notag\\
&+c\lambda_{\sfrak}(\abs{\hatvarphiell}^2)\vert_{\rb=2M}
-CM\lambda_{\sfrak}(r^{2\beta}\abs{\hatvarphiell}^2+r^{2\beta+2}\abs{\prb\hatvarphiell}^2)\vert_{\rb=\tb}.
\end{align}
Furthermore, in view of the following inequality
\begin{align*}
\int_{2M}^{\rb'}r^{2\beta+2}\abs{\prb\hatvarphiell}^2\di\rb
+(\mu r^{2\beta+3}\abs{\prb\hatvarphiell}^2)\vert_{\rb=\rb'}
\leq C \int_{2M}^{\rb'}\mu^{-2\sfrak}r^{-4\ell+2\beta}\abs{\prb(\mu^{1+\sfrak} r^{2\ell+2}\prb \hatvarphiell)}^2\di\rb
\end{align*}
which holds for any $\rb'>2M$ and follows simply by integrating
\begin{align*}
\hspace{4ex}&\hspace{-4ex}
\prb(\mu^{-2\sfrak-1}r^{-4\ell+2\beta-1}\abs{\mu^{1+\sfrak}r^{2\ell+2}\prb\hatvarphiell}^2)\notag\\
={}&-[(2\sfrak+1)\mu^{-2\sfrak-2}2Mr^{-1}+(4\ell-2\beta+1)\mu^{-2\sfrak-1}]r^{-4\ell+2\beta-2}\abs{\mu^{1+\sfrak}r^{2\ell+2}\prb\hatvarphiell}^2\notag\\
&
+2\mu^{-2\sfrak-1}r^{-4\ell+2\beta-1}\Re\big(\prb(\mu^{1+\sfrak}r^{2\ell+2}\prb\hatvarphiell)\mu^{1+\sfrak}r^{2\ell+2}\overline{\prb\hatvarphiell}\big)
\end{align*}
in $\rb$ from $2M$ to $\rb'$, and the following Hardy-type inequality
\begin{align*}
\int_{2M}^{\rb'}r^{2\beta}\abs{\hatvarphiell}^2\di\rb \leq C\int_{2M}^{\rb'}\mu^2 r^{2\beta+2}\abs{\prb\hatvarphiell}^2 \di\rb+C(\mu  r^{2\beta+1}\abs{\hatvarphiell}^2 )\vert_{\rb=\rb'}
\end{align*}
which follows from integrating  in $\rb$ from $2M$ to $\rb'$ for the following equality
\begin{align*}
\prb(\mu r^{2\beta+1}\abs{\hatvarphiell}^2 )= (2Mr^{-1}+(2\beta+1)\mu )r^{2\beta}\abs{\hatvarphiell}^2+2\mu r^{2\beta+1}\Re(\prb\hatvarphiell\overline{\hatvarphiell}),
\end{align*}
the first line of \eqref{eq:almost:int:m:inter:12} is bounded below by
\begin{align}
c\int_{2M}^{\tb}\left(r^{2\beta}\abs{\hatvarphiell}^2
+r^{2\beta+2}\abs{\prb\hatvarphiell}^2
+\mu^2 r^{2\beta+4}\abs{\prb^2\hatvarphiell}^2
\right)\di\rb
-C(r^{2\beta+1}\abs{\hatvarphiell}^2 )\vert_{\rb=\tb}.
\end{align}
Consequently, we arrive at
\begin{align*}
\hspace{4ex}&\hspace{-4ex}
\int_{2M}^{\tb}\left(r^{2\beta}\abs{\hatvarphiell}^2
+r^{2\beta+2}\abs{\prb\hatvarphiell}^2
+\mu^{-2\sfrak}r^{-4\ell+2\beta}\abs{\prb(\mu^{1+\sfrak} r^{2\ell+2}\prb \hatvarphiell)}^2
\right)\di\rb\notag\\
\lesssim{}&\int_{2M}^{\tb} \left(r^{2\beta+4}\abs{\Lxi\prb\hatvarphiell}^2 + r^{2\beta}\abs{\Lxi^2\hatvarphiell}^2+r^{2\beta+2}\abs{\Lxi\hatvarphiell}^2\right)
\di \rb\notag\\
&
+(r^{2\beta+1}\abs{\hatvarphiell}^2 +r^{2\beta+2}\abs{\prb\hatvarphiell}^2)\vert_{\rb=\tb}.
\end{align*}
This thus proves \eqref{eq:key:almost:int:m:gene} for $\reg=0$.

Since $\Lxi$ commutes with equation \eqref{eq:radialeq:ell:S1:nega},  it suffices to prove  the estimate \eqref{eq:key:almost:int:m:gene} with $\CDeris$ replaced by $\{\rb\prb\}$. We prove it by induction in $\reg$, that is,
assuming it holds for $\reg=n-1$, $n\in \mathbb{N}^+$,   we prove for $\reg=n$. We multiply both sides of equation \eqref{eq:radialeq:ell:S1:nega} by $\mu^{-\sfrak}$ and then commute $\rb \prb$, and since
\begin{align}
r\prb\Big(\mu^{-\sfrak}\prb(r^{2\ell+2}\mu^{1+\sfrak} \prb\hatvarphiell)\Big)={}&
\mu^{-(\sfrak+1)}\prb\Big(r^{2\ell+2}\mu^{1+(\sfrak+1)} \prb (r\prb\hatvarphiell)\Big)\notag\\
&+\big(O_{\infty}(1)\mu r^{2\ell+2}\prb^2+O_{\infty}(1)r^{2\ell+1}\prb +O_{\infty}(1)r^{2\ell}\big)\hatvarphiell,
\end{align}
where $O_{\infty}(1)$ are $O(1)$ functions and smooth everywhere in $\rb\in [2M,\infty)$,
we obtain for any $n\in \mathbb{N}^+$ that
\begin{align}
\label{eq:almost:int:m:highorder:ge}
\hspace{4ex}&\hspace{-4ex}\mu^{-(\sfrak+n)}\prb\Big(r^{2\ell+2}\mu^{1+(\sfrak+n)} \prb ((r\prb)^n\hatvarphiell)\Big)- 2M\lambda_{\sfrak} r^{2\ell-1}(r\prb)^n\hatvarphiell\notag\\
={}&\Lxi (r\prb)^n\Hhatvarphiell
+r^{2\ell}\bigg(\sum_{i=0}^n O_{\infty}(1)(r\prb)^i\hatvarphiell +O_{\infty}(1)\mu (r\prb)^{n+1}\hatvarphiell\bigg).
\end{align}
Similarly to the $\reg=0$ case, we take a modulus square of both sides, multiply by $r^{-4\ell+2\beta}$, and integrate in $\rb$ from $2M$ to $\tb$. The integral arising from the last term of \eqref{eq:almost:int:m:highorder:ge} is  controlled by the assumption in the induction, and the analysis for the integral coming from the LHS of \eqref{eq:almost:int:m:highorder:ge} is exactly the same as the above one in treating the $\reg=0$ case, which thus completes the proof.
\end{proof}

The above estimate \eqref{eq:key:almost:int:m:gene} is readily  applied to trade $r$-decay for $\tb$-decay and thus yield the almost Price's law for a fixed mode of the spin $-\sfrak$ component in the interior region $\{\rb\leq \tb\}$. The key observation is that on the RHS of the estimate \eqref{eq:key:almost:int:m:gene}, the integral terms all contain $\Lxi$ derivatives and hence satisfy faster $\tb$-decay, and the terms evaluated at $\rb=\tb$ enjoy extra $\tb$-decay when decreasing the parameter $\beta$, and due to this observation, it allows us to improve the $\tb$-decay for the spin $-\sfrak$ component upon the weak decay estimates in Proposition \ref{prop:weakdecay:summary}. 

\begin{prop}[Almost Price's law for a single mode of the spin $\pm \sfrak$ components in the interior region $\{\rb\leq \tb\}$]
\label{prop:almost:int:ipm:nvv}
Assume the  the spin $s=\pm \sfrak$ components are supported on a single $(m,\ell)$ mode, $\ell\geq \sfrak$ and $m\in \{-\ell,-\ell+1,\ldots, \ell\}$. Let the N--P constants $\NPCNs{\ell}$  be defined as in Definition \ref{def:NPCs}, and let $\varphisell$ be defined as in Proposition \ref{prop:totalderieq:generall:nega}.  Let $\vartheta\in (0,1/2)$ be arbitrary, and let $j\in \mathbb{N}$ and $\reg\in \mathbb{N}$.
\begin{enumerate}
\item[$(i)$] \label{pt:NPCneq0:l=l0:decay:int:Schw}
If the $\ell$-th N--P constant $\NPCNs{\ell}$  is nonzero,  there exist universal constants $C=C(\reg, j,\ell,\vartheta)$  and  $\regl=\regl(\reg, j, \ell,\vartheta)$ such that
in the interior region $\{\rb\leq \tb\}$,
\begin{subequations}
\label{eq:almost:int:pm:allorder:nv:gene}
\begin{align}
\label{eq:almost:int:p:allorder:nv:gene}
 \absCDeri{\Lxi^j(r^{-\ell-\sfrak-1}\Psipluss)}{\reg}\leq{}& C
  \tb^{ -2-2\ell-j+\frac{\vartheta}{2}}(\InizeroEnergypluss{\tb_0}{\ell_0}{\reg+\regl}{3-\vartheta}
 )^{\half},\\
\label{eq:almost:int:m:allorder:nv:gene}
\absCDeri{\Lxi^j\varphiminussell}{\reg}
\leq{}&C\tb^{ -2-2\ell-j+\frac{\vartheta}{2}}(\InizeroEnergyminuss{\tb_0}{\ell}{\reg+\regl}{3-\vartheta})^{\half},\\
\label{eq:almost:int:m:allorder:nv:gene:prbderi:be}
\absCDeri{\Lxi^j\prb\varphiminussell}{\reg}\leq{}& C \tb^{ -3-2\ell-j+\frac{\vartheta}{2}}(\InizeroEnergyminuss{\tb_0}{\ell}{\reg+\regl}{3-\vartheta})^{\half}.
\end{align}
\end{subequations}
\item[$(ii)$] \label{pt:NPCeq0:l=l0:decay:int:Schw}
If the $\ell$-th N--P constant $\NPCNs{\ell}$  is zero,  there exist universal constants $C=C(\reg, j,\ell,\vartheta)$  and  $\regl=\regl(\reg, j, \ell,\vartheta)$ such that
in the interior region $\{\rb\leq \tb\}$,
\begin{subequations}
\label{eq:almost:int:pm:allorder:v:gene}
\begin{align}
\label{eq:almost:int:p:allorder:v:gene}
 \absCDeri{\Lxi^j(r^{-\ell-\sfrak-1}\Psipluss)}{\reg}\leq{}& C
  \tb^{ -3-2\ell-j+\frac{\vartheta}{2}}(\InizeroEnergypluss{\tb_0}{\ell_0}{\reg+\regl}{5-\vartheta}
 )^{\half},\\
\label{eq:almost:int:m:allorder:v:gene}
\absCDeri{\Lxi^j\varphiminussell}{\reg}
\leq{}&C\tb^{ -3-2\ell-j+\frac{\vartheta}{2}}(\InizeroEnergyminuss{\tb_0}{\ell}{\reg+\regl}{5-\vartheta})^{\half},\\
\label{eq:almost:int:m:allorder:v:gene:prbderi:be}
\absCDeri{\Lxi^j\prb\varphiminussell}{\reg}\leq{}& C \tb^{ -4-2\ell-j+\frac{\vartheta}{2}}(\InizeroEnergyminuss{\tb_0}{\ell}{\reg+\regl}{5-\vartheta})^{\half}.
\end{align}
\end{subequations}\end{enumerate}
\end{prop}

\begin{proof}
We consider only the case that the $\ell$-th N--P constant $\NPCNs{\ell}\neq 0$, and the other case that this constant vanishes can be analogously treated.

By the pointwise estimates in Proposition \ref{prop:almost:ext:ipm:nvv}, the estimate \eqref{eq:key:almost:int:m:gene} implies that for any $j\in \mathbb{N}$, $\reg\in \mathbb{N}$, $\beta\in [0,\ell+\half]$ and $\vartheta\in (0,\half)$, there exist universal constants $C=C(\reg, j,\vartheta)$  and  $\regl=\regl(\reg, j, \ell,\vartheta)$ such that
\begin{align}
\label{eq:almost:int:m:highorder:ge:iter}
\norm{\Lxi^j\psiminuss}^2_{W_{2(-\ell+\sfrak+\beta)}^{\reg}(\Sigmatb^{\leq \tb})}
\leq{}&C\norm{\Lxi^{j+1}\psiminuss}^2_{W_{2(-\ell+\sfrak+\beta+1)}^{\reg}(\Sigmatb^{\leq \tb})}
+C\tb^{-3+2\beta-4\ell-2j+\vartheta}\InizeroEnergyminuss{\tb_0}{\ell}{\reg+\regl}{3-\vartheta}.
\end{align}
We use this estimate for all $\beta\in \{0,1,\ldots, \ell-\sfrak-1\}$  iteratively: we first prove the estimate for the term on the LHS with $\beta=\ell-\sfrak-1$; then iteratively, the estimate for the LHS with $\beta=\beta_0$, which gives an estimate for the first term on the RHS with $\beta=\beta_0-1$, yields the estimate for the LHS with $\beta=\beta_0-1$.  Specifically,  the pointwise estimates in Proposition \ref{prop:almost:ext:ipm:nvv} imply that the first term on the RHS of \eqref{eq:almost:int:m:highorder:ge:iter} is bounded by $C(\reg, j,\ell,\vartheta) \tb^{-3+2\beta-4\ell-2j+\vartheta}\InizeroEnergyminuss{\tb_0}{\ell}{\reg+\regl}{3-\vartheta}$, hence the term on the LHS is bounded by $C(\reg, j,\ell,\vartheta) \tb^{-3+2\beta-4\ell-2j+\vartheta}\InizeroEnergyminuss{\tb_0}{\ell}{\reg+\regl}{3-\vartheta}$ as well.  Iteratively, we eventually conclude
\begin{align}
\label{eq:almost:int:m:allorder:nv:gene:energy}
\norm{\Lxi^j\Psiminuss}^2_{W_{2(-\ell+\sfrak-1)}^{\reg}(\Sigmatb^{\leq \tb})}
\leq{}&C\tb^{-3-4\ell-2j+\vartheta}\InizeroEnergyminuss{\tb_0}{\ell}{\reg+\regl}{3-\vartheta}.
\end{align}
Applying the Sobolve-type inequality \eqref{eq:Sobolev:1} to this energy decay estimate, one obtains in the interior region $\{\rb\leq \tb\}$ that
\begin{align}\label{eq:almost:int:m:allorder:nv:gene:s1}
\absCDeri{\Lxi^j\hatvarphiell}{\reg}\leq C(\reg, j,\ell,\vartheta)r^{-\half} \tb^{ -\frac{3-\vartheta}{2}-2\ell-j}(\InizeroEnergyminuss{\tb_0}{\ell}{\reg+\regl(\reg, j, \ell,\vartheta)}{3-\vartheta})^{\half}.
\end{align}

From Proposition \ref{prop:totalderieq:generall:nega}, with $\hell$ defined as in Definition \ref{def:hell:pm}, the scalar $\varphiminussell=\hell^{-1}\hatvarphiell$ satisfies
 \begin{align}
\prb(r^{2\ell+2}\mu^{1+\sfrak}\hell^2\prb\varphiminussell)=\hell\Lxi \Hhatvarphisell.
\end{align}
Recall from Proposition \ref{prop:totalderieq:propofh:generall:nega} that there exist two positive universal constant $c$ and $C$ such that
$c\leq \hell\leq C$. We integrate this equation from $\rb=2M$ to any $\rb$ with $\rb\leq \tb$, and the boundary term at horizon vanishes because of the degenerate factor $\mu^{1+\sfrak}$, thus arriving at
\begin{align}
\mu^{1+\sfrak}r^{2\ell+2}\absCDeri{\Lxi^j\prb\varphiminussell}{\reg}\lesssim_{\reg} \mu^{1+\sfrak}r^{2\ell+2}\absCDeri{\Lxi^{j+1}\hatvarphiell}{\reg}.
\end{align}
By utilizing the  pointwise decay estimate \eqref{eq:almost:int:m:allorder:nv:gene:s1} to estimate the RHS, we achieve a better decay estimate for $\Lxi^j\prb\varphiminussell$:
\begin{align}\label{eq:almost:int:m:allorder:nv:gene:prbderi}
\absCDeri{\Lxi^j\prb\varphiminussell}{\reg}\leq C(\reg, j,\ell,\vartheta)r^{-\half} \tb^{ -\frac{5-\vartheta}{2}-2\ell-j}(\InizeroEnergyminuss{\tb_0}{\ell}{\reg+\regl(\reg, j, \ell,\vartheta)}{3-\vartheta})^{\half}.
\end{align}
We can now integrate $\prb\varphiminussell$ from $\rb=\tb$ to $\rb$ and use the estimate \eqref{eq:almost:int:m:allorder:nv:gene:s1} to estimate the boundary term at $\rb=\tb$ and the estimate \eqref{eq:almost:int:m:allorder:nv:gene:prbderi} to estimate the integral of $\prb\varphiminussell$. By doing so, we prove \eqref{eq:almost:int:m:allorder:nv:gene}.  Finally, we repeat this step of proving better decay estimate \eqref{eq:almost:int:m:allorder:nv:gene:prbderi} for $\Lxi^j\prb\varphiminussell$ by using the estimate \eqref{eq:almost:int:m:allorder:nv:gene}, which then proves \eqref{eq:almost:int:m:allorder:nv:gene:prbderi:be}.

In the end,  the approach of showing the estimate \eqref{eq:almost:int:p:allorder:nv:gene} for the spin $+\sfrak$ component is the same as the one in the proof of Proposition \ref{prop:almost:ext:ipm:nvv} where the decay of the spin $+\sfrak$ is proven by the estimate of the spin $-\sfrak$ together with an application of the TSI \eqref{eq:TSI:simpleform:l=2}, hence we omit it.
\end{proof}

Finally, we combine the above estimates and achieve the almost Price's law for the spin $\pm \sfrak$ components in the region $\{\rb\leq\tb\}$. This is contained in the following theorem. 

\begin{thm}[Almost Price's law for the spin $\pm \sfrak$ components in  the interior region $\{\rb\leq \tb\}$]
\label{thm:almost:int:ipm:nvv:hig}
Assume the spin $s=\pm \sfrak$ components are supported on $\ell\geq\ell_0$ modes, $\ell_0\geq \sfrak$. For each $m\in \{-\ell_0,-\ell_0+1,\ldots, \ell_0\}$,  let the N--P constants $\NPCNs{\ell_0}$ for the $(m,\ell_0)$ mode of the spin $+\sfrak$ component be defined as in Definition \ref{def:NPCs}.  Let $\vartheta\in (0,1/2)$ be arbitrary, and let $j\in \mathbb{N}$ and $\reg\in \mathbb{N}$.
\begin{enumerate}
\item[$(i)$] \label{pt:NPCneq0:l=l0:decay:int:Schw:highigh}
If not all of the $(m,\ell_0)$-th N--P constant $\NPCNs{\ell_0}$  are zero,  there exist universal constants $C=C(\reg, j,\ell_0,\vartheta)$  and  $\regl=\regl(\reg, j, \ell_0,\vartheta)$ such that
in the interior region $\{\rb\leq \tb\}$,
\begin{align}
\label{eq:almost:int:pm:allorder:nv:gene:hig}
 \absCDeri{\Lxi^j(r^{-\ell_0-s-1}\Psipms)}{\reg}\leq{}& C
  \tb^{ -2-2\ell_0-j+\frac{\vartheta}{2}}(\InizeroEnergyplussnv{\reg+\regl}{\vartheta}
 )^{\half}.
\end{align}
\item[$(ii)$] \label{pt:NPCeq0:l=l0:decay:int:Schw:hig}
If all of the $(m,\ell_0)$-th N--P constant $\NPCNs{\ell_0}$  are zero, there exist universal constants $C=C(\reg, j,\ell_0,\vartheta)$  and  $\regl=\regl(\reg, j, \ell_0,\vartheta)$ such that
in the interior region $\{\rb\leq \tb\}$,
\begin{align}
\label{eq:almost:int:pm:allorder:v:gene:hig}
 \absCDeri{\Lxi^j(r^{-\ell_0-s-1}\Psipms)}{\reg}\leq{}& C
  \tb^{ -3-2\ell_0-j+\frac{\vartheta}{2}}(\InizeroEnergyplussv{\reg+\regl}{\vartheta}
 )^{\half}.
\end{align}
\end{enumerate}
\end{thm}

\begin{proof}
It suffices to prove the estimate for the spin $-\sfrak$ component in the case that not all of the $(m,\ell_0)$-th N--P constant $\NPCNs{\ell_0}$  are zero, since the estimate for the spin $+\sfrak$ component can be obtained via the TSI and the proof for the other case that all of the $(m,\ell_0)$-th N--P constant $\NPCNs{\ell_0}$  are zero is analogous.
For each $(m,\ell_0)$ mode, its estimate   has been obtained in Proposition \ref{prop:almost:int:ipm:nvv}.
  For each $(m,\ell)$ mode, $\ell_0+1\leq \ell\leq 2\ell_0+1$, the proof in the previous sections also implies a slightly different decay estimate from \eqref{eq:weakdecay:ell:ipm:ext:nv:m}
  \begin{align}
  \absCDeri{\Lxi^j(r^{-1}(\Psiminuss)_{m,\ell})}{\reg-\regl}\leq{}& Cv^{-1}
  \tb^{-\frac{2\ell+2\sfrak-2+\vartheta}{2}-j-1}(\InizeroEnergypluss{\tb_0}{\ell}{\reg}{1+\vartheta})^{\half}.
  \end{align}
By applying Lemma \ref{lem:apl:int:n:ee} and going through the proof of Proposition \ref{prop:almost:int:ipm:nvv}, we obtain an analogous estimate as \eqref{eq:almost:int:m:allorder:nv:gene:energy}
 \begin{align}
\norm{\Lxi^j(\Psiminuss)_{m,\ell}}^2_{W_{2(-\ell+\sfrak-1)}^{\reg}(\Sigmatb^{\leq \tb})}
\leq{}&C\tb^{-1-4\ell-2j-\vartheta}\InizeroEnergyminuss{\tb_0}{\ell}{\reg+\regl}{1+\vartheta},
\end{align}
and thus an analog of the estimate \eqref{eq:almost:int:m:allorder:nv:gene}:
\begin{align}
\absCDeri{\Lxi^j(r^{-\ell+\sfrak-1}(\Psiminuss)_{m,\ell})}{\reg}
\leq{}&C\tb^{ -1-2\ell-j-\frac{\vartheta}{2}}(\InizeroEnergyminuss{\tb_0}{\ell}{\reg+\regl}{1+\vartheta})^{\half}.
\end{align}
Together with the weak decay estimate \eqref{eq:weakdecay:geqell0:ipm} applied to the remainder $(\Psiminuss)^{\ell\geq 2\ell_0+2}$, and using the Plancherel's lemma, this yields in the  interior region $\{\rb\leq \tb\}$,
\begin{subequations}
\label{eq:HighermodeshaveFasterdecayininterior}
\begin{align}
\absCDeri{\Lxi^j(r^{-\ell_0+\sfrak-1}(\Psiminuss)^{\ell\geq \ell_0+1})}{\reg}
\leq{}&C\tb^{ -3-2\ell_0-j-\frac{\vartheta}{2}}(\InizeroEnergyplussnv{\reg+\regl}{\vartheta}
 )^{\half},
\end{align}
the decay of which is faster than the one of $r^{-\ell_0+\sfrak-1}(\Psiminuss)_{m,\ell_0}$; Similarly, one has for the spin $+\sfrak$ component that in the interior region $\{\rb\leq \tb\}$,
\begin{align}
\absCDeri{\Lxi^j(r^{-\ell_0-\sfrak-1}(\Psipluss)^{\ell\geq \ell_0+1})}{\reg}
\leq{}&C\tb^{ -3-2\ell_0-j-\frac{\vartheta}{2}}(\InizeroEnergyplussnv{\reg+\regl}{\vartheta}
 )^{\half},
\end{align}
\end{subequations}
Adding this estimate for $\ell\geq \ell_0+1$ to the one of the $(m,\ell_0)$ mode $r^{-\ell_0+\sfrak-1}(\Psiminuss)_{m,\ell_0}$ then yields the desired estimate.
\end{proof}

\section{Price's law under non-vanishing Newman--Penrose constant condition}
\label{sect:NVNPC:PL}

This section derives the precise late-time asymptotics of a single $(m,\ell)$-th mode of the $\pm \sfrak$ components under the condition that the  N--P constant $\NPCPs{\ell}$ defined in Definition \ref{def:NPCs} is nonzero. First, in Section \ref{subsect:nv:ext:PL}, we compute the asymptotics 
of the $+\sfrak$ component in a region $\{v-u\geq v^\alpha\}$ for some $\alpha\in (\half,1)$, and the precise late-time asymptotics of the $-\sfrak$ component in this region is obtained by using the TSI.  Afterwards, in Section \ref{subsect:nv:int:PL} which treats  the remaining region $\{v-u\leq v^\alpha\}$, we however first compute the late-time asymptotics of the $-\sfrak$ component and then use the TSI to get the late-time asymptotics of $+\sfrak$ component. These estimates in different regions are collected together in Section \ref{subsect:GPL:nv:MT} to derive the global late-time asymptotics of the $\pm \sfrak$ components which are supported on $\ell\geq \ell_0$ modes. 

We will frequently use the double null coordinates $(u,v,\theta,\phi)$, and the DOC is divided into different regions as depicted in Figure \ref{fig:4}, where $\gamma_\alpha=\{v-u=v^\alpha\}$ for an $\alpha\in (\half,1)$. 

The following lemma lists all relations and estimates among $u$, $v$, $r$, and $\tb$ that will be utilized in these different regions. The proof is simple and omitted.
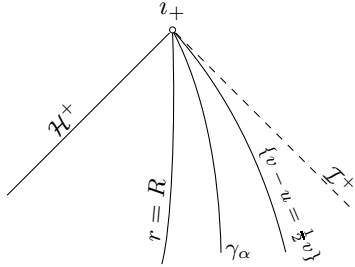
\begin{figure}[htbp]
\begin{center}
\begin{tikzpicture}[scale=0.8]
\draw[thin]
(0,0)--(2.45,2.45);
\draw[thin]
(0,0)--(-0.25,-0.25);
\draw[very thin]
(2.5,2.5) circle (0.05);
\coordinate [label=90:$i_+$] (a) at (2.5,2.5);
\draw[dashed]
(2.55,2.45)--(5.5,-0.5);
\node at (0.68,1.02) [rotate=45] {\small $\mathcal{H}^+$};
\node at (5.3,0) [rotate=-45] {\small $\mathcal{I}^+$};
\draw[thin]
(2.5,2.45) arc (20:-70:0.3 and 3);
\node at  (2.2,-0.5) [rotate=85] {\small $r=R$};
\draw[thin]
(2.52,2.45) arc (50:2:2.2 and 5);
\draw[thin]
(2.53,2.46) arc (50:18:6 and 8);
\node at (3.6,-1.2)  {\small $\gamma_{\alpha}$};
\node at (4.45,-0.4) [rotate=-67] {\scriptsize $\{v-u=\half v\}$};
\end{tikzpicture}
\end{center}
\caption{Useful hypersurfaces}
\label{fig:4}
\end{figure}

\begin{lemma}
For any $\alpha\in (\half, 1)$, let $\gamma_{\alpha}=\{v-u=v^\alpha\}$. For any $u$ and $v$, let $u_{\gamma_\alpha}(v)$ and $v_{\gamma_{\alpha}}(u)$ be such that $(u_{\gamma_\alpha}(v),v), (u, v_{\gamma_{\alpha}}(u))\in \gamma_{\alpha}$.
In the region $v-u\geq v^{\alpha}$,
\begin{subequations}
\begin{align}
\label{eq:rela:a}
&r\gtrsim{} v^{\alpha}+u^{\alpha},\\
\label{eq:rela:b}
&\abs{u-v_{\gamma_{\alpha}}(u)}\lesssim{}u^{\alpha},\\
\label{eq:rela:c}
&\abs{2r-(v-u)}\lesssim{}\log (r-2M);
\end{align}
in the region $\{v-u\geq v^\alpha\}\cap\{v-u\geq\frac{v}{2}\}$,
\begin{align}
\label{eq:rela:d}
&v+u\lesssim r\lesssim v;
\end{align}
in the region $\{v-u\geq v^{\alpha}\}\cap\{v-u\leq\frac{v}{2}\}$,
\begin{align}
\label{eq:rela:e}
&u \sim v, \quad r\gtrsim v^{\alpha};
\end{align}
in the region $\{r\geq R\}\cap\{v-u\leq v^{\alpha}\}$,
\begin{align}
\label{eq:rela:f}
&r\lesssim\min\{v^{\alpha}, u^{\alpha}\};
\end{align}
in the region $\{2M\leq r\leq R\}$, there is a constant $C_R$ depending on $R$ such that
\begin{align}
\label{eq:rela:g}
\abs{v|_{\Sigma_\tau}(R)-v|_{\Sigma_\tau}(r)}+\abs{v-\tb}\leq C_R.
\end{align}
\end{subequations}
On $\Sigmazero$, for $r$ large,
\begin{align}\label{asymp-factor}
\abs{r^{-1}v- 2-4Mr^{-1}\log(r-2M)}\lesssim r^{-1}.
\end{align}
\end{lemma}

Last, we remark that \textit{throughout subsections \ref{subsect:nv:ext:PL} and \ref{subsect:nv:int:PL},  only the sharp decay (i.e., the Price's law) for a fixed $(m,\ell)$ mode of the spin $\pm \sfrak$ components is considered, and the dependence on $m, \ell$ may be suppressed.} The last subsection \ref{subsect:GPL:nv:MT} is to gather these estimates to show global Price's law for the spin $\pm \sfrak$ component in the case of non-vanishing Newman--Penrose constant.

\subsection{Price's law in the region $\{v-u\geq v^{\alpha}\}$}
\label{subsect:nv:ext:PL}

Let us begin with making an assumption for the asymptotics of the initial data.

\begin{assump}[Initial data assumption to order $i$] \label{assump:initialdata:non-vanishing:pm1} Let $i\in \mathbb{N}$. Assume the $(m,\ell)$-th N--P constant of the spin $+\sfrak$ component satisfies $\NPCPs{\ell}\neq 0$. Assume on $\Sigma_{\tau_0}$ that  there are constants  $\beta\in (0,\half)$ and $0\leq D_0<\infty$ such that for all $0\leq i'\leq i$ and $r\geq 10M$,
\begin{align}
\Big|\prb^{i'}\Big(\VR\tildePhiplussHigh{\ell-\sfrak}(\tau_0,v)-{\rb^{-2}\NPCPs{\ell}}\Big)\Big|\lesssim D_0\rb^{-2-\beta-i'}.
\end{align}
\end{assump}

We next define the asymptotic profile of the spin $\pm \sfrak$ components. Roughly speaking, the asymptotic profile is the leading order term in the evolution if one expands the solution into Taylor series with respect to the time function $\tb$. 

Given a $\delta\in (0,\half)$ that will be fixed,  we define \begin{align}
\FBT=(\InizeroEnergypluss{\tb_0}{\ell}{\reg}{3-\delta})^{\half}+\abs{\NPCPs{\ell}}+ D_0,
\end{align}
where  $\InizeroEnergypluss{\tb_0}{\ell}{\reg}{3-\delta}$ is defined as in Definition \ref{def:energyonSigmazero:Schw} with $\delta$ is to be fixed in the proof,  $D_0$ is the constant appearing in Assumption \ref{assump:initialdata:non-vanishing:pm1},  and we have suppressed from this notation $\FBT$ the dependence on the mode parameter $m$, the regularity parameter $\reg$ which only depends on $\ell$, and the constant $\delta$. This scalar is necessary in determining the late-time asymptotics of the spin $\pm \sfrak$ components.

 The definition of asymptotic profile is given below.

\begin{definition}[Asymptotic profile]
For a spin $s$ scalar $\varphi$, we call $(\varphi)_{\text{AP}}$ the \textquotedblleft{asymptotic profile\textquotedblright}
 of $\varphi$ if there exists a $\veps>0$ such that
\begin{align}
\abs{\varphi-(\varphi)_{\text{AP}}}\lesssim \tb^{-\veps}(\NPCPs{\ell})^{-1}{\FBT}\cdot(\varphi)_{\text{AP}}.
\end{align}
\end{definition}

\begin{remark}
In the latter discussions, $\varphi$ is either the spin $+\sfrak$ or the spin $-\sfrak$ component and the function $(\varphi)_{\text{AP}}$ will be shown to always contain a factor $\NPCPs{\ell}$; this is the reason for the presence of the factor $ (\NPCPs{\ell})^{-1}$ on the RHS here.
\end{remark}

The proof of obtaining the Price's law in this region is divided into three steps.

\textbf{Step 1. Decay of $V\tildePhiplussHigh{\ell-\sfrak}$ by integrating its wave equation over constant $v$}

\begin{prop}
\label{prop:VtildePhil-1:nearinf:pm1}
Assume the initial data assumption \ref{assump:initialdata:non-vanishing:pm1}  holds to order $i_0$, then for $\delta$ sufficiently small and $\alpha=\alpha(\delta)$ sufficiently close to $1$, there exists an $\veps>0$ such that in the region $v-u\geq v^\alpha$,
\begin{align}
\label{eq:VPhiplusHigh1:generalell:largeregion}
\Big|
V \tildePhiplussHigh{\ell-\sfrak}
-4(v-u)^{2\ell} v^{-2\ell-2}\NPCPs{\ell}
\Big|
\lesssim{}&(v-u)^{2\ell} v^{-2\ell-2-\veps}\FBT.
\end{align}
\end{prop}

\begin{proof}
Recall the wave equations of $\tildePhiplussHigh{i}$ in Proposition \ref{prop:wave:Phihigh:pm1}: for $i=\ell-\sfrak$,
\begin{align}
-\mu Y \curlVR  \tildePhiplussHigh{\ell-\sfrak}
-2(\ell+1)(r-3M)r^{-2}\curlVR\tildePhiplussHigh{\ell-\sfrak}
+\sum_{j=0}^{\ell-\sfrak}O(r^{-1}) \PhiplussHigh{j}={}0.
\end{align}
One can rewrite it as
\begin{align}
-\mu^{-\ell}r^{2\ell +2} Y (\mu^{\ell}r^{-2\ell}  V\tildePhiplussHigh{\ell-\sfrak})
+\sum_{j=0}^{\ell-\sfrak}O(r^{-1}) \PhiplussHigh{j}={}0,
\end{align}
or equivalently,
\begin{align}
\label{eq:puVtildePhil-1:pm1}
\pu \big(\mu^{\ell}r^{-2\ell}  V\tildePhiplussHigh{\ell-\sfrak}\big)
={}&\sum_{j=0}^{\ell-\sfrak}O(r^{-2\ell-3}) \PhiplussHigh{j}.
\end{align}
Integrating this equation along constant $v$ from the intersection point with $\Sigmazero$ and using the estimates in Propositions \ref{prop:almost:ext:ipm:nvv} and \ref{prop:almost:int:ipm:nvv},  we achieve
\begin{align}
\hspace{4ex}&\hspace{-4ex}\Big|(\mu^\ell r^{-2\ell}v^{2\ell +2}V \tildePhiplussHigh{\ell-\sfrak})(u,v)
- (\mu^\ell r^{-2\ell}v^{2\ell +2}V \tildePhiplussHigh{\ell-\sfrak})(u_{\Sigmazero}(v),v)\Big|\notag\\
\lesssim{}&v^{2\ell +2} \int_{u_{\Sigmazero}(v)}^u\big(
r^{-2\ell-2}v^{-1}\tb^{-1+\delta/2} \big)(u',v)\di u' (\InizeroEnergypluss{\tb}{\ell}{\reg}{3-\delta})^{\half}\notag\\
\lesssim{}&v^{-\eta} (\InizeroEnergypluss{\tb}{\ell}{\reg}{3-\delta})^{\half},
\end{align}
with $\eta=(2\ell+2)\alpha-(2\ell+1)-\delta/2$.
Meanwhile, we utilize \eqref{asymp-factor} to obtain
\begin{align}
\label{eq:VPhinsHigh1:awayregion:2}
|\mu^\ell r^{-2\ell}v^{2\ell +2}V\tildePhiplussHigh{\ell-\sfrak}(u_{\Sigma_{\tau_0}}(v),v)-2^{2\ell+2}\NPCPs{\ell}|
\lesssim v^{-\beta}D_0  +v^{-\alpha}\log v \abs{\NPCPs{\ell}}.
\end{align}
The above two estimates together then yield
\begin{align}
\label{eq:VPhiplusHigh1:generalell:largeregion:pre}
\Big|
V \tildePhiplussHigh{\ell-\sfrak}
-4(v-u)^{2\ell} v^{-2\ell-2}\NPCPs{\ell}
\Big|
\lesssim{}&\big(r^{2\ell-1}\log r v^{-2\ell-2}+r^{2\ell} v^{-2\ell-2}(v^{-\beta}+v^{-\eta})\big)\FBT.
\end{align}
By taking $\delta$ sufficiently small and $\alpha$ sufficiently close to $1$, this proves the estimate \eqref{eq:VPhiplusHigh1:generalell:largeregion}.
\end{proof}

\textbf{Step 2. Decay of  $\tildePhiplussHigh{\ell-\sfrak}$ and its derivatives in $\{v-u \geq v^{\alpha}\}$.}

We then derive the asymptotics of $\tildePhiplussHigh{\ell-\sfrak}$ and its derivatives. To obtain the asymptotics for $\tildePhiplussHigh{\ell-\sfrak}$,
one integrates along $u=const$ as in Figure \ref{fig:6}
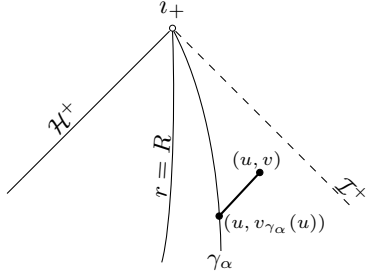
\begin{figure}[htbp]
\begin{center}
\begin{tikzpicture}[scale=0.8]
\draw[thin]
(0,0)--(2.45,2.45);
\draw[thin]
(0,0)--(-0.25,-0.25);
\draw[very thin]
(2.5,2.5) circle (0.05);
\coordinate [label=90:$i_+$] (a) at (2.5,2.5);
\draw[dashed]
(2.55,2.45)--(5.5,-0.5);
\node at (0.68,1.02) [rotate=45] {\small $\mathcal{H}^+$};
\node at (5.5,-0.2) [rotate=-45] {\small $\mathcal{I}^+$};
\draw[thin]
(2.5,2.45) arc (20:-70:0.3 and 3);
\node at  (2.3,0.2) [rotate=86] {\small $r=R$};
\draw[thin]
(2.52,2.45) arc (50:2:2.2 and 5);
\node at (3.3,-1.4)  {\small $\gamma_{\alpha}$};
\draw[thick]
(3.275,-0.625)--(3.95,0.1);
\draw[fill] (3.275,-0.625) circle (1.5pt);
\draw[fill] (3.95,0.1) circle (1.5pt);
\node at (4.2,-0.75) {\scriptsize $(u,v_{\gamma_\alpha}(u))$};
\node at (3.9,0.3) {\scriptsize $(u,v)$};
\end{tikzpicture}
\end{center}
\caption{For any point $(u,v)$ in $\{r\geq R\}\cap\{v-u\geq v^\alpha\}$, i.e.  $v\geq v_{\gamma_\alpha}(u)$, one integrates along $u=const$ from $(u, v_{\gamma_\alpha}(u))\in \gamma_\alpha$.}
\label{fig:6}
\end{figure}
 to obtain
\begin{align}\label{eq-v-hatPhips}
\tildePhiplussHigh{\ell-\sfrak}(u,v)={}&\tildePhiplussHigh{\ell-\sfrak}(u,v_{\gamma_\alpha}(u))
+\half \int_{v_{\gamma_\alpha}(u)}^{v}
 V\tildePhiplussHigh{\ell-\sfrak} (u,v')\di v'\notag\\
 ={}&\tildePhiplussHigh{\ell-\sfrak}(u,v_{\gamma_\alpha}(u))
+\half \int_{v_{\gamma_\alpha}(u)}^{v}
 \Big(V\tildePhiplussHigh{\ell-\sfrak}
-4(v-u)^{2\ell} v^{-2\ell-2}\NPCPs{\ell}\Big)(u,v')\di v'\notag\\
&
+2\NPCPs{\ell}\int_{v_{\gamma_\alpha}(u)}^{v}
 (v'-u)^{2\ell} (v')^{-2\ell-2}\di v'.
 \end{align}
By an integration by parts, one has
\begin{align*}
\int_{v_{\gamma_{\alpha}}(u)}^v (v'-u)^{2\ell} (v')^{-2\ell-1}
\di v'
={}&\frac{1}{2\ell+1}(v'-u)^{2\ell+1}(v')^{-2\ell-1}\vert_{v_{\gamma_{\alpha}}(u)}^v\notag\\
&
+\int_{v_{\gamma_{\alpha}}(u)}^v (v'-u)^{2\ell+1} (v')^{-2\ell-2}
\di v',
\end{align*}
and moving the last term to the LHS yields
\begin{align}\label{eq:PL:nv:p:integral:constu}
\int_{v_{\gamma_\alpha}(u)}^{v}
2 (v'-u)^{2\ell} (v')^{-2\ell-2}\di v'
={}&\frac{2}{2\ell+1}(v'-u)^{2\ell+1}u^{-1}(v')^{-2\ell-1}\vert_{v_{\gamma_{\alpha}}(u)}^v\notag\\
={}&\frac{2}{2\ell+1}
\frac{(v-u)^{2\ell+1}}{u v^{2\ell+1}}
+O(u^{-2\ell-2+(2\ell+1)\alpha}).
\end{align}
Hence, we conclude
\begin{align*}
\bigg|
2\NPCPs{\ell}\int_{v_{\gamma_\alpha}(u)}^{v}
 ((v-u)^{2\ell} v^{-2\ell-2})(u,v')\di v'
 -\frac{2}{2\ell+1}
\frac{(v-u)^{2\ell+1}}{u v^{2\ell+1}}\NPCPs{\ell}
 \bigg|
 \lesssim{}&u^{-2\ell-2+(2\ell+1)\alpha}
\abs{\NPCPs{\ell}}.
\end{align*}
The coefficient of the RHS is of lower order than the $\frac{(v-u)^{2\ell+1}}{u v^{2\ell+1}}$ behaviour by requiring $\alpha$ sufficiently close to $1$. Meanwhile, the second last integral on the RHS of
\eqref{eq-v-hatPhips} is of lower order than $\frac{(v-u)^{2\ell+1}}{u v^{2\ell+1}}$ as well from \eqref{eq:VPhiplusHigh1:generalell:largeregion}, and $\abs{\tildePhiplussHigh{\ell-\sfrak}(u,v_{\gamma_\alpha}(u))}\lesssim r v^{-1} \tb^{-1+\delta/2}$ which is also of lower order than $\frac{(v-u)^{2\ell+1}}{u v^{2\ell+1}}$ by choosing $\delta$ sufficiently small and $\alpha$ sufficiently close to $1$.
In total, there exists an $\veps>0$ such that
\begin{align}
\label{esti:PL:nv:p:ell-sfrak}
\bigg|
\tildePhiplussHigh{\ell-\sfrak}
 -\frac{2}{2\ell+1}
\frac{(v-u)^{2\ell+1}}{u v^{2\ell+1}}\NPCPs{\ell}
 \bigg|
 \lesssim{}&
\frac{(v-u)^{2\ell+1}}{u v^{2\ell+1}}\tb^{-\veps}\FBT.
\end{align}

By further commuting \eqref{eq:puVtildePhil-1:pm1} with $\pv^i$ $(i\leq i_0)$ implies
\begin{align}
\label{eq:puVtildePhil-1:pms:highpv}
\pu \Big(\pv^i\big(\mu^{\ell}r^{-2\ell}  V\tildePhiplussHigh{\ell-\sfrak}\big)\Big)
={}&\sum_{x=0}^i \sum_{j=0}^{\ell-\sfrak}O(r^{-2\ell-3-i}) (r\pv)^x\PhiplussHigh{j}.
\end{align}
We integrate along constant $v$ from $\Sigmazero$, and from \eqref{eq:rela:c} and the almost Price's law estimates in Propositions \ref{prop:almost:ext:ipm:nvv} and \ref{prop:almost:int:ipm:nvv},
we can choose $\delta$ sufficiently small and $\alpha=\alpha(i,\delta)$ sufficiently close to $1$ such that there exists an $\veps=\veps(\delta,\alpha,i)>0$
\begin{align}
\Big|\Big(\pv^i((v-u)^{-2\ell}V \tildePhiplussHigh{\ell-\sfrak})\Big)(u,v)
- \Big(\pv^i((v-u)^{-2\ell}V \tildePhiplussHigh{\ell-\sfrak})\Big)(u_{\Sigmazero}(v),v)\Big|
\lesssim{}&v^{-i-2\ell -2-\veps}\FBT.
\end{align}
By the initial data assumption \ref{assump:initialdata:non-vanishing:pm1} to order $i_0$,  for any $i\leq i_0$, we can take $\alpha=\alpha(i,\beta)$ sufficiently close to $1$ and achieve
\begin{align}
\Big|\Big(\pv^i((v-u)^{-2\ell}V \tildePhiplussHigh{\ell-\sfrak})\Big)(u_{\Sigmazero}(v),v)-4\pv^i(v^{-2\ell-2})\NPCPs{\ell}\Big|
\lesssim v^{-i-2\ell-2-\veps}\FBT.
\end{align}
As a result,
\begin{align}
\label{eq:PL:nv:p:far:pvhigh}
\Big|\Big(\pv^i((v-u)^{-2\ell}V \tildePhiplussHigh{\ell-\sfrak})-4\pv^i(v^{-2\ell-2})\NPCPs{\ell}\Big)(u,v)
\Big|
\lesssim{}&v^{-i-2\ell-2-\veps}\FBT.
\end{align}
This together with \eqref{esti:PL:nv:p:ell-sfrak} implies for any $i\leq i_0$,
\begin{align}
\label{eq:PL:nv:p:far:pvhighi:13}
\bigg|\pv^i\bigg((v-u)^{-2\ell} \tildePhiplussHigh{\ell-\sfrak}-\frac{2}{2\ell+1}\frac{v-u}{uv^{2\ell+1}}\NPCPs{\ell}\bigg)
\bigg|
\lesssim{}&v^{-i-2\ell-1-\veps}\FBT.
\end{align}

In the above estimate \eqref{eq:PL:nv:p:far:pvhigh}, we can replace one $\pv$ using $\pv=\Lxi-\pu$ and utilize equation \eqref{eq:puVtildePhil-1:pms:highpv} to estimate $\pu \pv^{i-1}((v-u)^{-2\ell}  V\tildePhiplussHigh{\ell-\sfrak})$, and this enables us to obtain
for any $i\leq i_0$ that
\begin{align}
\Big|\Big(\pv^{i-1}((v-u)^{-2\ell}V\Lxi \tildePhiplussHigh{\ell-\sfrak})-4\pv^{i-1}\Lxi(v^{-2\ell-2})\NPCPs{\ell}\Big)(u,v)
\Big|
\lesssim{}&v^{-i-2\ell-2-\veps}\FBT.
\end{align}
We can repeat this and conclude that for any $0\leq j\leq i$,
\begin{align}
\Big|\Big(\pv^{i-j}((v-u)^{-2\ell}V\Lxi^j \tildePhiplussHigh{\ell-\sfrak})-4\pv^{i-j}\Lxi^j(v^{-2\ell-2})\NPCPs{\ell}\Big)(u,v)
\Big|
\lesssim{}&v^{-i-2\ell-2-\veps}\FBT,
\end{align}
that is,
\begin{align}
\Big|\Big(\pv^{i-j}\big((v-u)^{-2\ell}\Lxi^j\big[\pv \tildePhiplussHigh{\ell-\sfrak}-2(v-u)^{2\ell}v^{-2\ell-2}\NPCPs{\ell}\big]\big)\Big)(u,v)
\Big|
\lesssim{}&v^{-i-2\ell-2-\veps}\FBT.
\end{align}
In particular, for $j=i$, we have
\begin{align}
\label{eq:PL:nv:p:far:puhighLxi:15}
\Big|\Big(\pv \Lxi^i\tildePhiplussHigh{\ell-\sfrak}-2\Lxi^i((v-u)^{2\ell}v^{-2\ell-2})\NPCPs{\ell}\Big)(u,v)
\Big|
\lesssim{}&(v-u)^{2\ell}v^{-i-2\ell-2-\veps}\FBT.
\end{align}
We now integrate $\pv \Lxi^i\tildePhiplussHigh{\ell-\sfrak}$ on constant-$u$ hypersurface from the intersection point $(u,v_{\gamma_{\alpha}}(u))$ with $\gamma_{\alpha}$:
\begin{align}
\label{eq:PL:nv:p:far:puhighLxi:16}
\Lxi^i\tildePhiplussHigh{\ell-\sfrak}(u,v)
-\Lxi^i\tildePhiplussHigh{\ell-\sfrak}(u,v_{\gamma_{\alpha}}(u))
={}& \int_{v_{\gamma_{\alpha}}(u)}^v\pv \Lxi^i\tildePhiplussHigh{\ell-\sfrak}(u,v')\di v'.
\end{align}
To estimate the RHS, we substitute \eqref{eq:PL:nv:p:far:puhighLxi:15} in and use the same way of arguing as in proving \eqref{eq:PL:nv:p:integral:constu} to conclude
\begin{align}
\bigg|\int_{v_{\gamma_{\alpha}}(u)}^v\pv \Lxi^i\tildePhiplussHigh{\ell-\sfrak}(u,v')\di v'-
\frac{2}{2\ell+1}
\Lxi^i\bigg(\frac{(v-u)^{2\ell+1}}{u v^{2\ell+1}}\bigg)\NPCPs{\ell}\bigg|
\lesssim (v-u)^{2\ell} v^{-i-2\ell-1-\veps }\FBT.
\end{align}
For the second term  on the LHS of \eqref{eq:PL:nv:p:far:puhighLxi:16}, which is evaluated at $(u, v_{\gamma_{\alpha}}(u))$, one can use $\pu=\Lxi-\frac{1}{2r}rV$ and the estimates in Proposition \ref{prop:almost:int:ipm:nvv} to achieve
\begin{align}
\big|\Lxi^{i}\tildePhiplussHigh{\ell-\sfrak}(u,v_{\gamma_{\alpha}}(u))\big|
\lesssim{}& (r(u,v_{\gamma_{\alpha}}(u)))^{2\ell+1}u^{-i-2\ell-2+\delta/2}\FBT\notag\\
\lesssim{}& u^{(2\ell+1)\alpha+(-i-2\ell-2+\delta/2)}\FBT.
\end{align}
These together imply that for any $i\leq i_0$,
\begin{align}
\label{eq:PL:nv:p:far:lxihighi:13}
\hspace{4ex}&\hspace{-4ex}
\bigg|\Lxi^i\bigg((v-u)^{-2\ell}\tildePhiplussHigh{\ell-\sfrak}-
\frac{2}{2\ell+1}
\frac{v-u}{u v^{2\ell+1}}\NPCPs{\ell}\bigg)\bigg|\notag\\
\lesssim {}&v^{-i-2\ell-1-\veps }\FBT
+(v-u)^{-2\ell}u^{(2\ell+1)\alpha+(-i-2\ell-2+\delta/2)}\FBT.
\end{align}
The last term can be easily verified to be bounded by $Cv^{-2\ell} \tb^{-i-1-\veps }\FBT$ by considering two cases $\{v-u\geq \half v\}$ and $\{v-u\leq \half v\}$ separately if taking $\delta$ sufficiently small and $\alpha$ sufficiently close to $1$.

Using $\pu=\Lxi-\pv$ and  combining the estimates \eqref{eq:PL:nv:p:far:pvhighi:13} and \eqref{eq:PL:nv:p:far:lxihighi:13},
we finally obtain the following statement for the asymptotic profiles of the derivatives of the scalar $(v-u)^{-2\ell}\tildePhiplussHigh{\ell-\sfrak}$ in the region $\{v-u \geq v^{\alpha}\}$.

\begin{prop}
\label{prop:asymp:tildePhiplussHighest}
Let the initial data condition \ref{assump:initialdata:non-vanishing:pm1} to order $i_0$ hold true. Let $j_1, j_2,j_3\in \mathbb{N}$ and $j_1+j_2+j_3\leq i_0$. There exists a sufficiently small $\delta$ and an $\alpha\in (\frac{1}{2},1)$ sufficiently close to $1$ such that in the region $\{v-u \geq v^{\alpha}\}$,
\begin{align}
\label{eq:PriceDeriactingonPhil-1:pm1:1}
(\Lxi^{j_1}\pu^{j_2}\pv^{j_3}((v-u)^{-2\ell}\tildePhiplussHigh{\ell-\sfrak}))_{\text{AP}}
={}&\frac{2\NPCPs{\ell}}{2\ell+1}
\Lxi^{j_1}\pu^{j_2}\pv^{j_3}\bigg(\frac{v-u}{u v^{2\ell+1}}\bigg).
\end{align}
\end{prop}

\textbf{Step 3. Decay of  all $\tildePhiplussHigh{i}$ and spin $\pm \sfrak$ components in $\{v-u \geq v^{\alpha'}\}$.}

Utilizing the above asymptotics of $\tildePhiplussHigh{\ell-\sfrak}$, we can derive the precise asymptotics of all $\tildePhiplussHigh{i}$, $i=0,1,\ldots, \ell-\sfrak$, hence also of the spin $+\sfrak$ component itself. The asymptotics of the spin $-\sfrak$ component is then calculated by the TSI in Section \ref{sect:TSI}.

\begin{thm}[Price's law in the region $\{v-u\geq v^{\alpha}\}$ in the non-vanishing Newman--Penrose constant case]
\label{thm:PL:anymodes:away:pm1}
Let $j\in\mathbb{N}$ and $\abs{\mathbf{a}}\leq j$. Let $\PriceDeri=\{\Lxi, \pu,\pv\}$. Let the initial data condition \ref{assump:initialdata:non-vanishing:pm1} to order $j+\ell+\sfrak$ hold true. Then, there exists an $\alpha\in (\frac{1}{2},1)$ sufficiently close to $1$ such that
in the region  $\{v-u \geq v^{\alpha}\}$,
the asymptotics $(\PriceDeri^{\mathbf{a}}(r^{-2\sfrak}\psipluss))_{\text{AP}}$ are
\begin{align}
\label{eq:asymp:tildePhiplusi:away:pm1:d}
(\PriceDeri^{\mathbf{a}}(r^{-2\sfrak}\psipluss))_{\text{AP}}={}&
2^{\ell+\sfrak+2}
\prod_{i=\ell+\sfrak+1}^{2\ell+1} i^{-1}{\NPCPs{\ell}}\PriceDeri^{\mathbf{a}}\big((v-u)^{\ell-\sfrak}u^{-\ell+\sfrak-1}v^{-\ell-\sfrak-1}\big)
\end{align}
and the asymptotics $(\PriceDeri^{\mathbf{a}}\psiminuss)_{\text{AP}}$ are
\begin{align}
\label{eq:asymp:psiminus:away:pm1}
(\PriceDeri^{\mathbf{a}}\psiminuss)_{\text{AP}}
={}&2^{\ell+\sfrak+2}
\prod_{i=\ell+\sfrak+1}^{2\ell+1} i^{-1}{\NPCPs{\ell}}\PriceDeri^{\mathbf{a}}\big((v-u)^{\ell-\sfrak}v^{-\ell+\sfrak-1}u^{-\ell-\sfrak-1}\big).
\end{align}
\end{thm}

\begin{proof}
Recall from equation \eqref{eq:Phiplushighi:Schw:generall:tildePhiplusi} that for all $i\in \{0,1,\ldots, \ell-\sfrak-1\}$,
\begin{align}
&-\mu Y \curlVR  \tildePhiplussHigh{i} +(\edthR'\edthR+f_{\sfrak,i,1})\tildePhiplussHigh{i}
+{f_{\sfrak,i,2}(r-3M)r^{-2}}\curlVR\tildePhiplussHigh{i}
+\sum_{j=0}^{i}O(r^{-1}) \PhiplussHigh{j}={}0.
\end{align}
We substitute in the values of $f_{\sfrak,i,1}$ and $f_{\sfrak,i,2}$ and put the above equation into the following form
\begin{align}
\hspace{4ex}&\hspace{-4ex}
-\mu Y  \tildePhiplussHigh{i+1} -(\ell-i-\sfrak)(\ell+i+\sfrak+1)\tildePhiplussHigh{i}
-2(i+\sfrak+1) (r-3M)r^{-2}\tildePhiplussHigh{i+1}\notag\\
={}&\sum_{j=0}^{i}\Big(O(r^{-1}) \tildePhiplussHigh{j}
+O(1) Y\tildePhiplussHigh{j}\Big).
\end{align}
This gives
\begin{align}
\hspace{4ex}&\hspace{-4ex}
(\ell-i-\sfrak)(\ell+i+\sfrak+1)(v-u)^{-2(i+\sfrak)}\tildePhiplussHigh{i}\notag\\
={}&
(v-u)^{-2(i+\sfrak)}\Big(-2\pu\tildePhiplussHigh{i+1}
-2(i+\sfrak+1) r^{-1}\tildePhiplussHigh{i+1}\notag\\
&
+\sum_{j=0}^{i}\big(O(1) \pu\tildePhiplussHigh{i}
+O(r^{-1}) \tildePhiplussHigh{j}\big)
+O(r^{-2})\tildePhiplussHigh{i+1}\Big)\notag\\
={}&
-2(v-u)^2\pu\Big((v-u)^{-2(i+\sfrak+1)}\tildePhiplussHigh{i+1}\Big)
\notag\\
&
+(v-u)^{-2(i+\sfrak)}\Big(O(1) \pu\tildePhiplussHigh{i}
+\sum_{j=0}^{i}O(r^{-1}) \tildePhiplussHigh{j}
+O(r^{-1})r^{-1}\log r \tildePhiplussHigh{i+1}\Big).
\end{align}
All the terms in the last line have (at least) extra $\tb^{-\veps}$ decay compared to the second last line,  thus the asymptotics $(\PriceDeri^{\mathbf{a}}\tildePhiplussHigh{i})_{\text{asymp}}$ for all $i\in \{0,1,\ldots, \ell-\sfrak\}$ are determined by
 iteratively solving
the following equations
\begin{align*}
(\PriceDeri^{\mathbf{a}}((v-u)^{-2\ell}\tildePhiplussHigh{\ell-\sfrak}))_{\text{AP}}={}&\frac{2\NPCPs{\ell}}{2\ell+1}
\PriceDeri^{\mathbf{a}}\bigg(\frac{v-u}{u v^{2\ell+1}}\bigg),\\
(\PriceDeri^{\mathbf{a}}((v-u)^{-2(i+\sfrak)}\tildePhiplussHigh{i}))_{\text{AP}}={}&\frac{-2(v-u)^2 (\pu\PriceDeri^{\mathbf{a}}(v-u)^{-2(i+\sfrak+1)}\tildePhiplussHigh{i+1}))_{\text{AP}}}{(\ell-i-\sfrak)(\ell+i+\sfrak+1)}
, \quad 0\leq i\leq \ell-\sfrak-1.
\end{align*}
Solving these equations yields
\begin{align}
\label{eq:asymp:tildePhiplusi:away:pm1}
(\PriceDeri^{\mathbf{a}}((v-u)^{-2\sfrak}\tildePhiplussHigh{0}))_{\text{AP}}={}&2^{\ell-\sfrak+1}
\prod_{i=\ell+\sfrak+1}^{2\ell+1} i^{-1}{\NPCPs{\ell}}\PriceDeri^{\mathbf{a}}\bigg(\bigg(\frac{v-u}{u}\bigg)^{\ell-\sfrak+1}v^{-\ell-\sfrak-1}\bigg).
\end{align}
Since $r^{-2\sfrak}\psipluss=\mu^{\sfrak} r^{-2\sfrak-1}\tildePhiplussHigh{0}$, equation \eqref{eq:asymp:tildePhiplusi:away:pm1:d} holds.

Consider next the asymptotics of the spin $-\sfrak$ component. As mentioned already, the asymptotics of the spin $-\sfrak$ component can be calculated explicitly from the  TSI \eqref{eq:otherTSI:simpleform:l=1} and the already proven asymptotics of the spin $+\sfrak$ component. In view of the  TSI \eqref{eq:otherTSI:simpleform:l=1} and \eqref{eq:otherTSI:simpleform:l=2}, and since the last term $-12 M\overline{\Lxi\psiminustwo}$ in \eqref{eq:otherTSI:simpleform:l=2} has (at least) faster $\tb^{-1+\veps}$ decay than $\psiminustwo$, one has
\begin{align}
(\PriceDeri^{\mathbf{a}}\psiminuss)_{\text{AP}}
={}&\frac{1}{(\ell-\sfrak+1)\cdots (\ell+\sfrak)}\PriceDeri^{\mathbf{a}}((Y^{2\sfrak}\psipluss)_{\text{AP}}).
\end{align}
One can expand $Y^{2\sfrak}\psipluss= 2^{2\sfrak}\pu^{2\sfrak}\psipluss+\sum\limits_{i=1}^{2\sfrak}O(r^{-1-i})\pu^{2\sfrak-i}\psipluss$ and the last term clearly has faster decay than the term $2^{2\sfrak}\pu^{2\sfrak}\psipluss$ in the region $\{v-u \geq v^{\alpha}\}$, hence this  yields
\begin{align}
(\PriceDeri^{\mathbf{a}}\psiminuss)_{\text{AP}}
={}&\frac{2^{2\sfrak}}{(\ell-\sfrak+1)\cdots (\ell+\sfrak)}\PriceDeri^{\mathbf{a}}((\pu^{2\sfrak}\psipluss)_{\text{AP}})\notag\\
={}&\frac{2^{2\sfrak}}{(\ell-\sfrak+1)\cdots (\ell+\sfrak)}(\PriceDeri^{\mathbf{a}}(\pu^{2\sfrak}((v-u)^{2\sfrak}(v-u)^{-2\sfrak}\psipluss))_{\text{AP}})\notag\\
={}&\frac{1}{(\ell-\sfrak+1)\cdots (\ell+\sfrak)}(\PriceDeri^{\mathbf{a}}(\pu^{2\sfrak}((v-u)^{2\sfrak}r^{-2\sfrak}\psipluss))_{\text{AP}}).
\end{align} 
In view of  equation \eqref{eq:asymp:tildePhiplusi:away:pm1:d}, and by the following equality
\begin{align}
 \hspace{3ex}&\hspace{-3ex}\pu^{2\sfrak}((v-u)^{\ell+\sfrak}u^{-\ell+\sfrak-1}v^{-\ell-\sfrak-1})\notag\\
={}&(\ell+\sfrak)(\ell+\sfrak-1)\cdots(\ell-\sfrak+1)(v-u)^{\ell-\sfrak}v^{-\ell+\sfrak-1}u^{-\ell-\sfrak-1},
\end{align}
this proves \eqref{eq:asymp:psiminus:away:pm1}.
\end{proof}

\subsection{Price's law in the region $\{v-u\leq v^{\alpha}\}$}
\label{subsect:nv:int:PL}

We show the Price's law for the spin $\pm \sfrak$ components in the remaining region $\{v-u\leq v^{\alpha}\}$ in this subsection. Again, the spin $\pm \sfrak$ components are assumed to be supported on an $(m,\ell)$ mode. 
The estimates are first proved for the spin $-\sfrak$ component, and these yield the estimates for the spin $+\sfrak$   component via the TSI of Section \ref{sect:TSI}.

\begin{thm}[Price's law in the region $\{v-u\leq v^{\alpha}\}$ in the non-vanishing Newman--Penrose constant case]
\label{thm:PL:anymodes:near:pm1}
Let $j\in\mathbb{N}$. Let the initial data condition \ref{assump:initialdata:non-vanishing:pm1} to order $j+\ell+\sfrak$ hold true.
Then, there exists an $\alpha\in (\half,1)$ sufficiently close to $1$ such that
in the region  $\{v-u \leq v^{\alpha}\}$, the asymptotic profiles $(\Lxi^j (r^{-2\sfrak}\psipluss))_{\text{AF}}$  and $(\Lxi^j \psiminuss)_{\text{AF}}$ are
\begin{subequations}
\begin{align}
\label{eq:asymp:psiminus:near:int:pm1:d}
(\Lxi^j\psiminuss(\tb,\rb))_{\text{AP}}={}&(-1)^j2^{2\ell+2}
\prod_{i=\ell+\sfrak+1}^{2\ell+1} i^{-1}
\frac{(2\ell+j+1)!}{(2\ell+1)!}{\NPCPs{\ell}}\rb^{\ell-\sfrak}\hell\tb^{-2\ell-j-2},\\
\label{eq:asymp:psiplus:near:int:pm1:d}
(\Lxi^j(r^{-2\sfrak}\psipluss(\tb,\rb)))_{\text{AP}}={}&(-1)^j2^{2\ell+2}
\prod_{i=\ell+\sfrak+1}^{2\ell+1} i^{-1}
\frac{(2\ell+j+1)!}{(2\ell+1)!}{\NPCPs{\ell}}\mu^{\sfrak}\rb^{\ell-\sfrak}\hell\tb^{-2\ell-j-2},
\end{align}
and for $\sfrak\neq 0$,
\begin{align}
\label{eq:asymp:psiplus:horizon:pm1:d}
(\Lxi^j(r^{-2\sfrak}\psipluss(\tb,\rb)))_{\text{AP}}\Big|_{\Horizon}={}&\frac{(-1)^{\sfrak+j}2^{2\ell+3}\sfrak (\ell-\sfrak)!(2\ell+j+2)!}{((2\ell+1)!)^2}
{\NPCPs{\ell}}(2M)^{\ell-\sfrak+1}\hell(2M)v^{-2\ell-j-3}.
\end{align}
\end{subequations}
\end{thm}

\begin{proof}
Consider first the spin $-\sfrak$ component $\psiminuss$. We have achieved in Proposition \ref{prop:almost:int:ipm:nvv} that
\begin{align}
\label{eq:pfofthm:PL:anymodes:near:pm1:12}
\abs{\Lxi^j\prb\varphiminussell}\lesssim{}&\tb^{-2\ell-j-3+\delta/2}\FBT.
\end{align}
Therefore, by integrating  $\prb\varphiminussell$ from any point $(\tb,\rb')\in \{v-u \leq v^{\alpha}\}$ along constant $\tb$ up to the intersection point with the curve $\gamma_{\alpha}$, one finds for small enough $\delta$ that
\begin{align}
\label{eq:pfofthm:PL:anymodes:near:pm1:1}
\bigg|\int_{\rb'}^{\rb_{\gamma_{\alpha}}(\tb)}{\Lxi^j\prb\varphi_{-1,\ell}}\di\rb\bigg|
\lesssim{}&\tb^{-2\ell-j-3+\delta/2+\alpha}\FBT\lesssim \tb^{-2\ell-j-2-\veps}\FBT.
\end{align}
From the definition \eqref{def:varphiell:-1} of $\varphiminussell$ and equation \eqref{eq:asymp:psiminus:away:pm1}, we obtain
\begin{align}
(\Lxi^j\varphiminussell)_{\text{AP}}\big\vert_{\gamma_{\alpha}}
={}&2^{2\ell+2}
\prod_{i=\ell-\sfrak+1}^{2\ell+1} i^{-1}{\NPCPs{\ell}}\Lxi^j\big((v-u)^{-\ell+\sfrak}\pu^{2\sfrak}\big((v-u)^{\ell+\sfrak}u^{-\ell+\sfrak-1}v^{-\ell-\sfrak-1}\big)\big)\big\vert_{\gamma_{\alpha}}\notag\\
={}&2^{2\ell+2}
\prod_{i=\ell+\sfrak+1}^{2\ell+1} i^{-1}{\NPCPs{\ell}}\Lxi^j\big(u^{-\ell+\sfrak-1}v^{-\ell-\sfrak-1}\big)\big\vert_{\gamma_{\alpha}}\notag\\
={}&(-1)^j2^{2\ell+2}
\prod_{i=\ell+\sfrak+1}^{2\ell+1} i^{-1}
\prod_{n=2\ell+2}^{2\ell+j+1} n{\NPCPs{\ell}}\tb^{-2\ell-j-2}.
\end{align}
From this equation and the estimate  \eqref{eq:pfofthm:PL:anymodes:near:pm1:1}, it holds at any point $(\tb,\rb)\in \{v-u \leq v^{\alpha}\}$ that
\begin{align}
\label{eq:pfofthm:PL:anymodes:near:pm1:111}
(\Lxi^j\varphiminussell(\tb,\rb))_{\text{AP}}={}&(-1)^j2^{2\ell+2}
\prod_{i=\ell+\sfrak+1}^{2\ell+1} i^{-1}
\prod_{n=2\ell+2}^{2\ell+j+1} n{\NPCPs{\ell}}\tb^{-2\ell-j-2}.
\end{align}
Equation \eqref{eq:asymp:psiminus:near:int:pm1:d} thus follows from this equation together with the definition \eqref{def:varphiell:-1} .

Consider next the spin $+\sfrak$ component. We can obtain its asymptotics by utilizing the TSI \eqref{eq:TSI:simpleform:l=1} and \eqref{eq:TSI:simpleform:l=2} together with the proven estimates for the spin $-\sfrak$ component.
Since the last term $12 M \overline{\Lxi\Phiplustwo}$ on the RHS of \eqref{eq:TSI:simpleform:l=2} has faster decay in $\tb$ than $\Phiplustwo$,
the TSI \eqref{eq:TSI:simpleform:l=1}  and \eqref{eq:TSI:simpleform:l=2} can be written as
\begin{align}
\label{eq:TSI:precise:ap:generals}
(\ell-\sfrak+1)\cdots(\ell+\sfrak)(\Lxi^j(r^{-2\sfrak}\psipluss))_{\text{AP}}={}& (\Lxi^j(\mu^{\sfrak} r^{-2\sfrak-1}\Phiminuss{2\sfrak}))_{\text{AP}}.
\end{align}
Note that $\Phiminuss{2\sfrak}=(r^2 \VR)^{2\sfrak}\Phiminuss{0}=(r^2 \VR)^{2\sfrak}(\mu^{\sfrak}r\psiminuss)=(r^2 (\prb+\Hhyp\Lxi))^{2\sfrak}(\mu^{\sfrak}r^{\ell+1-\sfrak}\hell \varphiminussell)$, hence
we can expand out $\Lxi^j(\mu^{\sfrak} r^{-2\sfrak-1}\Phiminuss{2\sfrak})$ as follows:
\begin{align}
\Lxi^j(\mu^{\sfrak} r^{-2\sfrak-1}\Phiminuss{2\sfrak})
={}&\mu^{\sfrak} r^{-2\sfrak-1}(r^2 (\prb+\Hhyp\Lxi))^{2\sfrak}(\mu^{\sfrak}r^{\ell+1-\sfrak}\hell \Lxi^j\varphiminussell)\notag\\
={}&\sum_{j_1\geq 0, j_2\geq 0, 1\leq j_1+j_2\leq 2\sfrak}O(1)r^{\ell-\sfrak+j_1+j_2}\prb^{j_1}\Lxi^{j_2+j}\varphiminussell\notag\\
&+\mu^{\sfrak} r^{-2\sfrak-1}(r^2 \prb)^{2\sfrak}(\mu^{\sfrak}r^{\ell+1-\sfrak}\hell )\Lxi^j\varphiminussell.
\end{align}
By the estimates \eqref{eq:pfofthm:PL:anymodes:near:pm1:12}, the second last line has decay $\tb^{-2\ell-j-2-\veps}$ in the region  $\{v-u \leq v^{\alpha}\}$. Therefore, to show the estimate \eqref{eq:asymp:psiplus:near:int:pm1:d}, it suffices to prove
\begin{align}
\big([(\ell-\sfrak+1)\cdots(\ell+\sfrak)]^{-1}\mu^{\sfrak} r^{-2\sfrak-1}(r^2 \prb)^{2\sfrak}(\mu^{\sfrak}r^{\ell+1-\sfrak}\hell )\Lxi^j\varphiminussell\big)_{\text{AP}}={}\text{RHS of } \eqref{eq:asymp:psiplus:near:int:pm1:d}.
\end{align}
In view of \eqref{eq:pfofthm:PL:anymodes:near:pm1:111}, it in turn suffices to prove
\begin{align}
[(\ell-\sfrak+1)\cdots(\ell+\sfrak)]^{-1}\mu^{\sfrak} r^{-2\sfrak-1}(r^2 \prb)^{2\sfrak}(\mu^{\sfrak}r^{\ell+1-\sfrak}\hell)=\mu^{\sfrak}r^{\ell-\sfrak}\hell.
\end{align}
This is exactly \eqref{TSI:hell}, hence this proves \eqref{eq:asymp:psiplus:near:int:pm1:d}.

 To compute the asymptotics for the spin $+\sfrak$ component on $\Horizon$,  we shall use again equation \eqref{eq:TSI:precise:ap:generals}. Recall that $\Lxi^j(\mu^{\sfrak} r^{-2\sfrak-1}\Phiminuss{2\sfrak})=\Lxi^j(\mu^{\sfrak} r^{-2\sfrak-1}(\mu^{-1}r^2 V)^{2\sfrak}(\mu^{\sfrak}r\psiminuss))$. Hence, for $\rb\geq 2M$, for $\sfrak=1$, this becomes
\begin{subequations}
\label{eq:TSI:expand:horizon:+12}
\begin{align}
\label{eq:TSI:expand:horizon:+1}
\Lxi^j(\mu r^{-3}\Phiminus{2})={}&\Lxi^j(\mu r^{-3}(\mu^{-1}r^2 V)^{2}(\mu r\psiminus))\notag\\
={}&\Lxi^j (r^{-1} V (r^2 \psiminus+r^3 V\psiminus))\notag\\
={}&\Lxi^j (2\mu r \psiminus+ rV\psiminus +3\mu rV\psiminus+r^2 V^2 \psiminus)\notag\\
={}&r\Lxi^j  V\psiminus + \mu \sum_{i=0,1}O(1)r^i\Lxi^j V^i \psiminus+O(1)\Lxi^j r^2V^{2}\psiminus,
\end{align}
and for $\sfrak=2$, it becomes
\begin{align}
\label{eq:TSI:expand:horizon:+2}
\hspace{2.3ex}&\hspace{-2.3ex}
\Lxi^j(\mu^2 r^{-5}\Phiminustwo{4})\notag\\
={}&\Lxi^j(\mu^2 r^{-5}(\mu^{-1}r^2 V)^{4}(\mu^2 r\psiminustwo))\notag\\
={}&\Lxi^j \big[\mu r^{-3} V(\mu^{-1}r^2 V)^2 (r^2 \partial_r (\mu^2 r)\psiminustwo+ \mu r^3 V\psiminustwo)\big]\notag\\
={}&\Lxi^j \big[\mu r^{-3} V(\mu^{-1}r^2 V) (r^2 \partial_r (r^2 \partial_r (\mu^2 r))\psiminustwo+\mu^{-1}r^4  \partial_r (\mu^2 r)V\psiminustwo
+ r^2 \partial_r(\mu r^3) V\psiminustwo
+r^5 V^2 \psiminustwo)\big]\notag\\
={}&\mu^2r^{-3}\partial_r (\mu^{-1}r^4  \partial_r (r^2\partial_r (\mu^2 r)))V\Lxi^j\psiminustwo  + \mu \sum_{i=0,1}O(1) r^i\Lxi^j V^i \psiminustwo+\sum_{i=2}^4 O(1)\Lxi^j r^iV^{i}\psiminustwo.
\end{align}
\end{subequations}
On the event horizon, we have $\mu=0$ and $V=2\Lxi$, hence the second last term of both \eqref{eq:TSI:expand:horizon:+1} and \eqref{eq:TSI:expand:horizon:+2} vanishes and the last term of both \eqref{eq:TSI:expand:horizon:+1} and \eqref{eq:TSI:expand:horizon:+2} has $v^{-2\ell-j-4}$ decay by the asymptotics \eqref{eq:asymp:psiminus:near:int:pm1:d} of $\psiminuss$. Since the coefficient of the first term in the last line of \eqref{eq:TSI:expand:horizon:+2}, $\mu^2r^{-3}\partial_r (\mu^{-1}r^4  \partial_r (r^2\partial_r (\mu^2 r)))$, equals $-4M$ when on $\Horizon$, we obtain, by the above discussions and from equation \eqref{eq:TSI:precise:ap:generals}, that
\begin{align}
(\ell-\sfrak+1)\cdots(\ell+\sfrak)(\Lxi^j(r^{-2\sfrak}\psipluss))_{\text{AP}}\big|_{\Horizon}={}& (-1)^{\sfrak-1}4\sfrak M(\Lxi^{j+1}\psiminuss)_{\text{AP}}\big|_{\Horizon}.
\end{align}
By the asymptotics \eqref{eq:asymp:psiminus:near:int:pm1:d} of $\psiminuss$ on $\Horizon$, this thus yields the asymptotic \eqref{eq:asymp:psiplus:horizon:pm1:d} for the spin $+\sfrak$ component on $\Horizon$.
\end{proof}

\subsection{Global Price's law in the non-vanishing Newman--Penrose constant case}
\label{subsect:GPL:nv:MT}

We collect the statements in Theorems \ref{thm:PL:anymodes:away:pm1} and \ref{thm:PL:anymodes:near:pm1} and summarize as follows:
  
\begin{thm}[Global Price's law  for a fixed mode of the spin $\pm \sfrak$ components in the non-vanishing Newman--Penrose constant case]
\label{summary:PL:nv:pm:global}
Let $j\in\mathbb{N}$ and let $\sfrak=0,1,2$. Assume the spin $s=\pm \sfrak$ components are supported on a single $(m,\ell)$ mode, $\ell\geq \sfrak$ and $-\ell\leq m\leq \ell$.  Let the $(m,\ell)$-th N--P constant $\NPCPs{\ell}$ be defined as in Definition \ref{def:NPCs} and be nonzero.  Let the initial data condition \ref{assump:initialdata:non-vanishing:pm1} to order $j+\ell+\sfrak$ hold true.  Let function $\hell$ and scalars $\tildePhiplussHigh{i}$ and $\tildePhiminuss{i}$ be defined as in Definitions \ref{def:hell:pm} and \ref{def:tildePhiplusandminusHigh}, respectively.

Then there exist constants $\delta>0$, $\reg=\reg(\ell,j)$, $\veps=\veps(j, \ell, \delta)>0$ and $C=C(\reg,\ell, j, \delta)$ such that for any $\tb\geq \tb_0$, 
\begin{subequations}
\label{eq:globalPL:nv:pm}
\begin{align}
  \hspace{4ex}&\hspace{-4ex}
  \bigg|\Lxi^j\Upsilon_{s} -2^{2\ell+2}
  \prod_{i=\ell+\sfrak+1}^{2\ell+1} i^{-1}\mu^{\frac{\sfrak+s}{2}}\hell  r^{\ell-\sfrak}\Lxi^j(v^{-\ell-s-1}\tau^{-\ell+s-1}) {\NPCPs{\ell}}Y_{m,\ell}^{s}(\cos\theta)e^{im\pb} 
  \bigg|\notag\\
  \leq{}&C
 r^{\ell-\sfrak}
  v^{-\ell-s-1}\tau^{-\ell+s-1-j-\veps}\FBT,
\end{align}
and on the future event horizon $\Horizon$, for $\sfrak\neq0$,
\begin{align}
\hspace{4ex}&\hspace{-4ex}
\bigg|\Lxi^j\NPR_{+\sfrak}\big|_{\Horizon}
-\frac{(-1)^{\sfrak+j}2^{2\ell+3}\sfrak(\ell-\sfrak)!}{(2\ell+1)!}
\prod_{n=2\ell+2}^{2\ell+j+2} n
{\NPCPs{\ell}}(2M)^{\ell-\sfrak+1}\hell(2M)v^{-2\ell-j-3}Y_{m,\ell}^{+\sfrak}(\cos\theta)e^{im\pb}\bigg|\notag\\
\leq{}&C
v^{-2\ell-j-3-\veps}\FBT,
\end{align}
\end{subequations}
with $
\FBT=(\InizeroEnergyplussell{\tb_0}{\ell}{\reg}{3-\delta})^{\half}+\abs{\NPCPs{\ell}}+ D_0$ and $\InizeroEnergyplussell{\tb_0}{\ell}{\reg}{3-\delta}$ defined as in Definition \ref{def:energyonSigmazero:Schw}.
\end{thm}

As a corollary, by making use of the Maxwell equations \eqref{eq:TSIsSpin1Kerr}, one can easily derive the precise late-time asymptotics for the middle component of the Maxwell field from the above asymptotics of the spin $\pm \sfrak$ components of the Maxwell field. 

\begin{cor}\label{globalPL:nv:middle:cor}
Under the assumptions in Theorem \ref{summary:PL:nv:pm:global} for $s=\pm 1$, we have for the middle component $\NPRzero$ of the Maxwell field that
for any $\tb\geq \tb_0$ and $\rb\geq 2M$,
\begin{align}
\label{globalPL:nv:middle}
\hspace{4ex}&\hspace{-4ex}
\Biggl|\Lxi^j\biggl(\NPRzero-r^{-2}\left(q_{\mathbf{E}} + iq_{\mathbf{B}}\right)+\frac{2^{2\ell+3}}{\sqrt{\ell(\ell+1)}}
\prod_{i=\ell+2}^{2\ell+1} i^{-1} Y(\mu\hellone  r^{\ell}
\tb^{-\ell}v^{-\ell-2}){\NPCP{\ell}}Y_{m,\ell}^{0}(\cos\theta)e^{im\pb}\biggr)\Bigg|\notag\\
\leq{}&C\tb^{-\veps}
\Lxi^j(Y(\mu r^{\ell}
\tb^{-\ell}v^{-\ell-2}))\FBT,
\end{align}
where $q_{\mathbf{E}}$ and $q_{\mathbf{B}}$ are defined as in Lemma \ref{lem:decomp:Maxwellfield} from initial data and $\FBT$ is defined as in Theorem \ref{summary:PL:nv:pm:global}.
\end{cor}

\begin{proof}
For the middle component of the Maxwell field, we can decompose it as in Lemma \ref{lem:decomp:Maxwellfield} into
$\NPRzero=\NPRzero(\mathbf{F}_{\text{sta}})+ \NPRzero(\mathbf{F}_{\text{rad}})$, where $\NPRzero(\mathbf{F}_{\text{sta}})=r^{-2}(q_{\mathbf{E}} + iq_{\mathbf{B}})$.
The Maxwell equation \eqref{eq:TSIsSpin1KerrAngular0With1} and formula \eqref{eq:ellipticop:eigenvalue:fixedmode} together thus give
 $-\sqrt{\ell (\ell+1)}\NPRzero(\mathbf{F}_{\text{rad}})=2Y(r^{-1}\psiplus)$. We then use the estimates for $r^{-2}\psiplus$ in Theorems \ref{summary:PL:nv:pm:global}  and conclude the estimate \eqref{globalPL:nv:middle} for $\rb\geq 2M$. 
\end{proof}

The main statement in this entire Section \ref{sect:NVNPC:PL} is to compute the precise late-time asymptotics for the spin $\pm \sfrak$ components that are supported on $\ell\geq \ell_0$ modes, $\ell_0\geq \sfrak$. Essentially, this is obtained by combining the above Theorem \ref{summary:PL:nv:pm:global} for $\ell=\ell_0$ mode and the proven Theorems \ref{thm:almost:ext:ipm:nvv:higher} and \ref{thm:almost:int:ipm:nvv:hig} on the almost Price's law for $\ell\geq \ell_0+1$ modes.

\begin{thm}[Global Price's law  for the spin $\pm \sfrak$ components in the non-vanishing Newman--Penrose constant case]
\label{summary:PL:nv:pm:global:HM}
Let $j\in\mathbb{N}$ and let $\sfrak=0,1,2$. Assume the spin $s=\pm \sfrak$ components are supported on $\ell\geq \ell_0$ modes, $\ell_0\geq \sfrak$.  Let function $\hell$ and scalars $\tildePhiplussHigh{i}$ and $\tildePhiminuss{i}$ be defined as in Definitions \ref{def:hell:pm} and \ref{def:tildePhiplusandminusHigh}, respectively. Let the $(m,\ell_0)$-th N--P constant $\NPCPs{\ell_0}$ be defined as in Definition \ref{def:NPCs} and assume that not all of $\{\NPCPs{\ell_0}\}_{-\ell_0\leq m\leq \ell_0}$ are zero.  Assume on $\Sigma_{\tau_0}$ that  there are constants  $\beta\in (0,\half)$ and $0\leq D_0<\infty$ such that for all $0\leq i\leq j+\ell+\sfrak$, $m\in\{-\ell_0,-\ell_0+1,\ldots,\ell_0\}$ and $\rb\geq 10M$,
\begin{align}
\Big|\prb^{i}\Big(\VR(\tildePhiplussHigh{\ell_0-\sfrak})_{m,\ell_0}-{\rb^{-2}\NPCPs{\ell_0}}\Big)\Big|\lesssim D_0\rb^{-2-\beta-i}.
\end{align}

Then there are constants $\delta>0$,\footnote{By examining the proof, one finds that any $\delta\in (0,\frac{1}{2})$ works fine.}  $\reg=\reg(\ell_0,j)$, $\veps=\veps(j, \ell_0, \delta)>0$ and $C=C(\reg,\ell_0, j, \delta)$ such that for any $\tb\geq \tb_0$, 
\begin{align}\label{eq:globalPL:nv:pm:HM}
\hspace{4ex}&\hspace{-4ex}
\bigg|\Lxi^j\Upsilon_s -2^{2\ell_0+2}
\prod_{i=\ell_0+\sfrak+1}^{2\ell_0+1} i^{-1} \mu^{\frac{\sfrak+s}{2}} \hellz r^{\ell_0-\sfrak} \Lxi^j(
v^{-\ell_0-s-1}\tb^{-\ell_0+s-1})\sum_{m=-\ell_0}^{\ell_0}{\NPCPs{\ell_0}} Y_{m,\ell_0}^{s}(\cos\theta)e^{im\pb} \bigg|\notag\\
\leq{}&Cr^{\ell_0-\sfrak}
v^{-\ell_0-s-1}\tau^{-\ell_0+s-1-j-\veps}\FBNV,
\end{align}
and, on the future event horizon, for  $\sfrak\neq0$,
\begin{align}
\bigg|\Lxi^j\NPR_{+\sfrak}\big|_{\Horizon}
-C_j v^{-2\ell_0-j-3} \sum_{m=-\ell_0}^{+\ell_0}
{\NPCPs{\ell_0}}Y_{m,\ell_0}^{+\sfrak}(\cos\theta)e^{im\pb}\bigg|
\leq{}C
v^{-2\ell_0-j-3-\veps}\FBNV,
\end{align}
where
\begin{align}
  C_j=\frac{(-1)^{\sfrak+j}2^{2\ell_0+3}\sfrak (\ell_0-\sfrak)!}{(2\ell_0+1)!} (2M)^{\ell_0-\sfrak+1}\hellz(2M)
  \prod_{n=2\ell_0+2}^{2\ell_0+j+2} n,
\end{align}
and $
\FBNV=(\InizeroEnergyplussnv{\reg}{\delta})^{\half}+\sum_{m=-\ell_0}^{\ell_0}\abs{\NPCPs{\ell_0}}+ D_0$ and $\InizeroEnergyplussnv{\reg}{\delta}$ is defined as in Definition \ref{def:initialenergy:highMODEs}.
\end{thm}

\begin{proof}
By decomposing the spin $\pm \sfrak$ components into $\ell=\ell_0$ mode and $\ell\geq \ell_0+1$ modes and using the previous theorem \ref{summary:PL:nv:pm:global} to achieve the asymptotics for its $(m,\ell_0)$ mode, it suffices to show that $|\Lxi^j (\NPR_{s})^{\ell\geq \ell_0+1}|$ are bounded by the RHS of each estimate in \eqref{eq:globalPL:nv:pm:HM}, and the rest estimates follow easily. This fact is in turn implied by the estimates \eqref{eq:weakdecay:ell:ipm:ext:nv:higher:even} in the exterior region and \eqref{eq:HighermodeshaveFasterdecayininterior} in the interior region.
\end{proof}

\begin{remark}
  Similarly, one can argue as in Corollary \ref{globalPL:nv:middle:cor} to derive the global precise late-time asymptotics for the middle component $\NPRzero$ of the Maxwell field.
\end{remark}


\section{Price's law under vanishing Newman--Penrose constant condition}
\label{sect:VNPC:PL}


This section differs from the previous section in the sense that we assume the Newman--Penrose constant therein vanishes, i.e., $\NPCPs{\ell}=0$. Under this vanishing Newman--Penrose constant condition, and assuming further initial data  condition, we prove the asymptotic profiles for the spin $\pm \sfrak$ components.  In Section \ref{subsect:TI:v}, we define the time integral of each fixed $(m,\ell)$ mode of the spin $+\sfrak$ component and compute the expressions for its derivatives. Then, in Section \ref{subsect:NPIT:esti}, we calculate the Newman--Penrose constant for the time integral and bound the initial norms of the time integral in terms of the  initial norm of the $(m,\ell)$ mode of the spin $+\sfrak$ component. Finally, we apply the results in Section \ref{subsect:GPL:nv:MT} to the time integral and derive the precise late-time asymptotics of the spin $\pm \sfrak$ components in Section \ref{subsect:PL:v:MT}.

\subsection{Time integral}
\label{subsect:TI:v}

In this subsection, we will always assume that the spin $+\sfrak$ component is supported on an $(m,\ell)$ mode, and without confusion, we might suppress the dependence on $m$ and $\ell$ unless specified. In addition, we denote $(\psipluss)_{m,\ell}$ by $\Ppsiell$.

Recall equation \eqref{eq:totalderieq:generall:nega} satisfied by $\Ppsiell$ $(\sfrak=0,1,2)$:
\begin{align}
\prb\left(r^{2\ell+2}\mu^{1+\sfrak}\Phell^2 \prb(\mu^{-\sfrak}\Phell^{-1} r^{-\ell-\sfrak}\Ppsiell)\right)
 ={}& r^{\ell-\sfrak-1}\Lxi(\Hpsi),
\end{align}
where
\begin{align}\label{de:Hpsi}\begin{split}
\Hpsi
=&\mu^{\sfrak}\Phell\Big\{\mu \partial_r\hhyp\curlVR(\mu^{-\sfrak}r\Ppsiell)-\mu \Hhyp r^2 \partial_\rho(\mu^{-\sfrak}r\Ppsiell)\\
&\qquad\quad+\big[\big(2\sfrak r-2M(3\sfrak+1)\big)\Hhyp -\mu r^2\partial_r\Hhyp \big]\mu^{-\sfrak}r\Ppsiell\Big\}.
\end{split}\end{align}

The time integral of the scalar $\Ppsiell$ is constructed and defined as follows. 

\begin{lemma}
For $\sfrak=0,1,2$, there exists a unique smooth $\Pgell$ to the following equation
\begin{align}
\label{eq:totalderieq:generall:nega-1}
\prb\left(r^{2\ell+2}\mu^{1+\sfrak}\Phell^2 \prb(\mu^{-\sfrak}\Phell^{-1} r^{-(\ell+\sfrak)}\Ppsiell)\right)
 ={}& r^{\ell-\sfrak-1}\Hpsi
\end{align}
 which satisfies both
\begin{align}\label{assump:gell:1}
\lim\limits_{\rho\to\infty}r^{-\ell-\sfrak}\Pgell\big|_{\Sigma_{\tau_0}}=0
\end{align}
and
\begin{align}
\label{eq:Lxipgell=ppsiell}
\Lxi \Pgell=\Ppsiell.
\end{align}
Furthermore, we have, on $\Sigma_{\tau_0}$,
\begin{align}\label{eq:s:pgell}
\Pgell(\rho)=
-\mu^{\sfrak}\Phell \rho^{\ell+\sfrak}\int_{\rho}^{+\infty}\bigg\{r^{-2\ell-2}\Phell^{-2}\mu^{-1-\sfrak}
\bigg[\csf+
\int_{2M}^{\rho_1}r^{\ell-\sfrak-1}\Hpsi(\tb_0, \rho_2)\di \rho_2\bigg]\bigg\}\di\rho_1,
\end{align}
where the constants $\csf$ are
\begin{subequations}
\label{eq:csf:012}
\begin{align}
\cfz={}&0,\\
\cfo={}&\frac{(2M)^{\ell-1}\Hpsi\big|_{\rb=2M}}{\ell(\ell+1)},\\
\cft={}&\frac{2(2M)^{\ell-2}}{3(\ell+2)(\ell+1)\ell(\ell-1)}\Big((4\ell^2 +\ell+3)\Hpsi\big|_{\rb=2M}-6M \prb\Hpsi\big|_{\rb=2M}\Big).
\end{align}
\end{subequations}
\end{lemma}

\begin{definition}[Time integral]
\label{def:TimeInt}
The unique scalar $\Pgell$ constructed from $\Ppsiell$ as in above lemma is called the \textbf{time integral} of $\Ppsiell$, that is, the time integral of the $(m,\ell)$ mode of the spin $+\sfrak$ component. 
\end{definition}

\begin{remark}
\label{rem:gplussol}
By equations \eqref{eq:totalderieq:generall:nega-1} and \eqref{eq:Lxipgell=ppsiell}, $\Pgell$ satisfies the same TME as the one of $\Ppsiell$, that is, $\Pgell Y_{m,\ell}^{+\sfrak} (\cos\theta)e^{im\pb}$ is also a solution to the TME \eqref{eq:TME} for $s=+\sfrak$.
\end{remark}

\begin{proof}
It suffices to uniquely determine $\Pgell$ on $\Sigma_{\tau_0}$, which means, we need to solve $\Pgell$ by the following equation on $\Sigma_{\tau_0}$
\begin{align}\label{eq:Pgell:tau0}
\prb\left(r^{2\ell+2}\mu^{1+\sfrak}\Phell^2 \prb(\Pthell^{-1}\Pgell)\right)
 ={}& r^{\ell-\sfrak-1}\Hpsi
\end{align}
with the asymptotic condition \eqref{assump:gell:1} and the smoothness up to $\rho=2M$, where $\Pthell=\mu^{\sfrak}\Phell r^{\ell+\sfrak}$.
By integrating \eqref{eq:Pgell:tau0} from $\rho_1$ to $\rho_2$ where $2M<\rho_1<\rho_2$, we get
\begin{align}
r^{2\ell+2}\mu^{1+\sfrak}\Phell^2 \prb(\Pthell^{-1}\Pgell)(\rho_2)
-r^{2\ell+2}\mu^{1+\sfrak}\Phell^2 \prb(\Pthell^{-1}\Pgell)(\rho_1)
=\int_{\rho_1}^{\rho_2}r^{\ell-\sfrak-1}\Hpsi(\rho)\di\rho.
\end{align}
Taking $\rho_1\to2M$, and by the fact that $\Hpsi$ is smooth away from $\rho=2M$ (one can easy check this by the definition of $\Hpsi$ in \eqref{de:Hpsi} and the smoothness of  $\Ppsiell$ on $\Sigmazero$),  we have for any $\rho>2M$,
\begin{align}\label{first:order:eq:g+}
r^{2\ell+2}\mu^{1+\sfrak}\Phell^2 \prb(\Pthell^{-1}\Pgell)
-\int_{2M}^{\rho}\big(r^{\ell-\sfrak-1}\Hpsi\big)(\rho_1)\di \rho_1=\csf,
\end{align}
for some constant $\csf$ to be determined. We integrate \eqref{first:order:eq:g+} again and obtain for $2M<\rho_1<\rho_2$,
\begin{align}\label{eq:Pgell}
\Pthell^{-1}\Pgell(\rho_2)-\Pthell^{-1}\Pgell(\rho_1)=
\int_{\rho_1}^{\rho_2}r^{-2\ell-2}\mu^{-1-\sfrak}\Phell^{-2}\Big\{
\csf+\int_{2M}^{\rho}\big(r^{\ell-\sfrak-1}\Hpsi\big)(\rho_3)\di \rho_3\Big\}\di \rho.
\end{align}
Thus, one easily finds that $\Pgell$ is smooth in $(2M,+\infty)$.

We shall now determine the values of the constants $\csf$ such that $\Pgell$ is smooth up to and including $\rho=2M$.

First, consider $\sfrak=0$ case. By requiring $\Pgell$ to be $C^1$ at $\rho=2M$, one has from \eqref{first:order:eq:g+} that
\begin{align}
\csf=\lim\limits_{\rho\to2M} r^{2\ell+2}\mu^{1+\sfrak}\Phell^2 \prb(\Pthell^{-1}\Pgell)=0.
\end{align}
Then, by \eqref{eq:Pgell} and since $\mu^{-1}\int_{2M}^{\rho}\big(r^{\ell-\sfrak-1}\Hpsi\big)(\rho_1)\di\rho_1$ is smooth at $\rho=2M$,
$\Pgell$ can be smoothly extended to $\rho=2M$. In \eqref{eq:Pgell}, taking  $\rho_1\to 2M$, $\rho_2\to+\infty$, and by
 \eqref{assump:gell:1}, we have
\begin{align}\label{Pgell:2m:0}
\lim\limits_{\rho_1\to2M}\big(\Pthell^{-1}\Pgell\big)(\rho_1)=
-\int_{2M}^{+\infty}\Big\{r^{-2\ell-2}\Phell^{-2}\mu^{-1}
\int_{2M}^{\rho}\big(r^{\ell-1}\Hpsi\big)(\rho_3)\di \rho_3\Big\}\di\rho,
\end{align}
where the RHS of \eqref{Pgell:2m:0} is integrable. Together with \eqref{eq:Pgell} and \eqref{Pgell:2m:0}, we get
\begin{align}
\Pthell^{-1}\Pgell(\rho)=
-\int_{\rho}^{+\infty}\Big\{r^{-2\ell-2}\Phell^{-2}\mu^{-1}
\int_{2M}^{\rho_1}\big(r^{\ell-1}\Hpsi\big)(\rho_2)\di \rho_2\Big\}\di\rho_1.
\end{align}
This is \eqref{eq:s:pgell} in the case $\sfrak=0$.

Second, we show that $\Pgell$ can be continuously extended to $2M$ with any $\csf$ for $\sfrak=1,2$. In fact, for $\sfrak=1,2$,
\begin{align}\label{lim:gell:0}\begin{split}
\lim\limits_{\rho\to2M}\Pgell
=&\lim\limits_{\rho\to2M}\frac{\Pthell^{-1}\Pgell}{\Pthell^{-1}}\\
=&\lim\limits_{\rho\to2M}\frac{\partial_\rho(\Pthell^{-1}\Pgell)}{-\Pthell^{-2}
\partial_\rho \Pthell}\\
=&\lim\limits_{\rho\to2M}-\frac{r^{2(\sfrak-1)}\mu^{\sfrak-1}(\csf+\int_{2M}^\rho \big(r^{\ell-\sfrak-1}\Hpsi\big)(\rho_1)\di\rho_1)}{\partial_\rho \hellt}\\
=&-\csf\sfrak^{-1}(\Phell(2M))^{-1}(2M)^{\sfrak-\ell-1}:=\gsf{0},
\end{split}\end{align}
where we have used the L'H\^{o}pital's rule in the second step and \eqref{first:order:eq:g+} in the third step.
Further, from \eqref{first:order:eq:g+},
 we have, for $\rho_2>\rho_1>2M$,
\begin{align}\label{eq:zero:order:g+}\begin{split}
\prb\big(\Pthell^{-1}(\Pgell-\gsf{0})\big)=r^{-2\ell-2}\Phell^{-2}\mu^{-1-\sfrak}W_1(\rho),
\end{split}\end{align}
where
\begin{align}\label{behav:eq:gell:2m:1}\begin{split}
W_1(\rho)
\triangleq{}&\csf+\gsf{0}r^{2-2\sfrak}
\mu^{1-\sfrak}\partial_\rho(\Pthell)
+\int_{2M}^{\rho}r^{\ell-\sfrak-1}\Hpsi(\rho_1)\di \rho_1.
\end{split}\end{align}

We now determine the values of $\csf$ such that $\Pgell$ can be smoothly extended to $2M$ for $\sfrak=1$. For $\rho_2>\rho_1>2M$, we have by integrating \eqref{eq:zero:order:g+} that
\begin{align}\label{eq:zero:order:g+:11}\begin{split}
\big(\Pthell^{-1}(\Pgell-\gsf{0})\big)(\rho_1)=\big(\Pthell^{-1}(\Pgell-\gsf{0})\big)(\rho_2)
-\int_{\rho_1}^{\rho_2}r^{-2\ell-2}\Phell^{-2}\mu^{-2}W_1(\rho)\di\rho.
\end{split}\end{align}
From \eqref{behav:eq:gell:2m:1},
$
W_1(\rho)
=\big(\gsf{0}\partial_\rho^2(\Pthell)(2M)+(2M)^{\ell-2}\Hpsi(2M)\big)\cdot(\rho-2M)
+O((\rho-2M)^2)$ as $\rho\to2M$. Further, notice from equation \eqref{eq:totalderieq:generall:nega:2} solved by $\hell$ that $\prb^2 \Pthell=\ell(\ell+1) r^{\ell-1}\hell$.
Hence, by taking $\gsf{0}=-\frac{(2M)^{\ell-2}\Hpsi(2M)}{\partial_\rho^2(\Pthell)(2M)}=-\frac{\Hpsi(2M)}{2\ell(\ell+1) M \hell(2M)}$, or equivalently, from \eqref{lim:gell:0}, taking
\begin{align}
\csf=\frac{(2M)^{\ell-1}\Hpsi(2M)}{\ell(\ell+1)},
\end{align}
then $\mu^{-2}W_1$ is smooth in $[2M,+\infty)$, which thus yields $\Pgell$ is smooth at $2M$ from \eqref{eq:zero:order:g+}.
 Furthermore, By \eqref{eq:zero:order:g+} and the assumption \eqref{assump:gell:1}, we have
\begin{align}\begin{split}
\lim\limits_{\rho\to2M}\Pthell^{-1}(\Pgell-\gsf{0})
=-\int_{2M}^{+\infty}r^{-2\ell-2}\Phell^{-2}\mu^{-2}W_1\di\rho.
\end{split}\end{align}
Using \eqref{eq:zero:order:g+:11} again, we have
\begin{align}\label{eq:zero:order:gsell}\begin{split}
\Pthell^{-1}(\Pgell-\gsf{0})
=-\int_{\rho}^{+\infty}\big(r^{-2\ell-2}\Phell^{-2}\mu^{-2}W_1\big)(\rho_1)\di\rho_1.
\end{split}\end{align}
This corresponds to \eqref{eq:s:pgell} in the case $\sfrak=1$.

Finally, we discuss the case $\sfrak=2$. Similarly as proving  the continuity of $\Pgell$ at $2M$ in the case $\sfrak=1$ via equation \eqref{lim:gell:0} , we can show that $\Pgell$ is actually $C^1$ at $2M$ for any constant $\csf$. In fact, for $\sfrak=2$,
\begin{align}
\label{lim:gell:1}\begin{split}
\lim\limits_{\rho\to2M}\frac{\Pgell(\rho)-\gsf{0}}{\rho-2M}&=
\lim\limits_{\rho\to2M}\frac{\Pthell^{-1}(\Pgell-\gsf{0})}{\mu r\Pthell^{-1}}\\
&=\lim\limits_{\rho\to2M}\frac{\partial_\rho(\Pthell^{-1}(\Pgell-\gsf{0}))}
{\partial_\rho(\mu r\Pthell^{-1})}\\
&=\lim\limits_{\rho\to2M}\frac{r^{-2\ell-2}\mu^{-1-\sfrak}\Phell^{-2}W_1(\rho)}
{\partial_\rho(\mu r\Pthell^{-1})}\\
&=\frac{r^{\ell-\sfrak-1}\Hpsi+\gsf{0}\partial_\rho( r^{-2}\mu^{-1}\partial_\rho(\hellt))}{-2Mr^{\ell-2}\hell}\Big|_{\rho=2M}
:=\gsf{1},
\end{split}\end{align}
where we have used the L'H\^{o}pital's rule in the second step and \eqref{eq:zero:order:g+} in the third step.  Furthermore, by \eqref{first:order:eq:g+}, we have, for $\rho_2>\rho_1>2M$,
\begin{align}\label{eq:gell:sfrak2}\begin{split}
\mu r\Pthell^{-1}\Big(\frac{\Pgell-\gsf{0}}{\mu r}-\gsf{1}\Big)(\rho_1)={}&
\mu r\Pthell^{-1}\Big(\frac{\Pgell-\gsf{0}}{\mu r}-\gsf{1}\Big)(\rho_2)\\
&\quad-\int_{\rho_1}^{\rho_2}r^{-2\ell-2}
\mu^{-3}\Phell^{-2}W_2(\rho)\di \rho,
\end{split}\end{align}
where
\begin{align}\begin{split}
W_2(\rho)
\triangleq{}&\csf+\gsf{0}r^{2-2\sfrak}\mu^{1-\sfrak}\partial_\rho(\Pthell)\\
&-
\gsf{1}r^{2\ell+2}\mu^{1+\sfrak}\Phell^2 \prb(\mu r\Pthell^{-1})+
\int_{2M}^{\rho}\big(r^{\ell-\sfrak-1}\Hpsi\big)(\rho_1)\di \rho_1.
\end{split}\end{align}
Using $\Pthell=\mu^{\sfrak}\Phell r^{\ell+\sfrak}$, we obtain for $\rb-2M$ small,
\begin{align}
\label{esti:W2}
W_2(\rho)={}&\csf+\gsf{0}\big(r^{-2}\mu^{-1}\partial_\rho(\Pthell)\big)(2M)+(\rb-2M)\cdot\Big[\gsf{0}\big(\prb(r^{-2}\mu^{-1}\partial_\rho(\Pthell))\big)(2M)\notag\\
&+
\gsf{1}(2M)^{\ell-1}\Phell(2M)+\big(r^{\ell-3}\Hpsi\big)(2M)\Big]\notag \\
&+\frac{1}{2}(\rb-2M)^2 \cdot \Big[\gsf{0}\big(\prb^2(r^{-2}\mu^{-1}\partial_\rho(\Pthell))\big)(2M)
+\big(\prb(r^{\ell-3}\Hpsi)\big)(2M)\notag\\
&
-2\gsf{1}\big(\prb(r^{2\ell+1}\mu^{2}\Phell^2 \prb(\mu r\Pthell^{-1}))\big)(2M)
\Big]+O((\rb-2M)^3).
\end{align}
We utilize then equations \eqref{lim:gell:1} and \eqref{lim:gell:0} to find the first two lines on the RHS of the above equation are vanishing identically.  By requiring the $(\rb-2M)^2$ terms vanish, one then finds that $\mu^{-3}W_2$ is smooth in $[2M,+\infty)$, and, from \eqref{eq:gell:sfrak2}, $\Pgell$ is thus smooth in $[2M,+\infty)$. This requirement is equivalent to setting the coefficient of $\frac{1}{2}(\rb-2M)^2$ term on the RHS of \eqref{esti:W2} to be zero, and, plugging in the expression \eqref{lim:gell:1} of $\gsf{1}$, it reduces to imposing the following condition
\begin{align*}
\hspace{4ex}&\hspace{-4ex}
\gsf{0}\big(\prb^2(r^{-2}\mu^{-1}\partial_\rho(\Pthell))\big)(2M)
+\big(\prb(r^{\ell-3}\Hpsi)\big)(2M)\notag\\
={}&2\gsf{1}\big(\prb(r^{\ell-1}\Phell-\mu^{-1}r^{-2}\prb \Pthell)\big)(2M)\notag\\
={}&2\big(\prb(r^{\ell-1}\Phell-\mu^{-1}r^{-2}\prb \Pthell)\big)(2M)\times\frac{r^{\ell-\sfrak-1}\Hpsi+\gsf{0}\partial_\rho( r^{-2}\mu^{-1}\partial_\rho(\hellt))}{-2Mr^{\ell-2}\hell}\Big|_{\rho=2M}.
\end{align*}
We use equation \eqref{eq:totalderieq:generall:nega:2} solved by $\hell$, which in particular implies $\prb(\mu^{-1}r^{-2}\prb \Pthell)=(\ell+2)(\ell-1)r^{\ell-2} \hell$ and $6M(\prb \hell) (2M)=\ell(\ell-2) \hell (2M)$, to solve
this equation,  and then use \eqref{lim:gell:0} to obtain the values of $\gsf{0}$ and $\csf$:
\begin{align}
\gsf{0}={}&-\frac{1}{(\ell+2)(\ell-1)}\bigg(\frac{r^{2}\hell\prb(r^{-1} \Hpsi)+((\ell-2\ell^2)\hell+2r\prb\hell)\Hpsi}{\hell\big[3r^{2}\prb \hell+(\ell-2\ell^2)\hell r\big]}\bigg)\bigg|_{\rb=2M},\notag\\
\csf={}&\frac{2(2M)^{\ell-1}}{(\ell+2)(\ell-1)}\bigg(\frac{r^{2}\hell\prb(r^{-1} \Hpsi)+((\ell-2\ell^2)\hell+2r\prb\hell)\Hpsi}{3r^{2}\prb \hell+(\ell-2\ell^2)\hell r}\bigg)\bigg|_{\rb=2M}\notag\\
={}&\frac{2(2M)^{\ell-2}}{3(\ell+2)(\ell+1)\ell(\ell-1)}\big((4\ell^2 +\ell+3)\Hpsi(2M)-6M (\prb\Hpsi)(2M)\big).
\end{align}
By \eqref{assump:gell:1}, we have
\begin{align}
\lim\limits_{\rho\to2M}\mu r\Pthell^{-1}\Big(\frac{\Pgell-\gsf{0}}{\mu r}-\gsf{1}\Big)(\rho)=
-\int_{2M}^{+\infty}r^{-2\ell-2}
\mu^{-3}\Phell^{-2}W_2(\rho)\di \rho,
\end{align}
and then,
\begin{align}
\Pthell^{-1}(\Pgell-\gsf{0}-\mu r\gsf{1})(\rho)=
-\int_{\rho}^{+\infty}\big(r^{-2\ell-2}
\mu^{-3}\Phell^{-2}W_2\big)(\rho_1)\di \rho_1,
\end{align}
which is equality \eqref{eq:s:pgell} for $\sfrak=2$.
\end{proof}

As an immediate consequence, we compute the derivatives of the time integral $\Pgell$ in terms of $\Ppsiell$.

\begin{lemma}[Derivatives of the  time integral $\Pgell$]
\label{lem:asfl:values}
Assume $(r^2\partial_\rho)^{i'}(\mu^{-\sfrak}r\Ppsiell)$, $0\leq i'\leq \ell-\sfrak+1$, are bounded near infinity.  Then for any $0\leq i\leq \ell-\sfrak$,
\begin{align}
\label{eq:higherradialderiofgplus:general}
\hspace{4ex}&\hspace{-4ex}
(r^2\partial_\rho)^i(\mu^{-\sfrak}r\Pgell)\notag\\
=&-\sum_{j=0}^{i}\sum_{n=0}^{j}C_i^j C_j^n[(r^2\prb)^{i-j}\hell]
\rho^{\ell+\sfrak+1+j}
\int_{\rho}^{\infty}\bigg(\big[(r^2\prb)^{j-n}(\hell^{-2}
\mu^{-1-\sfrak})\big]r^{-2\ell-j+n-2}\notag\\
&\qquad \qquad\times\bigg\{\asfl{n}+\int_{2M}^r\big(r^{\ell-\sfrak-n-1}(r^2\partial_\rho)^n\Hpsi\big)(\rb_1)\di \rb_1\bigg\}\bigg)(\rb_2)\di \rb_2\notag\\
&-\csf(r^2\prb)^i\Big\{\hell \rho^{\ell+\sfrak+1}\int_{\rho}^{\infty}r^{-2\ell-2}\hell^{-2}\mu^{-1-\sfrak}\di r\Big\},
\end{align}
where $\asfl{0}=0$, $\asfl{1}=(2M)^{\ell-\sfrak}\Hpsi(2M)$, and for $2\leq n\leq \ell-\sfrak$,
\begin{align}\begin{split}
\asfl{n}
={}&(2M)^{\ell-\sfrak-n+1}(r^2\partial_\rho)^{n-1}\Hpsi(2M)\\
&+\sum_{j=0}^{n-2}(-1)^{n-j-1}\frac{(\ell-\sfrak-(j+1))!}
{(\ell-\sfrak-n)!}(2M)^{\ell-\sfrak-j}(r^2\partial_\rho)^j
\Hpsi(2M).
\end{split}\end{align}
\end{lemma}
\begin{proof}
Let $E(\rho)
=\Phell^{-2}\mu^{-1-\sfrak}\rho^{-\ell+\sfrak}
\int_{2M}^{\rho}\big(r^{\ell-\sfrak-1}\Hpsi\big)(\rho_1)\di \rho_1$,
then $E(\rho)$ is bounded near infinity by assumption.  By this definition  and  equation \eqref{eq:s:pgell}, we have
\begin{align}
\label{eq:radPgell}
\mu^{-\sfrak}r\Pgell=&-\Phell \rho^{\ell+\sfrak+1}\int_{\rho}^{\infty}\big(r^{-\ell-\sfrak-2}E\big)(\rb_1)\di \rb_1\notag\\
&-\csf \Phell \rho^{\ell+\sfrak+1}\int_{\rho}^{\infty}\big(r^{-2\ell-2}\Phell^{-2}\mu^{-1-\sfrak}\big)(\rb_1)\di \rb_1.
\end{align}
This proves the $i=0$ case of \eqref{eq:higherradialderiofgplus:general}.
One can perform integration by parts and obtain
\begin{align}\begin{split}
\hspace{4ex}&\hspace{-4ex}r^2\partial_\rho\bigg(\rho^{-\ell+\sfrak}\int_{2M}^{\rho}\big(r^{\ell-\sfrak-1}
\Hpsi\big)(\rb_1)\di \rb_1\bigg)\\
={}&(2M)^{\ell-s}\Hpsi(2M)\rho^{-\ell+\sfrak+1}
+\rho^{-\ell+\sfrak+1}\int_{2M}^{\rho}\big(r^{\ell-\sfrak-2}r^2\partial_\rho\Hpsi\big)(\rb_1)\di \rb_1.
\end{split}\end{align}
A simple induction then yields that for any $j\in \mathbb{N}^+$,
\begin{align}\label{Def:r2par:1}\begin{split}
\hspace{4ex}&\hspace{-4ex}(r^2\partial_\rho)^j\bigg\{\rho^{-\ell+\sfrak}\int_{2M}^{\rho}\big(r^{\ell-\sfrak-1}
\Hpsi\big)(\rb_1)\di \rb_1\bigg\}\\
={}&\rho^{-\ell+\sfrak+j}\bigg\{\asfl{j}+
\int_{2M}^{\rho}\big(r^{\ell-\sfrak-j-1}(r^2\partial_\rho)^j\Hpsi\big)(\rb_1)\di \rb_1\bigg\}.
\end{split}\end{align}
 Further, we have for any $j\leq \ell-\sfrak$,
\begin{align}\label{Def:r2par:2}
(r^2\partial_\rho)^j\bigg\{\rho^{\ell+\sfrak+1}\int_\rho^\infty \big(r^{-\ell-\sfrak-2}E\big)(\rho_1)\di \rho_1\bigg\}=
\rho^{\ell+\sfrak+1+j}\int_{\rho}^{\infty} \big(r^{-(\ell+\sfrak+2+j)}(r^2\partial_\rho)^jE\big)(\rho_1)\di \rho_1,
\end{align}
where we have used integration by parts and  $\lim_{\rb\to \infty}r^{-\ell-\sfrak-1+j}(r^2\partial_\rho)^jE(\rho)=0$ which yields the boundary term at $\infty$ vanishes.
By taking $(r^2\partial_\rho)^i$ derivative over equation \eqref{eq:radPgell} and utilizing equations  \eqref{Def:r2par:1} and \eqref{Def:r2par:2}, \eqref{eq:higherradialderiofgplus:general} then follows.
\end{proof}

\subsection{Newman--Penrose constant of the time integral $\gplus$}
\label{subsect:NPIT:esti}

Consider a fixed $(m,\ell)$ mode of the spin $+\sfrak$ component. In this subsection, we compute the N--P constant for the time integral $\gplus$, which is fundamental in characterizing the precise late-time asymptotics for the spin $\pm \sfrak$ components. 

Recall from Remark \ref{rem:gplussol} that $\gplus Y_{m,\ell}^{+\sfrak} (\cos\theta)e^{im\pb}$ is also a solution to the TME, hence $\PhiplussHighTI{0}=\mu^{-\sfrak}r\gplus$ is a solution to \eqref{eq:Phi+1:Schw:generall}. One can then similarly define $\PhiplussHighTI{i}$ and $\tildePhiplussHighTI{i}$ as in Definition \ref{def:tildePhiplusandminusHigh}. As a result, the scalars $\PhiplussHighTI{i}$ and $\tildePhiplussHighTI{i}$ solve equations \eqref{eq:Phiplushighi:Schw:generall} and \eqref{eq:Phiplushighi:Schw:generall:tildePhiplusi}, respectively, that is,
\begin{subequations}
\begin{align}
\label{eq:PhiplussHighTI:wave}
&-\mu Y \curlVR  \PhiplussHighTI{i} +(\edthR'\edthR+f_{\sfrak, i,1})\PhiplussHighTI{i}
\notag\\
&
\qquad \quad
+{f_{\sfrak, i,2}(r-3M)r^{-2}}\curlVR\PhiplussHighTI{i}-6f_{\sfrak, i,3}Mr^{-1}\PhiplussHighTI{i}
+g_{\sfrak, i} M\PhiplussHighTI{i-1}={}0,\\
\label{eq:tildePhiplussHighTI:wave}
&-\mu Y \curlVR  \tildePhiplussHighTI{i} +(\edthR'\edthR+f_{\sfrak, i,1})\tildePhiplussHighTI{i}
+{f_{\sfrak, i,2}(r-3M)r^{-2}}\curlVR\tildePhiplussHighTI{i}
+\sum_{j=0}^{i}h_{\sfrak, i,j} \PhiplussHighTI{j}={}0.
\end{align}
\end{subequations}

The following lemma calculates the N--P constant of the time integral $\gplus$.

\begin{lemma}[The value of the N--P constant of the time integral $\gplus$]
\label{lem:higherderiofgplus:general}
Let $\asfl{\ell-\sfrak}$,  $\csf$,   and $\Hpsi$ be defined as in Lemma \ref{lem:asfl:values}, equation \eqref{eq:csf:012}, and equation \eqref{de:Hpsi}, respectively.  Assume on $\Sigmazero$ that the limits $\lim\limits_{\rb\to \infty}r \curlVR \tildePhiplussHigh{\ell-\sfrak}$ and $\lim\limits_{\rb\to\infty}(r^2\partial_\rho)^{i'}(\mu^{-\sfrak}r\Ppsiell)$, $0\leq i'\leq \ell-\sfrak+1$, exist and are finite. Then, the limits $\big\{\lim\limits_{\rb\to \infty}(r^2\prb)^i\PhiplussHighTI{0}\big\}_{i=0,1,\ldots,\ell-\sfrak}$ and $\big\{\lim\limits_{\rb\to \infty}\PhiplussHighTI{i}\big\}_{i=0,1,\ldots,\ell-\sfrak}$ exist, are finite and are given by
\begin{subequations}
\label{eq:limitofPhiplussHighTIi:exp}
\begin{align}
\label{eq:limitofPhiplussHighTIi:exp:low}
\hspace{4ex}&\hspace{-4ex}
\lim_{\rb\to \infty}(r^2\prb)^i\PhiplussHighTI{0}\notag\\
={}&\sum_{j=0}^{i}\sum_{n=0}^{j}C_i^j C_j^m (-1)^{i-j+1} (i-j)! M^{i-j}\hellh{i-j}(\ell+\sfrak+j+1)^{-1}(2\ell+j-n+1)^{-1}\notag\\
&\times \lim_{\rb\to \infty}\big[(r^2\partial_\rho)^{j-n}(h_{\sfrak,\ell}^{-2}
\mu^{-1-s})\big]\lim_{\rb\to \infty} (r^2\prb)^{n}\Hpsi, \quad \text{for } 0\leq i\leq \ell-\sfrak-1,\\
\label{eq:limitofPhiplussHighTIi:exp:low:ell-sfrak}
\hspace{4ex}&\hspace{-4ex}
\lim_{\rb\to \infty}(r^2\prb)^{\ell-\sfrak}\PhiplussHighTI{0}\notag\\
={}&\sum_{j=0}^{i}\sum_{n=0}^{\min(j,\ell-\sfrak-1)}C_i^j C_j^n (-1)^{i-j+1} (i-j)! M^{i-j}\hellh{i-j}(\ell+\sfrak+j+1)^{-1}(2\ell+j-n+1)^{-1}\notag\\
&\times \lim_{\rb\to \infty}\big[(r^2\partial_\rho)^{j-n}(h_{\sfrak,\ell}^{-2}
\mu^{-1-s})\big]\lim_{\rb\to \infty} (r^2\prb)^{n}\Hpsi
- (-1)^{\ell-\sfrak}\frac{\csf}{2\ell+1}(\ell-\sfrak)!\notag\\
&
- \frac{1}{2\ell+1}\asfl{\ell-\sfrak}
-\frac{1}{2\ell+1}\int_{2M}^{\infty}\rb^{-1}(r^2\prb)^{\ell-\sfrak}\Hpsi\di \rb,
\end{align}
\end{subequations}
and the limits $ \lim\limits_{\rb\to \infty}\PhiplussHighTI{j}$ with $j\in \{0,1,\ldots,\ell-\sfrak\}$ are computed via \eqref{eq:limitofPhiplussHighTIi:exp} and
\begin{align}
\label{eq:PhiplussHighTI:r2prbPhiplussHighTI}
\PhiplussHighTI{i}=\curlVR^i\PhiplussHighTI{0}=(r^2\partial_\rho)^i\PhiplussHighTI{0}+\mathcal{E}^i[\Phiplussell],
\end{align}
where $\mathcal{E}^0[\Phiplussell]=0$ and $\mathcal{E}^i[\Phiplussell]=r^2\Hhyp (r^2\partial_\rho)^{i-1}\Phiplussell
+\curlVR\mathcal{E}^{i-1}[\Phiplussell]$.

\underline{The N--P constant of the time integral $\gplus$, denoted by $\NPCITs{\ell}$, equals}
\begin{align}
\label{eq:N--Ptimeintegral:exp}
\NPCITs{\ell}={}&\frac{1}{2\ell+2}\sum_{j=0}^{\ell-\sfrak}\Big(
\lim_{\rb\to \infty}(rh_{\sfrak, \ell-\sfrak,j} )\times \lim_{\rb\to \infty}\PhiplussHighTI{j}\Big)-\frac{2}{2\ell+2} \lim_{\rb\to \infty}r \curlVR \tildePhiplussHigh{\ell-\sfrak},
\end{align}
where the limits $\lim\limits_{\rb\to \infty}(rh_{\sfrak, \ell-\sfrak,j} )$ are calculated from Proposition \ref{prop:wave:Phihigh:pm1}.
\end{lemma}

\begin{proof}
One can easily see from \eqref{eq:higherradialderiofgplus:general} that for $i\leq \ell-\sfrak-1$, the limit
$\lim_{\rb\to \infty}(r^2\prb)^i\PhiplussHighTI{0}$ exists and is finite. For $i=\ell-\sfrak$, the limit of the RHS exists by direct inspection except for one sub-term with $n=k=i=\ell-\sfrak$:
\begin{align}
\label{eq:jfkkdjfsoij}
\lim_{\rb\to \infty}-h_{\sfrak,\ell}
\rho^{2\ell+1}\int_{\rho}^{\infty}\bigg(h_{\sfrak,\ell}^{-2}
\mu^{-1-s}r^{-2l-2}\int_{2M}^r\big(r^{-1}(r^2\partial_\rho)^{\ell-\sfrak}\Hpsi\big)(r')\di r'\bigg)(\rb_1)\di  \rb_1.
\end{align}
By expanding $(r^2\partial_\rho)^{\ell-\sfrak}\Hpsi=e_{\ell-\sfrak+1}(r)\curlVR\tildePhiplussHigh{\ell-\sfrak}+\sum_{j=0}^{\ell-\sfrak}e_j(r) \tildePhiplussHigh{j}$, where we have utilized the expression \eqref{def:vectorVRintermsofprb} and equation \eqref{eq:Phiplushighi:Schw:generall} to rewrite $\Lxi \tildePhiplussHigh{j}$, it is manifest that all $\{e_j\}_{j=0,\ldots, \ell-\sfrak+1}$ are $O(1)$ and comparing the coefficient of the highest order derivative term gives $e_{\ell-\sfrak+1}=2+O(r^{-1})$, thus the integral from $2M$ to $r$ may lead to a $\log r$ growth and the limit \eqref{eq:jfkkdjfsoij} goes to infinity. However, a direct, but tedious, computation shows that  $e_{\ell-\sfrak+1}=O(1)$ and $\{e_j\}_{j=0,\ldots, \ell-\sfrak+1}$ are $O(r^{-1})$. Hence,
\begin{align}(r^2\partial_\rho)^{\ell-\sfrak}\Hpsi=&(2r^{-1}+O(r^{-2}))r\curlVR\tildePhiplussHigh{\ell-\sfrak}
+\sum_{j=0}^{\ell-\sfrak}O(r^{-1})\tildePhiplussHigh{j}\notag\\ 
=&(2r^{-1}+O(r^{-2}))r\curlVR\tildePhiplussHigh{\ell-\sfrak}
+\sum_{j=0}^{\ell-\sfrak}O(r^{-1})(r^2\prb)^j\PhiplussHigh{0},
\end{align}
 where we have used in the second step the relation $\tildePhiplussHigh{j}=\sum_{n=0}^jO(1)(r^2\prb)^n\PhiplussHigh{0}$ which in turn follows by the definition of  $\tildePhiplussHigh{j}$ and an application of equation \eqref{eq:Phiplushighi:Schw:generall}  to rewrite $\Lxi\PhiplussHigh{n}$, and this then guarantees the existence and finiteness of the limit \eqref{eq:jfkkdjfsoij}.  Meanwhile, this also yields that the limits $\{\lim_{\rb\to \infty}\tildePhiplussHigh{j}\}_{j=0,1,\ldots,\ell-\sfrak}$ exist.

A different way of proving the finiteness of the limit $\lim_{\rb\to \infty}\PhiplussHighTI{\ell-\sfrak}$ is to start with equation \eqref{eq:PhiplussHighTI:wave} with $i=\ell-\sfrak-1$, which is
\begin{align}
&-\mu Y \PhiplussHighTI{\ell-\sfrak} -2\ell\PhiplussHighTI{\ell-\sfrak-1}
-2\ell(r-3M)r^{-2}\PhiplussHighTI{\ell-\sfrak}\notag\\
&\qquad\qquad-6f_{\sfrak, \ell-\sfrak-1,3}Mr^{-1}\PhiplussHighTI{\ell-\sfrak-1}
+g_{\sfrak,\ell-\sfrak-1} M\PhiplussHighTI{\ell-\sfrak-2}={}0.
\end{align}
Since $-\mu Y= -(2-\mu H)\Lxi+\mu \prb$ and $\Lxi\PhiplussHighTI{i}=\PhiplussHigh{i}$, this equation can be written as
\begin{align}
\label{eq:PhiplussHighTIell-sfrak:deri}
\hspace{4ex}&\hspace{-4ex} \prb \Big((\mu r^{-2})^{\ell}\PhiplussHighTI{\ell-\sfrak} \Big)\notag\\
={}&\mu^{\ell-1}r^{-2\ell}\Big((2-\mu H)\PhiplussHigh{\ell-\sfrak}+2\ell\PhiplussHighTI{\ell-\sfrak-1}
-2(\ell-\sfrak-1)(\ell-1)(\ell+\sfrak-1)M\PhiplussHighTI{\ell-\sfrak-2}\Big)\notag\\
&+6f_{\sfrak, \ell-\sfrak-1,3}M\mu^{\ell-1}r^{-2\ell-1}\PhiplussHighTI{\ell-\sfrak-1}.
\end{align}
By \eqref{def:vectorVRintermsofprb} and \eqref{eq:propertyofHfunction}, we have, for any $j\geq1$,
\begin{align}
\curlVR^j=(r^2\partial_\rho+r^2\Hhyp\partial_\tau)^j=\sum_{p+q\leq j}(a_{p,q}^j+O(r^{-1}))
(r^2\partial_\rho)^p\partial_\tau^q\qquad \text{as}\ r\to\infty,
\end{align}
and, for any $i\leq \ell-\sfrak-1$ and $n\geq0$, $(r^2\prb)^i\partial_\tau^n\PhiplussHighTI{0}$ are equal to $b_{i,n}+O(r^{-1})$ as $r\to\infty$, where $b_{i,n}$ are finite constants. Hence, the RHS of the first line of \eqref{eq:PhiplussHighTIell-sfrak:deri} equals $r^{-2\ell}(c+O(r^{-1}))$ for some constant $c$ for $r$ large. We now prove $c=0$. In fact, the LHS of \eqref{eq:PhiplussHighTIell-sfrak:deri} equals
\begin{align*}
&2\ell\rho^{-2\ell-1}\cdot\rho^{2\ell+1}\int_{\rho}^{\infty}\Big(h_{\sfrak,\ell}^{-2}
\mu^{-1-s}r^{-2l-2}\int_{2M}^r\big(r^{-1}(r^2\partial_\rho)^{\ell-\sfrak}\Hpsi\big)(r')\di r'\Big)(\rb_1)\di  \rb_1\notag\\
&-\rho^{-2\ell}\partial_\rho\Big(\rho^{2\ell+1}\int_{\rho}^{\infty}\Big(h_{\sfrak,\ell}^{-2}
\mu^{-1-s}r^{-2l-2}\int_{2M}^r\big(r^{-1}(r^2\partial_\rho)^{\ell-\sfrak}\Hpsi \big)(r')\di r'\Big)(\rb_1)\di  \rb_1\Big)+O(\rho^{-2\ell-1}).
\end{align*}
This yields $\lim\limits_{\rho\to\infty} \rho^{2\ell}\times[\text{left-hand side of
\eqref{eq:PhiplussHighTIell-sfrak:deri}}]=0$, hence, by multiplying \eqref{eq:PhiplussHighTIell-sfrak:deri} by $r^{2\ell}$ and taking $r\to\infty$, we get $c=0$. That is, $ \text{right-hand side of
\eqref{eq:PhiplussHighTIell-sfrak:deri}}=O(\rb^{-2\ell-1})$.
We integrate equation \eqref{eq:PhiplussHighTIell-sfrak:deri} from $\infty$ and thus justify that the limit $\lim_{\rb\to \infty}\PhiplussHighTI{\ell-\sfrak} $ exists and is finite.

We turn to computing  the limits $\lim_{\rb\to \infty}(r^2\partial_\rho)^i\PhiplussHighTI{0}$ for $i=0,\ldots,\ell-\sfrak$. By Definition \ref{def:hell:pm} of the function $\hell$, we have $(r^2\prb)^{i-j}\hell=(-1)^{i-j} (i-j)! M^{i-j}\hellh{i-j}$. Meanwhile,
\begin{align*}
&
\lim_{\rb\to \infty}\rho^{\ell+\sfrak+1+j}\int_{\rho}^{\infty}\big[(r^2\partial_\rho)^{j-n}(\hell^{-2}
\mu^{-1-\sfrak})\big]r^{-2\ell-j+n-2}\asfl{n}\di r\notag\\
&\quad={}
\left\{
                  \begin{array}{ll}
                    0, & \hbox{if $n\leq \ell-\sfrak-1$;} \\
                    \frac{1}{2\ell+1}\asfl{\ell-\sfrak}, & \hbox{if $n=j=i=\ell-\sfrak$,}
                  \end{array}
                \right.
\end{align*}
and
\begin{align*}
&\lim_{\rb\to \infty}-\csf(r^2\prb)^i\bigg(h_{\sfrak,\ell} \rho^{\ell+\sfrak+1}\int_{\rho}^{\infty}r^{-2\ell-2}\hell^{-2}\mu^{-1-\sfrak}\di r\bigg)\\
&\quad={}
\left\{
                  \begin{array}{ll}
                    0, & \hbox{if $i\leq \ell-\sfrak-1$;} \\
                    (-1)^{\ell-\sfrak+1}\frac{\csf}{2\ell+1}(\ell-\sfrak)!, & \hbox{if $i=\ell-\sfrak$.}
                  \end{array}
                \right.
\end{align*}
Further, for $n\leq \ell-\sfrak-1$,
\begin{align*}
&\lim_{\rb_1\to \infty}(\rb_1)^{\ell+\sfrak+1+j}\int_{\rho_1}^{\infty}\big[(r^2\partial_\rho)^{j-n}(\hell^{-2}
\mu^{-1-\sfrak})\big]r^{-2\ell-j+n-2}\int_{2M}^\rb(r')^{\ell-\sfrak-n-1}((r^2\partial_\rho)^n\Hpsi)(r')\di r'\di \rb\notag\\
&={}\frac{1}{(\ell+\sfrak+j+1)(2\ell+j-n+1)}\lim_{\rb\to \infty}\big[(r^2\partial_\rho)^{j-n}(h_{\sfrak,\ell}^{-2}
\mu^{-1-\sfrak})\big]\lim_{\rb\to \infty} (r^2\partial_\rho)^{n}\Hpsi,
\end{align*}
and for $n=j=\ell-\sfrak$,
\begin{align*}
&\lim_{\rb_1\to \infty}(\rho_1)^{\ell+\sfrak+1+j}\int_{\rho_1}^{\infty}\big[(r^2\partial_\rho)^{j-n}(\hell^{-2}
\mu^{-1-\sfrak})\big]r^{-2\ell-j+n-2}\int_{2M}^\rb(r')^{\ell-\sfrak-n-1}((r^2\partial_\rho)^n\Hpsi)(r')\di r'\di \rb\notag\\
&={}\lim_{\rb_1\to \infty}(\rho_1)^{2\ell+1}\int_{\rho_1}^{\infty}\hell^{-2}
\mu^{-1-\sfrak}r^{-2\ell-2}\int_{2M}^\rb(r')^{-1}((r^2\partial_\rho)^{\ell-\sfrak}\Hpsi)(r')\di r'\di \rb\notag\\
&={}\frac{1}{2\ell+1}\int_{2M}^{\infty}\rb^{-1}(r^2\partial_\rho)^{\ell-\sfrak}\Hpsi\di \rb.
\end{align*}
In view of the above limits, equations \eqref{eq:limitofPhiplussHighTIi:exp} thus follow.

In the end, we estimate $\curlVR  \tildePhiplussHighTI{\ell-\sfrak}$, the limit of which as $\rb\to \infty$ equals the Newman--Penrose constant of the time integral $\gplus$.
We use equation \eqref{eq:tildePhiplussHighTI:wave} with $i=\ell-\sfrak$ and utilize equation \eqref{def:vectorVRintermsofprb} to write $-\mu Y= -(2-\mu \Hhyp )\Lxi+\mu \prb$, and in view of $\Lxi\tildePhiplussHighTI{i}=\tildePhiplussHigh{i}$ and the expressions of $f_{\sfrak,i,1}$ and $f_{\sfrak,i,2}$, we achieve
\begin{align*}
\mu \prb\curlVR  \tildePhiplussHighTI{\ell-\sfrak} -(2-\mu \Hhyp )\curlVR\tildePhiplussHigh{\ell-\sfrak}
-2(\ell+1)(r-3M)r^{-2}\curlVR\tildePhiplussHighTI{\ell-\sfrak}
+\sum_{j=0}^{\ell-\sfrak}h_{\sfrak, \ell-\sfrak,j} \PhiplussHighTI{j}={}0.
\end{align*}
This can be rewritten as
\begin{align}
\prb\big((\mu r^{-2})^{\ell+1}\curlVR  \tildePhiplussHighTI{\ell-\sfrak}\big)
={}&\mu^{\ell}r^{-2\ell-2}\Big((2-\mu \Hhyp )\curlVR\tildePhiplussHigh{\ell-\sfrak}-\sum_{j=0}^{\ell-\sfrak}h_{\sfrak, \ell-\sfrak,j} \PhiplussHighTI{j}\Big).
\end{align}
By integrating from $\infty$, we obtain
\begin{align}
\lim_{\rb\to \infty}\curlVR  \tildePhiplussHighTI{\ell-\sfrak}=&-\lim_{\rb'\to \infty} (\rb')^{2\ell+2}\int_{\rb'}^{\infty} \mu^{\ell}r^{-2\ell-2}\Big((2-\mu \Hhyp )\curlVR\tildePhiplussHigh{\ell-\sfrak}
-\sum_{j=0}^{\ell-\sfrak}h_{\sfrak, \ell-\sfrak,j} \PhiplussHighTI{j}\Big)\di \rb\notag\\
=&-\frac{2}{2\ell+2} \lim_{\rb\to \infty}r \curlVR \tildePhiplussHigh{\ell-\sfrak}
+\lim_{\rb'\to \infty} (\rb')^{2\ell+2}\int_{\rb'}^{\infty} \mu^{\ell}r^{-2\ell-2}\Big(
\sum_{j=0}^{\ell-\sfrak}h_{\sfrak, \ell-\sfrak,j} \PhiplussHighTI{j}\Big)\di \rb.
\end{align}
For the last term,
since the limits $\lim_{\rb\to \infty}(r^2\prb)^j\PhiplussHighTI{0}$ exist and are finite for all $j=0,1,\ldots, \ell-\sfrak$ and equation \eqref{eq:PhiplussHighTI:r2prbPhiplussHighTI} follows easily from equation \eqref{def:vectorVRintermsofprb} and Definition \ref{def:PsiPhietc}, we conclude that the limits $\lim_{\rb\to \infty}\PhiplussHighTI{j}$ exist and are finite for all $j=0,1,\ldots, \ell-\sfrak$. Hence, we obtain
\begin{align*}
\hspace{4ex}&\hspace{-4ex}\lim_{\rb'\to \infty} (\rb')^{2\ell+2}\int_{\rb'}^{\infty} \mu^{\ell}r^{-2\ell-2}\Big(
\sum_{j=0}^{i}h_{\sfrak, \ell-\sfrak,j} \PhiplussHighTI{j}\Big)\di \rb\notag\\
={}&\frac{1}{2\ell+2}\sum_{j=0}^{\ell-\sfrak}\bigg(
\lim_{\rb\to \infty}(rh_{\sfrak, \ell-\sfrak,j} )\times \lim_{\rb\to \infty}\PhiplussHighTI{j}\bigg).
\end{align*}
These yield equation \eqref{eq:N--Ptimeintegral:exp}.
\end{proof}

We will apply Theorem \ref{summary:PL:nv:pm:global} to the time integral $\gplus$ which solves the same equation as $(\psipluss)_{m,\ell}$, and to do so, we have to bound each terms in the expression of $\FBT$ for the time integral in terms of the initial data of $(\psipluss)_{m,\ell}$. This is achieved in the following lemma.

\begin{lemma}
\label{lem:initialnorm:psiandgplus}
\begin{enumerate}
\item
For any $\delta\in (0,\half)$,
\begin{align}
\label{eq:initialnorm:psiandgplus}
\InizeroEnergyplussellTI{\tb_0}{\ell}{\reg}{3-\delta}\lesssim_{\reg} \InizeroEnergyplussell{\tb_0}{\ell}{\reg+1}{5-\delta}
\end{align}
\item Assume on $\Sigmazero$ that the limit $\lim_{\rb\to \infty}r \curlVR \tildePhiplussHigh{\ell-\sfrak}$ exists and is finite and, for any $i\in \mathbb{N}$, there are constants $\beta\in (0,\half)$ and  $D_0\geq 0$ such that for all $0\leq i'\leq i+1$ and $r\geq R$,
\begin{align}
\Big|\prb^{i'}\Big(\VR\tildePhiplussHigh{\ell-\sfrak}-{\rb^{-3}\lim_{\rb\to \infty}r \curlVR \tildePhiplussHigh{\ell-\sfrak}}\Big)\Big|\lesssim D_0\rb^{-3-\beta-i'}.
\end{align}
Then, for all $0\leq i'\leq i$ and $r\geq R$,
\begin{align}
\Big|\prb^{i'}\Big(\VR\tildePhiplussHighTI{\ell-\sfrak}-{\rb^{-2}\NPCITs{\ell}}\Big)\Big|\lesssim D_0'\rb^{-2-\beta-i'},
\end{align}
with  $D_0'=D_0+(\InizeroEnergyplussell{\tb_0}{\ell}{i+\regl}{5-\delta})^{\half}$ for some constant $\regl$.
\end{enumerate}
\end{lemma}

\begin{proof}
It is manifest from Lemmas \ref{eq:higherradialderiofgplus:general} and \ref{lem:higherderiofgplus:general} that
\begin{align}
\sum_{i=0}^{\ell-\sfrak}\norm{(r^2V)^i\Psi_{\gplus}}^2_{W_{-2}^{\reg-i}(\Sigmazero)}\lesssim{}&\InizeroEnergyplussell{\tb_0}{\ell}{\reg+1}{5-\delta}.
\end{align}
Meanwhile, by examining the proof of Lemma \ref{lem:higherderiofgplus:general}, it also implies
\begin{align}
\hspace{4ex}&\hspace{-4ex}\norm{rV\tildePhiplussHighTI{\ell-\sfrak}}^2_{W_{3-\delta-2}^{\reg-\ell}(\Sigmazero\cap\{\rb\geq 4M\})}\notag\\
\lesssim{}&
\sum_{i=0}^{\ell-\sfrak}\norm{(r^2V)^i\Psipluss}^2_{W_{-2}^{\reg-i+1}(\Sigmazero)}+
\norm{rV\tildePhiplussHigh{\ell-\sfrak}}^2_{W_{5-\delta-2}^{\reg-\ell+1}(\Sigmazero\cap\{\rb\geq 4M\})}
\notag\\
={}&\InizeroEnergyplussell{\tb_0}{\ell}{\reg+1}{5-\delta}.
\end{align}
These two estimates together yield \eqref{eq:initialnorm:psiandgplus}.
The second claim follows from the proof of Lemma \ref{lem:higherderiofgplus:general} as well.
\end{proof}

\subsection{Global Price's law in the vanishing Newman--Penrose constant case}
\label{subsect:PL:v:MT}

This final subsection is devoted to deriving the globally precise late-time asymptotics of the spin $\pm\sfrak$ components under the vanishing N--P constant condition.  In addition, we give a proof of Theorem \ref{thm:main:intro} that follows easily from the following Theorem \ref{summary:PL:v:pm:global:HiGhmodes}. 

Analogous to Section \ref{subsect:GPL:nv:MT} which however treats the vanishing N--P constant case, we first show the Price's law in the case that the spin $\pm \sfrak$ components are supported on a single $(m,\ell)$ mode.

\begin{thm}[Global Price's law  for a fixed mode of the spin $\pm \sfrak$ components in the vanishing Newman--Penrose constant case]
\label{summary:PL:v:pm:global}
Let $j\in\mathbb{N^+}$ and let $\sfrak=0,1,2$. Suppose the spin $s=\pm \sfrak$ components are supported on an $(m, \ell)$ mode, $\ell\geq \sfrak$. Let function $\hell$ and scalars $\tildePhiplussHigh{i}$ and $\tildePhiminuss{i}$ be defined as in Definitions \ref{def:hell:pm} and \ref{def:tildePhiplusandminusHigh}, respectively.  Assume on $\Sigma_{\tau_0}$ that  there are constants $\beta\in (0,\half)$ and $0\leq D_0<\infty$ such that for all $0\leq i\leq j+\ell+\sfrak$ and $r\geq 10M$,
\begin{align}
\Big|\prb^{i}\Big(\VR\tildePhiplussHigh{\ell-\sfrak}-{\rb^{-3}\lim_{\rb\to \infty}r \curlVR \tildePhiplussHigh{\ell-\sfrak}}\Big)\Big|\lesssim D_0\rb^{-3-\beta-i}.
\end{align}

Then all the estimates in Theorem \ref{summary:PL:nv:pm:global} hold true by making the following replacements:
\begin{align}
&\Lxi^j \Upsilon_s \rightarrow \Lxi^{j-1} \Upsilon_s, 
\quad \NPCPs{\ell}\rightarrow \NPCIT{\ell}, \quad \FBT\rightarrow \FBTV,
\end{align}
where $
\FBTV=(\InizeroEnergyplussell{\tb_0}{\ell}{\reg}{5-\delta})^{\half}+\abs{\NPCITs{\ell}}+ D_0$, $\InizeroEnergyplussell{\tb_0}{\ell}{\reg}{5-\delta}$ is defined as in Definition \ref{def:energyonSigmazero:Schw}, and the N--P constant $\NPCITs{\ell}$ of the time integral $\gplus$ can be calculated as in Lemma \ref{lem:higherderiofgplus:general}.
\end{thm}

\begin{proof}
By Lemma \ref{lem:initialnorm:psiandgplus}, we can apply Theorem \ref{summary:PL:nv:pm:global} to the time integral $\gplus$. The statements for the spin $+\sfrak$ component  then follows simply from $(\psipluss)_{(m,\ell)}=\Lxi \gplus$. The estimates for the spin $-\sfrak$ component can be proven by utilizing the TSI in Lemma \ref{lem:TSI:gene}. 
\end{proof}

We also obtain in the following theorem the precise late-time asymptotics for the spin $\pm \sfrak$ components when they are supported on $\ell\geq \ell_0$ modes. This is our main theorem in the vanishing N--P constant case and shall be compared to Theorem \ref{summary:PL:nv:pm:global:HM} in the non-vanishing N--P constant case.  

\begin{thm}[Global Price's law  for the spin $\pm \sfrak$ components in the vanishing Newman--Penrose constant case]
\label{summary:PL:v:pm:global:HiGhmodes}
Let $j\in\mathbb{N^+}$ and let $\sfrak=0,1,2$. Suppose the spin $s=\pm \sfrak$ components are supported on $\ell\geq \ell_0$ mode, $\ell_0\geq \sfrak$. Let function $\hell$ and scalars $\tildePhiplussHigh{i}$ and $\tildePhiminuss{i}$ be defined as in Definitions \ref{def:hell:pm} and \ref{def:tildePhiplusandminusHigh}, respectively. Assume on $\Sigma_{\tau_0}$ that  there are constants $\NPCPs{\ell_0+1}\in \mathbb{R}$, $\beta\in (0,\half)$ and $0\leq D_0<\infty$ such that for  any $r\geq 10M$,
\begin{subequations}
\begin{align}
\Big|\prb^{i}\Big(\VR(\tildePhiplussHigh{\ell_0-\sfrak})_{m,\ell_0}-{\rb^{-3}\lim_{\rb\to \infty}r \curlVR \tildePhiplussHigh{\ell-\sfrak}}\Big)\Big|\lesssim {}&D_0\rb^{-3-\beta-i}, \quad \forall 0\leq i\leq j+\ell_0+\sfrak, \abs{m}\leq \ell_0,\\
\Big|\prb^{i}\Big(\VR(\tildePhiplussHigh{\ell_0+1-\sfrak})_{m,\ell_0+1}-{\rb^{-2}\NPCPs{\ell_0+1}}\Big)\Big|\lesssim {}&D_0\rb^{-2-\beta-i}, \quad \forall 0\leq i\leq j+\ell_0+\sfrak-1, \abs{m}\leq \ell_0+1.
\end{align}
\end{subequations}

Then there are constants $\delta>0$,\footnote{In fact, any $\delta\in (0,\frac{1}{2})$ works fine.} $\reg=\reg(\ell_0,j)$, $\veps=\veps(j, \ell_0, \delta)>0$ and $C=C(\reg,\ell_0, j, \delta)$ such that for any $\tb\geq \tb_0$, 
\begin{align}\label{eq:globalPL:v:pm:HM}
&\bigg|\Lxi^{j-1}\Upsilon_s-2^{2\ell_0+2}
\prod_{i=\ell_0+\sfrak+1}^{2\ell_0+1} i^{-1} \mu^{\frac{\sfrak+s}{2}} \hellz r^{\ell_0-\sfrak}\Lxi^j(
\tb^{-\ell_0+s-1}v^{-\ell_0-s-1})\sum_{m=-\ell_0}^{\ell_0}{\NPCIT{\ell_0}} Y_{m,\ell_0}^{s}(\cos\theta)e^{im\pb} \notag\\
&\qquad\quad\,\,\,
-2^{2\ell_0+4}
\prod_{i=\ell_0+\sfrak+2}^{2\ell_0+3} i^{-1}
\mu^{\frac{\sfrak+s}{2}}\hellzp r^{\ell_0+1-\sfrak}\Lxi^{j-1}(
\tb^{-\ell_0+s-2}v^{-\ell_0-s-2})\notag\\
&\qquad\quad\,\,\,\times\sum_{m=-\ell_0-1}^{\ell_0+1}{\NPCPs{\ell_0+1}} Y_{m,\ell_0+1}^{s}(\cos\theta)e^{im\pb} \bigg|\notag\\
&\leq{}Cr^{\ell_0-\sfrak}
\tb^{-\ell_0+s-1-j-\veps}v^{-\ell_0-s-1}\FBV,
\end{align}
and, on the future event horizon, for $\sfrak\neq0$,
\begin{align}
\label{eq:VNPC:Thm:pluss:horizon}
\bigg|\Lxi^{j-1}\NPR_{+\sfrak}\Big|_{\Horizon}-\mathfrak{q}_{\sfrak,j-1}v^{-2\ell_0-4-(j-1)}
\sum_{m=-\ell_0}^{\ell_0}{\NPCIT{\ell_0}} Y_{m,\ell_0}^{+\sfrak}(\cos\theta)e^{im\pb}\bigg|
\leq{}C
v^{-2\ell_0-4-(j-1)-\veps}\FBV,
\end{align}
where
\begin{align}
  \mathfrak{q}_{\sfrak,j-1}=\frac{(-1)^{\sfrak+j}2^{2\ell_0+3}\sfrak (\ell_0-\sfrak)!(2M)^{\ell_0-\sfrak+1}\hellz(2M)}{(2\ell_0+1)!}
  \prod_{n=2\ell_0+2}^{2\ell_0+j+2} n,
\end{align}
$
\FBV=(\InizeroEnergyplussv{\reg}{\delta})^{\half}+\sum\limits_{m=-\ell_0}^{\ell_0}\abs{\NPCIT{\ell_0}}+\sum\limits_{m=-\ell_0-1}^{\ell_0+1}\abs{\NPCPs{\ell_0+1}}+ D_0$  and $\InizeroEnergyplussv{\reg}{\delta}$ is defined as in Definition \ref{def:initialenergy:highMODEs}.
\end{thm}

\begin{proof}
To prove this theorem, it suffices to prove the estimates \eqref{eq:globalPL:v:pm:HM}. This follows by simply applying Theorem \ref{summary:PL:v:pm:global} to estimate each $(m,\ell_0)$ mode and Theorem \ref{summary:PL:nv:pm:global:HM} to estimate the remainder $\ell\geq \ell_0+1$ modes.
 \end{proof}

 \begin{remark}\label{summary:PL:v:pm:global:HiGhmodes:version2}
  By expanding the expression $\Lxi(\tb^{-\ell_0+s-1}v^{-\ell_0-s-1})$, we can rewrite \eqref{eq:globalPL:v:pm:HM} as
  \begin{align}
  \label{eq:VNPC:Thm:ver2}
    &\biggl|\Lxi^{j-1}\Bigl\{\Upsilon_s-\mu^{\frac{\sfrak+s}{2}}\big[(-\ell_0+s-1)\mathfrak{d}_s^{\ell_0}\hellz r^{\ell_0-\sfrak}
    +\mathfrak{e}_s^{\ell_0+1}\hellzp r^{\ell_0+1-\sfrak}v^{-1}\big]\tau^{-\ell_0+s-2}v^{-\ell_0-s-1}
    \notag\\
    &\qquad\qquad\, -\mu^{\frac{\sfrak+s}{2}}(-\ell_0-s-1)\mathfrak{d}_s^{\ell_0}\hellz r^{\ell_0-\sfrak}\tau^{-\ell_0+s-1}v^{-\ell_0-s-2}
    \Bigr\}\biggr|\notag\\
    &\leq{}C
r^{\ell_0-\sfrak}
\tb^{-\ell_0+s-1-j-\veps}v^{-\ell_0-s-1}\FBV,
  \end{align}
  where
  \begin{subequations}
  \begin{align}
    &\mathfrak{d}_s^{\ell_0}=2^{2\ell_0+2}
    \prod_{i=\ell_0+\sfrak+1}^{2\ell_0+1} i^{-1} \sum_{m=-\ell_0}^{\ell_0}{\NPCIT{\ell_0}} Y_{m,\ell_0}^{s}(\cos\theta)e^{im\pb}, \\
    &\mathfrak{e}_s^{\ell_0+1}=2^{2\ell_0+4}
    \prod_{i=\ell_0+\sfrak+2}^{2\ell_0+3} i^{-1} \sum_{m=-\ell_0-1}^{\ell_0+1}{\NPCPs{\ell_0+1}} Y_{m,\ell_0+1}^{s}(\cos\theta)e^{im\pb}.
  \end{align}
  \end{subequations}
 \end{remark}

 \begin{remark}
   Under the assumptions in Theorem \ref{summary:PL:v:pm:global:HiGhmodes}, we remark without discussing in details that the global asymptotics for the middle component $\NPRzero$ of the Maxwell field can be obtained in a similar manner as in Corollary \ref{globalPL:nv:middle:cor}. 
\end{remark}

To end this section, we show that Theorem \ref{thm:main:intro}  immediately follows from the above Theorem \ref{summary:PL:v:pm:global:HiGhmodes} and Remark \ref{summary:PL:v:pm:global:HiGhmodes:version2}.

First, we show point \ref{point:suffidecay:intro} of Theorem \ref{thm:main:intro} .  Since the initial data are compactly supported away from infinity (or decay rapidly enough), this implies $ \NPCPs{\ell_0+1}=0$ (hence $\mathfrak{e}_s^{\ell_0+1}=0$). The estimates in point \ref{point:suffidecay:intro} of Theorem \ref{thm:main:intro} manifestly hold true with $C_{\sfrak;m,\ell_0}=\NPCIT{\ell_0}$ from the estimates \eqref{eq:VNPC:Thm:ver2} and \eqref{eq:VNPC:Thm:pluss:horizon}.

Next, we give a proof of point \ref{point:ccinfty:intro} in Theorem \ref{thm:main:intro}.
If the initial data of the spin $s$ components on $t=t_0$ hypersurface are supported in a compact region of $(2M,\infty)$, we can choose $\Sigmazero$ such that it coincides with $t=t_0$ hypersurface in the compact support of the initial data and on $\Sigmazero$, if away from the intersection of these two hypersurfaces, the initial data vanish identically. Then, one finds by \eqref{eq:csf:012} that all $\csf=0$. Since $\Hhyp =\partial_r \hhyp =\mu^{-1}$, we have from \eqref{de:Hpsi} that $\Hpsi=\mu^{-1}\hell r^{2\sfrak+1}(r^2\partial_t (\NPRpluss)_{m,\ell}+2\sfrak(r-3M)(\NPRpluss)_{m,\ell})$.
By Lemma \ref{lem:higherderiofgplus:general}, the N--P constant of the time integral $\gplus$ is
\begin{align}
\NPCITs{\ell}={}&\frac{1}{2\ell+2}\lim_{r\to \infty}(rh_{\sfrak, \ell-\sfrak,\ell-\sfrak} )\times \lim_{r\to \infty}(r^2\partial_r)^{\ell-\sfrak}\PhiplussHighTI{0}\notag\\
={}&-\frac{1}{(2\ell+1)(2\ell+2)}\lim_{r\to \infty}(rh_{\sfrak, \ell-\sfrak,\ell-\sfrak} )\times \int_{2M}^{\infty}r^{-1}(r^2\partial_r)^{\ell-\sfrak}\Hpsi\di r\notag\\
={}&-\frac{1}{(2\ell+1)(2\ell+2)}\lim_{r\to \infty}(rh_{\sfrak, \ell_0-\sfrak,\ell_0-\sfrak} )\notag\\
&\times \int_{2M}^{\infty}r^{-1}(r^2\partial_r)^{\ell-\sfrak}\Big(\mu^{-1}\hell  r^{2\sfrak+1}(r^2\partial_t (\NPRpluss)_{m,\ell}+2\sfrak(r-3M)(\NPRpluss)_{m,\ell})\Big)\di r\notag\\
={}&\frac{(-1)^{\ell-\sfrak+1}(\ell-\sfrak)!}{(2\ell+1)(2\ell+2)}\lim_{r\to \infty}(rh_{\sfrak, \ell_0-\sfrak,\ell_0-\sfrak} )\notag\\
&\times \int_{2M}^{\infty}\mu^{-1}\hell  r^{\ell+\sfrak}(r^2\partial_t (\NPRpluss)_{m,\ell}+2\sfrak(r-3M)(\NPRpluss)_{m,\ell})\di r,
\end{align}
where in the last step we have used integration by parts in $r$ and the assumption that the initial data is supported on a compact region in $(2M,\infty)$.
This completes the proof of point \ref{point:ccinfty:intro} in Theorem \ref{thm:main:intro}.

\section*{Acknowledgement}

The first author S. M. acknowledges the support by the ERC grant ERC-2016 CoG 725589 EPGR.  Both authors are grateful for the comments and suggestions by anonymous reviewers, as well as an anonymous comment on Barack--Ori's result. 


\newcommand{\arxivref}[1]{\href{http://www.arxiv.org/abs/#1}{{arXiv.org:#1}}}
\newcommand{\mnras}{Monthly Notices of the Royal Astronomical Society}
\newcommand{\prd}{Phys. Rev. D}
\newcommand{\apj}{Astrophysical J.}

\bibliographystyle{amsplain}

\providecommand{\MR}{\relax\ifhmode\unskip\space\fi MR }
\providecommand{\MRhref}[2]{%
  \href{http://www.ams.org/mathscinet-getitem?mr=#1}{#2}
}
\providecommand{\href}[2]{#2}

\end{document}